\RequirePackage{ifpdf}          
\ifpdf 
\documentclass[10pt,pdftex]{article}
\usepackage{hyperref}
\hypersetup{colorlinks, citecolor=MyDarkRed, filecolor=MyDarkBlue,
	linkcolor=MyDarkBlue, urlcolor=MyDarkBlue}
\else 
\documentclass[10pt,dvips]{article}
\usepackage{hyperref}
\hypersetup{colorlinks, citecolor=MyDarkRed, filecolor=MyDarkBlue,
	linkcolor=MyDarkRed, urlcolor=MyDarkBlue}
\usepackage{breakurl}     
\fi

\usepackage[in]{fullpage}       
\usepackage{amsmath,amssymb}	
\usepackage{pslatex}		 
\usepackage[usenames]{color}   
\definecolor{MyDarkBlue}{rgb}{0, 0.0, 0.45} 
\definecolor{MyDarkRed}{rgb}{0.45, 0.0, 0} 
\definecolor{MyDarkGreen}{rgb}{0, 0.45, 0} 
\definecolor{MyLightGray}{gray}{.90}
\definecolor{MyLightGreen}{rgb}{0.5, 0.99, 0.5}
\DefineNamedColor{named}{Peach}         {cmyk}{0,0.50,0.70,0}
\DefineNamedColor{named}{Cyan}          {cmyk}{1,0,0,0}

\usepackage{graphicx} 
\usepackage{floatflt}
\usepackage{wrapfig}
\usepackage[parfill]{parskip}    
\usepackage{amsfonts, amscd, amssymb, amsthm, amsmath}
\usepackage{mathrsfs}
\usepackage{urwchancal}
\usepackage{xcolor}
\usepackage{MnSymbol}
\usepackage{pdfsync}
\usepackage{enumerate}
\usepackage{enumitem} 
\theoremstyle{plain}
\newtheorem{thm}{Theorem}[section]
\newtheorem*{theorem-non}{Theorem}
\newtheorem{lem}[thm]{Lemma}
\newtheorem{prop}[thm]{Proposition}
\newtheorem{cor}[thm]{Corollary}
\newtheorem{conj}[thm]{Conjecture}

\theoremstyle{definition}

\newtheorem{rem}{Remark}

\usepackage{color}
\definecolor{dred}{rgb}{.65, 0, 0.15}

\let\svthefootnote\thefootnote

      \def\fg{\mathfrak{g}} \def\fh{\mathfrak{h}}          \def\fs{\mathfrak{s}}

\def\<{\langle} \def\>{\rangle}

\def\dim{\mathrm{dim}}

\usepackage{listings}           
\pagecolor{white}		

\begin{document}

\title{Secondary terms in the asymptotics of moments of
$L$-functions}
\author{Adrian Diaconu\footnote{{School of Mathematics,
      University of Minnesota, Minneapolis, MN 55455,
      Email: \!cad@umn.edu}} \,
  and Henry Twiss\footnote{{School of Mathematics,
      University of Minnesota, Minneapolis, MN 55455,
      Email: \!twiss013@umn.edu}}         
}

\date{}	                
\maketitle

\begin{abstract} 
\noindent
We propose a refined version of the existing conjectural
asymptotic formula for the moments of the family of
quadratic {D}irichlet $L$-functions over rational function
fields. Our prediction is motivated by two natural
conjectures that provide\linebreak sufficient information
to determine the analytic properties (meromorphic
continuation, location of poles, and the residue at each pole)
of a certain generating function of moments of quadratic
$L$-functions. The number field analogue of
our asymptotic formula can be obtained by a similar
procedure, the only difference being the contributions
coming from the archimedean and even places, which
require a separate analysis. To avoid this additional technical
issue, we present, for simplicity, the asymptotic formula
only in the rational function field setting. This has also the
advantage of being much easier to test.
\end{abstract}

\tableofcontents
\vskip-15pt
\let\thefootnote\relax
\let\thefootnote\relax\footnote{A. Diaconu was partially supported by a Simons Fellowship in Mathematics. H. Twiss was partially supported by a Barry Goldwater Scholarship.}

\section{Introduction} 
Over the past two decades, a great deal of work in analytic number theory focused on understanding various aspects of\linebreak moments of families of automorphic $L$-functions. In particular, the problem of obtaining asymptotics for the moments of central values of $L$-functions in families received considerable attention, beginning with the work of Hardy and Lit-\linebreak tlewood \cite{Ha-Li}, who obtained the asymptotic formula for the Riemann zeta-function on the critical line, 
\[
\int_{0}^{T}|\zeta(\tfrac{1}{2} + it)|^{2}\, dt \sim T\log T    \quad \text{(as $T\to \infty$)}.     
\] 
About eight years later, Ingham \cite{Ing} obtained the asymptotic formula for the fourth moment 
\[
\int_{0}^{T}|\zeta(\tfrac{1}{2} + it)|^{4}\, dt \sim \frac{1}{2\pi^2}T \log^4 T  \quad \text{(as $T\to \infty$)}.
\] 
No other asymptotic formula for higher moments of the Riemann zeta-function is known, but it is conjectured that\linebreak 
\[ 
\int_{0}^{T}|\zeta(\tfrac{1}{2} + it)|^{2k}\, dt \sim C_{k}T\log^{k^2} T \quad \text{(for $k\ge 3$)}.
\] 
The precise value of the constant $C_{k}$ was conjectured by Keating and Snaith \cite{Ke-Sn} using random matrix theory; see also \cite{CFKRS} for a more accurate conjecture, where the full main terms in the asymptotic formula are predicted.

Another classical example of interest is furnished by the moments of the 
family of quadratic {D}irichlet $L$-functions 
\[
M_{r}(D) \, : = \sum_{0 < d < D} L(\tfrac{1}{2}, \chi_{d})^{r}
\] 
the sum being over real {\it primitive} Dirichlet characters. In 1981, Jutila \cite{Jut} computed the first and second moment, and\linebreak Soundararajan \cite{Sound}, \cite{DGH}, and Young \cite{Young} obtained asymptotics for the third moment with different error terms. \!Under the generalized Riemann hypothesis, \!Shen \cite{Sh} obtained an asymptotic formula for the fourth moment (using the method of Soundararajan and Young \cite{SY}). 
It was conjectured \cite{CFKRS} that 
\begin{equation} \label{eq: incomplete-asympt}
M_{r}(D) \sim DQ(\log D) \quad \text{(as $D\to \infty$)} 
\end{equation}
for some polynomial $Q$ of degree $r (r + 1)\slash 2.$ An explicit description of this polynomial is given in \cite{GHRR}.

An important feature of the family of quadratic Dirichlet
$L$-functions is that its moments admit a perfect function field
analogue. \!In this context, asymptotics are known only for the first
four moments by the work of Florea \cite{Florea1, Florea2, Florea3},
and\linebreak the corresponding conjectural asymptotic formula
for all moments was given in \cite{AK}; see also \cite{RW}.

In \cite{DGH}, the first-named author, Goldfeld and Hoffstein conjectured the existence of additional terms in \eqref{eq: incomplete-asympt}, when $r \ge 3;$\linebreak if $r = 3$ there is one additional term of size $D^{\frac{3}{4}},$ and if $r \ge 4,$ we have the following:
\vskip5pt
\begin{conj} --- Let $N \ge 1$ be an integer, and take 
$(N + 1)^{-1} < \Theta < N^{-1}\!.$ Then, \!as $D\to \infty,$ \!we have
\begin{equation} \label{eq: Moment-conjecture-over-rationals}
M_{r}(D) = \sum_{n = 1}^{N} D^{\scriptscriptstyle \left(\frac{1}{2} + \frac{1}{2n}\right)}Q_{n, r}(\log D)
+ O\big(D^{\scriptscriptstyle  (1 + \Theta)\slash 2}\big)
\end{equation}
for some polynomials $Q_{n, r}(x).$ 
\end{conj} 

\vskip3pt
Evidence in favor of this conjecture, both theoretical and numerical, was provided by Zhang \cite{Zha1} and the survey \cite{Zha2}, and by Alderson and Rubinstein \cite{AR}. When $r = 3,$ the existence of the additional term was proved very recently in the function field case by the first-named author \cite{Dia}, and by the first-named author and Whitehead \cite{Di-Wh}, for a smoothed version of $M_{3}(D).$ 

Our aim in this paper is to develop a method to determine the polynomials $Q_{n, r}(x)$ occurring in the asymptotic formula \eqref{eq: Moment-conjecture-over-rationals}. However, when dealing with this problem over the rationals (or its extension to number fields), one is invariably facing the unpleasant task of computing the contributions to $Q_{n, r}(x)$ corresponding to the archimedean and even places. To avoid this technical issue, we shall discuss in this paper only the (rational) function field analogue of \eqref{eq: Moment-conjecture-over-rationals}.

To state the function field version of \eqref{eq: Moment-conjecture-over-rationals}, let $\mathbb{F}_{\! q}$ be a finite field of odd characteristic; for simplicity, we assume that the cardinality $q$ of 
$\mathbb{F}_{\! q}$ is congruent to $1$ modulo $4.$ For $d \in \mathbb{F}_{\! q}[x]$ monic and square-free, let $L(s, \chi_{d})$ denote the {D}irichlet {$L$}-function associated to the quadratic symbol 
$\chi_{d}(m) = (d \slash m).$

\vskip5pt
\begin{conj} \label{Main Conjecture: Moment-asymptotics-central-point-ff} --- Let $D, N \ge 1$ and $r \ge 4$ be integers. Then, for any $(N + 1)^{\scriptscriptstyle -1} \! < \Theta  < N^{\scriptscriptstyle -1}\!,$ we have 
\[ 
\sum_{\substack{d - \mathrm{monic \; \& \; sq. \; free} \\ \deg \, d\, = \, D}}\;\,
L\big(\tfrac{1}{2}, \chi_{d}\!\big)^{\scalebox{1.25}{$\scriptscriptstyle r$}} 
= \sum_{n \, \le \, N} Q_{r,\, n}(D, q)q^{\left(\!\scalebox{.9}{$\scriptscriptstyle \frac{1}{2}$} 
+ \scalebox{.9}{$\scriptscriptstyle \frac{1}{2n}$}\!\right)\scalebox{1.2}{$\scriptscriptstyle D$}} + \, 
O_{\scalebox{.85}{$\scriptscriptstyle \Theta$}\scriptscriptstyle, \,  q, \, r}
\!\left(q^{\scalebox{.9}{$\scriptscriptstyle D$} \scalebox{.9}{$\scriptscriptstyle (1 + \Theta)\slash 2$}}\right)
\] 
for computable $Q_{r,\, n}(D, q).$
\end{conj}

For notational convenience, we shall suppress the dependence of $Q_{r,\,  n}(D, q)$ upon $r.$ To determine 
these coefficients, we\linebreak will investigate the {\it conjectural} analytic properties (meromorphic continuation, location of poles and the residue at each pole) of the generating function 
\[
\mathscr{W}\!(s_{\scriptscriptstyle 1}, \ldots, s_{r}, \xi)  = \, \sum_{D \ge 0}
\; \Bigg(\, \sum_{\substack{d - \mathrm{monic \; \& \; sq. \; free} \\ \deg \, d\, = \, D}} \;\,
\prod_{k = 1}^{r}L(s_{\scriptscriptstyle k}, \chi_{d}) 
\Bigg)\, \xi^{\scriptscriptstyle D}
\] 
for fixed $s_{\scriptscriptstyle 1}, \ldots, s_{r}\in \mathbb{C}$ with 
$\Re(s_{\scriptscriptstyle k}) = \frac{1}{2}.$ Rather than approaching this function directly (as it does not seem to provide suf-\linebreak ficient information), we proceed as in \cite{Dia}, and write 
\begin{equation} \label{eq: intro-MDS-congruence-condition}
\mathscr{W}\!(s_{\scriptscriptstyle 1}, \ldots, s_{r}, q^{- s_{r + {\scriptscriptstyle 1}}}) \; = 
\sum_{h - \mathrm{monic}}\; \underbrace{\mu(h)}_{\textrm{\textcolor{MyDarkBlue}
		{${\scriptstyle{\text{M\"obius function on $\mathbb{F}_{\! q}[x]$}}}$}}} \hskip-15pt 
\!Z(s_{\scriptscriptstyle 1}, \ldots, s_{r + {\scriptscriptstyle 1}}, 1\!; h) 
\qquad \text{(for $\Re(s_{r + {\scriptscriptstyle 1}}) > 1$)}
\end{equation} 
where $Z(s_{\scriptscriptstyle 1}, \ldots, s_{r + {\scriptscriptstyle 1}}, 1\!; h)$ is a finite sum whose terms involve the (suitably normalized) twisted {W}eyl group multiple {D}irichlet series introduced in \cite{DV}, see \ref{introref-generating-function-moments}. These multiple {D}irichlet series satisfy a group of functional equations iso-\linebreak morphic to the {W}eyl group $W$ of the Kac-Moody Lie algebra $\fg(A)$ associated to the generalized Cartan matrix $A = (A_{ij})_{1 \le i, j \le r + 1}$ whose entries are given by 
\[
A_{ij} = 
\begin{cases} 
2 &\mbox{if $i = \!j$} \\
-1 &\mbox{if $i \ne \!j$ and either $i = r + 1$ or $j = r + 1$} \\
0 & \mbox{otherwise.}
\end{cases} 
\] 
We note that, unlike the finite-dimensional case, when $r \ge 4$ ($W$ is infinite in this case), these series are not determined by the group of functional equations alone. In fact, the complexity grows with $r,$ which can be seen by inspecting some\linebreak low-index coefficients of the multiple {D}irichlet series that can be computed using the explicit comparison between the Grothendieck-Lefschetz and Arthur-Selberg trace formulas established by Weissauer \cite{Weiss1} and \cite{Weiss2}; see also \cite{Peter}.\linebreak 

\vskip-5pt
We conjecture that these twisted {W}eyl group multiple {D}irichlet series admit meromorphic continuation to the interior of a complexified convex cone $X_{\scalebox{1.1}{$\scriptscriptstyle 0$}}^{\ast}$ -- a transformation of the interior of the complexified Tits cone of $W.$ More precisely, we conjecture that these functions are holomorphic on $X_{\scalebox{1.1}{$\scriptscriptstyle 0$}}^{\ast}\!,$ except for possible simple poles corresponding (via a simple change of variables) to the positive real roots of $\fg(A).$ The twisted {W}eyl group multiple {D}irichlet series occurring in the\linebreak expression of 
$
Z(s_{\scriptscriptstyle 1}, \ldots, s_{r + {\scriptscriptstyle 1}}, 1\!; h)
$ 
are twists of the multiple {D}irichlet series constructed in \cite{DV}, normalized by a certain $W$-invariant Euler product with the local factors at the primes dividing $h$ removed, and whose local factors are supported on squares. The effect of this normalization (which we introduced in order to remove potential singularities cor-\linebreak responding to the imaginary roots of $\fg(A)$) is completely annihilated in \eqref{eq: intro-MDS-congruence-condition}. We also conjecture that the series in \eqref{eq: intro-MDS-congruence-condition}\linebreak is absolutely convergent for 
$
\Re(s_{\scriptscriptstyle i}) \ge \frac{1}{2},
$ 
$i = 1, \ldots, r,$ and $\Re(s_{r + {\scriptscriptstyle 1}}) > \frac{1}{2},$ away from the poles of 
$ 
Z(s_{\scriptscriptstyle 1}, \ldots, s_{r + {\scriptscriptstyle 1}}, 1\!; h),
$ 
which combined with the previous conjecture, would provide the information we need to obtain the relevant analytic properties of the function 
$
\mathscr{W}\!(s_{\scriptscriptstyle 1}, \ldots, s_{r}, \xi).
$ 
In particular, all singularities of this function (in the variable $\xi$) should be simple poles at points corresponding to positive real roots of $\fg(A);$ if 
$
\alpha =  \!\sum k_{\scriptscriptstyle i}\alpha_{\scriptscriptstyle i}
$ 
is a positive real root, for a choice of positive simple roots $\alpha_{\scriptscriptstyle i}$ ($i = 1, \ldots, r+1$), the residue at the corresponding (simple) pole can be computed explicitly, and we shall do so in Section \ref{Poles-residues-ref-section}.

Our first conjecture is supported, for example, by an extension of the
so-called Eisenstein Conjecture ({\bf EC}) \cite{BBF1, BBF2, BBF3, CO,
  McN, P-P1} to Kac-Moody groups. \!This extension of {\bf EC}
predicts that twisted {W}eyl group multiple {D}irichlet
series\linebreak occur as Whittaker coefficients of a minimal
parabolic Eisenstein series on a metaplectic cover of a (root datum
related) Kac-Moody group; in this respect, a preliminary result has
been recently obtained by Patnaik and Pusk\'as \cite{P-P2}, who proved
a Casselman-Shalika type formula for Whittaker functions on
metaplectic covers of Kac-Moody groups over non-archimedean local
fields. It is conceivable that the constant term of this minimal
parabolic Eisenstein series can be\linebreak expressed via a
generalized Gindikin-Karpelevich formula in terms of a ratio of an
{\it infinite} product of zeta and $L$-functions. Presumably, this
infinite product is structurally similar to the denominator of the
character of an irreducible highest-weight $\fg(A)$-module, which
converges on the interior of the complexified Tits cone of $W.$ Thus
the constant term would be meromorphic on
$X_{\scalebox{1.1}{$\scriptscriptstyle 0$}}^{\ast},$ and by a
well-known principle, the minimal parabolic Eisenstein series should
have meromorphic continuation to the same domain. \!However,
developing the Langlands-Shahidi method in the con-\linebreak text of
Kac-Moody groups is quite problematic at the moment, a serious
obstacle in doing so being the lack of an adequate integration theory
over the relevant unipotent radicals allowing us to transfer the
``meromorphy property" from an Eisenstein series to its Whittaker
coefficients; see \cite{GMP} and the reference therein for specific
information in the\linebreak special case of Eisenstein series on loop
groups. Nonetheless, in some special cases, it is possible to obtain
the mero-\linebreak morphic continuation of the corresponding
{W}eyl group multiple {D}irichlet series more directly
(see, for example, \cite{BD}, \cite{White1}, and especially the
forthcoming manuscripts \cite{DPP1}, \cite{DPP2}, where the conjecture
is proved in the important case of {\it untwisted} {W}eyl group
multiple {D}irichlet series of type
$
D_{\scriptscriptstyle 4}^{\scriptscriptstyle (1)}
$
over rational function fields).

Our conjecture on the absolute convergence of \eqref{eq:
  intro-MDS-congruence-condition} can be motivated by an analogue of
the Lindel\"of hypothesis for the twisted {W}eyl group multiple
{D}irichlet series in the twisting parameters combined with a
Ramanujan type bound for the coefficients of the $p$-parts of the
{W}eyl group multiple {D}irichlet series (see Assumption \eqref{eq:
  Estimate-H1}).

By assuming the two conjectures (and, for consistency, also the
Ramanujan bound for the coefficients of the $p$-parts) to\linebreak be
true, we show that Conjecture
\ref{Main Conjecture: Moment-asymptotics-central-point-ff} holds with
$Q_{n}(D, q)$ given by a limit of a sum of contributions coming
from the singularities of the generating function
$
\mathscr{W}\!(s_{\scriptscriptstyle 1}, \ldots, s_{r}, \xi)
$ 
corresponding to the set of positive real roots of $\fg(A)$ with 
$k_{r + {\scriptscriptstyle 1}}=n.$ 
For example, when $n = 1,$ the set of roots contributing to
$Q_{\scriptscriptstyle 1}(D, q)$ is
\[
\left\{\sum_{i = 1}^{r} k_{\scriptscriptstyle i}\alpha_{\scriptscriptstyle i} \, +\, 
  \alpha_{r + {\scriptscriptstyle 1}}:
 \text{$k_{\scriptscriptstyle i} =0$ or $1$} \right\}
\]
(see \cite{AK} for the computation of $Q_{\scriptscriptstyle 1}(D, q)$),
and when $n = 2,$ the set of corresponding roots is
\[
\Bigg\{\sum_{\substack{j = 1 \\ j \neq
    j_{\scalebox{.71}{$\scriptscriptstyle 1$}}, \,
    j_{\scalebox{.71}{$\scriptscriptstyle 2$}}, \,
    j_{\scalebox{.71}{$\scriptscriptstyle 3$}}}}^{r}
k_{\!\scriptscriptstyle j}\alpha_{\!\scriptscriptstyle j} 
+ \alpha_{\!\scriptscriptstyle j_{\scalebox{.62}{$\scriptscriptstyle 1$}}}  \!+
\alpha_{\!\scriptscriptstyle j_{\scalebox{.62}{$\scriptscriptstyle 2$}}}  
\! +  \alpha_{\!\scriptscriptstyle j_{\scalebox{.62}{$\scriptscriptstyle 3$}}}  
\! +  2\alpha_{r +  {\scriptscriptstyle 1}} :
\text{$1 \le j_{\scriptscriptstyle 1} < j_{\scriptscriptstyle 2} <
  j_{\scriptscriptstyle 3} \le r,$ and $k_{\!\scriptscriptstyle j} \in \{0, 2\}$
for 
$
j \neq \! j_{\scriptscriptstyle 1}, j_{\scriptscriptstyle 2},
j_{\scriptscriptstyle 3}
$}
\Bigg\}
\]
and $Q_{\scriptscriptstyle 2}(D, q)$ is a polynomial of degree
$(r - 3) (r + 10)\slash 2$ in $D$ with leading coefficient given by
\begin{equation*}
\begin{split}
2^{{\scriptscriptstyle 19} - {\scriptscriptstyle 7}r}
\frac{1!\, 2!\, \cdots \, (r - 4)!}
{7!\,  9!\,  \cdots \, (2r - 1)!}
\Bigg\{&\big(1 \, +  \,
q^{\,\scriptscriptstyle 1\slash 4} \, + \, 10\, q^{\,\scriptscriptstyle 1\slash 2} 
\, + \, 7 q^{\,\scriptscriptstyle 3\slash 4} \, + \, 20\, q \, + \, 7 q^{\,\scriptscriptstyle 5\slash 4} \, + \, 10\, q^{\,\scriptscriptstyle 3\slash 2} \, + \, q^{\,\scriptscriptstyle 7\slash 4} \, + \, q^{\scriptscriptstyle 2}\big)
\zeta\big(\tfrac{1}{2}\big)^{\!\scriptscriptstyle 7}
\!\prod_{p} P_{r}\left(\frac{1}{\sqrt{|p|}}\right) \\
& + (-1)^{\scriptscriptstyle D} \big(1 \, -  \,
q^{\,\scriptscriptstyle 1\slash 4} \, + \, 10\,
q^{\,\scriptscriptstyle 1\slash 2} \, - \, 7 q^{\,\scriptscriptstyle
  3\slash 4} \, + \, 20\, q \, - \, 7 q^{\,\scriptscriptstyle 5\slash 4} \,
+ \, 10\, q^{\,\scriptscriptstyle 3\slash 2} \, - \,
q^{\,\scriptscriptstyle 7\slash 4}\, + \, 
q^{\scriptscriptstyle 2}\big)\zeta\big(\tfrac{1}{2}\big)^{\!\scriptscriptstyle 7}
\!\prod_{p} P_{r}\left(\frac{1}{\sqrt{|p|}}\right) \\
& + i^{\scriptscriptstyle D}\big(1 \, -  \, i
q^{\,\scriptscriptstyle 1\slash 4} \, - \, 4\, q^{\,\scriptscriptstyle
  1\slash 2} \, + \, 7i q^{\,\scriptscriptstyle 3\slash 4}
\, + \, 6\, q \, - \, 7i q^{\,\scriptscriptstyle 5\slash 4} 
\, - \, 4\, q^{\,\scriptscriptstyle 3\slash 2}
\, + \, i q^{\,\scriptscriptstyle 7\slash 4}
\, + \, q^{\scriptscriptstyle 2}\big)
L\big(\tfrac{1}{2}\big)^{\!\scriptscriptstyle 7}
\!\prod_{p} P_{r}\left(\frac{(-1)^{\scriptscriptstyle \deg\,  p}}{\sqrt{|p|}}\right)\\
& + (-i)^{\scriptscriptstyle D}\big(1 \, +  \, i
q^{\,\scriptscriptstyle 1\slash 4}
\, - \, 4\, q^{\,\scriptscriptstyle 1\slash 2} 
\, - \, 7i q^{\,\scriptscriptstyle 3\slash 4} \, + \, 6\, q
\, + \, 7i q^{\,\scriptscriptstyle 5\slash 4} 
\, - \, 4\, q^{\,\scriptscriptstyle 3\slash 2} \, - \, i
q^{\,\scriptscriptstyle 7\slash 4} \, + \, 
q^{\scriptscriptstyle 2}\big)
L\big(\tfrac{1}{2}\big)^{\!\scriptscriptstyle 7}
\!\prod_{p} P_{r}\left(\frac{(-1)^{\scriptscriptstyle \deg\, p}}{\sqrt{|p|}}\right)
\!\Bigg\}.
\end{split}
\end{equation*} 
Here
$
\zeta\big(\tfrac{1}{2}\big) = (1 - \sqrt{q})^{\scriptscriptstyle -1}\!,
$
$
L\big(\tfrac{1}{2}\big) = (1 + \sqrt{q})^{\scriptscriptstyle -1}\!, 
$
and
\begin{equation*} 
\begin{split} 
P_{r}(t) = &\, \left(1 -  t\right)^{\!(r^{\scalebox{.85}{$\scriptscriptstyle 2$}} + 7r - 14)\slash 2}    
\!\left(1 +  t \right)^{\!(r^{\scalebox{.85}{$\scriptscriptstyle 2$}}  + 7r - 28)\slash 2}\\
&\cdot \, \left[\left(t^{} +  t^{2} \right)
\left(t^{} + 6 t^{2} + t^{3}\right) \, + \, 
\tfrac{1}{2}\left(1 +  t^{} \right)^{4 - r}  \! + \, 
\tfrac{1}{2}\left(1  -  t^{}\right)^{- r} 
\left(1 + 10\, t^{} + 20\, t^{2} + 10\, t^{3} + t^{4} \right) \right].\end{split}
\end{equation*}
In general, the explicit form of the leading coefficient of
$Q_{n}(D, q)$ involves $2n$ contributions, each
corresponding to a $2n$-th root of $1.$ The computation of
these contributions is similar, albeit increasingly involved as
$n$ grows. By compari-\linebreak son, the coefficient
$Q_{\scriptscriptstyle 2}(x)$ in
\eqref{eq: Moment-conjecture-over-rationals} is a
polynomial of degree
$
(r - 3) (r + 10)\slash 2,
$
and we have
\[
Q_{\scriptscriptstyle 2}(x) \, \sim \, 
2^{{\scriptscriptstyle 19} - {\scriptscriptstyle 7}r}
\frac{0!\, 1!\, 2!\, \cdots \, (r - 4)!}
{7!\,  9!\, 11!\,  \cdots \, (2r - 1)!}
c_{\scriptscriptstyle \infty}c_{\scriptscriptstyle 2}
\zeta^{\scriptscriptstyle (2)}\big(\tfrac{1}{2}\big)^{\!\scriptscriptstyle 7}
\!\prod_{p \ne 2} P_{r}\Big(p^{\scriptscriptstyle -\frac{1}{2}}\Big) 
\,\cdot \, x^{(r - 3) (r + 10)\slash 2} \quad \text{(as $x\to \infty$)}
\]
where $c_{\scriptscriptstyle 2}$ and $c_{\scriptscriptstyle \infty}$
are contributions at $2$ and the archimedean place,
respectively, and
$
\zeta^{\scriptscriptstyle (2)}\big(\tfrac{1}{2}\big)
$
is the Riemann zeta-function with the local factor at $2$ removed.

\subsection{Structure of the paper}
In Section \ref{Prelim}, we recall the necessary background from
the theory of Kac-Moody Lie algebras and quadratic {D}irichlet
$L$-\linebreak functions. In Section \ref{General-MDS},
we introduce a class of formal power series in several variables
satisfying certain properties. In\linebreak Section
\ref{WMDS-assoc-moments}, we introduce the relevant multiple
{D}irichlet series associated to moments of quadratic {D}irichlet
$L$-functions, and discuss their general properties. In particular,
the group of functional equations satisfied by these multiple
{D}irichlet series is a consequence of the fact that the local parts
(or $p$-parts) of these series satisfy the properties of the power
series introduced in Section \ref{General-MDS}. We then proceed
by formulating two conjectures, giving in particular,
the meromorphic continuation of the generating function of
moments of quadratic $L$-series. In Section
\ref{Poles-residues-ref-section}, we study the poles and residues
of the functions introduced in $\S$\ref{WMDS-assoc-moments},
and in Section \ref{Asympt-moments-introduction}, we combine
the bound \eqref{eq: Estimate-H1} and Conjectures
\ref{Conjecture-Meromorphic continuation}, \ref{Conjecture-Meromorphic
  continuation-Z0} with the infor-\linebreak mation obtained in
Section \ref{Poles-residues-ref-section} to deduce asymptotics for moments
of quadratic $L$-functions. Finally, in Section
\ref{computations-introduction}, we study in detail the first two
terms of the asymptotics obtained in Section
\ref{Asympt-moments-introduction}.

\vskip5pt
{\bf Acknowledgements.} The authors would like to thank Vicen\c tiu Pa\c sol and Alex Popa for some very helpful discussions
and comments.

\section{Preliminaries} \label{Prelim}

\subsection{ Kac-Moody Lie algebras}  \label{KM-alg} 
We first recall some basic facts about root systems of Kac-Moody Lie
algebras; for a detailed exposition of the theory of Kac-Moody
algebras, the reader may consult the standard reference \cite{Kac}.

Let $A = (A_{ij})_{1 \le i, j \le n}$ be a {\it generalized Cartan matrix,} that is, the entries $A_{ij}$ satisfy the following conditions: 
\begin{enumerate}[label=(\roman*)]
	\item $A_{ii} = 2$ for $i = 1, \ldots, n;$ 
	
	\item $A_{ij}$ are non-positive integers for $i\ne j;$
	
	\item $A_{ij} = 0$ implies $A_{\! ji} = 0.$
\end{enumerate} 
For simplicity, let us assume throughout that $A \! = \!DS,$ where $D$ is an $n \times n$ invertible diagonal matrix, and $S$ is a symmetric matrix. \!Set $I =\{1, \ldots, n\},$ \!and let $(\fh,\, \Pi,\, \Pi^{\scriptscriptstyle \vee}\!)$ be a {\it realization} of $A.$ Thus $\fh$ is a complex vector space of dimension $2n - \mathrm{rank}\; A,$ and 
$
\Pi = \{\alpha_{\scriptscriptstyle i} : i \in I\} \subset \fh^{*}\!,
$ 
$
\Pi^{\scriptscriptstyle \vee} = \{\alpha_{\scriptscriptstyle i}^{\scriptscriptstyle \vee} : i \in I\} 
\subset \fh
$ 
are linearly independent sets satisfying the condition 
\[
\alpha_{\scriptscriptstyle j}(\alpha_{\scriptscriptstyle i}^{\scriptscriptstyle \vee}) = A_{ij} \qquad 
\text{(for all $i, j \in I$).}
\]

The Kac-Moody algebra associated to the matrix $A$ (see \cite[p. \!159]{Kac}) is then the Lie algebra $\fg(A)$ on generators $e_{\scriptscriptstyle i}, f_{\scriptscriptstyle i}$ ($i \in I$), $\fh,$ and the defining relations: 

\begin{itemize}

\item $[e_{\scriptscriptstyle i}, f_{\! \scriptscriptstyle j}] 
= \delta_{\scriptscriptstyle ij}\alpha_{\scriptscriptstyle i}^{\scriptscriptstyle \vee},$ where 
$\delta_{\scriptscriptstyle ij}$ is the Kronecker delta; 

\item $[h, e_{\scriptscriptstyle i}] = \alpha_{\scriptscriptstyle i}(h)e_{\scriptscriptstyle i}$ \, and \, $[h, f_{\scriptscriptstyle i}] = -\, \alpha_{\scriptscriptstyle i}(h)f_{\scriptscriptstyle i}$ \, (for $h\in \fh$);

\item $[h, h'] = 0$ \, (for $h, h' \in \fh$); 

\item $\mathrm{ad}(e_{\scriptscriptstyle i})^{\scriptscriptstyle 1 - A_{ij}}(e_{\! \scriptscriptstyle j}) = 0$ \, and \, $\mathrm{ad}(f_{\scriptscriptstyle i})^{\scriptscriptstyle 1 - A_{ij}}(f_{\! \scriptscriptstyle j}) = 0$ \, (for $i, j \in I$ with $i\ne j$).

\end{itemize} 
One of the fundamental results of the theory of Kac-Moody algebras is the direct sum decomposition 
\begin{equation*} 
\fg(A) = \fh \, \oplus \, \Big(\bigoplus_{\alpha \in \Delta} \, \fg_{\alpha}\!\Big)
\end{equation*} 
where 
$
\Delta \subset \sum_{i \in I} \mathbb{Z}\alpha_{\scriptscriptstyle i}
\setminus \{ \bf{0} \}
$ 
is the set of all {\it roots} of $\fg(A),$ and for each $\alpha \in \Delta,$ 
$
\fg_{\alpha} : = \{x \in \fg(A) : \text{$[h, x] = \alpha(h)x$ (for all $h\in \fh$)}\}
$ 
is the {\it root space} attached to $\alpha.$ Letting 
$
\Delta_{+} = \Delta \, \cap \sum_{i \in I} \mathbb{N}\alpha_{\scriptscriptstyle i}
$ 
(resp. \!$\Delta_{-}$) denote the set of all {\it positive} (resp. \!all {\it negative}) roots, we have 
\begin{equation} \label{eq: Roots1} 
\Delta_{-} = - \Delta_{+} \;\;\; \mathrm{and} \;\;\; \Delta = \Delta_{+} \cup \Delta_{-}.
\end{equation} 
Moreover, if $m_{\alpha} : = \dim\; \fg_{\alpha} \ge 1$ is the multiplicity of $\alpha,$ we have 
$m_{\alpha} = m_{-\alpha}$ for all $\alpha \in \Delta.$ For each $i\in I,$ 
$\alpha_{\scriptscriptstyle i} \in \Delta_{+},$ $m_{\alpha_{\scriptscriptstyle i}} \! = 1$ and $\mathbb{Z}\alpha_{\scriptscriptstyle i}\cap\Delta=\{\alpha_{\scriptscriptstyle i}, 
- \alpha_{\scriptscriptstyle i}\}.$ The elements $\alpha_{\scriptscriptstyle i}$ ($i\in I$) are called the {\it simple} roots of $\fg(A).$

We also have the following important fact, see \cite[Lemma~1.3]{Kac}: {\it If 
$
\beta \in \Delta_{+}\setminus \{\alpha_{\scriptscriptstyle i}\}
$ then 
$
(\beta + \mathbb{Z}\alpha_{\scriptscriptstyle i}) \cap \Delta \subset \Delta_{+}. 
$}

For each $i \in I,$ we define the {\it fundamental reflection} $w_{\scriptscriptstyle i}$ of $\fh^{*}$ by 
$$
w_{\scriptscriptstyle i}(\lambda) = \lambda \, - \, \lambda(\alpha_{\scriptscriptstyle i}^{\scriptscriptstyle \vee})\alpha_{\scriptscriptstyle i} \qquad \text{(for $\lambda \in \fh^{*}$).}
$$ 
The subgroup $W : = \langle w_{\scriptscriptstyle i} : i \in I \rangle$ of $\mathrm{GL}(\fh^{*})$ is the {\it {W}eyl group} of $\fg(A).$ The set of roots $\Delta$ is stabilized by $W,$ and $m_{w(\alpha)} = m_{\alpha}$ for $\alpha \in \Delta$ and $w \in W.$ The elements of the subset $\Delta^{\mathrm{re}} : = W\Pi$ of $\Delta$ are called {\it real roots}, and the elements of $\Delta^{\mathrm{im}} : = \Delta \setminus \Delta^{\mathrm{re}}$ are called {\it imaginary roots}. Both sets $\Delta^{\mathrm{re}}$ and $\Delta^{\mathrm{im}}$ split according to \eqref{eq: Roots1} as 
$ 
\Delta^{\mathrm{re}} = \Delta^{\mathrm{re}}_{+} \cup \Delta^{\mathrm{re}}_{-},
$ 
$
\Delta^{\mathrm{im}} = \Delta^{\mathrm{im}}_{+} \cup \Delta^{\mathrm{im}}_{-},
$ 
and $\Delta^{\mathrm{re}}_{-} = - \Delta^{\mathrm{re}}_{+},$ $\Delta^{\mathrm{im}}_{-} = - \Delta^{\mathrm{im}}_{+}.$

Recall that the Lie algebra $\fg(A)$ is assumed to be symmetrizable. \!Thus, by \cite[Theorem~2.2]{Kac}, $\fg(A)$ admits a non-degenerate symmetric bilinear $\mathbb{C}$-valued form $(\cdot \, \vert \, \cdot)$ -- playing the role of a ``Killing-Cartan form" when $\fg(A)$ is infinite-dimensional. Using this bilinear form, we have the following characterization of the imaginary roots, see \cite[Prop.~5.2, c]{Kac}: $\alpha \in \Delta$ is an imaginary root if and only if $(\alpha \, \vert \, \alpha) \le 0.$ Thus $(\alpha \, \vert \, \alpha) > 0$ if 
$\alpha \in \Delta^{\mathrm{re}}.$

The real roots also enjoy the following important property, see \cite[Prop.~5.1, c]{Kac}. For 
$\alpha \in  \Delta^{\mathrm{re}}$ and $\beta \in \Delta,$ let 
$S(\alpha, \beta) = \{\beta + k\alpha : k \in \mathbb{Z}\} \cap \Delta$ denote the $\alpha$-string through 
$\beta.$ Then there exist $u, v \in \mathbb{N}$ such that 
\begin{equation} \label{eq: Root-strings} 
S(\alpha, \beta) = \{\beta - u\alpha, \ldots,\, \beta - \alpha,\, \beta,\, \beta + \alpha, 
\ldots,\, \beta + v\alpha \}
\end{equation} 
and $u - v = \beta(\alpha^{\scriptscriptstyle \vee}).$ Moreover, the reflection 
$
w_{\scriptscriptstyle \alpha}(\lambda) = \lambda  -  \lambda(\alpha^{\scriptscriptstyle \vee})\alpha, 
$ 
$\lambda \in \fh^{*}\!,$ reverses the sequence of elements in $S(\alpha, \beta).$

For 
$
\alpha =  \!\sum_{i \in I} k_{\scriptscriptstyle i}\alpha_{\scriptscriptstyle i}
$ 
with $k_{\scriptscriptstyle i} \in \mathbb{Z},$ the integer 
$d(\alpha) : =  \sum_{i \in I} k_{\scriptscriptstyle i}$ is called the \addtocounter{footnote}{-2}\let\thefootnote\svthefootnote{\it height}\footnote{In \cite{Kac} the height of an element $\alpha$ in the root lattice is denoted by $\mathrm{ht} \; \alpha$ instead of $d(\alpha).$} \!of $\alpha.$ The elements of $\Delta$ 
of height one are $\alpha_{\scriptscriptstyle i}$ ($i \in I$). The following proposition can be used to construct the set 
$\Delta_{+}$ inductively.

\vskip10pt
\begin{prop}[{\bf Moody} \cite{Moo}] \label{properties-roots} --- Let 
$
\alpha =  \! \sum_{i \in I} k_{\scriptscriptstyle i}\alpha_{\scriptscriptstyle i}
$ 
with $k_{\scriptscriptstyle i} \in \mathbb{N}$ be of height $d(\alpha) > 1.$ Then $\alpha \in \Delta_{+}$ if and only if 
either there exists a fundamental reflection $w_{\scriptscriptstyle i}$ ($i \in I$) such that 
$w_{\scriptscriptstyle i}(\alpha)$ is a root and $d(w_{\scriptscriptstyle i}(\alpha)) < d(\alpha),$ or for all $i,$ 
$d(w_{\scriptscriptstyle i}(\alpha)) \ge d(\alpha)$ and there exists $i \in I$ such that $\alpha - \alpha_{\scriptscriptstyle i}$ 
is a root. Furthermore, if $\alpha \in \Delta_{+}$ then $\alpha \in \Delta^{\mathrm{re}}_{+}$ if and only if 
$w_{\scriptscriptstyle i}(\alpha)$ is a real root and $d(w_{\scriptscriptstyle i}(\alpha)) < d(\alpha)$ for some 
fundamental reflection $w_{\scriptscriptstyle i}.$
\end{prop} 

From now on, we shall assume that 
\[
A \,= \begin{pmatrix} 2 & & & & -1 \cr  & 2 & & & -1 \cr  & & & &\vdots \cr 
& & & 2 & -1 \cr  -1 &-1 &\cdots & -1 & 2\end{pmatrix}
\] 
where we take $n = r + 1$ with $r \ge 1.$ Thus $A = (A_{ij})_{1 \le i, j \le r + 1}$ is the symmetric matrix whose entries are: 

\[
A_{ij} = 
\begin{cases} 
2 &\mbox{if $i = \!j$} \\
-1 &\mbox{if $i \ne \!j$ and either $i = r + 1$ or $j = r + 1$} \\
0 & \mbox{otherwise.}
\end{cases} 
\] 
We have that $\det \, A = - 2^{r - {\scriptscriptstyle 1}}(r - 4).$ By \cite[Theorem~4.3]{Kac}, one finds at once that 
$A$ is of finite, affine, or indefinite type when $r \le 3,$ $r = 4,$ or $r \ge 5,$ respectively. For example, the Dynkin diagram of $A$ is $D_{4}$ when $r = 3$ and $\tilde{D}_{4}$ when $r = 4;$ for $r \ge 5,$ the Dynkin diagram of $A$ has a star-like shape.

With notation as above, we have:

\vskip10pt
\begin{lem}\label{Finiteness-roots-level-n} --- Let 
$
\alpha =  \!\sum k_{\scriptscriptstyle i}\alpha_{\scriptscriptstyle i} \in \Delta_{+} 
$ 
with $k_{r + \scriptscriptstyle 1} \ge 1.$ Then $k_{\scriptscriptstyle i} \le k_{r + \scriptscriptstyle 1}$ for all $i = 1, \ldots, r.$
\end{lem}

\begin{proof} \!From the definition of the Cartan matrix $A,$ we have 
\begin{equation*}
u_{\scriptscriptstyle i} - v_{\scriptscriptstyle i} 
= \, \alpha(\alpha_{\scriptscriptstyle i}^{\scriptscriptstyle \vee}) 
= 2k_{\scriptscriptstyle i} - \, k_{r + \scriptscriptstyle 1}
\end{equation*} 
for all $1 \le i \le r.$ Since $k_{r + \scriptscriptstyle 1} > 0,$ it follows from \eqref{eq: Root-strings} that $\alpha - u_{\scriptscriptstyle i}\alpha_{\scriptscriptstyle i}\in \Delta_{+}.$ In particular, 
$
k_{r + \scriptscriptstyle 1} - \,k_{\scriptscriptstyle i} - \, v_{\scriptscriptstyle i} \ge 0,
$ 
and our assertion follows.
\end{proof}

Now let 
$
\alpha = \!\sum_{j} k_{\!\scriptscriptstyle j}\alpha_{\!\scriptscriptstyle j} 
$ 
($k_{\!\scriptscriptstyle j} \in \mathbb{Z}$) be an element of the root lattice. \!Since 
$
w_{\scriptscriptstyle i}(\alpha_{\!\scriptscriptstyle j}) 
= \alpha_{\!\scriptscriptstyle j} - A_{ij}\alpha_{\scriptscriptstyle i},
$ 
the effect of the fundamental reflection $w_{\scriptscriptstyle i}$ on the coefficients of $\alpha$ is given by 
\begin{equation} \label{eq: effect-coeff-roots} 
k_{\!\scriptscriptstyle j} \mapsto 
\begin{cases} 
k_{\!\scriptscriptstyle j} &\mbox{if $j \ne i$} \\ 
k_{r + {\scriptscriptstyle 1}} - k_{\scriptscriptstyle i}  &\mbox{if $j = i \le r$} \\
-\, k_{r + {\scriptscriptstyle 1}} \, + \,  \sum_{l = 1}^{r} k_{\scriptscriptstyle l}
& \mbox{if $j = i = r + 1.$}
\end{cases} 
\end{equation} 
(One will notice that these relations also give the conclusion of Lemma \ref{Finiteness-roots-level-n}.) 
Thus to determine the set of positive roots it suffices to find the elements 
$\alpha \in \Delta_{+}$ whose coefficients satisfy the condition 
\begin{equation} \label{eq: condition*} 
k_{\!\scriptscriptstyle j} \le k_{r + {\scriptscriptstyle 1}} \slash 2 \qquad \text{(for $j = 1, \ldots, r$).} \tag{$*$}
\end{equation} 
It follows at once from Proposition \ref{properties-roots} that, if $\alpha \notin \Pi$ is a positive real root satisfying the condition \eqref{eq: condition*}, then $d(w_{r + {\scriptscriptstyle 1}}(\alpha)) < d(\alpha),$ i.e. 
\begin{equation} \label{eq: condition**} 
\sum_{j = 1}^{r} k_{\!\scriptscriptstyle j} \le 2 k_{r + {\scriptscriptstyle 1}} - 1. \tag{$**$}
\end{equation} 
This gives a simple inductive procedure to determine the set of positive real roots. \!Another special feature of the set $\Delta^{\mathrm{re}}$ is expressed in the following lemma:

\vskip10pt
\begin{lem} --- We have $\Delta^{\mathrm{re}} = W \alpha_{r + {\scriptscriptstyle 1}}.$
\end{lem}

\begin{proof} \!This follows at once from the fact that 
$
w_{\scriptscriptstyle i} 
w_{r + {\scriptscriptstyle 1}}(\alpha_{\scriptscriptstyle i}) = \alpha_{r + {\scriptscriptstyle 1}}
$ 
for all $1 \le i \le r + 1.$
\end{proof}

{\it Denominator formula.} Since $A$ has real entries, we can take
a realization
$
(\fh_{\scriptscriptstyle \mathbb{R}},\, \Pi,\, \Pi^{\scriptscriptstyle
  \vee}\!)
$
of $A$ over $\mathbb{R}.$ Thus $\fh_{\scriptscriptstyle \mathbb{R}}$
is a vector\linebreak space of dimension $2r + 2 - \mathrm{rank}\; A$ over
$\mathbb{R},$ so that
$
(\fh,\, \Pi,\, \Pi^{\scriptscriptstyle \vee}\!)
$
($\fh : = \fh_{\scriptscriptstyle \mathbb{R}}
\otimes_{\scriptscriptstyle \mathbb{R}}
\mathbb{C}$) is the realization of $A$ over $\mathbb{C}.$ The set
\[
  C = \{h \in \fh_{\scriptscriptstyle \mathbb{R}}
:\alpha_{\scriptscriptstyle i}(h) \ge 0 \;\; \text{for $i = 1, \ldots, r+1$}\}
\]
is called the {\it fundamental chamber,} and the union
\[
\mathbf{X} \; = \bigcup_{w \in W} w(C)
\]
is called the {\it Tits cone.} The {\it complexified Tits cone}
$\mathbf{X}_{\scriptscriptstyle \mathbb{C}}$ is deﬁned to be
\[
\mathbf{X}_{\scriptscriptstyle \mathbb{C}}
= \{x + \scalebox{1.15}{$\scriptscriptstyle \sqrt{-1}$}\, y : x \in \mathbf{X},
\; y \in \fh_{\scriptscriptstyle \mathbb{R}}\}.
\]
Let $\rho \in \fh^{*}$ be a fixed element
such that
$
\rho(\alpha_{\scriptscriptstyle i}^{\scriptscriptstyle \vee}) = 1
$
($i = 1, \ldots, r + 1$), and, for $\lambda \in \fh^{*}\!,$ define
$e^{\scriptscriptstyle \lambda}$ by setting
$
e^{\scriptscriptstyle \lambda}(h) : = e^{\scriptscriptstyle \lambda(h)}
$
($h \in \fh$). Since $A$ is a symmetric matrix, hence $\fg(A)$ admits
a standard invariant bilinear form \cite[\textsection2.3]{Kac},
we have the Weyl-Kac {\it denominator formula}
\cite[Theorem~10.4 and (10.4.4)]{Kac}:
\[
\prod_{\alpha \in \Delta_{+}}(1 - e^{\scriptscriptstyle
  -\alpha})^{m_{\scalebox{.85}{$\scriptscriptstyle \alpha$}}}
\, = \sum_{w \in W} \epsilon(w) e^{\scriptscriptstyle w(\rho) - \rho}
\]
where $\epsilon(w) = (-1)^{\scriptscriptstyle l(w)}\!,$ $l(w)$ being the
length of $w.$ The natural domain of absolute convergence of both
sides of the formula is the interior (in metric topology) of
$\mathbf{X}_{\scriptscriptstyle \mathbb{C}},$
see \cite[Proposition~10.6]{Kac}. 

The denominator formula can be used to compute the multiplicities of the imaginary roots of $\fg(A),$ see \cite{BM}.

\subsection{Quadratic $L$-series}
Let $\mathbb{F}_{\! q}$ be a finite field of odd characteristic.
For simplicity, we shall assume throughout that $q\, \equiv 1
\!\!\pmod 4.$ For $m\in$ $\mathbb{F}_{\! q}[x],$ $m \ne 0,$ put
$|m| = q^{\deg \, m},$ and, for $d, m \in \mathbb{F}_{\! q}[x]$ with
$m$ monic, let $(d \slash m)$ denote the quadratic symbol; \!we take 
$(d \slash 1) = 1.$ Since $q\equiv 1 \!\!\pmod 4,$ we have the
quadratic reciprocity law  
\begin{equation*}
\left(\frac{d}{m}\right) =  \left(\frac{m}{d}\right) \qquad 
\text{(for coprime non-constant monic polynomials
$d, m \in \mathbb{F}_{\! q}[x]$).}
\end{equation*} 
In addition, if $b \in \mathbb{F}_{\! q}^{\times}$ then
$\left(\frac{b}{m}\right) = \mathrm{sgn}(b)^{\deg \, m}$
for all non-constant 
$m \in \mathbb{F}_{\! q}[x],$ where, for
$
d(x) = b_{\scriptscriptstyle 0}\, x^{n} + b_{\scriptscriptstyle 1}\,  x^{n - 1} + \cdots +  b_{\scriptscriptstyle n}
\! \in \mathbb{F}_{\! q}[x]
$ 
($b_{\scriptscriptstyle 0} \ne 0$), $\mathrm{sgn}(d) = 1$ if 
$
b_{\scriptscriptstyle 0} \in (\mathbb{F}_{\! q}^{\times})^{\scriptscriptstyle 2}
$ 
and $\mathrm{sgn}(d) = - 1$ if 
$
b_{\scriptscriptstyle 0} \notin (\mathbb{F}_{\! q}^{\times})^{\scriptscriptstyle 2}.
$

For $d$ square-free, put 
$
\chi_{d}(m) = (d \slash m).
$ 
The $L$-function associated to $\chi_{d}$ is defined by 
\begin{equation*}
  L(s, \chi_{d}) \;\;\, = \sum_{\substack{m \in \mathbb{F}_{\! q}[x] \\
      m - \text{monic}}}
  \chi_{d}(m) |m|^{-s}
  \;\;\; =  \prod_{p - \text{monic \& irreducible}}
  (1 - \chi_{d}(p)|p|^{-s})^{\scriptscriptstyle -1}
\qquad \text{(for complex $s$ with $\Re(s) > 1$).}
\end{equation*} 
It is well-known that $L(s, \chi_{d})$ is a polynomial in $q^{-s}$ of degree $\deg \, d - 1$ 
when $d$ is non-constant; if $d \in \mathbb{F}_{\! q}^{\times},$  
\begin{equation*}
L(s, \chi_{d}) =  \zeta(s) \, = \frac{1}{1 - q^{{\scriptscriptstyle 1} - s}} 
\;\;\;\; \text{(when $\mathrm{sgn}(d) = 1$)}
\;\;\; \text{and} \;\;\;  L(s, \chi_{d}) \, = \frac{1}{1 + q^{{\scriptscriptstyle 1} - s}} 
\;\;\;\; \text{(when $\mathrm{sgn}(d) = - 1$).}
\end{equation*} 
If one defines $\gamma_{\scriptscriptstyle q}(s,\,  d)$ by 
\begin{equation*} \label{eq: gamma-s-d}
\gamma_{\scriptscriptstyle q}(s,\,  d) : = 
q^{\frac{1}{2}{\scriptscriptstyle \left(3 \, + \, (-1)^{\deg \, d} \right)} \left(s - \frac{1}{2}\right)}\, 
\big(\! 1  -  \mathrm{sgn}(d) q^{- s} \big)^{\scriptscriptstyle \left(1 \, + \, (-1)^{\deg \, d}\right)\slash 2}\, 
\big(\! 1 - \mathrm{sgn}(d) q^{s - 1} \big)^{\scriptscriptstyle - \,  \left(1 \, + \, (-1)^{\deg \, d}\right)\slash 2}
\end{equation*} 
then $L(s, \chi_{d})$ satisfies the functional equation 
\begin{equation*} \label{eq: funct-eq-L} 
L(s, \chi_{d}) =\gamma_{\scriptscriptstyle q}(s,\,  d) 
|d|^{\frac{1}{2} - s} L(1 - s, \chi_{d}).
\end{equation*}

\section{Functions attached to the Weyl group $W$} \label{General-MDS}
We start with a power series 
\begin{equation*}
f(\mathrm{z}; q) = 1 + \sum_{\mathrm{k}\ne \bf{0}} a(\mathrm{k}; q) \mathrm{z}^{\mathrm{k}}
\end{equation*} 
where the sum is over all tuples 
$
\mathrm{k} = (k_{\scriptscriptstyle 1}, \ldots, k_{r}, k_{r +
  {\scriptscriptstyle 1}})\in \mathbb{N}^{r + {\scriptscriptstyle
    1}}\setminus \{\bf{0}\}\!,
$ 
and the coefficients $a(\mathrm{k}; q)$ are functions of odd prime powers $q;$ the power series is assumed to converge absolutely and uniformly in a small open polydisk (depending upon $q$) centered at ${\bf 0}.$ We take this function such that: 
\begin{enumerate}

\item If we set $\underline{z} = (z_{\scriptscriptstyle 1}, \ldots, z_{r}),$ 
$\underline{k} = (k_{\scriptscriptstyle 1}, \ldots, k_{r})$ and 
$l = k_{r + {\scriptscriptstyle 1}},$ then we can write 
\begin{equation*}
\begin{split}
f(\mathrm{z}; q) \, & = \,  
\frac{\sum_{l - \text{even}} 
P_{\scalebox{.85}{$\scriptscriptstyle l$}}(\underline{z}; q)
z_{r + {\scriptscriptstyle 1}}^{\scriptscriptstyle l}}
{\prod_{j = 1}^{r} (1 - z_{\scriptscriptstyle j})} \; + 
\sum_{l - \text{odd}} P_{\scalebox{.85}{$\scriptscriptstyle l$}}(\underline{z}; q)
z_{r + {\scriptscriptstyle 1}}^{\scriptscriptstyle l}\\
& = \, \frac{\sum_{|\underline{k}| - \text{even}}\, 
Q_{\scriptscriptstyle \underline{k}}(z_{r + {\scriptscriptstyle 1}}; q)\underline{z}^{\scriptscriptstyle \underline{k}}}{1 - z_{r + {\scriptscriptstyle 1}}}  \;\;  + 
\sum_{|\underline{k}| - \text{odd}} 
Q_{\scriptscriptstyle \underline{k}}(z_{r + {\scriptscriptstyle 1}}; q)\underline{z}^{\scriptscriptstyle \underline{k}}
\end{split}
\end{equation*} 
where $P_{\scalebox{.85}{$\scriptscriptstyle l$}}(\underline{z}; q)$ and 
$Q_{\scriptscriptstyle \underline{k}}(z_{r + {\scriptscriptstyle 1}}; q)$ are polynomials. 
Here $|\underline{k}| : = k_{\scriptscriptstyle 1} + \cdots + k_{r}.$

\item The power series obtained by expanding 
\[
\sum_{l \, \ge \, 0} P_{\scalebox{.85}{$\scriptscriptstyle l$}}(\underline{z}; q)
z_{r + {\scriptscriptstyle 1}}^{\scriptscriptstyle l}
\] 
is absolutely convergent for arbitrary 
$\underline{z} \in \mathbb{C}^{r},$ provided 
$|z_{r + {\scriptscriptstyle 1}}|$ is sufficiently small, and the power series obtained by expanding
\[
\sum_{|\underline{k}| \, \ge \, 0} 
Q_{\scriptscriptstyle \underline{k}}(z_{r + {\scriptscriptstyle 1}}; q)\underline{z}^{\scriptscriptstyle \underline{k}}
\] 
is absolutely convergent for any $z_{r + {\scriptscriptstyle 1}} \in \mathbb{C},$ provided all 
$|z_{\scriptscriptstyle 1}|, \ldots, |z_{r}|$ are sufficiently small.

\item We have 
\[
P_{\scriptscriptstyle 0}(\underline{z}; q) = P_{\scriptscriptstyle 1}(\underline{z}; q) = Q_{\scriptscriptstyle \underline{0}}(z_{r + {\scriptscriptstyle 1}}; q) \equiv 1.
\]

\item The polynomials $P_{\scalebox{.85}{$\scriptscriptstyle l$}}(\underline{z}; q)$ are symmetric in $\underline{z},$ and if $l$ is odd then $P_{\scalebox{.85}{$\scriptscriptstyle l$}}(\underline{z}; q)$ is even, i.e., 
$
P_{\scalebox{.85}{$\scriptscriptstyle l$}}(\underline{z}; q) = P_{\scalebox{.85}{$\scriptscriptstyle l$}}(-\underline{z}; q).
$

\item We have the functional equations 
\begin{equation} \label{eq: poly-P-Q-func-eq}
P_{\scalebox{.85}{$\scriptscriptstyle l$}}(z_{\scriptscriptstyle 1}, z_{\scriptscriptstyle 2}, \ldots, z_{r}; q) 
= (\sqrt{q} \, z_{\scriptscriptstyle 1})^{l - \delta_{\scriptscriptstyle l}} 
P_{\scalebox{.85}{$\scriptscriptstyle l$}}\bigg(\frac{1}{q \, z_{\scriptscriptstyle 1}}, 
z_{\scriptscriptstyle 2}, \ldots, z_{r}; q\bigg)
\;\;\;\; \mathrm{and} \;\;\;\; 
Q_{\scriptscriptstyle \underline{k}}(z_{r + {\scriptscriptstyle 1}}; q)
= (\sqrt{q} \, z_{r + {\scriptscriptstyle 1}})^{|\underline{k}| - \delta_{\scriptscriptstyle |\underline{k}|}}
Q_{\scriptscriptstyle \underline{k}}\bigg(\frac{1}{q \, z_{r + {\scriptscriptstyle 1}}}; q\bigg)
\end{equation} 
with $\delta_{\scriptscriptstyle n} \! = 0$ or $1$ according as $n$ is even or odd.

\end{enumerate}

For our purposes, we shall need a specific power series satisfying all
these properties, see the next section for definitions.

Put 
$f_{\!\scriptscriptstyle \mathrm{even}}(\mathrm{z}; q) 
: = (f(\underline{z}, z_{r + {\scriptscriptstyle 1}}; q) + f(\underline{z}, - z_{r + {\scriptscriptstyle 1}}; q))\slash 2
$ 
and 
$
f_{\!\scriptscriptstyle \mathrm{odd}}(\mathrm{z}; q) 
: = (f(\underline{z}, z_{r + {\scriptscriptstyle 1}}; q) - f(\underline{z}, - z_{r + {\scriptscriptstyle 1}}; q))\slash 2,
$ 
and define 
\[
M(u, q) : = \begin{pmatrix}
- \frac{1 - qu}{qu(1 - u)}& & \\
& \frac{1}{\sqrt{q}u}& \\
& & \frac{1 + qu}{qu(1 + u)} \\
\end{pmatrix}
\;\;\;\; \mathrm{and} \;\;\;\; 
{\bf{f}}(\mathrm{z}; q) : = 
\begin{pmatrix}
f_{\!\scriptscriptstyle \mathrm{even}}(\underline{z}, z_{r + {\scriptscriptstyle 1}}; q) \\
f_{\!\scriptscriptstyle \mathrm{odd}}(\underline{z}, z_{r + {\scriptscriptstyle 1}}; q)\\
f_{\!\scriptscriptstyle \mathrm{even}}(- \underline{z}, z_{r + {\scriptscriptstyle 1}}; q) \\
\end{pmatrix};
\] 
also, let 
\begin{equation*}
U: = \begin{pmatrix} 
1\slash 2 & 1 & 1\slash 2\\
1\slash 2 & 0 & - 1\slash 2\\
1\slash 2 & - 1 & 1\slash 2\\
\end{pmatrix}.
\end{equation*} 
Note that $U^{\scriptscriptstyle 2} = I,$ where $I$ is the identity
$3\times 3$ matrix.

For $i = 1, \ldots, r + 1$ and
$\mathrm{z} \in \mathbb{C}^{r + {\scriptscriptstyle 1}}\!,$ define
$
w_{\scriptscriptstyle i} \mathrm{z} =
\mathrm{z}\scalebox{.95}{$\scriptscriptstyle '$} 
\in \mathbb{C}^{r + {\scriptscriptstyle 1}}
$
by
\[
z_{\scriptscriptstyle j}' = 
\begin{cases} 
z_{\scriptscriptstyle j}
&\mbox{if $i, j \le r,\,  i \ne \!j$} \\
1 \slash q z_{\scriptscriptstyle i}
& \mbox{if $i = \!j$} \\
\sqrt{q}\, z_{\scriptscriptstyle i} z_{\scriptscriptstyle j}
& \mbox{if $i \ne \!j,$ and either $i = r + 1$ or $j = r + 1.$}
\end{cases}
\]
The group generated by 
$w_{\scriptscriptstyle 1}, \ldots, w_{r + {\scriptscriptstyle 1}}$
is easily seen to be isomorphic to the Weyl group $W$ in
\ref{KM-alg}. In what follows, \!by abuse of notation, we shall still
denote the group
$
\langle w_{\scriptscriptstyle i} : i = 1, \ldots, r + 1 \rangle
$
by $W.$ From the above properties of $f(\mathrm{z}; q),$ one checks
that
\begin{equation} \label{eq: f-e-loc-prelim}
  {\bf{f}}(\mathrm{z}; q) = M(z_{\scriptscriptstyle i}, q)
  {\bf{f}}(w_{\scriptscriptstyle i}\mathrm{z}; q) \;\;\;\, \text{(for $i = 1,\ldots, r$)}
  \;\;\;\; \mathrm{and} \;\;\;\;
  U{\bf{f}}(\mathrm{z}; q) = M(z_{r + {\scriptscriptstyle 1}}, q)U
  {\bf{f}}(w_{r + {\scriptscriptstyle 1}}\mathrm{z}; q).
\end{equation} 
Let $\mathrm{M}_{w}(\mathrm{z}) = \mathrm{M}_{w}(\mathrm{z}; q),$ for 
$w \in W,$ be defined as follows. Put 
\begin{equation} \label{eq: def-cocycle1}
\mathrm{M}_{w_{\scalebox{.65}{$\scriptscriptstyle i$}}}\!(z_{\scriptscriptstyle i}) = 
\begin{cases} 
M(z_{\scriptscriptstyle i}, q)
&\mbox{if $1 \le i \le r$} \\
U M(z_{r + {\scriptscriptstyle 1}}, q) U
& \mbox{if $i = r + 1$}
\end{cases}  
\end{equation} 
and extend this function to $W$ by the $1$-cocycle relation:   
\begin{equation} \label{eq: def-cocycle2}
\mathrm{M}_{w w'}(\mathrm{z}) 
= \mathrm{M}_{w'}(\mathrm{z})\mathrm{M}_{w}(w'\!\mathrm{z}) 
\qquad \text{(for $w, w' \in W$).}
\end{equation} 
The cocycle $\mathrm{M}_{w}(\mathrm{z}; q)$ is well-defined, and by \eqref{eq: f-e-loc-prelim}, we have 
\begin{equation} \label{eq: f-e-loc}
{\bf{f}}(\mathrm{z}; q) = \mathrm{M}_{w}(\mathrm{z}; q)
{\bf{f}}(w\mathrm{z}; q)
\end{equation}	
for every $w\in W.$

\section{Multiple {D}irichlet series}  \label{WMDS-assoc-moments} 
In this section, we shall introduce multiple {D}irichlet series (MDS) naturally associated with moments of quadratic {D}irichlet $L$-functions. The construction of these series, originally given in \cite{DV} (see also \cite{Saw}, where more general MDS\linebreak 
are discussed), is based upon a set of axioms that {\it uniquely} characterize them. By simply adapting the proof of \cite[Pro-\linebreak position~2.2.1]{White2} in the case of MDS for affine {W}eyl groups, one can show that the so-called {\it $p$-parts} of our MDS 
satisfy the conditions 1--5 in the previous section. \!An almost immediate consequence is that the multiple {D}irichlet series we 
consider here satisfy a group of functional equations isomorphic to the {W}eyl group of the Kac-Moody algebra in \ref{KM-alg}.

\subsection{Generating series} \label{Formal-MDS} 
Let $\mathbb{F}_{\! q}$ be a finite field of odd characteristic. We fix throughout an algebraic closure 
$\overline{\mathbb{F}}_{\! q}$ of $\mathbb{F}_{\! q}.$ For 
$
\mathbf{s} = (s_{\scriptscriptstyle 1}, \ldots, s_{r + {\scriptscriptstyle 1}}) \in \mathbb{C}^{r + {\scriptscriptstyle 1}}\!,
$ 
consider a {\it formal} series $Z(\mathbf{s})$ of the form 
\begin{equation} \label{eq: MDS-definition-initial} 
Z(\mathbf{s})
\; = \sum_{\substack{m_{\scalebox{.8}{$\scriptscriptstyle 1$}}\!, \ldots, \, m_{r}, \, d - \mathrm{monic} \\  
		d = d_{\scriptscriptstyle 0}^{} d_{\scriptscriptstyle 1}^{2}, \; d_{\scriptscriptstyle 0}^{} \; \mathrm{square \; free}}}  
	\, \frac{\chi_{d_{\scriptscriptstyle 0}}\!(\widehat{m}_{\scriptscriptstyle 1}) \, \cdots \, 
	\chi_{d_{\scriptscriptstyle 0}}\!(\widehat{m}_{r})}
{|m_{\scriptscriptstyle 1}|^{s_{\scalebox{.8}{$\scriptscriptstyle 1$}}} \cdots \, |m_{r}|^{s_{r}} 
	|d|^{s_{r + {\scalebox{.8}{$\scriptscriptstyle 1$}}}}} \cdot A(m_{\scriptscriptstyle 1}, \ldots, m_{r}, d)
\end{equation} 
where $\widehat{m}_{\scriptscriptstyle i}$ ($i = 1, \ldots, r$) denotes the part of 
$m_{\scriptscriptstyle i}$ coprime to $d_{\scriptscriptstyle 0}.$ We assume that the coefficients 
$
A(m_{\scriptscriptstyle 1}, m_{\scriptscriptstyle 2}, m_{\scriptscriptstyle 3}, \ldots, m_{r}, d)
$ 
are multiplicative, in the sense that 
\[
A(m_{\scriptscriptstyle 1}, \ldots, m_{r}, d)
\;\, = \prod_{\substack{p^{\, k_{\scalebox{.9}{$\scriptscriptstyle i$}}} \parallel 
		\, m_{\scalebox{.9}{$\scriptscriptstyle i$}}\\ p^{\, l} \parallel \, d}} 
A\big(p^{\, k_{\scalebox{.8}{$\scriptscriptstyle 1$}}}\!, \ldots, p^{\, k_{r}}\!, p^{\, l} \big)
\] 
the product being taken over monic irreducibles $p  \in  \mathbb{F}_{\! q}[x].$ If we set 
$
z_{\scriptscriptstyle i} = q^{ - s_{\scalebox{.9}{$\scriptscriptstyle i$}}}\!,
$ 
for $1 \le i \le r + 1,$ we can also write \eqref{eq: MDS-definition-initial} as a formal power series 
\[
\sum_{k_{\scalebox{.8}{$\scriptscriptstyle 1$}}\!, \ldots,  k_{r}\!,\, l \, \ge \, 0} \,
b(k_{\scriptscriptstyle 1}, \ldots, k_{r}, l; q)
\, z_{\scriptscriptstyle 1}^{k_{\scalebox{.8}{$\scriptscriptstyle 1$}}} \cdots \, z_{r}^{k_{r}} z_{r + {\scriptscriptstyle 1}}^{l}. 
\] 
We specify the coefficients 
$ 
A\big(p^{\, k_{\scalebox{.8}{$\scriptscriptstyle 1$}}}\!, \ldots, p^{\, k_{r}}\!, p^{\, l} \big) 
$ 
and 
$
b(k_{\scriptscriptstyle 1}, \ldots, k_{r}, l; q)
$ 
by considering formal power series of the form 
\begin{equation} \label{eq: formal-ps-Weil}
\sum_{\mathrm{k}}\, c(\mathrm{k}; q) \mathrm{z}^{\mathrm{k}}
\end{equation} 
with coefficients $c(\mathrm{k}; q)$ ($\mathrm{k}\in \mathbb{N}^{r + {\scriptscriptstyle 1}}$) finite sums 
\begin{equation} \label{eq: q-Weil-integers-sums} 
c(\mathrm{k}; q) = \sum_{j}\, {\mathbf P}_{\scriptscriptstyle j}^{\scriptscriptstyle (\mathrm{k})}\!(q) 
\lambda_{\scriptscriptstyle j}
\end{equation} 
where 
$
{\mathbf P}_{\scriptscriptstyle j}^{\scriptscriptstyle (\mathrm{k})}\!(x) 
\in  \mathbb{Q}[x],
$ 
and $\lambda_{\scriptscriptstyle j}$ are \addtocounter{footnote}{0}\let\thefootnote\svthefootnote distinct\footnote{In fact we shall always take an expression \eqref{eq: q-Weil-integers-sums} in reduced form, by which we mean that $\lambda' \! = q^{n}\lambda$ for some $n\in \mathbb{N}$ if and only if $\lambda' \!= \lambda.$} $q$-Weil algebraic integers of weights $\nu_{\!\scriptscriptstyle j} \in \mathbb{N}.$ 
We are assuming that: 
\begin{enumerate}[label=(\roman*)] 
\item Every $q$-Weil integer $\lambda_{\scriptscriptstyle j}$ occurs in 
	\eqref{eq: q-Weil-integers-sums} together with all its complex conjugates. 

\item If $\lambda_{\scriptscriptstyle j}$ and $\lambda_{\scriptscriptstyle j\scalebox{.8}{$\scriptscriptstyle '$}}$ are conjugates over $\mathbb{Q},$ then 
$
{\mathbf P}_{\scriptscriptstyle j}^{\scriptscriptstyle (\mathrm{k})}\!(x) 
= 
{\mathbf P}_{\scriptscriptstyle j\scalebox{.8}{$\scriptscriptstyle '$}}^{\scriptscriptstyle (\mathrm{k})}\!(x).
$ 

\item For each $j,$ 
$ 
\deg\, {\mathbf P}_{\scriptscriptstyle j}^{\scriptscriptstyle (\mathrm{k})} 
\! + \, \nu_{\!\scriptscriptstyle j} \le |\mathrm{k}|
$ 
and 
$
{\mathbf P}_{\scriptscriptstyle j}^{\scriptscriptstyle (\mathrm{k})}\!(x)^{\scriptscriptstyle 2} 
\equiv \,  0 \pmod{x^{\scriptscriptstyle |\mathrm{k}| - \nu_{\!\scalebox{.7}{$\scriptscriptstyle j$}} + 2}} 
$ 
when $|\mathrm{k}| = k_{1} + \cdots + k_{r + {\scriptscriptstyle 1}} \ge 2.$ 
\end{enumerate} 

\vskip5pt 
It is proved in \cite{DV} that there exists a unique series \eqref{eq: MDS-definition-initial}, and a unique series 
\eqref{eq: formal-ps-Weil}, whose coefficients are of the form 
\eqref{eq: q-Weil-integers-sums}\linebreak with all properties ($\mathrm{\romannumeral 1}$) $-$ ($\mathrm{\romannumeral 3}$), 
such that the following conditions are satisfied: 
\begin{enumerate}[label=(\alph*)]
	\item The subseries
	\[
	\sum_{k_{\scalebox{.8}{$\scriptscriptstyle 1$}}\!, \ldots,  k_{r} \, \ge \, 0}
	b(k_{\scriptscriptstyle 1}, \ldots, k_{r}, 0; q)
	\, z_{\scriptscriptstyle 1}^{k_{\scalebox{.8}{$\scriptscriptstyle 1$}}} \cdots \, z_{r}^{k_{r}} 
	=\; \prod_{i = 1}^{r} \frac{1}{1 - qz_{\scriptscriptstyle i}}
	\] 
	i.e., is a product of $r$ zeta functions. In addition,
	\[
	\sum_{k_{\scalebox{.8}{$\scriptscriptstyle 1$}}\!, \ldots,  k_{r} \, \ge \, 0}
	b(k_{\scriptscriptstyle 1}, \ldots, k_{r}, 1; q)
	\, z_{\scriptscriptstyle 1}^{k_{\scalebox{.8}{$\scriptscriptstyle 1$}}} \cdots \, z_{r}^{k_{r}} =\, q 
	\] 
	and 
	\[
	\sum_{l \, \ge \, 0}\,  b(0, \ldots, 0, l; q)\, z_{r + {\scriptscriptstyle 1}}^{l}
	= \frac{1}{1 - qz_{r + {\scriptscriptstyle 1}}}.
	\] 
	In particular, $A(1, \ldots, 1, 1) = b(0, \ldots, 0, 0; q) = 1.$

	\item For every $(k_{\scriptscriptstyle 1}, \ldots, k_{r}, l) \in \mathbb{N}^{r + {\scriptscriptstyle 1}}\!,$ the coefficient 
	$b(k_{\scriptscriptstyle 1}, \ldots, k_{r}, l; q^{n})$ corresponding to the series \eqref{eq: MDS-definition-initial} over 
	any finite field extension $\mathbb{F}_{\! q^{n}} \subset \overline{\mathbb{F}}_{\! q}$ of $\mathbb{F}_{\! q}$ is given by 
	\[
	b(k_{\scriptscriptstyle 1}, \ldots, k_{r}, l; q^{n})
	= 	c(k_{\scriptscriptstyle 1}, \ldots, k_{r}, l; q^{n}) 
	= \sum_{j}\, {\mathbf P}_{\scriptscriptstyle j}^{\scriptscriptstyle (k_{\scalebox{.75}{$\scriptscriptstyle 1$}}\!, \ldots, \, k_{r}, \, l)}\!(q^{n}) 
	\lambda_{\scriptscriptstyle j}^{\! n}. 
	\] Moreover, the polynomials ${\mathbf P}_{\scriptscriptstyle j}^{\scriptscriptstyle (\mathrm{k})}$ 
	\!are independent of the prime power $q,$ for all $\mathrm{k}\in \mathbb{N}^{r + {\scriptscriptstyle 1}}$ \!and all $j.$  
	
	\item For every monic irreducible $p \in \mathbb{F}_{\! q}[x]$ of degree $e \ge 1,$ the coefficients 
	$A\big(p^{\, k_{\scalebox{.8}{$\scriptscriptstyle 1$}}}\!, \ldots, p^{\, k_{r}}\!, p^{\, l} \big)$ are given by 
	\[
	A\big(p^{\, k_{\scalebox{.8}{$\scriptscriptstyle 1$}}}\!, \ldots, p^{\, k_{r}}\!, p^{\, l} \big) 
	= 	q^{e(k_{\scalebox{.75}{$\scriptscriptstyle 1$}} + \cdots + k_{r} + l)}
	c(k_{\scriptscriptstyle 1}, \ldots, k_{r}, l; q^{-e}) 
	= q^{e(k_{\scalebox{.75}{$\scriptscriptstyle 1$}} + \cdots + k_{r} + l)}
	\!\sum_{j}\, {\mathbf P}_{\scriptscriptstyle j}^{\scriptscriptstyle (k_{\scalebox{.75}{$\scriptscriptstyle 1$}}\!, \ldots, \, k_{r}, \, l)}\!(q^{- e}) 
	\lambda_{\scriptscriptstyle j}^{\! - e}.
	\]
\end{enumerate} 
\vskip-5pt 
In what follows, we will need a twisted version of \eqref{eq: MDS-definition-initial}. To define these series, we fix 
throughout an element 
$
\theta_{\scriptscriptstyle 0} \in \mathbb{F}_{\! q}^{\times} \setminus
$ $(\mathbb{F}_{\! q}^{\times})^{2}.$ 
Let $c \in \mathbb{F}_{\! q}[x]$ be monic and square free, and fix a factorization 
$c = c_{\scriptscriptstyle 1}c_{\scriptscriptstyle 2}c_{\scriptscriptstyle 3},$ with $c_{\scriptscriptstyle i}$ monic. 
For $a_{\scriptscriptstyle 1}, \, a_{\scriptscriptstyle 2} \in \{1, \theta_{\scriptscriptstyle 0} \},$ we define 
$
Z^{(c)}(\mathbf{s};  \chi_{a_{\scriptscriptstyle 2} c_{\scriptscriptstyle 2}}, \chi_{a_{\scriptscriptstyle 1} c_{\scriptscriptstyle 1}}) 
$ 
by 
\begin{equation} \label{eq: MDS-definition-vers0} 
Z^{(c)}(\mathbf{s};  \chi_{a_{\scriptscriptstyle 2} c_{\scriptscriptstyle 2}}, \chi_{a_{\scriptscriptstyle 1} c_{\scriptscriptstyle 1}})
\;\; = \sum_{\substack{m_{\scalebox{.8}{$\scriptscriptstyle 1$}}\!, \ldots, \, m_{r}, \, d - \mathrm{monic} \\  
		d = d_{\scriptscriptstyle 0}^{} d_{\scriptscriptstyle 1}^{2}, \; d_{\scriptscriptstyle 0}^{} \; \mathrm{square \; free} \\ 
		(m_{\scalebox{.8}{$\scriptscriptstyle 1$}} \cdots \, m_{r} \, d, \, c) = 1}}  \, \frac{\chi_{a_{\scriptscriptstyle 1} c_{\scriptscriptstyle 1} d_{\scriptscriptstyle 0}}\!(\widehat{m}_{\scriptscriptstyle 1}) \, \cdots \, 
	\chi_{a_{\scriptscriptstyle 1} c_{\scriptscriptstyle 1}d_{\scriptscriptstyle 0}}\!(\widehat{m}_{r})\,
	\chi_{a_{\scriptscriptstyle 2} c_{\scriptscriptstyle 2}}\!(d_{\scriptscriptstyle 0})}
{|m_{\scriptscriptstyle 1}|^{s_{\scalebox{.8}{$\scriptscriptstyle 1$}}} \cdots \, |m_{r}|^{s_{r}} |d|^{s_{r 
+ {\scalebox{.8}{$\scriptscriptstyle 1$}}}}} \cdot 
A(m_{\scriptscriptstyle 1}, \ldots, m_{r}, d)
\end{equation} 
with the same coefficients $A(m_{\scriptscriptstyle 1}, \ldots, m_{r}, d)$ as before.
	
\vskip5pt
\begin{rem} --- The analogue of \eqref{eq: MDS-definition-vers0} over $\mathbb{Q}$ (or over any number field) is constructed similarly from the same (unique) series \eqref{eq: formal-ps-Weil}. Concretely, for 
$d\in \mathbb{Z}$ non-zero and square free, let $\chi_{d}(m)$ be the quadratic character defined by 
	\[ 
	\chi_{d}(m) = 
	\begin{cases} 
	\left(\frac{d}{m}\right) & \mbox{if $d \equiv 1 \!\!\!\!\!\pmod 4$} \\
	\left(\frac{4d}{m}\right) & \mbox{if $d \equiv 2, \, 3 \!\!\!\!\!\pmod 4.$}
	\end{cases}  
	\] 
	Fix an odd, positive, square free integer $c.$ Let $a_{\scriptscriptstyle 1}, a_{\scriptscriptstyle 2} \in \{\pm 1, \, \pm2\},$ 
	and let $c_{\scriptscriptstyle 1}, c_{\scriptscriptstyle 2}$ divide $c.$ Define 
	\[
	Z^{(c)}(\mathbf{s}; \chi_{a_{\scriptscriptstyle 2}c_{\scriptscriptstyle 2}}, \chi_{a_{\scriptscriptstyle 1}c_{\scriptscriptstyle 1}})
	\;\;\;  = \sum_{\substack{m_{\scalebox{.8}{$\scriptscriptstyle 1$}}\!, \ldots, \, m_{r}, \, d  \ge 1 \\  
			d = d_{\scriptscriptstyle 0}^{} d_{\scriptscriptstyle 1}^{2}, \; d_{\scriptscriptstyle 0}^{} \; \mathrm{square \; free} \\ 
			(m_{\scalebox{.8}{$\scriptscriptstyle 1$}} \cdots \, m_{r} \, d, \, 2c) = 1}} \frac{\chi_{a_{\scriptscriptstyle 1} c_{\scriptscriptstyle 1} d_{\scriptscriptstyle 0}}\!(\widehat{m}_{\scriptscriptstyle 1} \cdots \, \widehat{m}_{r})\,
		\chi_{a_{\scriptscriptstyle 2} c_{\scriptscriptstyle 2}}\!(d_{\scriptscriptstyle 0})
	\, A(m_{\scriptscriptstyle 1}, \ldots, m_{r}, d)}
	{m_{\scriptscriptstyle 1}^{s_{\scalebox{.8}{$\scriptscriptstyle 1$}}} \cdots \, m_{r}^{s_{r}} d_{\phantom{x}}^{s_{r + {\scalebox{.8}{$\scriptscriptstyle 1$}}}}}
	\] 
	with $A(m_{\scriptscriptstyle 1}, \ldots, m_{r}, d)$ multiplicative, and 
	$
	A\big(p^{\, k_{\scalebox{.8}{$\scriptscriptstyle 1$}}}\!, \ldots, p^{\, k_{r}}\!, p^{\, l} \big),
	$ 
	for any prime $p \ge 3,$ given by 
\[
A\big(p^{\, k_{\scalebox{.8}{$\scriptscriptstyle 1$}}}\!, \ldots, p^{\, k_{r}}\!, p^{\, l} \big) 
= p^{k_{\scalebox{.75}{$\scriptscriptstyle 1$}} + \cdots + k_{r} + l}
\cdot \sum_{j}\, {\mathbf P}_{\scriptscriptstyle j}^{\scriptscriptstyle (k_{\scalebox{.75}{$\scriptscriptstyle 1$}}\!, \ldots, \, k_{r}, \, l)}
\!(p^{\scriptscriptstyle - 1}) 
\lambda_{\scriptscriptstyle j}^{\!\scriptscriptstyle - 1}.
\] 
This function should (conjecturally) have all the analytic properties that will be shortly discussed in the function-field case. 
\end{rem}

\subsection{Functional equations} 
With the notations introduced in \ref{Formal-MDS}, let $a(\mathrm{k}; q),$ for 
$\mathrm{k} \in \mathbb{N}^{r + \scalebox{1.0}{$\scriptscriptstyle 1$}}\setminus \{\bf{0}\}\!,$ be defined by 
\begin{align*}
a(\mathrm{k}; q) & : = q^{\scriptscriptstyle |\mathrm{k}|}\sum_{j}\, 
{\mathbf P}_{\scriptscriptstyle j}^{\scriptscriptstyle (\mathrm{k})}\!(q^{\scalebox{.95}{$\scriptscriptstyle -1$}}) 
\lambda_{\scriptscriptstyle j}^{\!\scalebox{.95}{$\scriptscriptstyle -1$}}  \\
& =  \sum_{j}\, 
{\mathbf P}_{\scriptscriptstyle j}^{\scriptscriptstyle (\mathrm{k})}\!(q^{\scriptscriptstyle -1}) 
q^{\scriptscriptstyle |\mathrm{k}| - \nu_{\!\scalebox{.7}{$\scriptscriptstyle j$}}}\lambda_{\scriptscriptstyle j}
\end{align*} 
where ${\mathbf P}_{\scriptscriptstyle j}^{\scriptscriptstyle (\mathrm{k})}$ and 
$\lambda_{\scriptscriptstyle j}$ refer to the unique sets of polynomials and $q$-Weil algebraic integers, respectively, determined by 
the\linebreak conditions ($\mathrm{a}$) $-$ ($\mathrm{c}$) in \ref{Formal-MDS}. Note that by the first condition in 
($\mathrm{\romannumeral 3}$), \ref{Formal-MDS}, 
$
{\mathbf P}_{\scriptscriptstyle j}^{\scriptscriptstyle (\mathrm{k})}\!(x^{\scriptscriptstyle -1}) 
x^{\scriptscriptstyle |\mathrm{k}| - \nu_{\!\scalebox{.7}{$\scriptscriptstyle j$}}} \! \in  \mathbb{Q}[x] 
$ 
when $|\mathrm{k}| > 1,$ and by the sec-\linebreak ond condition in ($\mathrm{\romannumeral 3}$), \ref{Formal-MDS}, 
we have 
\[
\max_{j}\!\left\{\tfrac{\nu_{\! \scalebox{.7}{$\scriptscriptstyle j$}}}{2} 
+ \deg\;  {\mathbf P}_{\scriptscriptstyle j}^{\scriptscriptstyle (\mathrm{k})}\!(x^{\scriptscriptstyle -1}) 
x^{\scriptscriptstyle |\mathrm{k}| - \nu_{\!\scalebox{.7}{$\scriptscriptstyle j$}}}\right\}
\le  \tfrac{|\mathrm{k}|}{2} - 1  
\qquad \text{(if $|\mathrm{k}| > 1$).} 
\] 
This leads us to make the following:

\vskip5pt
{\bf Assumption.} For every small $\varepsilon  > 0$ and $|\mathrm{k}| \ge 2,$ we have
\begin{equation} \label{eq: Estimate-H1} 
a(\mathrm{k}; q) \ll_{\varepsilon} 
	q^{\left(\!\scalebox{1.15}{$\scriptscriptstyle \frac{1}{2}$} 
		+ \scalebox{1.35}{$\scriptscriptstyle \varepsilon$}\!\right)\scalebox{1.3}{$\scriptscriptstyle |\mathrm{k}|$} - 1}. 
	\tag{H1}
\end{equation}

\begin{rem} --- It is conceivable that the method used in \cite{BDPW} to bound the Betti numbers 
of the moduli space of hyperel-\linebreak liptic curves of any fixed genus can be adapted to prove a more natural bound of the form 
\[
|a(\mathrm{k}; q)|  <   C^{|\mathrm{k}|} q^{\scalebox{1.05}{$\scriptscriptstyle \frac{|\mathrm{k}|}{2}$} - 1}	
\;\;\;\;  \text{($|\mathrm{k}| > 1$)}
\] 
for an absolute constant $C$ depending only upon $r.$ 
\end{rem} 

\vskip5pt
We now take the function $f(\mathrm{z}; q)$ at the beginning of Section \ref{General-MDS} to be 
\[ 
f(\mathrm{z}; q) : = 1 + \sum_{\mathrm{k}\ne \bf{0}} a(\mathrm{k}; q) \mathrm{z}^{\mathrm{k}} 
= 1 + \sum_{\mathrm{k}\ne \bf{0}} \bigg\{\sum_{j}\, 
{\mathbf P}_{\scriptscriptstyle j}^{\scriptscriptstyle (\mathrm{k})}\!(q^{\scriptscriptstyle -1}) 
q^{\scriptscriptstyle - \nu_{\!\scalebox{.7}{$\scriptscriptstyle j$}}}\lambda_{\scriptscriptstyle j}\bigg\}
(q \mathrm{z})^{\scalebox{1.}{$\scriptscriptstyle \mathrm{k}$}};
\] 
by \eqref{eq: Estimate-H1}, it converges absolutely in the polydisk 
$
|z_{\scriptscriptstyle i}| < q^{\scalebox{.95}{$\scriptscriptstyle - \frac{1}{2}$}}
$ 
($i = 1, \ldots, r+1$). The same idea as in the proof of \cite[Proposi-\linebreak tion~2.2.1]{White2} shows that 
the function $f(\mathrm{z}; q)$ satisfies the conditions 1--5 in Section \ref{General-MDS}, and thus, the 
corresponding vector function ${\bf{f}}(\mathrm{z}; q)$ satisfies the functional equations \eqref{eq: f-e-loc}.

The series 
$
Z^{(c)}(\mathbf{s};  \chi_{a_{\scriptscriptstyle 2} c_{\scriptscriptstyle 2}}, \chi_{a_{\scriptscriptstyle 1} c_{\scriptscriptstyle 1}})
$ 
are constructed from $f(\mathrm{z}; q)$ exactly as in \cite[Section 3]{Dia}. By \eqref{eq: Estimate-H1}, it is clear that 
they converge absolutely for $\Re(s_{\scriptscriptstyle i}) > 1$ ($i = 1, \ldots, r+1$), and in this domain we can write, 
see loc. cit., 
\begin{equation}  \label{eq: MDS-vers1}
Z^{(c)}(\mathbf{s};  \chi_{a_{\scriptscriptstyle 2} c_{\scriptscriptstyle 2}}, \chi_{a_{\scriptscriptstyle 1} c_{\scriptscriptstyle 1}}) 
\; = \sum_{\substack{(d, \, c) = 1 \\  d = d_{\scriptscriptstyle 0}^{} d_{\scriptscriptstyle 1}^{2}}}  \, 
\frac{\prod_{i = 1}^{r}\, L^{\scriptscriptstyle (c_{\scriptscriptstyle 2}c_{\scriptscriptstyle 3})}
	(s_{\scriptscriptstyle i}, \chi_{a_{\scriptscriptstyle 1} c_{\scriptscriptstyle 1}  d_{\scriptscriptstyle 0}}) \, \cdot 
	\,\chi_{a_{\scriptscriptstyle 2} c_{\scriptscriptstyle 2}}\!(d_{\scriptscriptstyle 0})\, 
	P_{\scriptscriptstyle d}(\mathbf{s}\scalebox{1.}{$\scriptscriptstyle '$}; \chi_{a_{\scriptscriptstyle 1} c_{\scriptscriptstyle 1} d_{\scriptscriptstyle 0}})}{|d|^{s_{r + {\scalebox{.8}{$\scriptscriptstyle 1$}}}}}.
\end{equation} 
Here 
$
L^{\scriptscriptstyle (c_{\scriptscriptstyle 2}c_{\scriptscriptstyle 3})}
(s_{\scriptscriptstyle i}, \chi_{a_{\scriptscriptstyle 1} c_{\scriptscriptstyle 1}  d_{\scriptscriptstyle 0}}) = 
L(s_{\scriptscriptstyle i}, \chi_{a_{\scriptscriptstyle 1} c_{\scriptscriptstyle 1}  d_{\scriptscriptstyle 0}}) 
\cdot \prod_{p \mid c_{\scriptscriptstyle 2}c_{\scriptscriptstyle 3}} 
(1 -  \chi_{a_{\scriptscriptstyle 1} c_{\scriptscriptstyle 1}  d_{\scriptscriptstyle 0}}\!(p)  
|p|^{-s_{\scalebox{.8}{$\scriptscriptstyle i$}}}),
$ 
the product being over the monic irreducible divisors of 
$c_{\scriptscriptstyle 2}c_{\scriptscriptstyle 3},$ 
and 
$  
P_{\scriptscriptstyle d}(\mathbf{s}\scalebox{1.}{$\scriptscriptstyle '$}; \chi_{a_{\scriptscriptstyle 1} c_{\scriptscriptstyle 1} d_{\scriptscriptstyle 0}}),
$ 
where $\mathbf{s}\scalebox{1.}{$\scriptscriptstyle '$} : = (s_{\scriptscriptstyle 1}\!, \ldots, \, s_{r}),$ is the {D}irichlet polynomial 
defined by 
\begin{equation*} \label{eq: polyPd}
\begin{split}
P_{\scriptscriptstyle d}(\mathbf{s}\scalebox{1.}{$\scriptscriptstyle '$}; 
\chi_{a_{\scriptscriptstyle 1} c_{\scriptscriptstyle 1} d_{\scriptscriptstyle 0}}) \;\; = \!
& \prod_{\substack{p^{\, l} \parallel \, d \\ l \, \equiv \, 1 \!\!\!\!\!\!\!\pmod 2}} \,
P_{\scalebox{.85}{$\scriptscriptstyle l$}}\!\left(|p|^{\, - s_{\scalebox{.8}{$\scriptscriptstyle 1$}}}\!, \ldots, \, |p|^{\, - s_{r}}; \, q^{\deg \, p}\right) \\
& \cdot \prod_{\substack{p \, \mid \, d_{\scriptscriptstyle 1} \\ p^{\, l} \parallel \, d \\ l \, \equiv \, 0 \!\!\!\!\!\!\!\pmod 2}} \,
P_{\scalebox{.85}{$\scriptscriptstyle l$}}\!\left(\chi_{a_{\scriptscriptstyle 1} c_{\scriptscriptstyle 1} d_{\scriptscriptstyle 0}}\!(p) 
|p|^{\, - s_{\scalebox{.8}{$\scriptscriptstyle 1$}}}\!, \ldots, \, 
\chi_{a_{\scriptscriptstyle 1} c_{\scriptscriptstyle 1} d_{\scriptscriptstyle 0}}\!(p)|p|^{\, - s_{r}};\,  q^{\deg \, p}\right)
\end{split}
\end{equation*} 
where $P_{\scalebox{.85}{$\scriptscriptstyle l$}}(\underline{z}; q)$ are the polynomials defined by the condition 1 in Section \ref{General-MDS}. Notice that by \eqref{eq: poly-P-Q-func-eq}, 
$
P_{\scriptscriptstyle d}(\mathbf{s}\scalebox{1.}{$\scriptscriptstyle '$}; 
\chi_{a_{\scriptscriptstyle 1} c_{\scriptscriptstyle 1} d_{\scriptscriptstyle 0}})
$ 
satisfies a\linebreak functional equation in each of the variables $s_{\scriptscriptstyle 1}, \ldots, s_{r}.$ 
Moreover, it is clear that, for fixed 
$
s_{\scriptscriptstyle 1}, \ldots, s_{r} \in \mathbb{C} \setminus \{1\},
$ 
the series \eqref{eq: MDS-vers1} converges absolutely when $s_{r + {\scriptscriptstyle 1}}$ has sufficiently large real part.

We can also write 
\begin{equation} \label{eq: MDS-vers2}
Z^{(c)}(\mathbf{s};  \chi_{a_{\scriptscriptstyle 2} c_{\scriptscriptstyle 2}}, \chi_{a_{\scriptscriptstyle 1} c_{\scriptscriptstyle 1}}) 
\;\; = \sum_{\substack{(m_{\scalebox{.8}{$\scriptscriptstyle 1$}} \cdots \, m_{r}, \, c) = 1 \\ 
		m_{\scalebox{.8}{$\scriptscriptstyle 1$}} \cdots \, m_{r} \, = \, n_{\scriptscriptstyle 0}^{}  n_{\scriptscriptstyle 1}^{2}}}  \, 
\frac{L^{\scriptscriptstyle (c_{\scriptscriptstyle 1} c_{\scriptscriptstyle 3})}
	(s_{r +  {\scalebox{1.}{$\scriptscriptstyle 1$}}}, \chi_{a_{\scriptscriptstyle 2} c_{\scriptscriptstyle 2} n_{\scriptscriptstyle 0}})\, 
	\chi_{a_{\scriptscriptstyle 1}  c_{\scriptscriptstyle 1}}\!(n_{\scriptscriptstyle 0}^{})\,
	Q_{\underline{m}}(s_{r +  {\scalebox{1.}{$\scriptscriptstyle 1$}}}; 
		\chi_{a_{\scriptscriptstyle 2} c_{\scriptscriptstyle 2} n_{\scriptscriptstyle 0}})}
{|m_{\scriptscriptstyle 1}|^{s_{\scalebox{.8}{$\scriptscriptstyle 1$}}} \cdots \, |m_{r}|^{s_{r}}}
\end{equation} 
where, for 
$
\underline{m} = (m_{\scriptscriptstyle 1}, \ldots, \, m_{r}),
$ 
the {D}irichlet polynomial 
$  
Q_{\underline{m}}(s_{r +  {\scalebox{1.}{$\scriptscriptstyle 1$}}}; 
\chi_{a_{\scriptscriptstyle 2} c_{\scriptscriptstyle 2} n_{\scriptscriptstyle 0}})
$ 
is given by 
\begin{equation*} \label{eq: polyQm}
Q_{\underline{m}}(s_{r + \scriptscriptstyle 1}; 
\chi_{a_{\scriptscriptstyle 2}  c_{\scriptscriptstyle 2} n_{\scriptscriptstyle 0}}) \;\; = \!
\prod_{\substack{p^{\, k_{\scriptscriptstyle i}} \parallel \, m_{\scriptscriptstyle i}
		\\ |\underline{k}| \, \equiv \, 1 \!\!\!\!\!\!\!\pmod 2}} 
Q_{\underline{k}}\!\left(|p|^{\, - s_{r + {\scalebox{.8}{$\scriptscriptstyle 1$}}}}; \, q^{\deg \, p}\right) \;\; \cdot
\prod_{\substack{p \, \mid \, n_{\scriptscriptstyle 1} \\ p^{\, k_{\scriptscriptstyle i}} \parallel \, m_{\scriptscriptstyle i} 
		\\ |\underline{k}| \, \equiv \, 0 \!\!\!\!\!\!\pmod 2}}  
Q_{\underline{k}}\!\left(\chi_{a_{\scriptscriptstyle 2} c_{\scriptscriptstyle 2} n_{\scriptscriptstyle 0}}\!(p)
|p|^{\, - s_{r + {\scalebox{.8}{$\scriptscriptstyle 1$}}}}; \, q^{\deg \, p}\right).
\end{equation*} 
Again, by \eqref{eq: poly-P-Q-func-eq}, the polynomials 
$
Q_{\underline{m}}(s_{r + \scriptscriptstyle 1}; 
\chi_{a_{\scriptscriptstyle 2}  c_{\scriptscriptstyle 2} n_{\scriptscriptstyle 0}})
$ 
satisfy a functional equation as $s_{r + \scriptscriptstyle 1} \to 1 - s_{r + \scriptscriptstyle 1},$ and for every 
$s_{r + \scriptscriptstyle 1} \in \mathbb{C} \setminus \{1\},$ the series \eqref{eq: MDS-vers2} converges absolutely 
as long as all $s_{\scriptscriptstyle 1}, \ldots, s_{r}$ have sufficiently large real parts.

As in \cite[3.1]{Dia}, the expressions \eqref{eq: MDS-vers1} and \eqref{eq: MDS-vers2} of 
$
Z^{(c)}(\mathbf{s};  \chi_{a_{\scriptscriptstyle 2} c_{\scriptscriptstyle 2}}, \chi_{a_{\scriptscriptstyle 1} c_{\scriptscriptstyle 1}})
$ 
can be used to show that this family of multiple {D}irichlet series satisfies a group of functional equations isomorphic to 
the {W}eyl group of the Kac-Moody algebra in \ref{KM-alg}. To state explicitly the functional equations corresponding to the generators $w_{\scriptscriptstyle 1}, \ldots, w_{r + {\scriptscriptstyle 1}}$ of $W,$ define $U_{\scriptscriptstyle m}(s_{r + {\scriptscriptstyle 1}}) = 1$ if $m = 1,$ and 
\[
U_{\scriptscriptstyle m}(s_{r + \scriptscriptstyle 1}) = \prod_{p \, \mid \, m} 
\frac{|p|^{s_{r + {\scalebox{.8}{$\scriptscriptstyle 1$}}} - 1}
	\big(1 - |p|^{1 - 2 s_{r + {\scalebox{.8}{$\scriptscriptstyle 1$}}}} \big)}
{1 - |p|^{-  \scriptscriptstyle 1}}
\] 
for (non-constant) monic square free $m \in \mathbb{F}_{\! q}[x],$ the product being over the monic irreducible divisors of $m.$ If we set\linebreak 
\[ 
^{{\scriptscriptstyle w_{\scalebox{.7}{$\scriptscriptstyle i$}}}}\mathbf{s} :
= \big(s_{\scriptscriptstyle 1}, \ldots, 1 - s_{\scriptscriptstyle i}, \ldots, s_{r}, s_{i} + s_{r + \scriptscriptstyle 1} - \tfrac{1}{2}\big)\;\; 
\; \text{for $i\le r,$ and} \;\;\,  
^{{\scriptscriptstyle w_{\scalebox{.85}{$\scriptscriptstyle r$} {\scriptscriptstyle +}  
\scalebox{.70}{$\scriptscriptstyle 1$}}}}\mathbf{s} 
: = \big(s_{\scriptscriptstyle 1} + s_{r + \scriptscriptstyle 1} - \tfrac{1}{2}, \ldots, s_{r} + s_{r + \scriptscriptstyle 1} - \tfrac{1}{2}, 
1 - s_{r + \scriptscriptstyle 1}\big) 
\] 
then the functional equations of $L(s, \chi_{d}),$ 
$P_{\scriptscriptstyle d}(\mathbf{s}\scalebox{1.}{$\scriptscriptstyle '$}; \chi_{a_{\scriptscriptstyle 1} c_{\scriptscriptstyle 1} d_{\scriptscriptstyle 0}}),$ and 
$Q_{\underline{m}}(s_{r +  {\scalebox{1.}{$\scriptscriptstyle 1$}}}; 
\chi_{a_{\scriptscriptstyle 2} c_{\scriptscriptstyle 2} n_{\scriptscriptstyle 0}})$ imply that 
\begin{equation} \label{eq: functional-eq-Z-sigma_r+1}
\begin{split}
& Z^{(c)}(\mathbf{s};  \chi_{a_{\scriptscriptstyle 2} c_{\scriptscriptstyle 2}}, 
\chi_{a_{\scriptscriptstyle 1} c_{\scriptscriptstyle 1}})\,  
= \,  \tfrac{1}{2}\,  |c_{\scriptscriptstyle 2}|^{\frac{1}{2} - s_{r +  {\scalebox{.8}{$\scriptscriptstyle 1$}}}} 
\frac{\varphi(c_{\scriptscriptstyle 1} c_{\scriptscriptstyle 3})}{|c_{\scriptscriptstyle 1} c_{\scriptscriptstyle 3}|}
\prod_{p \, \mid \, c_{\scriptscriptstyle 1}  c_{\scriptscriptstyle 3}} \big(1  -  
|p|^{2 s_{r + {\scalebox{.8}{$\scriptscriptstyle 1$}}} - 2}\big)^{\scriptscriptstyle -1}\\
&\cdot \sum_{\vartheta \in \{1, \, \theta_{\scriptscriptstyle 0} \}}                
\chi_{a_{\scriptscriptstyle 1} \!{\scriptscriptstyle \vartheta}}(c_{\scriptscriptstyle 2})
\!\left\{\gamma_{\scriptscriptstyle q}^{+}(s_{r + {\scriptscriptstyle 1}}; \, a_{\scriptscriptstyle 2})  +  
\mathrm{sgn}(a_{\scriptscriptstyle 1} \!\vartheta)\, \gamma_{\scriptscriptstyle q}^{-}(s_{r + {\scriptscriptstyle 1}})\right\} 
\!\!\!\!\sum_{\substack{m \, \mid \, c_{\scriptscriptstyle 1} c_{\scriptscriptstyle 3} \\ (c_{\scriptscriptstyle 1}\!, \, m) = e}}
\chi_{a_{\scriptscriptstyle 2} c_{\scriptscriptstyle 2}}\!(m) \, 
U_{\scriptscriptstyle m}(s_{r + {\scriptscriptstyle 1}})\,  
Z^{(c)}(^{{\scriptscriptstyle w_{\scalebox{.85}{$\scriptscriptstyle r$} {\scriptscriptstyle +}  
\scalebox{.70}{$\scriptscriptstyle 1$}}}}\mathbf{s}; 
\chi_{a_{\scriptscriptstyle 2} c_{\scriptscriptstyle 2}}, 
\chi_{{\scriptscriptstyle \vartheta} c_{\scriptscriptstyle 1} m \slash e^{\scriptscriptstyle 2}})
\end{split}
\end{equation} 
and 
\begin{equation} \label{eq: functional-eq-Z-sigma1}
\begin{split}
& Z^{(c)}(\mathbf{s};  \chi_{a_{\scriptscriptstyle 2} c_{\scriptscriptstyle 2}}, 
\chi_{a_{\scriptscriptstyle 1} c_{\scriptscriptstyle 1}})\, = \, \tfrac{1}{2}\,  
|c_{\scriptscriptstyle 1}|^{\frac{1}{2} - s_{\scalebox{.8}{$\scriptscriptstyle 1$}}}
\frac{\varphi(c_{\scriptscriptstyle 2} c_{\scriptscriptstyle 3})}{|c_{\scriptscriptstyle 2} c_{\scriptscriptstyle 3}|}
\prod_{p \, \mid \, c_{\scriptscriptstyle 2}  c_{\scriptscriptstyle 3}} 
\big(1  -  |p|^{2 s_{\scalebox{.8}{$\scriptscriptstyle 1$}} - 2}\big)^{\scriptscriptstyle -1}\\ 
& \cdot \sum_{\vartheta' \in \{1, \, \theta_{\scriptscriptstyle 0} \}} 
\chi_{a_{\scriptscriptstyle 2} {\scriptscriptstyle \vartheta'}}(c_{\scriptscriptstyle 1})
\!\left\{\gamma_{\scriptscriptstyle q}^{+}(s_{\scriptscriptstyle 1}; \, a_{\scriptscriptstyle 1}) +  
\mathrm{sgn}(a_{\scriptscriptstyle 2} \vartheta')\, 
\gamma_{\scriptscriptstyle q}^{-}(s_{\scriptscriptstyle 1})\right\} 
\!\!\!\!\sum_{\substack{\ell \, \mid \, c_{\scriptscriptstyle 2} c_{\scriptscriptstyle 3} \\ (c_{\scriptscriptstyle 2}, \, \ell) = b}}
\chi_{a_{\scriptscriptstyle 1} c_{\scriptscriptstyle 1}}\!(\ell) \, 
U_{ \scriptscriptstyle \ell}(s_{\scriptscriptstyle 1}) 
Z^{(c)}(^{{\scriptscriptstyle w_{\scalebox{.70}{$\scriptscriptstyle 1$}}}}\mathbf{s}; 
\chi_{{\scriptscriptstyle \vartheta'} \!c_{\scriptscriptstyle 2} \ell \slash b^{\scriptscriptstyle 2}}, 
\chi_{a_{\scriptscriptstyle 1} c_{\scriptscriptstyle 1}})
\end{split}
\end{equation} 
where $\varphi$ is Euler's totient function over $\mathbb{F}_{\! q}[x].$ Since 
$
Z^{(c)}(\mathbf{s};  \chi_{a_{\scriptscriptstyle 2} c_{\scriptscriptstyle 2}}, 
\chi_{a_{\scriptscriptstyle 1} c_{\scriptscriptstyle 1}})
$ 
is symmetric in the variables $s_{\scriptscriptstyle 1}, \ldots, s_{r},$ we have similar functional equations 
$w_{\scriptscriptstyle 2}, \ldots, w_{r}$ in the variables $s_{\scriptscriptstyle 2}, \ldots, s_{r},$ respectively.

\subsection{Analytic continuation} \label{Conjecture-MDS-completed-imag-roots}
For the purpose of the present discussion, it will be convenient to substitute 
$
s_{\scriptscriptstyle i} \! = \fs_{\scriptscriptstyle i} +  \frac{1}{2}, 
$ 
for $i = 1, \ldots, r+1.$

We recall from Section \ref{KM-alg} that the Weyl-Kac denominator 
$
\prod_{\alpha \in \Delta_{+}}(1 - e^{\scriptscriptstyle -\alpha})^{m_{\scalebox{.85}{$\scriptscriptstyle \alpha$}}}
$ 
(where $e^{\scriptscriptstyle \lambda}(h) = e^{\scriptscriptstyle \lambda(h)}\!,$ for $\lambda \in \fh^{*}$ \!and $h \in \fh$), is a holomorphic function on the interior of the complexified Tits cone $\mathbf{X}_{\scriptscriptstyle \mathbb{C}}.$ Substituting 
$
e^{\scriptscriptstyle - \alpha_{\scalebox{.7}{$\scriptscriptstyle i$}}(h)} 
\mapsto q^{{\scriptscriptstyle -2} \fs_{\scalebox{.7}{$\scriptscriptstyle i$}}}\!,
$ 
for $i = 1,\ldots,  r + 1,$\linebreak 
the Weyl-Kac denominator becomes 
$
\prod_{\alpha \in \Delta_{+}}\!\left(1 - q^{\scriptscriptstyle - 2\alpha(\fs)}\right)^{m_{\scalebox{.85}{$\scriptscriptstyle \alpha$}}}\!,
$ 
where, for 
$
\alpha \, = \sum k_{\scriptscriptstyle i}\alpha_{\scriptscriptstyle i} \in \Delta
$ 
and 
$ 
\fs = (\fs_{\scriptscriptstyle 1}, \ldots, \fs_{r + {\scriptscriptstyle 1}}) 
\in \mathbb{C}^{r + {\scriptscriptstyle 1}}\!,
$ 
we set 
$
\alpha(\fs) : = \sum k_{\scriptscriptstyle i}\fs_{\scriptscriptstyle i};$ 
it is holomorphic on $X_{\scalebox{1.1}{$\scriptscriptstyle 0$}}^{\ast}:$ the interior of the corresponding complexified convex cone. 
Note that, for $\epsilon > 0,$\linebreak $(0, \ldots, 0, \epsilon) \in X_{\scalebox{1.1}{$\scriptscriptstyle 0$}}^{\ast}\!.$

Let 
$
\mathrm{Q}^{\scalebox{.95}{$\scriptscriptstyle \mathrm{ev}$}}_{+} 
\subset \sum \mathbb{N}\alpha_{\scriptscriptstyle i}
$ 
consist of the elements $\alpha = \sum k_{\scriptscriptstyle i}\alpha_{\scriptscriptstyle i}$ with both 
$
k_{\scriptscriptstyle 1} + \cdots + k_{r}
$ 
and $k_{r + {\scriptscriptstyle 1}}$ {\it non-zero even} natural numbers. \!For a real 
parameter $t > 1,$ let $\mathscr{S}_{\scriptscriptstyle W}$ consist of the absolutely convergent series on $X_{\scalebox{1.1}{$\scriptscriptstyle 0$}}^{\ast}\!,$ 
\[
1 \;  + \sum_{\alpha \in \mathrm{Q}^{\scalebox{.95}{$\scriptscriptstyle \mathrm{ev}$}}_{+} \setminus \{ \bf{0} \}}  
\! f_{\scriptscriptstyle \alpha}(t)t^{\scalebox{1.2}{$\scriptscriptstyle - \alpha(\scalebox{1.15}{$\scriptscriptstyle \fs$})$}}
\] 
with
$
t^{\scriptscriptstyle d(\alpha)\slash 2}\!f_{\scriptscriptstyle
  \alpha}(t) \in \mathbb{Q}[t]
$
of degree $\le (d(\alpha) - 2)\slash 2,$
$
f_{\scriptscriptstyle \alpha}(t) \ll_{\varepsilon} 
t^{{\scriptscriptstyle - 1} + \varepsilon \scalebox{1.2}{$\scriptscriptstyle d(\alpha)$}}
$ 
for every small $\varepsilon  > 0,$ which are $W$-invariant. \!Here 
$
\alpha(\fs) = \sum k_{\scriptscriptstyle i}\fs_{\scriptscriptstyle i}, 
$
for $\alpha \, = \sum k_{\scriptscriptstyle i}\alpha_{\scriptscriptstyle i}.$
Finally, let $\mathrm{D}^{\scriptscriptstyle \mathrm{re}}\!(\fs)$ be
defined by
\[ 
\mathrm{D}^{\scriptscriptstyle \mathrm{re}}\!(\fs) 
\, = \prod_{\alpha \in \Delta^{\mathrm{re}}_{+}}
\left(1 - q^{\scalebox{1.2}{$\scriptscriptstyle 1 - 2\alpha(\scalebox{1.15}{$\scriptscriptstyle \mathbf{\fs}$})$}}\right).
\] 
\begin{conj} \label{Conjecture-Meromorphic continuation} --- 
	Let 
	$
	a_{\scriptscriptstyle 1}, \, a_{\scriptscriptstyle 2} \in \{1, \theta_{\scriptscriptstyle 0} \},
	$ 
	$c \in \mathbb{F}_{\! q}[x]$ monic and square free, and write $c = c_{1}c_{2}c_{3},$ with $c_{\scriptscriptstyle i},$ 
	$i = 1, 2, 3,$ monic. 
	Then there exists 
	$
	\mathbf{I}(\fs, t) \in \mathscr{S}_{\scriptscriptstyle W},
	$ 
	symmetric in $\fs_{\scriptscriptstyle 1}, \ldots, \fs_{r},$ and 
	non-vanishing on $X_{\scalebox{1.1}{$\scriptscriptstyle 0$}}^{\ast}$ for all $t > 1,$ such that the function 
	\[
	\prod_{p \nmid c} \mathbf{I}\left(\fs, |p|\right)
	\, \cdot \, \mathrm{D}^{\scriptscriptstyle \mathrm{re}}\!(\fs) 
	\, Z^{(c)}\left(\fs + \mathbf{\tfrac{1}{2}};  \chi_{a_{\scriptscriptstyle 2} c_{\scriptscriptstyle 2}}, \chi_{a_{\scriptscriptstyle 1} c_{\scriptscriptstyle 1}}\!\right), \quad \text{$\mathbf{\tfrac{1}{2}} : = (\tfrac{1}{2}, \ldots, \tfrac{1}{2}) \in 
	\mathbb{C}^{r + {\scriptscriptstyle 1}}$}
	\] 
	is holomorphic on $X_{\scalebox{1.1}{$\scriptscriptstyle 0$}}^{\ast}\!;$ the product $\prod_{p} \mathbf{I}\left(\fs, |p|\right)$ 
	converges absolutely to a holomorphic function on $X_{\scalebox{1.1}{$\scriptscriptstyle 0$}}^{\ast}\!.$ 
\end{conj} 

\vskip10pt
 \begin{rem} --- When $r \ge 5,$ we do not
   have any prediction for the
   series $\mathbf{I}(\fs, t)$ in the conjecture.
   If $r = 4$ and $c = 1,$ one can take $\mathbf{I}(\fs, t) \equiv 1$
   (see \cite{DPP2}), and we conjecture that the same should be true for
   arbitrary $c.$ 
\end{rem} 
From now on we shall assume that Conjecture \ref{Conjecture-Meromorphic continuation} holds. Since 
$
\prod_{p \nmid c} \mathbf{I}\left(\fs, |p|\right)
$ 
is a {D}irichlet series supported on tuples $(m_{\scriptscriptstyle 1}, \ldots, m_{r}, d)$ 
of monic polynomials with $m_{\scriptscriptstyle 1} \cdots \, m_{r}$ and $d$ squares, 
we can express 
$ 
\prod_{p \nmid c} \mathbf{I}\left(\mathbf{s} - \mathbf{\tfrac{1}{2}}, |p|\right)
 \cdot Z^{(c)}(\mathbf{s};  \chi_{a_{\scriptscriptstyle 2} c_{\scriptscriptstyle 2}}, \chi_{a_{\scriptscriptstyle 1} c_{\scriptscriptstyle 1}})
$ 
as a series of the form \eqref{eq: MDS-definition-vers0}, for some multiplicative coefficients 
$
\tilde{A}(m_{\scriptscriptstyle 1}, \ldots, m_{r}, d)
$ 
satisfying 
\[
\tilde{A}\big(p^{\, k_{\scalebox{.8}{$\scriptscriptstyle 1$}}}\!, \ldots, p^{\, k_{r}}\!, p^{\, l} \big) 
= A\big(p^{\, k_{\scalebox{.8}{$\scriptscriptstyle 1$}}}\!, \ldots, p^{\, k_{r}}\!, p^{\, l} \big) = 1 
\] 
for all monic irreducibles $p,$ when either $k_{\scriptscriptstyle 1}
= \cdots = k_{r} \! = 0$ or $l = 0.$

In what follows, we shall denote
$
\prod_{p \nmid c} \mathbf{I}\left(\mathbf{s} - \mathbf{\tfrac{1}{2}}, |p|\right)
\cdot Z^{(c)}\!\left(\mathbf{s};
 \chi_{a_{\scriptscriptstyle 2} c_{\scriptscriptstyle 2}},
 \chi_{a_{\scriptscriptstyle 1} c_{\scriptscriptstyle 1}}\!\right)
$
by
$
\widetilde{Z}^{(c)}\!\left(\mathbf{s};
 \chi_{a_{\scriptscriptstyle 2} c_{\scriptscriptstyle 2}},
 \chi_{a_{\scriptscriptstyle 1} c_{\scriptscriptstyle 1}}\!\right).
$
Since $\mathbf{I}(\fs, t)$ is $W$-invariant, the functions
$
Z^{(c)}\!\left(\mathbf{s};
 \chi_{a_{\scriptscriptstyle 2} c_{\scriptscriptstyle 2}},
 \chi_{a_{\scriptscriptstyle 1} c_{\scriptscriptstyle 1}}\!\right)
$
and
$
\widetilde{Z}^{(c)}\!\left(\mathbf{s};
 \chi_{a_{\scriptscriptstyle 2} c_{\scriptscriptstyle 2}},
 \chi_{a_{\scriptscriptstyle 1} c_{\scriptscriptstyle 1}}\!\right)
$
satisfy the same functional equations.

\subsection{Generating functions for moments of $L$-functions} \label{introref-generating-function-moments} 
For
$
\mathbf{s} = (s_{\scriptscriptstyle 1}, \ldots, s_{r + {\scriptscriptstyle 1}})
\in \mathbb{C}^{r + {\scriptscriptstyle 1}}
$
with $\Re(s_{\scriptscriptstyle k}) > 1,$ and
$
a_{\scriptscriptstyle 2} \in \{1, \theta_{\scriptscriptstyle 0}\},
$
define
\[
 Z_{\scriptscriptstyle 0}(\mathbf{s}, \chi_{a_{\scalebox{.7}{$\scriptscriptstyle 2$}}})
 \;\; = 
 \sum_{d - \mathrm{monic \; \& \; sq. \; free}}\;\,
 \prod_{k = 1}^{r} L(s_{\scriptscriptstyle k}, \chi_{d})
 \, \cdot \, \chi_{a_{\scalebox{.7}{$\scriptscriptstyle 2$}}}\!(d)\,
|d|^{-s_{r + {\scalebox{.8}{$\scriptscriptstyle 1$}}}}
\]
and, for $h$ square-free monic, put
\[
Z(\mathbf{s}, \chi_{a_{\scalebox{.7}{$\scriptscriptstyle 2$}}}\!; h)\;\;  = 
\sum_{\substack{m_{\scalebox{.8}{$\scriptscriptstyle 1$}}\!, \ldots, \, m_{r}, \, d - \mathrm{monic} \\  
d = d_{\scriptscriptstyle 0}^{} d_{\scriptscriptstyle 1}^{2}, \;
d_{\scriptscriptstyle 0}^{} \; \mathrm{square \; free} \\ d_{\scriptscriptstyle 1}^{} \equiv \, 0 \!\!\!\!\pmod h}} \, 
\frac{\chi_{d_{\scriptscriptstyle 0}}\!(\widehat{m}_{\scriptscriptstyle 1})\, \cdots \, 
  \chi_{d_{\scriptscriptstyle 0}}\! (\widehat{m}_{r})\,
\chi_{a_{\scalebox{.7}{$\scriptscriptstyle 2$}}}\!(d_{\scriptscriptstyle 0})}
{|m_{\scriptscriptstyle 1}|^{s_{\scalebox{.8}{$\scriptscriptstyle 1$}}}
 \cdots \, |m_{r}|^{s_{r}} |d|^{s_{r + {\scalebox{.8}{$\scriptscriptstyle 1$}}}}}
\cdot \tilde{A}(m_{\scriptscriptstyle 1}, \ldots, m_{r}, d).
\]
Recalling that the coefficients
$
\tilde{A}(m_{\scriptscriptstyle 1}, \ldots, m_{r}, d)
$
are multiplicative,
$
\tilde{A}\big(p^{\, k_{\scalebox{.8}{$\scriptscriptstyle 1$}}}\!,
\ldots, p^{\, k_{r}}\!, 1\big) = 1, 
$
and
\begin{equation*}
 \tilde{A}\big(p^{\, k_{\scalebox{.8}{$\scriptscriptstyle 1$}}}\!,
\ldots, p^{\, k_{r}}\!, p \big) = 
\begin{cases}
1 & \mbox{if}  \;\,     k_{\scriptscriptstyle 1} = \cdots = k_{r} = 0\\ 
0 & \mbox{otherwise}
\end{cases}
\end{equation*}
for every monic irreducible $p,$ we can write (see
\cite[Lemma~5.1]{Dia})
\begin{equation} \label{eq: sum-sq-free-vs-MDS}
 Z_{\scriptscriptstyle 0}(\mathbf{s}, \chi_{a_{\scalebox{.7}{$\scriptscriptstyle 2$}}})\;\, = 
 \sum_{h - \mathrm{monic}}\,  \mu(h)
 Z(\mathbf{s}, \chi_{a_{\scalebox{.7}{$\scriptscriptstyle 2$}}}\!; h).
\end{equation}
Here $\mu(h)$ is the M\"obius function defined for non-zero
polynomials over $\mathbb{F}_{\! q}$ by $\mu(h) = (-
1)^{\scriptscriptstyle \omega(h)}$ if $h$ is square-free, and $h$ is a
constant times a product of $\omega(h)$ distinct monic irreducibles,
and $\mu(h) = 0$ if $h$ is {\it not} square-free; it is understood
that $\mu(h) = 1$ if $h \in \mathbb{F}_{\! q}^{\times}.$

We can express
$
Z(\mathbf{s}, \chi_{a_{\scalebox{.7}{$\scriptscriptstyle 2$}}}\!; h)
$
in terms of the functions
$
\widetilde{Z}^{(c)}\!\left(\mathbf{s};
 \chi_{a_{\scalebox{.7}{$\scriptscriptstyle 2$}} c_{\scalebox{.7}{$\scriptscriptstyle 2$}}},
 \chi_{a_{\scalebox{.7}{$\scriptscriptstyle 1$}} c_{\scalebox{.7}{$\scriptscriptstyle 1$}}}\!\right)
$
as follows. \!Set
$ 
\tilde{f}(\mathrm{z}; q) : =
\mathbf{I}(\sqrt{q}\, \mathrm{z}, q) f(\mathrm{z}; q),
$
where
$ 
\mathbf{I}(\sqrt{q}\, \mathrm{z}, q)
$
denotes the function
$
\mathbf{I}\left(\mathbf{s} - \mathbf{\tfrac{1}{2}}, q\right),
$
with $q^{\, - s_{\scalebox{.7}{$\scriptscriptstyle i$}}}$
replaced by $z_{\scriptscriptstyle i},$ $i = 1, \ldots, r + 1,$ and
$f(\mathrm{z}; q)$ is the power series defined in 3.2. Then
$\tilde{f}(\mathrm{z}; q)$ satisfies the conditions 1--5 in Section
\ref{General-MDS}, and if we write
\[
\tilde{f}(\mathrm{z}; q) = 1 + \sum_{\mathrm{k}\ne \bf{0}}
\tilde{a}(\mathrm{k}; q)\mathrm{z}^{\mathrm{k}}
\]
the coefficients $\tilde{a}(\mathrm{k}; q)$ satisfy the estimate
\eqref{eq: Estimate-H1}. Set
$
\tilde{f}_{\!\scriptscriptstyle \mathrm{even}}(\mathrm{z}; q) 
= \mathbf{I}(\sqrt{q}\, \mathrm{z}, q)
f_{\!\scriptscriptstyle \mathrm{even}}(\mathrm{z}; q),
$
$
\tilde{f}_{\!\scriptscriptstyle \mathrm{odd}}(\mathrm{z}; q) 
= \mathbf{I}(\sqrt{q}\, \mathrm{z}, q)
f_{\!\scriptscriptstyle \mathrm{odd}}(\mathrm{z}; q),
$
and let
$ 
\tilde{{\bf{f}}}(\mathrm{z}; q)
\, = \, ^{t}\!\big(\tilde{f}_{\!\scriptscriptstyle \mathrm{even}}(\underline{z}, z_{r + {\scriptscriptstyle 1}}; q),
\tilde{f}_{\!\scriptscriptstyle \mathrm{odd}}(\underline{z}, z_{r + {\scriptscriptstyle 1}}; q),
\tilde{f}_{\!\scriptscriptstyle \mathrm{even}}(- \underline{z}, z_{r +
  {\scriptscriptstyle 1}}; q)\big).
$ 

We have the following: 

\vskip5pt
\begin{lem}\label{Local-vect-estimate} --- For
  $\sigma > \frac{1}{2},$ and sufficiently small
  $\varepsilon > 0,$ we have
\[ 
\tilde{{\bf{f}}}(\mathrm{z}; q)
\; =\,
^{^{^{t}}}\!\!\bigg(
\prod_{j = 1}^{r}\left(1 - z_{\!\scriptscriptstyle j}\right)^{\scriptscriptstyle -1}\!, \, 
z_{r + {\scriptscriptstyle 1}}, 
\prod_{j = 1}^{r}\left(1 + z_{\!\scriptscriptstyle
    j}\right)^{\scriptscriptstyle -1}\bigg)
\; + \; O_{\varepsilon, \, \sigma}
\left(q^{{\scriptscriptstyle -
    \scalebox{.9}{$\scriptscriptstyle 2$}\sigma} +
  \scalebox{.9}{$\scriptscriptstyle (4$}r +
\scalebox{.9}{$\scriptscriptstyle 2)$}
{\scriptscriptstyle
  \varepsilon}}\right)
\]
in the polydisk
$
|z_{\scriptscriptstyle i}| \le q^{- \frac{1}{2} + \varepsilon},
$
$i = 1, \ldots, r,$ and $|z_{r + {\scriptscriptstyle 1}}| \le
q^{-\sigma}.$
\end{lem}

\begin{proof} \!Write
\[
\prod_{j = 1}^{r} (1 - z_{\scriptscriptstyle j})
\, \cdot \,
\tilde{f}_{\!\scriptscriptstyle \mathrm{even}}(\mathrm{z}; q)
= 1 \;\; + \!\sum_{\substack{\mathrm{k} \\
    k_{r +  \scalebox{.8}{$\scriptscriptstyle 1$}} \ge 2 \\
    k_{r +  \scalebox{.8}{$\scriptscriptstyle 1$}}
    - \, \text{even}}}
\tilde{b}(\mathrm{k}; q)\mathrm{z}^{\mathrm{k}}
= 1 \; + \; \prod_{j = 1}^{r} (1 - z_{\scriptscriptstyle j})
\;\; \cdot \! \sum_{\substack{\mathrm{k} \\
    k_{r +  \scalebox{.8}{$\scriptscriptstyle 1$}} \ge 2
    \\ k_{r +  \scalebox{.8}{$\scriptscriptstyle 1$}}
- \, \text{even}}}
\tilde{a}(\mathrm{k}; q)\mathrm{z}^{\mathrm{k}}
\]
and
\[
\tilde{f}_{\!\scriptscriptstyle \mathrm{odd}}(\mathrm{z}; q)
= z_{r + {\scriptscriptstyle 1}} \;\,  +
\!\sum_{\substack{\mathrm{k}\\
k_{r +  \scalebox{.8}{$\scriptscriptstyle 1$}} \ge 3 \\ k_{r +  \scalebox{.8}{$\scriptscriptstyle 1$}}
  - \, \text{odd}}}
\tilde{a}(\mathrm{k}; q)\mathrm{z}^{\mathrm{k}}.
\]
Since the coefficients $\tilde{a}(\mathrm{k}; q)$
satisfy the estimate \eqref{eq: Estimate-H1},
$ 
\tilde{b}(\mathrm{k}; q) \ll_{\varepsilon} 
q^{\left(\!\scalebox{1.15}{$\scriptscriptstyle \frac{1}{2}$} +
    \scalebox{1.35}{$\scriptscriptstyle
      \varepsilon$}\!\right)\scalebox{1.3}{$\scriptscriptstyle
    |\mathrm{k}|$} - 1}.
$
Moreover, since
$
\prod_{j = 1}^{r} (1 - z_{\scriptscriptstyle j})
\, \cdot \,
\tilde{f}_{\!\scriptscriptstyle \mathrm{even}}(\mathrm{z}; q)
$
is invariant under $w_{\scriptscriptstyle i},$ for
$1\le i \le r,$ the coefficients
$
\tilde{b}(\mathrm{k}; q)
$
vanish unless
$
\mathrm{k} = (k_{\scriptscriptstyle 1}, \ldots, k_{r},
k_{r + {\scriptscriptstyle 1}}) 
$
is such that
$
k_{\scriptscriptstyle i}
\le k_{r + {\scriptscriptstyle 1}} 
$
for all $i.$ Similarly, from the functional equations
$
\tilde{f}_{\!\scriptscriptstyle \mathrm{odd}}(\mathrm{z}; q)= (\sqrt{q}\,z_{\scriptscriptstyle i})^{\scriptscriptstyle - 1}
\!\tilde{f}_{\!\scriptscriptstyle \mathrm{odd}}
(w_{\scriptscriptstyle i}\mathrm{z}; q),
$
for $1\le i \le r,$ we deduce that the coefficients of
$
\tilde{f}_{\!\scriptscriptstyle \mathrm{odd}}(\mathrm{z}; q)
$
vanish unless
$
\mathrm{k} = (k_{\scriptscriptstyle 1}, \ldots, k_{r},
k_{r + {\scriptscriptstyle 1}}) 
$
satisfies the condition
$
k_{\scriptscriptstyle i}
\le k_{r + {\scriptscriptstyle 1}} - 1
$
for all $i = 1, \ldots, r.$ Accordingly, for
$
|z_{\scriptscriptstyle i}| \le q^{- \frac{1}{2} + \varepsilon}
$
($i = 1, \ldots, r$), $|z_{r + {\scriptscriptstyle 1}}| \le
q^{-\sigma},$ and $\varepsilon >0$ such that
$
1 - 2\sigma + (4r +2)\varepsilon < 0,
$
\begin{equation*}
\begin{split}
\left|
\prod_{j = 1}^{r} (1 - z_{\scriptscriptstyle j})
\, \cdot \,
\tilde{f}_{\!\scriptscriptstyle \mathrm{even}}(\mathrm{z}; q)
\, - \, 1\right| \, &\ll_{\varepsilon}\,  
q^{\scriptscriptstyle -1} \, \cdot \sum_{\substack{
    k_{r +  \scalebox{.8}{$\scriptscriptstyle 1$}} \ge 2 \\
    k_{r +  \scalebox{.8}{$\scriptscriptstyle 1$}}
    - \, \text{even}}}
q^{\left(\!{\scriptscriptstyle \frac{1}{2} - \sigma
    + \varepsilon}\!\right)\scalebox{1.2}{$\scriptscriptstyle k_{r +  \scalebox{.8}{$\scriptscriptstyle 1$}}$}}
\sum_{\underline{k}\,:\,
k_{\scalebox{.83}{$\scriptscriptstyle i$}}
\le k_{r +  \scalebox{.8}{$\scriptscriptstyle 1$}}}
q^{{\scriptscriptstyle \scalebox{.9}{$\scriptscriptstyle 2$}\varepsilon}
  \scalebox{1.15}{$\scriptscriptstyle
 |\underline{k}|$}}\\
&\ll_{\varepsilon}\,  q^{\scriptscriptstyle - 1}
(1 - q^{{\scriptscriptstyle - \scalebox{.9}{$\scriptscriptstyle 2$}\varepsilon}})^{- \scalebox{1.1}{$\scriptscriptstyle r$}}\, \cdot \sum_{\substack{
    k_{r +  \scalebox{.8}{$\scriptscriptstyle 1$}} \ge 2 \\
    k_{r +  \scalebox{.8}{$\scriptscriptstyle 1$}}
    - \, \text{even}}}
q^{\left[{\scriptscriptstyle \frac{1}{2} -
\sigma} + \scalebox{.9}{$\scriptscriptstyle (2$}r +
\scalebox{.9}{$\scriptscriptstyle 1)$}
{\scriptscriptstyle
  \varepsilon}\right]\scalebox{1.2}{$\scriptscriptstyle k_{r +
  \scalebox{.8}{$\scriptscriptstyle 1$}}$}}\\
&\ll_{\varepsilon}\,
q^{{\scriptscriptstyle -
    \scalebox{.9}{$\scriptscriptstyle 2$}\sigma} +
  \scalebox{.9}{$\scriptscriptstyle (4$}r +
\scalebox{.9}{$\scriptscriptstyle 2)$}
{\scriptscriptstyle
  \varepsilon}}
(1 - q^{{\scriptscriptstyle - \scalebox{.9}{$\scriptscriptstyle
      2$}\varepsilon}})^{- \scalebox{1.1}{$\scriptscriptstyle r$}}
(1 - q^{{\scriptscriptstyle \scalebox{.9}{$\scriptscriptstyle 1$} -
\scalebox{.9}{$\scriptscriptstyle 2$}\sigma} + \scalebox{.9}{$\scriptscriptstyle (4$}r +
\scalebox{.9}{$\scriptscriptstyle 2)$}
{\scriptscriptstyle
  \varepsilon}})^{\scriptscriptstyle - 1}\\
&\ll_{\varepsilon,\,  \sigma}\,
q^{{\scriptscriptstyle -
    \scalebox{.9}{$\scriptscriptstyle 2$}\sigma} +
  \scalebox{.9}{$\scriptscriptstyle (4$}r +
\scalebox{.9}{$\scriptscriptstyle 2)$}
{\scriptscriptstyle
  \varepsilon}}.
\end{split}
\end{equation*}
Similarly,
\[
\left|
\tilde{f}_{\!\scriptscriptstyle \mathrm{odd}}(\mathrm{z}; q)
\, - \, z_{r + {\scriptscriptstyle 1}}\right| \, \ll_{\varepsilon}\, q^{\scriptscriptstyle -1} \, \cdot \sum_{\substack{
    k_{r +  \scalebox{.8}{$\scriptscriptstyle 1$}} \ge 3 \\
    k_{r +  \scalebox{.8}{$\scriptscriptstyle 1$}}
    - \, \text{odd}}}
q^{\left(\!{\scriptscriptstyle \frac{1}{2} - \sigma
    + \varepsilon}\!\right)\scalebox{1.2}{$\scriptscriptstyle k_{r +  \scalebox{.8}{$\scriptscriptstyle 1$}}$}}
\sum_{\underline{k}\,:\,
k_{\scalebox{.83}{$\scriptscriptstyle i$}}
\le k_{r +  \scalebox{.8}{$\scriptscriptstyle 1$}} - 1}
q^{{\scriptscriptstyle \scalebox{.9}{$\scriptscriptstyle 2$}\varepsilon}
  \scalebox{1.15}{$\scriptscriptstyle
    |\underline{k}|$}}
\, \ll_{\varepsilon,\,  \sigma}\,
q^{{\scriptscriptstyle \frac{1}{2} -
    \scalebox{.9}{$\scriptscriptstyle 3$}\sigma} +
  \scalebox{.9}{$\scriptscriptstyle (4$}r +
\scalebox{.9}{$\scriptscriptstyle 3)$}
{\scriptscriptstyle
  \varepsilon}}
\]
which completes the proof.
\end{proof}

If we now define,
\[ 
F\!(\mathrm{z}; q) : = 
z_{r +  {\scriptscriptstyle 1}}^{\scriptscriptstyle - 3}\,
\tilde{f}_{\scriptscriptstyle \mathrm{odd}}(\underline{z}, z_{r +  {\scriptscriptstyle 1}}; q) 
\, - \, z_{r +  {\scriptscriptstyle 1}}^{\scriptscriptstyle - 2}
\]
and, for $a\in \{0, 1\},$
\[
G^{\scriptscriptstyle (a)}(\mathrm{z}; q) : = 
\frac{1}{2}
\left\{\tilde{f}_{\scriptscriptstyle \mathrm{even}}(\underline{z}, 
z_{r +  {\scriptscriptstyle 1}}; q) 
\, - \, \prod_{k = 1}^{r}(1 - z_{\scriptscriptstyle k})^{\scriptscriptstyle -1} \right\}
z_{r +  {\scriptscriptstyle 1}}^{\scriptscriptstyle - 2} 
 \, + \, \frac{(- 1)^{a + {\scriptscriptstyle 2}}}{2}
\left\{\tilde{f}_{\scriptscriptstyle \mathrm{even}}(- \, \underline{z}, z_{r +  {\scriptscriptstyle 1}}; q) 
\, - \, \prod_{k = 1}^{r}(1 + z_{\scriptscriptstyle k})^{\scriptscriptstyle -1} \right\}
z_{r +  {\scriptscriptstyle 1}}^{\scriptscriptstyle - 2}
\]
then, as in \cite[Section~5]{Dia}, we can write
\begin{equation} \label{eq: fundamental-eq-h}
\begin{split}
Z(\mathbf{s}, \chi_{a_{\scalebox{.7}{$\scriptscriptstyle 2$}}}\!; h) = \,
|h|^{- 2s_{r + {\scalebox{.8}{$\scriptscriptstyle 1$}}}} \!\!
\sum_{h =
  c_{\scalebox{.7}{$\scriptscriptstyle 1$}}
  c_{\scalebox{.7}{$\scriptscriptstyle 2$}}
  c_{\scalebox{.7}{$\scriptscriptstyle 3$}}} \;
&\chi_{a_{\scalebox{.7}{$\scriptscriptstyle 2$}}
  c_{\scalebox{.7}{$\scriptscriptstyle 2$}}}
\!(c_{\scriptscriptstyle 1})
\widetilde{Z}^{(h)}(\mathbf{s};  \chi_{a_{\scalebox{.7}{$\scriptscriptstyle 2$}}
  c_{\scalebox{.7}{$\scriptscriptstyle 2$}}},
\chi_{c_{\scalebox{.7}{$\scriptscriptstyle 1$}}}) 
\prod_{p \, \mid \, c_{\scalebox{.7}{$\scriptscriptstyle 1$}}}
F(|p|^{\, - s_{\scalebox{.8}{$\scriptscriptstyle 1$}}}\!, \ldots, \, |p|^{\, - s_{r + {\scalebox{.8}{$\scriptscriptstyle 1$}}}}; \, 
|p|)\,  |p|^{\, - s_{r + {\scalebox{.8}{$\scriptscriptstyle 1$}}}}\\
& \cdot \, \prod_{p \, \mid \, c_{\scalebox{.7}{$\scriptscriptstyle
      2$}}}
G^{\scriptscriptstyle (1)}(|p|^{\, - s_{\scalebox{.8}{$\scriptscriptstyle 1$}}}\!,
\ldots, \, |p|^{\, - s_{r + {\scalebox{.8}{$\scriptscriptstyle 1$}}}};
\, |p|)
\prod_{p \, \mid \, c_{\scalebox{.7}{$\scriptscriptstyle 3$}}}
G^{\scriptscriptstyle (0)}(|p|^{\, - s_{\scalebox{.8}{$\scriptscriptstyle 1$}}}\!,
\ldots, \, |p|^{\, - s_{r + {\scalebox{.8}{$\scriptscriptstyle 1$}}}}; \, |p|).
\end{split}
\end{equation}

Conjecture \ref{Conjecture-Meromorphic continuation}
combined with the following conjecture gives the meromorphic
continuation of the function
$
Z_{\scriptscriptstyle 0}(\mathbf{s},
\chi_{a_{\scalebox{.7}{$\scriptscriptstyle 2$}}}).
$ 

\vskip10pt
\begin{conj} \label{Conjecture-Meromorphic continuation-Z0} --- 
  The series
  \[ 
 \sum_{h - \mathrm{monic}}\,  \mu(h)
 Z(\mathbf{s}, \chi_{a_{\scalebox{.7}{$\scriptscriptstyle 2$}}}\!; h) 
\]
is absolutely convergent for
$
 \Re(s_{\scriptscriptstyle i}) \ge \frac{1}{2},
$
$i = 1, \ldots, r,$ and
$\Re(s_{r + {\scriptscriptstyle 1}}) > \frac{1}{2},$ away from the
zero-set of
$
\mathrm{D}^{\scriptscriptstyle \mathrm{re}}\!\left(\mathbf{s} -
  \mathbf{\tfrac{1}{2}}\right).
$
\end{conj} 

\vskip10pt
In the next sections, we shall see how the asymptotics for
moments of L-functions, summed over monic square-free
polynomials of fixed degree, can be obtained from
Assumption \eqref{eq: Estimate-H1}, and Conjectures
\ref{Conjecture-Meromorphic continuation},
\ref{Conjecture-Meromorphic continuation-Z0}.

\section{Poles and residues}  \label{Poles-residues-ref-section}
Our first goal in this section is to compute the residues at the poles
of the functions
$
\widetilde{Z}^{(c)}\!\left(\mathbf{s};
 \chi_{a_{\scriptscriptstyle 2} c_{\scriptscriptstyle 2}},
 \chi_{a_{\scriptscriptstyle 1} c_{\scriptscriptstyle 1}}\!\right)
$
introduced at the end of \ref{Conjecture-MDS-completed-imag-roots}. By
Conjecture \ref{Conjecture-Meromorphic continuation}, the poles of
$
\widetilde{Z}^{(c)}\!\left(\mathbf{s};
 \chi_{a_{\scriptscriptstyle 2} c_{\scriptscriptstyle 2}},
 \chi_{a_{\scriptscriptstyle 1} c_{\scriptscriptstyle 1}}\!\right)
$
are among the zeros of the infinite product
$
\mathrm{D}^{\scriptscriptstyle \mathrm{re}}\!\left(\mathbf{s} -
  \mathbf{\tfrac{1}{2}}\right).
$\linebreak
In particular, each pole is simple, and corresponds to a positive real
root of the Kac-Moody Lie algebra with the gener-\linebreak
alized Cartan matrix in \ref{KM-alg}.

Taking $s_{\scriptscriptstyle 1}, \ldots, s_{r}$ 
with sufficiently large real parts, we see that 
$
\widetilde{Z}^{(c)}(\mathbf{s}; \chi_{a}, \chi_{a_{\scriptscriptstyle 1} c_{\scriptscriptstyle 1}}\!) 
$ 
(i.e., if $a_{\scriptscriptstyle 2} = a$ and $c_{\scriptscriptstyle 2} = 1$) has a simple pole\linebreak when 
$
q^{- s_{r + \scalebox{.8}{$\scriptscriptstyle 1$}}} \! 
= \mathrm{sgn}(a)q^{\scriptscriptstyle  -1}.
$ 
The part of 
$
\widetilde{Z}^{(c)}(\mathbf{s}; \chi_{a}, \chi_{a_{\scriptscriptstyle 1} c_{\scriptscriptstyle 1}}\!) 
$ 
that contributes to this pole is 
\begin{equation} \label{eq: Residue-principal}
L^{\scriptscriptstyle (c)}(s_{r + \scriptscriptstyle 1}, \chi_{a})
\; \cdot \sum_{\substack{(m_{\scalebox{.8}{$\scriptscriptstyle 1$}} \cdots \, m_{r}, \, c) = 1 \\ 
m_{\scalebox{.8}{$\scriptscriptstyle 1$}} \cdots \, m_{r} \, = \, \square}}  \, 
\frac{\tilde{Q}_{\underline{m}}(s_{r + \scriptscriptstyle 1}; \chi_{a})}
{|m_{\scriptscriptstyle 1}|^{s_{\scalebox{.8}{$\scriptscriptstyle 1$}}} \cdots \, |m_{r}|^{s_{r}}}
\, = \, L^{\scriptscriptstyle (c)}(s_{r + \scriptscriptstyle 1}, \chi_{a})
\prod_{p \nmid c} R_{\! p}(\mathbf{s}, \chi_{a})
\end{equation} 
where $R_{\! p}(\mathbf{s}, \chi_{a})$ is given by 
\begin{equation} \label{eq: Local-factor-residue-principal}
\begin{split}
& R_{\! p}(\mathbf{s}, \chi_{a}) \;  = \sum_{|\underline{k}| - \text{even}} 
\frac{\tilde{Q}_{\scriptscriptstyle \underline{k}}(\chi_{a}(p)
|p|^{\, - s_{r + \scalebox{.8}{$\scriptscriptstyle 1$}}}; \, |p|)}
{|p|^{k_{\scalebox{.8}{$\scriptscriptstyle 1$}}
s_{\scalebox{.8}{$\scriptscriptstyle 1$}} + \cdots + k_{r}s_{r}}}\\ 
& = (1 - \chi_{a}(p)|p|^{\, - s_{r + \scalebox{.8}{$\scriptscriptstyle 1$}}})
\big(\tilde{f}(|p|^{- s_{\scalebox{.8}{$\scriptscriptstyle 1$}}}\!, \ldots,\, 
|p|^{-s_{r}}\!,\, \chi_{a}(p)|p|^{\, - s_{r + \scalebox{.8}{$\scriptscriptstyle 1$}}};
\, |p|) 
+ \tilde{f}(-\, |p|^{- s_{\scalebox{.8}{$\scriptscriptstyle 1$}}}\!, \ldots, 
- \, |p|^{-s_{r}}\!,\, \chi_{a}(p)|p|^{\, - s_{r + \scalebox{.8}{$\scriptscriptstyle 1$}}};\, |p|)\big)\slash 2.
\end{split}
\end{equation} 
In what follows, it will be convenient to express \eqref{eq: Residue-principal} as  
\[
\frac{L(s_{r + \scriptscriptstyle 1}, \chi_{a})R(\mathbf{s},
  \chi_{a})}{L_{c}(s_{r + \scriptscriptstyle 1}, \chi_{a})
\prod_{p \mid c} R_{\! p}(\mathbf{s}, \chi_{a})}
\] 
where we set 
\begin{equation*}
R(\mathbf{s}, \chi_{a}) : = \prod_{p}R_{\! p}(\mathbf{s}, \chi_{a}).
\end{equation*} 
Thus we get: 
\[
\left.\frac{\widetilde{Z}^{(c)}(\mathbf{s}; \chi_{\scriptscriptstyle \theta_{\scriptscriptstyle 0}}, \chi_{a_{\scriptscriptstyle 1} c_{\scriptscriptstyle 1}}\!)}
{L(s_{r + \scriptscriptstyle 1}, \chi_{\scriptscriptstyle \theta_{\scriptscriptstyle 0}})}
\right\vert_{q^{ - s_{r + \scalebox{.8}{$\scriptscriptstyle 1$}}}
= \,- q^{\scriptscriptstyle -1}} = \; 
\left.\frac{\widetilde{Z}^{(c)}(\mathbf{s}; \chi_{\scriptscriptstyle 1}\!, \chi_{a_{\scriptscriptstyle 1} c_{\scriptscriptstyle 1}}\!)}{\zeta(s_{r + {\scriptscriptstyle 1}})} 
\right\vert_{s_{r + \scalebox{.8}{$\scriptscriptstyle 1$}} = 1} 
= \;\;  \frac{R(\mathbf{s}', 1, \chi_{\scriptscriptstyle
    1})}{\zeta_{c}(1)\prod_{p \mid c} R_{\! p}(\mathbf{s}', 1, \chi_{\scriptscriptstyle 1})}
\] 
where $\mathbf{s}' : = (s_{\scriptscriptstyle 1}, \ldots, s_{r}),$ and
$\chi_{\scriptscriptstyle 1}$ is the trivial character.


Let 
$
\alpha =  \sum k_{\scriptscriptstyle i}\alpha_{\scriptscriptstyle i} 
$ 
(with $k_{\scriptscriptstyle i} \in \mathbb{N}$) be a fixed positive real root, and let $w \in W$ be an element (in reduced form) sending $\alpha$ to the simple root $\alpha_{r + \scriptscriptstyle 1};$ we shall denote the length of $w$ by $l(w).$ Fix $a\in \{1, \, \theta_{\scriptscriptstyle 0} \},$ and $\zeta_{a}\in \mathbb{C}$ such that $\zeta_{a}^{k_{r + \scalebox{.8}{$\scriptscriptstyle 1$}}} \! = \mathrm{sgn}(a).$

To this data, we attach the residue $R_{w}(\mathbf{s}'; c, a)$ defined by 
\begin{equation*}
R_{w}(\mathbf{s}'; c, a) \; : =  \!\lim_{q^{\scriptscriptstyle
    (d(\alpha) + 1 - 2\alpha(\mathbf{s}))\slash
    2k_{r + \scalebox{.8}{$\scriptscriptstyle 1$}}}
\to \, {\scriptscriptstyle \zeta_{a}^{-1}}} 
k_{r + \scriptscriptstyle 1}
\big(1 - \zeta_{a}q^{\scriptscriptstyle (d(\alpha) + 1 - 2\alpha(\mathbf{s}))\slash 2k_{r + \scalebox{.8}{$\scriptscriptstyle 1$}}}\big)
\, \widetilde{Z}^{(c)}\!(^{{\scriptscriptstyle w}}\!\mathbf{s}; 
\chi_{a}, \chi_{{\scriptscriptstyle \vartheta} \bar{c}_{\scriptscriptstyle 1}}\!)
\end{equation*} 
where 
$
\alpha(\mathbf{s})  =  \sum k_{\scriptscriptstyle i}s_{\scriptscriptstyle i},
$ 
$
\vartheta \in \{1, \, \theta_{\scriptscriptstyle 0} \}
$ 
and $\bar{c}_{\scriptscriptstyle 1}$ is a square-free divisor of $c.$
Here, it is understood that $s_{r + \scriptscriptstyle 1}$ is
expressed in terms of the other variables. The factor $k_{r + \scriptscriptstyle 1}$ should be thought of as the limit 
\begin{equation*}
  \lim_{q^{\scriptscriptstyle (d(\alpha) + 1 - 2\alpha(\mathbf{s}))
      \slash 2k_{r + \scalebox{.8}{$\scriptscriptstyle 1$}}}
	\to \, {\scriptscriptstyle \zeta_{a}^{-1}}} 
\frac{1 - \mathrm{sgn}(a)q^{\scriptscriptstyle (d(\alpha) + 1 - 2\alpha(\mathbf{s}))\slash 2}}
{1 - \zeta_{a}q^{\scriptscriptstyle (d(\alpha) + 1 -
    2\alpha(\mathbf{s}))
    \slash 2k_{r + \scalebox{.8}{$\scriptscriptstyle 1$}}}}
\end{equation*} 
and thus
\begin{equation*}
\begin{split}
R_{w}(\mathbf{s}'; c, a) \;  & =  \!\lim_{q^{\scriptscriptstyle (d(\alpha) + 1 - 2\alpha(\mathbf{s}))\slash 2}
	\to \, {\mathrm{sgn}(a)}} 
\big(1 - \mathrm{sgn}(a)q^{\scriptscriptstyle (d(\alpha) + 1 - 2\alpha(\mathbf{s}))\slash 2} \big)
\, \widetilde{Z}^{(c)}\!(^{{\scriptscriptstyle w}}\!\mathbf{s}; 
\chi_{a}, \chi_{{\scriptscriptstyle \vartheta} \bar{c}_{\scriptscriptstyle 1}}\!)\\
& = \left.\frac{\widetilde{Z}^{(c)}\!(^{{\scriptscriptstyle w}}\!\mathbf{s}; 
	\chi_{a}, \chi_{{\scriptscriptstyle \vartheta} \bar{c}_{\scriptscriptstyle 1}}\!)}
{L\Big(\alpha(\mathbf{s}) - \frac{d(\alpha)}{2} + \frac{1}{2}, \chi_{a}\Big)}
\right\vert_{q^{{\scriptscriptstyle - \alpha(\mathbf{s}) + \frac{d(\alpha)}{2} - \frac{1}{2}}} 
\,=\, \mathrm{sgn}(a)q^{-1}} \\ 
& =  
\left. \frac{\prod_{p \nmid c}R_{\! p}(^{{\scriptscriptstyle w}}\!\mathbf{s}, \chi_{a})}{L_{c}\Big(\alpha(\mathbf{s}) - \frac{d(\alpha)}{2} + \frac{1}{2}, \chi_{a}\Big)}
\right\vert_{q^{{\scriptscriptstyle - \alpha(\mathbf{s}) + \frac{d(\alpha)}{2} - \frac{1}{2}}} 
\,=\, \mathrm{sgn}(a)q^{-1}}\\
& = \left. \frac{R(^{{\scriptscriptstyle w}}\!\mathbf{s}, \chi_{a})}{L_{c}\Big(\alpha(\mathbf{s}) - \frac{d(\alpha)}{2} + \frac{1}{2}, \chi_{a}\Big) 
\prod_{p \mid c}R_{\! p}(^{{\scriptscriptstyle w}}\mathbf{s}, \chi_{a})}
\right\vert_{q^{{\scriptscriptstyle - \alpha(\mathbf{s}) + \frac{d(\alpha)}{2} - \frac{1}{2}}} 
\,=\, \mathrm{sgn}(a)q^{-1}}.
\end{split}
\end{equation*} 
Here $R_{\! p}(\mathbf{s}, \chi_{a})$ is defined by \eqref{eq: Local-factor-residue-principal}. 

\vskip10pt
\begin{prop}\label{MS-residue-general} --- Let notations be as above. Then there exist functions $\Gamma_{\! w}(a_{\scriptscriptstyle 2}, a; \zeta_{a}),$ independent of 
	$c_{\scriptscriptstyle 1}, c_{\scriptscriptstyle 2}, c_{\scriptscriptstyle 3},$ and 
	$L_{w, p}(z_{\scriptscriptstyle 1}, \ldots, z_{r}; \zeta_{a}),$ parametrized by the prime divisors of $c,$ such that 
	\begin{equation*}
	\begin{split}
	& \left. \big(1 - \zeta_{a}q^{\scriptscriptstyle (d(\alpha) + 1 - 2\alpha(\mathbf{s}))\slash 2k_{r + \scalebox{.8}{$\scriptscriptstyle 1$}}}\!\big)\,
	\widetilde{Z}^{(c)}\!(\mathbf{s}; \chi_{a_{\scriptscriptstyle 2} c_{\scriptscriptstyle 2}}, \chi_{c_{\scriptscriptstyle 1}}\!)
      \right\vert_{q^{{\scriptscriptstyle -\alpha(\mathbf{s})
            \slash k_{r + \scalebox{.8}{$\scriptscriptstyle 1$}}}}
		=\;  {\scriptscriptstyle \zeta_{a}^{-1}} q^{-
(d(\alpha) + 1)\slash 2k_{r + \scalebox{.8}{$\scriptscriptstyle 1$}}}}\\
	& = \frac{\chi_{a_{\scriptscriptstyle 2} c_{\scriptscriptstyle 2}}\!(c_{\scriptscriptstyle 1})}{2^{{\scriptscriptstyle l(w)}}k_{r + \scriptscriptstyle 1}}
	\Gamma_{\! w}(a_{\scriptscriptstyle 2}, a; \zeta_{a})\, 
	R_{w}(\mathbf{s}'; c, a)
	\prod_{p \, \mid \, c}
	L_{w, p}(|p|^{-s_{\scalebox{.8}{$\scriptscriptstyle 1$}}}\!, \ldots, |p|^{-s_{\scriptscriptstyle r}}; \zeta_{a}).
	\end{split}
	\end{equation*} 
\end{prop}

\begin{proof} \!Set 
$
\widetilde{Z}_{\chi}^{(c)}\!(\mathbf{s};  \chi_{a_{\scriptscriptstyle 2} c_{\scriptscriptstyle 2}}, 
\chi_{a_{\scriptscriptstyle 1} c_{\scriptscriptstyle 1}})
: =\chi_{a_{\scriptscriptstyle 1}}\!(c_{\scriptscriptstyle 2})
\chi_{a_{\scriptscriptstyle 2}}\!(c_{\scriptscriptstyle 1})
\left(\!\frac{c_{\scriptscriptstyle 1}}{c_{\scriptscriptstyle 2}} \!\right)
\widetilde{Z}^{(c)}(\mathbf{s};  \chi_{a_{\scriptscriptstyle 2} c_{\scriptscriptstyle 2}}, 
\chi_{a_{\scriptscriptstyle 1} c_{\scriptscriptstyle 1}}).
$ 
Then 
$
\widetilde{Z}_{\chi}^{(c)}\!(\mathbf{s};  \chi_{a_{\scriptscriptstyle 2} c_{\scriptscriptstyle 2}}, 
\chi_{a_{\scriptscriptstyle 1} c_{\scriptscriptstyle 1}})
$ 
satisfies the (somewhat) simpler functional equation 
\begin{equation} \label{eq: functional-eq-Z-chi-absorbed-sigma_r+1}
\begin{split}
& \widetilde{Z}_{\chi}^{(c)}\!(\mathbf{s};  \chi_{a_{\scriptscriptstyle 2} c_{\scriptscriptstyle 2}}, 
\chi_{a_{\scriptscriptstyle 1} c_{\scriptscriptstyle 1}})\,  
= \,  \tfrac{1}{2}\,  |c_{\scriptscriptstyle 2}|^{\frac{1}{2} - s_{r + \scalebox{.8}{$\scriptscriptstyle 1$}}} 
\frac{\varphi(c_{\scriptscriptstyle 1} c_{\scriptscriptstyle 3})}{|c_{\scriptscriptstyle 1} c_{\scriptscriptstyle 3}|}
\prod_{p \, \mid \, c_{\scriptscriptstyle 1}  c_{\scriptscriptstyle 3}} \big(\! 1 \, - \, 
|p|^{2 s_{r + \scalebox{.8}{$\scriptscriptstyle 1$}} - 2}\big)^{\scriptscriptstyle -1}\\
&\cdot \sum_{\vartheta \in \{1, \, \theta_{\scriptscriptstyle 0} \}}
\!\left(\gamma_{\scriptscriptstyle q}^{+}(s_{r + \scriptscriptstyle 1}; \, a_{\scriptscriptstyle 2}) \, + \, 
\mathrm{sgn}(a_{\scriptscriptstyle 1}\vartheta)\, 
\gamma_{\scriptscriptstyle q}^{-}(s_{r + \scriptscriptstyle 1})\right) 
\!\!\!\!\sum_{\substack{m \, \mid \, c_{\scriptscriptstyle 1} c_{\scriptscriptstyle 3} \\ (c_{\scriptscriptstyle 1}\!, \, m) = e}} 
\!U_{ \scriptscriptstyle m}(s_{r + \scriptscriptstyle 1})\,  
\widetilde{Z}_{\chi}^{(c)}\!(^{{\scriptscriptstyle
w_{r + \scalebox{.8}{$\scriptscriptstyle 1$}}}}\mathbf{s}; 
\chi_{a_{\scriptscriptstyle 2} c_{\scriptscriptstyle 2}}, 
\chi_{{\scriptscriptstyle \vartheta} c_{\scriptscriptstyle 1} m \slash e^{\scriptscriptstyle 2}}).
\end{split}
\end{equation}
We have similar functional equations with respect to all the
other transformations 
$
w_{\scriptscriptstyle i}
$
($i = 1, \ldots, r$), and hence with respect to all $w\in W.$ Notice
that the quadratic characters get absorbed into the functions
$\widetilde{Z}_{\chi}^{(c)}.$ One should also note that the expression
\begin{equation*}
|c_{\scriptscriptstyle 2}|^{\frac{1}{2} - s_{r + \scalebox{.8}{$\scriptscriptstyle 1$}}} 
\frac{\varphi(c_{\scriptscriptstyle 1} c_{\scriptscriptstyle 3})}{|c_{\scriptscriptstyle 1} c_{\scriptscriptstyle 3}|}
\prod_{p \, \mid \, c_{\scriptscriptstyle 1}  c_{\scriptscriptstyle 3}} \big(\! 1 \, - \, 
|p|^{2 s_{r + \scalebox{.8}{$\scriptscriptstyle 1$}} - 2}\big)^{\scriptscriptstyle -1}
U_{ \scriptscriptstyle m}(s_{r + \scriptscriptstyle 1})
\qquad \text{(with $m  \mid  c_{\scriptscriptstyle 1} c_{\scriptscriptstyle 3}$)}
\end{equation*}
occurring in the right-hand side of \eqref{eq: functional-eq-Z-chi-absorbed-sigma_r+1}, factors over irreducibles dividing $c.$

The functional equation corresponding to $w$ relates 
$
\widetilde{Z}_{\chi}^{(c)}\!(\mathbf{s};  \chi_{a_{\scriptscriptstyle 2} c_{\scriptscriptstyle 2}}, 
\chi_{c_{\scriptscriptstyle 1}}\!)
$ 
to a sum of functions of the form 
$
\widetilde{Z}_{\chi}^{(c)}\!(^{{\scriptscriptstyle w}}\!\mathbf{s}; 
\chi_{{\scriptscriptstyle \vartheta'}\bar{c}_{\scriptscriptstyle 2}}, 
\chi_{{\scriptscriptstyle \vartheta} \bar{c}_{\scriptscriptstyle 1}}\!)
$ 
with 
$
\vartheta, \vartheta' \in \{1, \, \theta_{\scriptscriptstyle 0} \},
$ 
and $\bar{c}_{\scriptscriptstyle 1}, \bar{c}_{\scriptscriptstyle 2}$ monic square-free and coprime divisors of $c.$ The latter functions could have a pole when $q^{-(^{{\scriptscriptstyle w}}\!\mathbf{s})} = \mathrm{sgn}(a)q^{\scriptscriptstyle -1}$ (where 
$
q^{-(^{{\scriptscriptstyle w}}\!\mathbf{s})}
$ 
has the obvious meaning) only if $\vartheta' = a$ and
$\bar{c}_{\scriptscriptstyle 2} = 1.$ We have
\begin{equation*}
\widetilde{Z}_{\chi}^{(c)}\!(^{{\scriptscriptstyle w}}\!\mathbf{s}; 
\chi_{a}, \chi_{{\scriptscriptstyle \vartheta} \bar{c}_{\scriptscriptstyle 1}}\!)
=\chi_{a}(\bar{c}_{\scriptscriptstyle 1})
\widetilde{Z}^{(c)}\!(^{{\scriptscriptstyle w}}\!\mathbf{s}; 
\chi_{a}, \chi_{{\scriptscriptstyle \vartheta} \bar{c}_{\scriptscriptstyle 1}}\!)
\end{equation*}
and thus 
\begin{equation*}
\left. \big(1 - \zeta_{a}q^{\scriptscriptstyle (d(\alpha) + 1 -
2\alpha(\mathbf{s}))\slash 2k_{r + \scalebox{.8}{$\scriptscriptstyle 1$}}}\!\big)\,
\widetilde{Z}_{\chi}^{(c)}\!(^{{\scriptscriptstyle w}}\!\mathbf{s}; 
\chi_{a}, \chi_{{\scriptscriptstyle \vartheta} \bar{c}_{\scriptscriptstyle 1}}\!)
\right\vert_{q^{{\scriptscriptstyle -\alpha(\mathbf{s})
      \slash k_{r + \scalebox{.8}{$\scriptscriptstyle 1$}}}}
  =\;  {\scriptscriptstyle \zeta_{a}^{-1}} q^{- (d(\alpha) + 1)
    \slash 2k_{r + \scalebox{.8}{$\scriptscriptstyle 1$}}}}
 = \; \frac{\chi_{a}(\bar{c}_{\scriptscriptstyle 1})}{k_{r + {\scriptscriptstyle 1}}}
 R_{w}(\mathbf{s}'; c, a).
\end{equation*} 
The sum of products of the factors involving 
$
\gamma_{\scriptscriptstyle q}^{+}, \gamma_{\scriptscriptstyle q}^{-}
$ 
factors out, giving rise to $\Gamma_{\! w}(a_{\scriptscriptstyle 2}, a; \zeta_{a}).$ This function is clearly independent of $c_{\scriptscriptstyle 1}, c_{\scriptscriptstyle 2}$ and $c_{\scriptscriptstyle 3}.$ The remaining sum factors over the prime divisors of 
$
c = c_{\scriptscriptstyle 1}c_{\scriptscriptstyle 2}c_{\scriptscriptstyle 3}.
$ 
This completes the proof. 
\end{proof} 

We shall use this proposition in order to determine the residues at the poles of the function 
$
Z_{\scriptscriptstyle 0}(\mathbf{s}, \chi_{a_{\scriptscriptstyle 2}}).
$ 
For this purpose, it will be convenient to index the family of functions 
$
\big(L_{w, p}(z_{\scriptscriptstyle 1}, \ldots, z_{r}; \zeta_{a})\big)_{\! p \mid c}
$ 
by $L_{w, p}^{\scriptscriptstyle (j)}$ with $j = 1, 2, 3$ according as $p$ is a divisor of 
$
c_{\scriptscriptstyle 1}, c_{\scriptscriptstyle 2}
$ 
or $c_{\scriptscriptstyle 3},$ respectively.

\subsection{Computation of $\Gamma_{w}(a_{\scriptscriptstyle 2}, a; \zeta_{a})$} 
To compute $\Gamma_{\! w}(a_{\scriptscriptstyle 2}, a; \zeta_{a}),$ we shall use the fact that this function is independent of 
$
c_{\scriptscriptstyle 1}, c_{\scriptscriptstyle 2}
$ 
and $c_{\scriptscriptstyle 3}.$ Indeed, by taking $c = 1,$ the functional equation of 
$
Z(\mathbf{s}; \chi_{a_{\scriptscriptstyle 2}}, \chi_{a_{\scriptscriptstyle 1}}\!)
: = \widetilde{Z}^{(1)}\!(\mathbf{s}; \chi_{a_{\scriptscriptstyle 2}}, \chi_{a_{\scriptscriptstyle 1}}\!)
$ 
with respect to $w\in W$ will give, after taking the appropriate residue, the precise relationship between $\Gamma_{\! w}$ and $\mathrm{M}_{w}.$ 

To see this, set 
$
t_{\scriptscriptstyle i} = q^{\, - s_{\scalebox{.65}{$\scriptscriptstyle i$}}}
$ 
($i = 1, \ldots, r + 1$) in 
$
Z(\mathbf{s}; \chi_{a_{\scriptscriptstyle 2}}, \chi_{a_{\scriptscriptstyle 1}}\!),
$ 
and let 
$
\mathrm{Z}(\mathrm{t}; \chi_{a_{\scriptscriptstyle 2}}, \chi_{a_{\scriptscriptstyle 1}}\!)
$ 
denote the resulting function. If we define\linebreak 
$
\vec{\mathrm{Z}}(\mathrm{t}; q) 
: = \, ^{t}\!(\mathrm{Z}(\mathrm{t}; 1, 1 \!), \, \mathrm{Z}(\mathrm{t}; \chi_{\scriptscriptstyle \theta_{\scriptscriptstyle 0}}\!, 1 \!), \, \mathrm{Z}(\mathrm{t}; 1\!, \chi_{\scriptscriptstyle \theta_{\scriptscriptstyle 0}} \!))
$ 
then, by \eqref{eq: functional-eq-Z-sigma_r+1}, \addtocounter{footnote}{0}\let\thefootnote\svthefootnote\eqref{eq:
functional-eq-Z-sigma1}\footnote{Recall that the functions
$
Z^{(c)}\!\left(\mathbf{s};
 \chi_{a_{\scriptscriptstyle 2} c_{\scriptscriptstyle 2}},
 \chi_{a_{\scriptscriptstyle 1} c_{\scriptscriptstyle 1}}\!\right)
$
and
$
\widetilde{Z}^{(c)}\!\left(\mathbf{s};
 \chi_{a_{\scriptscriptstyle 2} c_{\scriptscriptstyle 2}},
 \chi_{a_{\scriptscriptstyle 1} c_{\scriptscriptstyle 1}}\!\right)
$
satisfy the same functional equations.} and the fact that 
$
\mathrm{Z}(\mathrm{t}; \chi_{a_{\scriptscriptstyle 2}}, \chi_{a_{\scriptscriptstyle 1}}\!)
$ 
is symmetric in the vari-\linebreak ables
$
t_{\scriptscriptstyle 1}, \ldots, t_{r},
$ 
we have the functional equations 
\begin{equation} \label{eq: f-e-glob-matrix}
\vec{\mathrm{Z}}(\mathrm{t}; q) = B^{\scriptscriptstyle -1}\mathrm{M}_{w}(q\mathrm{t}; 1\slash q)B
\, \cdot \, \vec{\mathrm{Z}}(w\mathrm{t}; q) \;\; \qquad \;\; \text{(for $w\in W$)}
\end{equation} 
where 
\begin{equation*}
B: = \begin{pmatrix} 
1\slash 2 & 1\slash 2 & 0\\
1\slash 2 & - 1\slash 2 & 0\\
- 1\slash 2 & 1\slash 2 & 1\\
\end{pmatrix}.
\end{equation*} 
Here $\mathrm{M}_{w}$ is the cocycle defined by \eqref{eq: def-cocycle1} and \eqref{eq: def-cocycle2}. We note that the function 
$
\mathrm{Z}(\mathrm{t}; \chi_{\scriptscriptstyle \theta_{\scriptscriptstyle 0}}\!, \chi_{\scriptscriptstyle \theta_{\scriptscriptstyle 0}} \!), 
$ 
which could have been included in
$\vec{\mathrm{Z}}(\mathrm{t}; q),$ can be easily expressed in terms of
the other three functions as 
\begin{equation*}
\mathrm{Z}(\mathrm{t}; \chi_{\scriptscriptstyle \theta_{\scriptscriptstyle 0}}\!, \chi_{\scriptscriptstyle \theta_{\scriptscriptstyle 0}}\!) 
= - \, \mathrm{Z}(\mathrm{t}; 1, 1 \!) + \mathrm{Z}(\mathrm{t}; \chi_{\scriptscriptstyle \theta_{\scriptscriptstyle 0}}\!, 1 \!) + \mathrm{Z}(\mathrm{t}; 1\!, \chi_{\scriptscriptstyle \theta_{\scriptscriptstyle 0}} \!).
\end{equation*} 
Multiplying \eqref{eq: f-e-glob-matrix} by 
$
1 - \zeta_{a}q^{\scriptscriptstyle (d(\alpha) + 1 -
  2\alpha(\mathbf{s}))
  \slash 2k_{r + \scalebox{.8}{$\scriptscriptstyle 1$}}}
$ 
and then taking the limit as 
$ 
q^{{\scriptscriptstyle -\alpha(\mathbf{s})
    \slash k_{r + \scalebox{.8}{$\scriptscriptstyle 1$}}}}
\to \, \zeta_{a}^{\scriptscriptstyle -1} q^{{\scriptscriptstyle -
    (d(\alpha) + 1)
    \slash 2k_{r + \scalebox{.8}{$\scriptscriptstyle 1$}}}}\!,
$ 
it follows from Proposition \ref{MS-residue-general} that 
\begin{equation} \label{eq: function-Gammaw}
\begin{split}
& 2^{{\scriptscriptstyle - l(w)}}
\Gamma_{\! w}(a_{\scriptscriptstyle 2}, a; \zeta_{a})\\
& = \big(\varepsilon^{\scriptscriptstyle +}\!(a_{\scriptscriptstyle 2}), \, \varepsilon^{\scriptscriptstyle -}\!(a_{\scriptscriptstyle 2}), \, 0 \big)
\, \cdot \, B^{\scriptscriptstyle -1}
\left(\!\left. \mathrm{M}_{w}\big(q^{{\scriptscriptstyle 1} - s_{\scalebox{.8}{$\scriptscriptstyle 1$}}}\!, 
\ldots,\, q^{{\scriptscriptstyle 1} - s_{r + \scalebox{.8}{$\scriptscriptstyle 1$}}}; 1\slash q \big)
\right\vert_{q^{{\scriptscriptstyle -\alpha(\mathbf{s})
      \slash k_{r + \scalebox{.8}{$\scriptscriptstyle 1$}}}}
  =\;  {\scriptscriptstyle \zeta_{a}^{-1}} q^{- (d(\alpha) + 1)
    \slash 2k_{r + \scalebox{.8}{$\scriptscriptstyle 1$}}}}\!\right)B
\, \cdot \,  
^{t}\!(\varepsilon^{\scriptscriptstyle +}\!(a), \, \varepsilon^{\scriptscriptstyle -}\!(a), \, \varepsilon^{\scriptscriptstyle +}\!(a))
\end{split}
\end{equation} 
with 
$
\varepsilon^{\scriptscriptstyle \pm}\!(\vartheta) = (1 \pm \mathrm{sgn}(\vartheta))\slash 2 
$ 
for $\vartheta \in \{1, \, \theta_{\scriptscriptstyle 0} \}.$

\subsection{Residues of $Z_{\scriptscriptstyle 0}(\mathbf{s}, \chi_{a_{\scriptscriptstyle 2}})$}
  Let all notations be as before. By using
  \eqref{eq: sum-sq-free-vs-MDS},
  \eqref{eq: fundamental-eq-h} and
  Proposition \ref{MS-residue-general}, it is not hard
  to check that 
\[
  \lim_{q^{{\scriptscriptstyle -\alpha(\mathbf{s})
        \slash k_{r + \scalebox{.8}{$\scriptscriptstyle 1$}}}} \to \,
    {\scriptscriptstyle \zeta_{a}^{\scriptscriptstyle -1}} q^{-
      (d(\alpha) + 1) \slash 2k_{r + \scalebox{.8}{$\scriptscriptstyle 1$}}}}
\big(1 - \zeta_{a}q^{\scriptscriptstyle (d(\alpha) + 1 -
  2\alpha(\mathbf{s}))\slash 2k_{r + \scalebox{.8}{$\scriptscriptstyle 1$}}}\!\big)\,
Z_{\scriptscriptstyle 0}(\mathbf{s},
    \chi_{a_{\scalebox{.7}{$\scriptscriptstyle 2$}}})
\; = \; \frac{\Gamma_{\! w}(a_{\scriptscriptstyle 2}, a; \zeta_{a})}
{2^{{\scriptscriptstyle l(w)}}k_{r + {\scriptscriptstyle 1}}}
R(^{\scriptscriptstyle w}\!\mathbf{s}\scalebox{.95}{$\scriptscriptstyle '$}\!, \chi_{a})
\, \cdot \, \prod_{p}\frac{S_{\! p}^{w}\!(\mathbf{s}
\scalebox{.95}{$\scriptscriptstyle '$}\!, \zeta_{a})}
{R_{\! p}(^{\scriptscriptstyle w}\mathbf{s}\scalebox{.95}{$\scriptscriptstyle '$}\!, \chi_{a})}
\]
with
$
S_{\! p}^{w}\!(\mathbf{s}
\scalebox{.95}{$\scriptscriptstyle '$}\!, \zeta_{a})
$
($w$ being an element in reduced form of
$W$ sending $\alpha$ to $\alpha_{r + \scriptscriptstyle 1}$)
given by
\begin{equation*}
\begin{split}
  & S_{\! p}^{w}\!(\mathbf{s}
\scalebox{.95}{$\scriptscriptstyle '$}\!, \zeta_{a})
  = \Big(L_{p}({\scriptstyle \alpha(\mathbf{s}) \, - \, \frac{d(\alpha)}{2} \, + \, \frac{1}{2}}, \chi_{a})^{\scriptscriptstyle -1}\cdot \,  \big(L_{p}({\scriptstyle \alpha(\mathbf{s}) \, - \, \frac{d(\alpha)}{2} \, + \, \frac{1}{2}}, \chi_{a}) 
R_{\! p}(^{{\scriptscriptstyle w}}\!\mathbf{s}, \chi_{a}) \, - \,
L_{w, p}^{\scriptscriptstyle (1)} F(|p|^{\, -
  s_{\scalebox{.8}{$\scriptscriptstyle 1$}}}\!, \ldots, \,
|p|^{\, - s_{r + \scalebox{.8}{$\scriptscriptstyle 1$}}}; \, 
|p|) |p|^{- 3 s_{r + \scalebox{.8}{$\scriptscriptstyle 1$}}}\\ 
& \left. - \, L_{w, p}^{\scriptscriptstyle (2)}G^{\scriptscriptstyle
    (1)}(|p|^{\, - s_{\scalebox{.8}{$\scriptscriptstyle 1$}}}\!, \ldots, \,
  |p|^{\, - s_{r + \scalebox{.8}{$\scriptscriptstyle 1$}}}; \, |p|)
  |p|^{- 2 s_{r + \scalebox{.8}{$\scriptscriptstyle 1$}}} 
- \, L_{w, p}^{\scriptscriptstyle (3)}G^{\scriptscriptstyle (0)}(|p|^{\, - s_{\scalebox{.8}{$\scriptscriptstyle 1$}}}\!, \ldots, \, |p|^{\, - s_{r + \scalebox{.8}{$\scriptscriptstyle 1$}}}; \, |p|)|p|^{- 2 s_{r + \scalebox{.8}{$\scriptscriptstyle 1$}}}\big)\Big)\right\vert_{q^{{\scriptscriptstyle -\alpha(\mathbf{s})\slash k_{r + \scalebox{.8}{$\scriptscriptstyle 1$}}}}
=\;  {\scriptscriptstyle \zeta_{a}^{-1}}
q^{- (d(\alpha) + 1)\slash 2k_{r + \scalebox{.8}{$\scriptscriptstyle 1$}}}}\!.
\end{split}
\end{equation*}
Here
$
R_{\! p}(^{\scriptscriptstyle w}\!\mathbf{s}\scalebox{.95}{$\scriptscriptstyle '$}\!, \chi_{a})
$
(resp. $R(^{\scriptscriptstyle w}
\!\mathbf{s}\scalebox{.95}{$\scriptscriptstyle '$}\!, \chi_{a})$)
denotes the function
$
R_{\! p}(^{\scriptscriptstyle w}\!\mathbf{s}, \chi_{a})
$ 
(resp.
$
R(^{\scriptscriptstyle w}\!\mathbf{s}, \chi_{a})
$) 
with $q^{\, - s_{r + \scalebox{.8}{$\scriptscriptstyle 1$}}}$ such that
$
q^{{\scriptscriptstyle -\alpha(\mathbf{s})
    \slash k_{r + \scalebox{.8}{$\scriptscriptstyle 1$}}}}
= \zeta_{a}^{\scriptscriptstyle -1}
q^{{\scriptscriptstyle - (d(\alpha) + 1)
\slash 2k_{r + \scalebox{.8}{$\scriptscriptstyle 1$}}}}.
$ 
Note that
\begin{equation*}
\left.
L_{p}\bigg(\alpha(\mathbf{s})  -  \frac{d(\alpha)}{2}  +  \frac{1}{2},
\chi_{a}\bigg)^{-1}\right\vert_{q^{{\scriptscriptstyle
    -\alpha(\mathbf{s})
    \slash k_{r + \scalebox{.8}{$\scriptscriptstyle 1$}}}}
=\;  {\scriptscriptstyle \zeta_{a}^{-1}} q^{- (d(\alpha) + 1)
\slash 2k_{r + \scalebox{.8}{$\scriptscriptstyle 1$}}}} =\; 1 - |p|^{\scriptscriptstyle -1}
\end{equation*} 
and by \eqref{eq: Local-factor-residue-principal}, 
\begin{equation*}
\begin{split} 
L_{p}(s_{r + \scriptscriptstyle 1}, \chi_{a}) R_{\! p}(\mathbf{s}, \chi_{a})
&  = \chi_{a}(p)\tilde{f}_{\!\scriptscriptstyle \mathrm{odd}}(|p|^{- s_{\scalebox{.8}{$\scriptscriptstyle 1$}}}\!, \ldots,\, 
|p|^{-s_{r}}\!,\, |p|^{\, - s_{r + \scalebox{.8}{$\scriptscriptstyle 1$}}}; \, |p|)\\ 
& + \, \big(\tilde{f}_{\!\scriptscriptstyle \mathrm{even}}(|p|^{- s_{\scalebox{.8}{$\scriptscriptstyle 1$}}}\!, \ldots,\, 
|p|^{-s_{r}}\!,\, |p|^{\, - s_{r + \scalebox{.8}{$\scriptscriptstyle 1$}}}; \, |p|) 
\, + \, \tilde{f}_{\!\scriptscriptstyle \mathrm{even}}(-\, |p|^{- s_{\scalebox{.8}{$\scriptscriptstyle 1$}}}\!, \ldots, 
- \, |p|^{-s_{r}}\!,\, |p|^{\, - s_{r + \scalebox{.8}{$\scriptscriptstyle 1$}}}; \, |p|)\big)\slash 2.
\end{split}
\end{equation*} 
We replace 
$
L_{p}({\scriptstyle \alpha(\mathbf{s}) \, - \, \frac{d(\alpha)}{2} \, + \, \frac{1}{2}}, \chi_{a}) 
R_{\! p}(^{{\scriptscriptstyle w}}\!\mathbf{s}, \chi_{a})
$ 
in $S_{\! p}^{w}$ using this identity, and the $F,$
$G^{\scriptscriptstyle (a)}$ ($a\in \{0, 1\}$) in $S_{\! p}^{w}$
using the definition of these functions. To
$\tilde{f}_{\!\scriptscriptstyle
  \mathrm{odd}}$ and
$\tilde{f}_{\!\scriptscriptstyle \mathrm{even}}$
coming from 
$
L_{p}({\scriptstyle \alpha(\mathbf{s}) \, - \, \frac{d(\alpha)}{2} \, + \, \frac{1}{2}}, \chi_{a}) 
R_{\! p}(^{{\scriptscriptstyle w}}\!\mathbf{s}, \chi_{a}),
$ 
we apply the functional equation\linebreak corresponding to 
$w^{\scriptscriptstyle -1}$ to get back in the variables 
$
s_{\scriptscriptstyle 1}, \ldots, s_{r + \scriptscriptstyle 1}.
$ 
\!Viewing 
$
\tilde{f}_{\!\scriptscriptstyle \mathrm{odd}}(
|p|^{- s_{\scalebox{.8}{$\scriptscriptstyle 1$}}}\!, \ldots,\, 
|p|^{-s_{r}}\!,\, |p|^{\, - s_{r + \scalebox{.8}{$\scriptscriptstyle 1$}}}; \, |p|)
$ 
and 
$ 
\tilde{f}_{\!\scriptscriptstyle \mathrm{even}}(\!\pm\, |p|^{-
  s_{\scalebox{.8}{$\scriptscriptstyle 1$}}}\!,$
$\ldots,
\pm \, |p|^{-s_{r}}\!,\, |p|^{\, - s_{r + \scalebox{.8}{$\scriptscriptstyle 1$}}}; \, |p|)
$ 
as independent variables, it is then not hard to see that
the functions 
$
L_{w, p}^{\scriptscriptstyle (j)} 
$ 
are precisely those that cancel out the 
$
\tilde{f}_{\!\scriptscriptstyle \mathrm{odd}}
$'s 
and the 
$
\tilde{f}_{\!\scriptscriptstyle \mathrm{even}}
$'s 
in the expression of $S_{\! p}^{w}$ when
$q^{\, - s_{r + \scalebox{.8}{$\scriptscriptstyle 1$}}}$ is such that 
$
q^{{\scriptscriptstyle -\alpha(\mathbf{s})
\slash k_{r + \scalebox{.8}{$\scriptscriptstyle 1$}}}}
= \zeta_{a}^{\scriptscriptstyle -1} q^{{\scriptscriptstyle -
    (d(\alpha) + 1)
\slash 2k_{r + \scalebox{.8}{$\scriptscriptstyle 1$}}}}.
$ 

Concretely, assume, without loss of generality, that $p$ is linear, and set 
$
z_{\scriptscriptstyle i} = q^{\, - s_{\scalebox{.65}{$\scriptscriptstyle i$}}}
$ 
for $1 \le i \le r + 1.$ Then the contributions of 
$
\tilde{f}_{\!\scriptscriptstyle \mathrm{odd}}(\underline{z}, z_{r + {\scriptscriptstyle 1}}; q)
$ 
and 
$ 
\tilde{f}_{\!\scriptscriptstyle \mathrm{even}}(\pm \underline{z}, z_{r + {\scriptscriptstyle 1}}; q)
$ 
to 
$
L_{p}(\cdots, \chi_{a})R_{\! p}(\cdots, \chi_{a})
$ 
is given by 
\begin{equation*} 
(1\slash 2, \mathrm{sgn}(a), 1\slash 2)
\mathrm{M}_{w^{\scriptscriptstyle -1}}(w\mathrm{z}; q)
\tilde{{\bf{f}}}(\mathrm{z}; q);
\end{equation*} 
the remaining contribution of these functions is 
\begin{equation*}
\big(L_{w, p}^{\scriptscriptstyle (1)}, \, L_{w, p}^{\scriptscriptstyle (2)}, \, L_{w, p}^{\scriptscriptstyle (3)}\big)
\begin{pmatrix}
0 & 1 & 0\\
1\slash 2& 0 & - 1\slash 2\\
1\slash 2& 0 & 1\slash 2 \\
\end{pmatrix}{\tilde{\bf{f}}}(\mathrm{z}; q).
\end{equation*} 
Setting 
\[
{\bf{g}}(\mathrm{z}; q) \, := \,
\begin{pmatrix}
0 & 1 & 0\\
1\slash 2& 0 & - 1\slash 2\\
1\slash 2& 0 & 1\slash 2 \\
\end{pmatrix}{\tilde{\bf{f}}}(\mathrm{z}; q)
\] 
we must have 
\begin{equation} \label{eq: functions-Lwp} 
\big(L_{w, p}^{\scriptscriptstyle (1)}, \, L_{w, p}^{\scriptscriptstyle (2)}, \, L_{w, p}^{\scriptscriptstyle (3)}\big)
{\bf{g}}(\mathrm{z}; q) 
= (1\slash 2, \mathrm{sgn}(a), 1\slash 2)
\mathrm{M}_{w^{\scriptscriptstyle -1}}(w\mathrm{z}; q)
\begin{pmatrix}
0 & 1 & 1\\
1 & 0 & 0\\
0& -1 & 1 \\
\end{pmatrix}
{\bf{g}}(\mathrm{z}; q).
\end{equation} 
Here the last component of $\mathrm{z}$ is taken in agreement with the relation 
$ 
q^{{\scriptscriptstyle -\alpha(\mathbf{s})
\slash k_{r + \scalebox{.8}{$\scriptscriptstyle 1$}}}}
=\;  \zeta_{a}^{\scriptscriptstyle -1} 
q^{{\scriptscriptstyle - (d(\alpha) + 1)\slash 
2k_{r + \scalebox{.8}{$\scriptscriptstyle 1$}}}}.
$ 
Thus the functions 
$
L_{w, p}^{\scriptscriptstyle (j)} 
$ 
are the coefficients of the entries of ${\bf{g}}(\mathrm{z}; q)$ 
in the right-hand side of 
\eqref{eq: functions-Lwp}. To obtain 
$
L_{w, p}^{\scriptscriptstyle (j)}(|p|^{-s_{\scalebox{.8}{$\scriptscriptstyle 1$}}}\!, \ldots, |p|^{-s_{\scriptscriptstyle r}}; \zeta_{a})
$ 
for arbitrary monic irreducible $p,$ we just replace 
$\mathrm{sgn}(a)$ by 
$
\chi_{a}(p) = \mathrm{sgn}(a)^{\deg \, p},
$ 
$
z_{\scriptscriptstyle i} = |p|^{\, - s_{\scalebox{.65}{$\scriptscriptstyle i$}}}
$ 
and $q$ by $|p|.$

From this discussion, it follows that
$
S_{\! p}^{w}\!(\mathbf{s}
\scalebox{.95}{$\scriptscriptstyle '$}\!, \zeta_{a})
$
above can be expressed as
\begin{equation} \label{eq: Final-expression-Sp}
S_{\! p}^{w}\!(\mathbf{s}
\scalebox{.95}{$\scriptscriptstyle '$}\!, \zeta_{a})
= \big(1 - |p|^{\scriptscriptstyle -1} \big)
\bigg(L_{w, p}^{\scriptscriptstyle (1)}
|p|^{- s_{r + \scalebox{.8}{$\scriptscriptstyle 1$}}}  \, + \; 
\frac{L_{w, p}^{\scriptscriptstyle (2)} \, + \, L_{w, p}^{\scriptscriptstyle (3)}}{2\prod_{k = 1}^{r}\big(1 - |p|^{\, - s_{\scalebox{.65}{$\scriptscriptstyle k$}}}\!\big)} 
\left. \, + \;\, \frac{L_{w, p}^{\scriptscriptstyle (3)} \, - \, L_{w,
      p}^{\scriptscriptstyle (2)}}{2\prod_{k = 1}^{r}\big(1 + |p|^{\,
      - s_{\scalebox{.65}{$\scriptscriptstyle k$}}}\!\big)}\bigg)\right\vert_{q^{{\scriptscriptstyle
      -\alpha(\mathbf{s})
\slash k_{r + \scalebox{.8}{$\scriptscriptstyle 1$}}}}
	=\;  {\scriptscriptstyle \zeta_{a}^{-1}} q^{- (d(\alpha) + 1)
\slash 2k_{r + \scalebox{.8}{$\scriptscriptstyle 1$}}}}\!.
\end{equation}
As expected, we have the equality
\begin{equation} \label{eq: Regularization-prod-general-residue-Z-zero}
R(^{\scriptscriptstyle w}\!\mathbf{s}
\scalebox{.95}{$\scriptscriptstyle '$}\!, \chi_{a})
\, \cdot \, \prod_{p}\frac{S_{\! p}^{w}\!(\mathbf{s}
\scalebox{.95}{$\scriptscriptstyle '$}\!, \zeta_{a})}
{R_{\! p}(^{\scriptscriptstyle w}\mathbf{s}
\scalebox{.95}{$\scriptscriptstyle '$}\!, \chi_{a})}
= \prod_{p}S_{\! p}^{w}\!(\mathbf{s}
\scalebox{.95}{$\scriptscriptstyle '$}\!, \zeta_{a})
\end{equation}
but since the product
$
\prod_{p}S_{\! p}^{w}\!(\mathbf{s}
\scalebox{.95}{$\scriptscriptstyle '$}\!, \zeta_{a})
$
is divergent on any neighborhood of
$
s_{\scriptscriptstyle 1} = \cdots = s_{r} = \frac{1}{2},
$
a (non-trivial when $k_{r + \scriptscriptstyle 1} \ge 2$)
regularization process is required to justify
\eqref{eq: Regularization-prod-general-residue-Z-zero}; we shall
merely discuss the regularization process when
$k_{r + \scriptscriptstyle 1} = 2,$ see \ref{Case n=2} -- the general
case following, essentially, the same argument.

\section{Asymptotics of moments}
\label{Asympt-moments-introduction}
For $n \ge 1,$ let $\Phi_{n}$ denote the subset of $\Delta^{\mathrm{re}}_{+}$ with 
$k_{r + \scriptscriptstyle 1} = n.$ By Lemma \ref{Finiteness-roots-level-n}, $\Phi_{n}$ is a finite set. For each $\alpha \in \Phi_{n},$ fix $w_{\scriptscriptstyle \alpha} \in W$ in reduced form such that 
$
w_{\scriptscriptstyle \alpha}^{\scriptscriptstyle  -1}(\alpha_{r + \scriptscriptstyle 1}) = \alpha.
$ 
Let $\frak{S}_{n}(\mathrm{z}, a_{\scriptscriptstyle 2}),$ with 
$
\mathrm{z} = (z_{\scriptscriptstyle 1}^{}, \ldots, z_{r + {\scriptscriptstyle 1}}^{})
$ 
($
z_{\scriptscriptstyle i}^{}
= q^{ - s_{\scalebox{.65}{$\scriptscriptstyle i$}}}
$) and $a_{\scriptscriptstyle 2} \in \{1, \theta_{\scriptscriptstyle 0} \},$ be defined by 
\begin{equation}  \label{eq: Secondary-princ-parts-general}
\frak{S}_{n}(\mathrm{z}, a_{\scriptscriptstyle 2}) \, 
: = \, n^{\scriptscriptstyle -1}
\cdot\sum_{\alpha \in \Phi_{n}}\; \sum_{a\in \{1, \, \theta_{\scriptscriptstyle 0} \}}\;
\sum_{\zeta_{a}^{n}  = \mathrm{sgn}(a)}
\frac{\Gamma_{\!\! w_{\scalebox{.85}{$\scriptscriptstyle \alpha$}}}\!(a_{\scriptscriptstyle 2}, a; \zeta_{a})}{2^{l(w_{\scalebox{.85}{$\scriptscriptstyle \alpha$}})}}
\,\big(1 \, - \, \zeta_{a}q^{\scriptscriptstyle (d(\alpha) + 1)\slash 2n}
\mathrm{z}^{\scriptscriptstyle \alpha\slash n} \!\big)^{\scriptscriptstyle -1}
\prod_{p}S_{\! p}^{w_{\scalebox{.85}{$\scriptscriptstyle \alpha$}}}
\!(\underline{z}, \zeta_{a})
\end{equation} 
where 
$
\Gamma_{\!\! w_{\scalebox{.85}{$\scriptscriptstyle \alpha$}}}\!(a_{\scriptscriptstyle 2}, a; \zeta_{a})
$ 
and
$
S_{\! p}^{w_{\scalebox{.85}{$\scriptscriptstyle
      \alpha$}}}\!(\underline{z}, \zeta_{a})
$
are obtained from \eqref{eq: function-Gammaw} and \eqref{eq:
  functions-Lwp}, \eqref{eq: Final-expression-Sp}, respectively. It is
clear that the expression of\linebreak
$\frak{S}_{n}(\mathrm{z}, a_{\scriptscriptstyle 2})$ is independent of the choice of the elements $w_{\scriptscriptstyle \alpha}.$ 
Via \eqref{eq: Regularization-prod-general-residue-Z-zero},
the function $\frak{S}_{n}(\mathrm{z}, a_{\scriptscriptstyle 2})$
gives (after substituting 
$
z_{\scriptscriptstyle i}^{}
\!= q^{ - s_{\scalebox{.65}{$\scriptscriptstyle i$}}}\!,
$ 
$i = 1, \ldots, r + 1$) the sum of the principal parts of 
$
Z_{\scriptscriptstyle 0}(\mathbf{s}, \chi_{a_{\scriptscriptstyle 2}}\!)
$ 
at the poles corresponding to the roots in $\Phi_{n}.$


Now let $s_{\scriptscriptstyle 1}, \ldots, s_{r}$ be (fixed) {\it
  distinct} complex numbers with $\Re(s_{\scriptscriptstyle k}) = \frac{1}{2},$ and consider the function 
$
Z_{\scriptscriptstyle 0}(s_{\scriptscriptstyle 1}, \ldots, s_{r + {\scriptscriptstyle 1}}, \chi_{\scriptscriptstyle 1}).
$ 
For convenience, we replace $q^{- s_{r + {\scriptscriptstyle 1}}}$ in $Z_{\scriptscriptstyle 0}$ by $\xi,$ 
and denote the resulting function by $\mathscr{W}\!(\xi).$ Thus 
\begin{equation*}
\mathscr{W}\!(\xi) = \mathscr{W}\!(\mathbf{s}'\!, \xi)  = \, \sum_{D \ge 0}
\; \Bigg(\, \sum_{\substack{d - \mathrm{monic \; \& \; sq. \; free} \\ \deg \, d\, = \, D}} \;\,
\prod_{k = 1}^{r}L(s_{\scriptscriptstyle k}, \chi_{d}) 
\Bigg)\, \xi^{\scriptscriptstyle D}.
\end{equation*} 
By our assumptions, this function is meromorphic in the open disk $|\xi| < q^{\scriptscriptstyle - 1\slash 2}\!,$ 
with the only possible poles at 
\[ 
\scalebox{1.6}{$\scriptscriptstyle \xi$}_{\alpha,\,  \scalebox{.95}{$\scriptscriptstyle \zeta$}_{a}}
\!\!: =\scalebox{1.45}{$\scriptscriptstyle \zeta$}_{a}^{\scriptscriptstyle -1}
q^{\scriptscriptstyle \frac{\alpha(\mathbf{s}\scalebox{.75}{$\scriptscriptstyle '$})}{k_{\scalebox{.95}{$\scriptscriptstyle r$}
+ \scalebox{.8}{$\scriptscriptstyle 1$}}} \, - \, \frac{d(\alpha) + 1}{2k_{\scalebox{.95}{$\scriptscriptstyle r$} + \scalebox{.8}{$\scriptscriptstyle 1$}}}}
\] 
where, for 
$
\alpha = \!\sum k_{\scriptscriptstyle i}\alpha_{\scriptscriptstyle i} \in \Delta^{\mathrm{re}}_{+},
$ 
we set $\alpha(\mathbf{s}') : =  \sum_{i \le r} k_{\scriptscriptstyle i}s_{\scriptscriptstyle i};
$ 
to avoid difficulties, we will initially assume that the poles of $\mathscr{W}\!(\xi)$ are \addtocounter{footnote}{0}\let\thefootnote\svthefootnote simple\footnote{This happens generically. 
	In fact, two poles corresponding to positive real roots $\alpha$ and $\alpha'$ could coincide only it $\alpha, \alpha' \! \in \Phi_{n},$ for some $n \ge 1.$ Thus,\linebreak by choosing 
	$s_{\scalebox{.9}{$\scriptscriptstyle k$}} = \frac{1}{2} + it_{\scalebox{.9}{$\scriptscriptstyle k$}},$ with 
	$\pi \slash \log q, \, t_{\scalebox{.9}{$\scriptscriptstyle 1$}}, \ldots, t_{\scalebox{1.05}{$\scriptscriptstyle r$}}$ 
	$\mathbb{Q}$-linearly independent, we ensure that the poles of $\mathscr{W}\!(\xi)$ are simple.}\!\!. For $n, D \ge 1,$ let 
$
Q_{n}(\mathbf{s}'; D, q)
$ 
be defined by 
\[
Q_{n}(\mathbf{s}'; D, q) 
= n^{\scriptscriptstyle -1}
\cdot\sum_{\alpha \in \Phi_{n}} \bigg\{\sum_{a\in \{1, \, \theta_{\scriptscriptstyle 0} \}}\;
\sum_{\zeta_{a}^{n}  = \mathrm{sgn}(a)}
\frac{\Gamma_{\!\! w_{\scalebox{.85}{$\scriptscriptstyle \alpha$}}}
\!(1, a; \zeta_{a})}{2^{l(w_{\scalebox{.85}{$\scriptscriptstyle \alpha$}})}}
S_{\scriptscriptstyle \alpha}(\mathbf{s}'\!, \zeta_{a})\zeta_{a}^{\scriptscriptstyle D}\bigg\}
q^{\scalebox{.95}{$\scriptscriptstyle \frac{D\, (d(\alpha) + 1 - 2\alpha(\mathbf{s}
\scalebox{.7}{$\scriptscriptstyle '$}))}{2n}$}}
\] 
where 
$
S_{\scriptscriptstyle \alpha}(\mathbf{s}'\!, \zeta_{a}) 
: = \prod_{p}S_{\! p}^{w_{\scalebox{.85}{$\scriptscriptstyle \alpha$}}},
$ 
with $S_{\! p}^{w_{\scalebox{.85}{$\scriptscriptstyle \alpha$}}}$ given by \eqref{eq: Final-expression-Sp}.

With this notation, our main result is the following:

\vskip5pt
 \begin{thm} \label{Main Theorem: Full-moment-asymptotics} --- Let $D, N \ge 1$ be integers, and suppose that 
   $r \ge 4.$ Then, under the Assumption
   \eqref{eq: Estimate-H1}, and
   Conjectures \ref{Conjecture-Meromorphic continuation},
   \ref{Conjecture-Meromorphic continuation-Z0}, we have
 	\[ 
 	\sum_{\substack{d - \mathrm{monic \; \& \; sq. \; free} \\ \deg \, d\, = \, D}}\;\,
 	\prod_{k = 1}^{r}L(s_{\scriptscriptstyle k}, \chi_{d}) 
 	\, = \sum_{n \, \le \, N} Q_{n}(\mathbf{s}'; D, q) \, + \, 
 	O_{\scalebox{.85}{$\scriptscriptstyle \Theta$}\scriptscriptstyle, \,  q, \, r}
 	\!\left(q^{\scalebox{.9}{$\scriptscriptstyle D$} \scalebox{.9}{$\scriptscriptstyle (1 + \Theta)\slash 2$}}\right)
 	\] 
 	for any $(N + 1)^{\scriptscriptstyle -1} \! < \Theta  < N^{\scriptscriptstyle -1}\!.$
 \end{thm} 

\begin{proof} \!For $(N + 1)^{\scriptscriptstyle -1} \! < \Theta  < N^{\scriptscriptstyle -1}\!,$ let 
	$
	\mathscr{A}_{\scriptscriptstyle  \Theta} = \{\xi \in \mathbb{C} : q^{\scalebox{.9}{$\scriptscriptstyle - 2$}} \le |\xi| \le 
	q^{\scalebox{.9}{$\scriptscriptstyle - (1 + \Theta)\slash 2$}}\},
	$ 
	and consider the integral 
	\[
	I(D) = \frac{1}{2 \pi \sqrt{-1}}\oint_{\partial 	\mathscr{A}_{\scriptscriptstyle  \Theta}} 
	\!\frac{\mathscr{W}\!(\xi)}{\xi^{\scriptscriptstyle D + 1}}\, d\xi.
	\] 
	where $\partial \mathscr{A}_{\scriptscriptstyle  \Theta}$ denotes the boundary of 
	$\mathscr{A}_{\scriptscriptstyle  \Theta}.$ In view of our
        assumptions, the function $\mathscr{W}\!(\xi)$ is meromorphic in 
	$|\xi| < q^{\scriptscriptstyle - 1\slash 2}\!,$ and since $\Re(s_{\scriptscriptstyle k}) = \frac{1}{2},$ for $k = 1, \ldots, r,$ 
	its poles are on the circles 
	\[
	|\xi| =  
|\scalebox{1.6}{$\scriptscriptstyle \xi$}_{\alpha,\,  \scalebox{.95}{$\scriptscriptstyle \zeta$}_{a}}|
= q^{\scriptscriptstyle -\frac{n + 1}{2n}}
\qquad \text{(with $\alpha \in \Phi_{n},$ $\zeta_{a}^{2 n} = 1$ for $n \ge 1$).}
\] 
Thus, by our choice of $\Theta,$ this function has no poles on the boundary of $\mathscr{A}_{\scriptscriptstyle  \Theta}.$ On the other hand, we have 
\[
\sum_{\substack{d - \mathrm{monic \; \& \; sq. \; free} \\ \deg \, d\, = \, D}}\;\,
\prod_{k = 1}^{r}L(s_{\scriptscriptstyle k}, \chi_{d}) 
= \frac{1}{2 \pi \sqrt{-1}}\; \oint\limits_{|\xi| \, = \, {\scalebox{1.5}{$\scriptscriptstyle q$}}^{\, \scriptscriptstyle - 2}} 
\!\frac{\mathscr{W}\!(\xi)}{\xi^{\scriptscriptstyle D + 1}}\, d\xi
\] 
and 
\[
\frac{1}{2 \pi \sqrt{-1}}\; \int\limits_{|\xi| \, = \, q^{\scalebox{.9}{$\scriptscriptstyle - (1 + \Theta)\slash 2$}}} 
\!\frac{\mathscr{W}\!(\xi)}{\xi^{\scriptscriptstyle D + 1}}\, d\xi 
\, \ll_{\scalebox{.85}{$\scriptscriptstyle \Theta$}\scriptscriptstyle, \,   q, \,  r} 
\, q^{\scalebox{.9}{$\scriptscriptstyle D$} \scalebox{.9}{$\scriptscriptstyle (1 + \Theta)\slash 2$}}
\] 
where the implied constant is taken to be the maximum of $|\mathscr{W}\!(\xi)|$ on the circle 
$|\xi| = q^{\scalebox{.9}{$\scriptscriptstyle - (1 + \Theta)\slash 2$}}.$ By applying the residue theorem, 
our assertion follows at once from \eqref{eq: Secondary-princ-parts-general} and the definition of $Q_{n}(\mathbf{s}'; D, q).$
\end{proof}

\begin{rem} --- It would be interesting to study the analytic properties of the generating series 
	$
	\sum_{n \ge 1} Q_{n}(\mathbf{s}'; D, q),
	$ 
	with fixed $D,$ as a function of $s_{\scriptscriptstyle 1}, \ldots, s_{r},$ and $q$ (not necessarily a prime power). 
\end{rem}

In order to get the asymptotic formula at the center of the critical strip, we have to study in detail the behavior of the function 
$\mathscr{W}\!(\mathbf{s}'\!, \xi)$ in a neighborhood of a pole. To do so, fix $n \ge 1,$ $\zeta$ a $2n$-th root of $1,$ and 
recall that
\[
\prod_{\alpha \in \Delta^{\mathrm{re}}_{+}}
\!\left(1 - q^{\scalebox{.95}{$\scriptscriptstyle d(\alpha) + 1 - 2\alpha(\mathbf{s}\scalebox{.7}{$\scriptscriptstyle '$})$}}
\xi^{2k_{\scalebox{.95}{$\scriptscriptstyle r$} + \scalebox{.8}{$\scriptscriptstyle 1$}}}\right)
\, \cdot \, \mathscr{W}\!(\mathbf{s}'\!, \xi)
\] 
where, as before, $\alpha =  \!\sum k_{\scriptscriptstyle i}\alpha_{\scriptscriptstyle i},$ is holomorphic in a neighborhood
of the point $s_{\scriptscriptstyle i} = \frac{1}{2}$ ($i = 1, \ldots, r$), and 
$
\xi = q^{\scalebox{.95}{$\scriptscriptstyle - \frac{n + 1}{2n}$}}\zeta; 
$ 
we can eliminate the unnecessary factors of the product, that is, those that do not vanish at this point, 
and consider instead the function
\begin{equation} \label{eq: aux-function-poles-study}
\mathscr{W}_{n,\,  {\scriptscriptstyle \zeta}}(\mathbf{s}'\!, \xi)
\; = \prod_{\alpha \in \Phi_{n}}
\!\left(1 - q^{\scalebox{.95}{$\scriptscriptstyle (d(\alpha) + 1 - 2\alpha(\mathbf{s}\scalebox{.7}{$\scriptscriptstyle '$}))\slash 2n$}}
\zeta^{\scriptscriptstyle -1}\xi\right)
\, \cdot \, \mathscr{W}\!(\mathbf{s}'\!, \xi).
\end{equation}
Let 
$
\frak{s}_{\scriptscriptstyle i} \! = s_{\scriptscriptstyle i} - \frac{1}{2},
$ 
$
u = \xi - q^{\scalebox{.95}{$\scriptscriptstyle -\frac{n + 1}{2n}$}}\zeta,
$ 
and consider the Weierstrass polynomial 
$
\mathbf{W}_{n,\,  {\scriptscriptstyle \zeta}}(\frak{s}_{\scriptscriptstyle 1}, \ldots, \frak{s}_{r}, u)
$ 
defined by 
\[
\mathbf{W}_{n,\,  {\scriptscriptstyle \zeta}}(\frak{s}_{\scriptscriptstyle 1}, \ldots, \frak{s}_{r}, u) 
\; = \prod_{\alpha \in \Phi_{n}} 
\!\left(u  - q^{\scalebox{.95}{$\scriptscriptstyle \frac{
2k_{\scalebox{.6}{$\scriptscriptstyle 1$}}\frak{s}_{\scalebox{.6}{$\scriptscriptstyle 1$}} + \, \cdots \, + \, 
2k_{\scalebox{.8}{$\scriptscriptstyle r$}}\frak{s}_{\scalebox{.8}{$\scriptscriptstyle r$}} - \, n \, -  \, 1}{2n}$}}\zeta 
+ q^{\scalebox{.95}{$\scriptscriptstyle -\frac{n + 1}{2n}$}}\zeta\right)\!.
\] 
By applying the Weierstrass division theorem (see, for example, \cite{Ga-Ro}), there exist unique functions 
$\mathbf{R}_{n,\,  {\scriptscriptstyle \zeta}}, \mathbf{S}_{n,\,  {\scriptscriptstyle \zeta}},$ both ho-\linebreak 
lomorphic in a neighborhood of ${\bf 0} \in \mathbb{C}^{r + \scriptscriptstyle{1}}\!,$ and 
$\mathbf{R}_{n,\,  {\scriptscriptstyle \zeta}}$ polynomial in $u$ with 
$
\deg_{u} \mathbf{R}_{n,\,  {\scriptscriptstyle \zeta}} 
< \deg_{u} \mathbf{W}_{n,\,  {\scriptscriptstyle \zeta}},
$ 
such that 
\[
\mathscr{W}_{n,\,  {\scriptscriptstyle \zeta}}
\!\left(\frak{s}_{\scriptscriptstyle 1} + \tfrac{1}{2},\, 
\frak{s}_{\scriptscriptstyle 2} + \tfrac{1}{2}, \ldots,\, 
u + q^{\scalebox{.95}{$\scriptscriptstyle -\frac{n + 1}{2n}$}}\zeta \right) 
= (\mathbf{S}_{n,\,  {\scriptscriptstyle \zeta}}
\cdot \mathbf{W}_{n,\,  {\scriptscriptstyle \zeta}})(\frak{s}_{\scriptscriptstyle 1}, \ldots, \frak{s}_{r}, u) 
+ \mathbf{R}_{n,\,  {\scriptscriptstyle \zeta}}(\frak{s}_{\scriptscriptstyle 1}, \ldots, \frak{s}_{r}, u).
\] 
Dividing this by the product in \eqref{eq: aux-function-poles-study}, it follows that we can express 
$\mathscr{W}\!(\mathbf{s}'\!, \xi)$ as 
\begin{equation} \label{conseq-Weierstrass-division-theorem} 
\mathscr{W}\!(\mathbf{s}'\!, \xi) 
\,= \, \tilde{\mathbf{S}}_{n,\,  {\scriptscriptstyle \zeta}}(\mathbf{s}'\!, \xi) 
\, + \, \frac{\tilde{\mathbf{R}}_{n,\,  {\scriptscriptstyle \zeta}}(\mathbf{s}'\!, \xi)}
{\prod_{\alpha \in \Phi_{n}}
	\!\left(1 - q^{\scalebox{.95}{$\scriptscriptstyle (d(\alpha) + 1 - 2\alpha(\mathbf{s}\scalebox{.7}{$\scriptscriptstyle '$}))\slash 2n$}}
	\zeta^{\scriptscriptstyle -1}\xi\right)}
\end{equation}
with 
$\tilde{\mathbf{S}}_{n,\,  {\scriptscriptstyle \zeta}}$ and 
$\tilde{\mathbf{R}}_{n,\,  {\scriptscriptstyle \zeta}}$ holomorphic in a neighborhood of the point 
$
(\mathbf{s}'\!, \xi) = \big(\frac{1}{2}, \ldots, \frac{1}{2}, 
q^{\scalebox{.95}{$\scriptscriptstyle - \frac{n + 1}{2n}$}}\zeta \big) \in \mathbb{C}^{r + {\scriptscriptstyle 1}};
$ 
the function $\tilde{\mathbf{R}}_{n,\,  {\scriptscriptstyle \zeta}}(\mathbf{s}'\!, \xi)$ is a polynomial in the variable $\xi$ of degree 
smaller than $|\Phi_{n}|.$

We have the following: 

\vskip5pt
\begin{prop} \label{Analytic-continuation-frakSn} --- For $n \ge 1$ and $\zeta$ a $2n$-th root of $1,$ let 
$\frak{S}_{n,\,  {\scriptscriptstyle \zeta}}(\mathbf{s}'\!, \xi)$ be defined by 
\[
\frak{S}_{n,\,  {\scriptscriptstyle \zeta}}(\mathbf{s}'\!, \xi)
=  \scalebox{1.5}{$\scriptscriptstyle \frac{1}{n}$}
\, \cdot\sum_{\alpha \in \Phi_{n}}
\frac{2^{\scriptscriptstyle - l(w_{\scalebox{.85}{$\scriptscriptstyle \alpha$}})}
\Gamma_{\!\! w_{\scalebox{.85}{$\scriptscriptstyle \alpha$}}}\!(1, a; \zeta^{\scriptscriptstyle -1}) 
S_{\scriptscriptstyle \alpha}(\mathbf{s}'\!, \zeta^{\scriptscriptstyle -1})}
{1 - 
q^{\scalebox{.95}{$\scriptscriptstyle (d(\alpha) + 1 - 2\alpha(\mathbf{s}\scalebox{.7}{$\scriptscriptstyle '$}))\slash 2n$}}
\zeta^{\scriptscriptstyle -1}\xi}
\] 
with $a = 1$ or $\theta_{\scriptscriptstyle 0}$ according as $\zeta^{n} = 1$ or $-1.$ 
Then, for $\mathbf{s}' = (s_{\scriptscriptstyle 1}, \ldots, s_{r})$ in a neighborhood of $(\frac{1}{2}, \ldots, \frac{1}{2}) \in \mathbb{C}^{r}\!,$ and 
$
\xi \ne q^{\scalebox{.95}{$\scriptscriptstyle (2\alpha(\mathbf{s}\scalebox{.7}{$\scriptscriptstyle '$}) - d(\alpha) - 1)\slash 2n$}}
\zeta
$ 
\,for all $\alpha \in \Phi_{n},$ we have 
\[
\frak{S}_{n,\,  {\scriptscriptstyle \zeta}}(\mathbf{s}'\!, \xi)
= \frac{\tilde{\mathbf{R}}_{n,\,  {\scriptscriptstyle \zeta}}(\mathbf{s}'\!, \xi)}
{\prod_{\alpha \in \Phi_{n}}
	\!\left(1 - q^{\scalebox{.95}{$\scriptscriptstyle (d(\alpha) + 1 - 2\alpha(\mathbf{s}\scalebox{.7}{$\scriptscriptstyle '$}))\slash 2n$}}
	\zeta^{\scriptscriptstyle -1}\xi\right)}.
\] 
In particular, the function 
$
\frak{S}_{n}
\!\left(q^{-s_{\scalebox{.65}{$\scriptscriptstyle 1$}}}\!, \ldots, q^{-s_{\scalebox{.85}{$\scriptscriptstyle r$}}}\!, \xi, 1\right) 
$ 
can be analytically continued for $\mathbf{s}' = (s_{\scriptscriptstyle 1}, \ldots, s_{r})$ in a neighborhood of 
$(\frac{1}{2}, \ldots, \frac{1}{2}),$ and 
$
\xi \ne q^{\scalebox{.95}{$\scriptscriptstyle (2\alpha(\mathbf{s}\scalebox{.7}{$\scriptscriptstyle '$}) - d(\alpha) - 1)\slash 2n$}}
\zeta
$ 
with $\alpha \in \Phi_{n},$ and $\zeta$ any $2n$-th root of $1.$  
\end{prop}

\begin{proof} Let $D_{\scalebox{1.}{$\scriptscriptstyle {\mathbf{1\slash 2}}$}}$ (resp. $D_{n,  {\scriptscriptstyle \zeta}}$) 
	be a polydisk centered at ${\mathbf{\frac{1}{2}}}: = \big(\frac{1}{2}, \ldots, \frac{1}{2}\big)\in \mathbb{C}^{r}$ 
	(resp. a disk centered at $q^{\scalebox{.95}{$\scriptscriptstyle - \frac{n + 1}{2n}$}}\zeta$) 
	such that \eqref{conseq-Weierstrass-division-theorem}\linebreak holds on 
	$
	V : = D_{\scalebox{1.}{$\scriptscriptstyle {\mathbf{1\slash 2}}$}} \times 
	D_{n,  {\scriptscriptstyle \zeta}}
	$ 
	(away from the singularities of the right-hand side), \!and 
	$
	q^{\scalebox{.95}{$\scriptscriptstyle (2\alpha(\mathbf{s}\scalebox{.7}{$\scriptscriptstyle '$}) - d(\alpha) - 1)\slash 2n$}}\zeta 
	\in D_{n,  {\scriptscriptstyle \zeta}}
	$ 
	for every 
	$
	\mathbf{s}' \in 
	D_{\scalebox{1.}{$\scriptscriptstyle {\mathbf{1\slash 2}}$}},
	$ 
	and all $\alpha \in \Phi_{n}.$ Let $a_{\scalebox{.97}{$\scriptscriptstyle 1$}}, \ldots, a_{r}$ be $\mathbb{Q}$-linearly independent 
	complex numbers, and put ${\mathbf{a}} : = (a_{\scalebox{.97}{$\scriptscriptstyle 1$}}, \ldots, a_{r}).$ 
	If $\rho$ is a sufficiently small positive number, then 
	$
	\mathbf{s}_{\scalebox{1.}{$\scriptscriptstyle \rho$}}{\!\!\!\scalebox{1.1}{$\scriptscriptstyle '$}} 
	: = {\mathbf{\frac{1}{2}}} + \rho {\mathbf{a}} \in D_{\scalebox{1.}{$\scriptscriptstyle {\mathbf{1\slash 2}}$}},
	$ 
	and the numbers 
	$
	\left(q^{\scalebox{1.}{$\scriptscriptstyle (d(\alpha) - 2\alpha(\mathbf{s}_{\scalebox{.8}{$\scriptscriptstyle \rho$}}
{\!\!\!\scalebox{.85}{$\scriptscriptstyle '$}}))\slash 2n$}}
	\right)_{\scriptscriptstyle \alpha \in \Phi_{\scalebox{.85}{$\scriptscriptstyle n$}}}
	$ 
	are mutually distinct; thus, by continuity, this property holds throughout a polydisk 
	$
	D_{\mathbf{s}_{\scalebox{.85}{$\scriptscriptstyle \rho$}}{\!\!\!\scalebox{.95}{$\scriptscriptstyle '$}}}
	\subset D_{\scalebox{1.}{$\scriptscriptstyle {\mathbf{1\slash 2}}$}}
	$ 
	centered at 
	$
	\mathbf{s}_{\scalebox{1.}{$\scriptscriptstyle \rho$}}{\!\!\!\scalebox{1.1}{$\scriptscriptstyle '$}}. 
	$ 
	
	Now, for each 
	$
	\mathbf{s}' \in 
	D_{\mathbf{s}_{\scalebox{.85}{$\scriptscriptstyle \rho$}}{\!\!\!\scalebox{.95}{$\scriptscriptstyle '$}}},
	$ 
	we can write (via a partial fraction decomposition in $\xi$) 
	\[
	\frac{\tilde{\mathbf{R}}_{n,\,  {\scriptscriptstyle \zeta}}(\mathbf{s}'\!, \xi)}
	{\prod_{\alpha \in \Phi_{n}}
		\!\left(1 - q^{\scalebox{.95}{$\scriptscriptstyle (d(\alpha) + 1 - 2\alpha(\mathbf{s}\scalebox{.7}{$\scriptscriptstyle '$}))\slash 2n$}}
		\zeta^{\scriptscriptstyle -1}\xi\right)}
	\, = \sum_{\alpha \in \Phi_{n}} 
	\frac{\mathbf{C}_{n,\,  {\scriptscriptstyle \zeta},\, \alpha}(\mathbf{s}')}
	{1 - q^{\scalebox{.95}{$\scriptscriptstyle (d(\alpha) + 1 - 2\alpha(\mathbf{s}\scalebox{.7}{$\scriptscriptstyle '$}))\slash 2n$}}
		\zeta^{\scriptscriptstyle -1}\xi}
	\] 
	for\, 
	$
		\xi \ne q^{\scalebox{.95}{$\scriptscriptstyle (2\alpha(\mathbf{s}\scalebox{.7}{$\scriptscriptstyle '$}) - d(\alpha) - 1)\slash 2n$}}
		\zeta
		$ 
		\,($\alpha \in \Phi_{n}$). Substituting this into \eqref{conseq-Weierstrass-division-theorem}, we find that, for each $\alpha \in \Phi_{n},$ 
\[
\mathbf{C}_{n,\,  {\scriptscriptstyle \zeta},\, \alpha}(\mathbf{s}') 
\;\;= \, \lim_{\xi \, \to \, q^{\scalebox{.95}{$\scriptscriptstyle (2\alpha(\mathbf{s}\scalebox{.7}{$\scriptscriptstyle '$}) - d(\alpha) - 1)\slash 2n$}}\zeta} \left(1 - q^{\scalebox{.95}{$\scriptscriptstyle (d(\alpha) + 1 - 2\alpha(\mathbf{s}\scalebox{.7}{$\scriptscriptstyle '$}))\slash 2n$}}
\zeta^{\scriptscriptstyle -1}\xi\right)\!\mathscr{W}\!(\mathbf{s}'\!, \xi) 
= \frac{\Gamma_{\!\! w_{\scalebox{.85}{$\scriptscriptstyle \alpha$}}}\!(1, a; \zeta^{\scriptscriptstyle -1})}
{2^{\scalebox{.95}{$\scriptscriptstyle l(w_{\scalebox{.85}{$\scriptscriptstyle \alpha$}})$}} n}
S_{\scriptscriptstyle \alpha}(\mathbf{s}'\!, \zeta^{\scriptscriptstyle -1}).
\] 
This gives our first assertion when 
$
\mathbf{s}' \in 
D_{\mathbf{s}_{\scalebox{.85}{$\scriptscriptstyle \rho$}}{\!\!\!\scalebox{.95}{$\scriptscriptstyle '$}}}; 
$ 
it extends to $D_{\scalebox{1.}{$\scriptscriptstyle {\mathbf{1\slash 2}}$}}$ by analytic continuation. 

We now have
\[
\frak{S}_{n}
\!\left(q^{-s_{\scalebox{.65}{$\scriptscriptstyle 1$}}}\!, \ldots, q^{-s_{\scalebox{.85}{$\scriptscriptstyle r$}}}\!, \xi, 1\right) 
\, = \sum_{\zeta}\,  \frac{\tilde{\mathbf{R}}_{n,\,  {\scriptscriptstyle \zeta}}(\mathbf{s}'\!, \xi)}
{\prod_{\alpha \in \Phi_{n}}
	\!\left(1 - q^{\scalebox{.95}{$\scriptscriptstyle (d(\alpha) + 1 - 2\alpha(\mathbf{s}\scalebox{.7}{$\scriptscriptstyle '$}))\slash 2n$}}
	\zeta^{\scriptscriptstyle -1}\xi\right)}
\] 
which proves the second assertion, and completes the proof. 
\end{proof}

\vskip5pt
\begin{thm}  --- Let $D, N \ge 1$ be integers, and suppose that 
  $r \ge 4.$ Then, under the Assumption
   \eqref{eq: Estimate-H1}, and
   Conjectures \ref{Conjecture-Meromorphic continuation},
   \ref{Conjecture-Meromorphic continuation-Z0}, we have
\[ 
\sum_{\substack{d - \mathrm{monic \; \& \; sq. \; free} \\ \deg \, d\, = \, D}}\;\,
L\big(\tfrac{1}{2}, \chi_{d}\!\big)^{\scalebox{1.25}{$\scriptscriptstyle r$}} 
= \sum_{n \, \le \, N} Q_{n}(D, q)q^{\left(\!\scalebox{.9}{$\scriptscriptstyle \frac{1}{2}$} 
+ \scalebox{.9}{$\scriptscriptstyle \frac{1}{2n}$}\!\right)\scalebox{1.2}{$\scriptscriptstyle D$}} + \, 
O_{\scalebox{.85}{$\scriptscriptstyle \Theta$}\scriptscriptstyle, \,  q, \, r}
	\!\left(q^{\scalebox{.9}{$\scriptscriptstyle D$} \scalebox{.9}{$\scriptscriptstyle (1 + \Theta)\slash 2$}}\right)
	\] 
	for any $(N + 1)^{\scriptscriptstyle -1} \! < \Theta  < N^{\scriptscriptstyle -1}\!,$ where 
	\[
	Q_{n}(D, q)q^{\left(\!\scalebox{.9}{$\scriptscriptstyle \frac{1}{2}$} 
		+ \scalebox{.9}{$\scriptscriptstyle \frac{1}{2n}$}\!\right)\scalebox{1.2}{$\scriptscriptstyle D$}} 
	= \sum_{\zeta^{\scriptscriptstyle \scalebox{.9}{$\scriptscriptstyle 2$} n} = 1}\, \underset{\xi \, = \, q^{\scriptscriptstyle - (n + \scalebox{.9}{$\scriptscriptstyle 1$})
			\slash \scalebox{.9}{$\scriptscriptstyle 2$}n}\zeta}{\mathrm{Res}}
\, \frak{S}_{n,\,  {\scriptscriptstyle \zeta}}\!\left({\mathbf{\tfrac{1}{2}}}, \xi \right)\xi^{\scriptscriptstyle - D - 1}
	\] 
	for $n \ge 1.$ 
\end{thm}

\begin{proof} \!Our assumptions imply that, for every small positive $\varepsilon,$ 
	the function $\mathscr{W}\!(\xi) = \mathscr{W}\!\big({\mathbf{\frac{1}{2}}}, \xi\big)$ is meromorphic in the disk\linebreak 
	$|\xi| < q^{\, \scriptscriptstyle - \frac{1}{2} - \varepsilon}$ with the only possible poles at 
	$  
	\xi = q^{\scriptscriptstyle - (n + \scalebox{.9}{$\scriptscriptstyle 1$})
		\slash \scalebox{.9}{$\scriptscriptstyle 2$}n}\zeta,
	$ 
	$n = 1, 2, \ldots,$ and for each $n,$ $\zeta$ is any $2n$-th root of $1.$ By Proposition 
	\ref{Analytic-continuation-frakSn}, the principal part of $\mathscr{W}\!(\xi)$ at a pole 
	$  
	\xi = q^{\scriptscriptstyle - (n + \scalebox{.9}{$\scriptscriptstyle 1$})
		\slash \scalebox{.9}{$\scriptscriptstyle 2$}n}\zeta
	$ 
	is given by  
	\[
\frac{\tilde{\mathbf{R}}_{n,\,  {\scriptscriptstyle \zeta}}\!\left({\mathbf{\frac{1}{2}}}, \xi \right)}
	{\left(1 - q^{\scalebox{.95}{$\scriptscriptstyle (n + \scalebox{.9}{$\scriptscriptstyle 1$})
	\slash \scalebox{.9}{$\scriptscriptstyle 2$}n$}}
		\zeta^{\scriptscriptstyle -1}\xi\right)^{\scriptscriptstyle |\Phi_{ \scalebox{.85}{$\scriptscriptstyle n$}}|}}
		= \frak{S}_{n,\,  {\scriptscriptstyle \zeta}}\!\left({\mathbf{\tfrac{1}{2}}}, \xi \right)\!.
	\] 
	From here on, the argument proceeds exactly as in the proof of Theorem \ref{Main Theorem: Full-moment-asymptotics}. 
\end{proof}

\begin{rem} --- The order of the pole of the rational function 
	$
	\frak{S}_{n,\,  {\scriptscriptstyle \zeta}}\!\left({\mathbf{\tfrac{1}{2}}}, \xi \right)
	$ 
	at 
	$  
	\xi = q^{\scriptscriptstyle - (n + \scalebox{.9}{$\scriptscriptstyle 1$})
		\slash \scalebox{.9}{$\scriptscriptstyle 2$}n}\zeta
	$ 
	\,may be smaller than the maxi-\linebreak mal limit $|\Phi_{n}|,$ and this happens, for example, when $n = 1.$ While, for 
	any given $n \ge 1,$ the function $\frak{S}_{n,\,  {\scriptscriptstyle \zeta}}(\mathbf{s}'\!, \xi)$ (and 
	hence the coefficient $Q_{n}(D, q)$ in the above asymptotic formula) can, in principle, be computed explicitly, 
	we do not have any prediction, at the moment, for what the order of the pole of 
	$
	\frak{S}_{n,\,  {\scriptscriptstyle \zeta}}\!\left({\mathbf{\tfrac{1}{2}}}, \xi \right)
	$ 
	at 
	$\xi = q^{\scriptscriptstyle - (n + \scalebox{.9}{$\scriptscriptstyle 1$})\slash \scalebox{.9}{$\scriptscriptstyle 2$}n}\zeta
	$ 
	should, in general, be. 
\end{rem}

\section{Some explicit computations}
\label{computations-introduction}

\subsection{The case $n = 1$} \label{Case n=1}
By Lemma \ref{Finiteness-roots-level-n}, each 
$
\alpha \in \Phi_{\scriptscriptstyle 1}
$ 
is of the form 
\[
\alpha = \sum_{i = 1}^{r} k_{\scriptscriptstyle i}\alpha_{\scriptscriptstyle i} \, +\, 
\alpha_{r + {\scriptscriptstyle 1}}
\] 
with $k_{\scriptscriptstyle i} = 0$ or $1$ for all $i = 1, \ldots, r.$ Conversely, any element 
$
\alpha_{\scriptscriptstyle i_{\scalebox{.62}{$\scriptscriptstyle 1$}}} \! 
+ \, \alpha_{\scriptscriptstyle i_{\scalebox{.62}{$\scriptscriptstyle 2$}}} \! 
+ \cdots + \alpha_{\scriptscriptstyle i_{\scalebox{.75}{$\scriptscriptstyle m$}}} \! 
+ \, \alpha_{r + {\scriptscriptstyle 1}}
$ 
is a root, since 
$$
w_{\scriptscriptstyle i_{\scalebox{.62}{$\scriptscriptstyle 1$}}} \! \cdots \, 
w_{\scriptscriptstyle i_{\scalebox{.75}{$\scriptscriptstyle m$}}}\!(\alpha_{r + {\scriptscriptstyle 1}}) 
= \alpha_{\scriptscriptstyle i_{\scalebox{.62}{$\scriptscriptstyle 1$}}}\! + \cdots + 
\alpha_{\scriptscriptstyle i_{\scalebox{.75}{$\scriptscriptstyle m$}}} \! + \,
\alpha_{r + \scriptscriptstyle 1};
$$ 
the subgroup $W_{\scriptscriptstyle 0} :  = \langle w_{\scriptscriptstyle i} : i = 1, \ldots, r \rangle$ is isomorphic to 
$(\mathbb{Z}\slash 2\mathbb{Z})^{r}.$ Taking $w$ to be the identity, we find by \eqref{eq: functions-Lwp} that 
$
L_{{\scriptscriptstyle \mathrm{id}}, p}^{\scriptscriptstyle (1)} = \mathrm{sgn}(a)^{\deg \, p}\!,
$ 
$
L_{{\scriptscriptstyle \mathrm{id}}, p}^{\scriptscriptstyle (2)} = 0
$ 
and 
$L_{{\scriptscriptstyle \mathrm{id}}, p}^{\scriptscriptstyle (3)} = 1.$ Thus we get: 
\begin{equation*}
S_{\! p}^{{\scriptscriptstyle \mathrm{id}}}  = \left(1 - \frac{1}{|p|}\right)
\left(\frac{1}{|p|}  +  \frac{1}{2\prod_{k = 1}^{r}\left(1 - |p|^{\, - s_{\scalebox{.75}{$\scriptscriptstyle k$}}}\!\right)} 
+  \frac{1}{2\prod_{k = 1}^{r}\left(1 + |p|^{\, - s_{\scalebox{.75}{$\scriptscriptstyle k$}}}\!\right)}\right). 
\end{equation*} 
Similarly, by \eqref{eq: function-Gammaw}, 
\begin{equation*}
\Gamma_{\! {\scriptscriptstyle \mathrm{id}}}(a_{\scriptscriptstyle 2}, a; \zeta_{a})
= \varepsilon^{\scriptscriptstyle +}\!(a_{\scriptscriptstyle 2})\varepsilon^{\scriptscriptstyle +}\!(a) \, + \, \varepsilon^{\scriptscriptstyle -}\!(a_{\scriptscriptstyle 2})\varepsilon^{\scriptscriptstyle -}\!(a).
\end{equation*} 
Summing now over $a,$ we find that the contribution to $\frak{S}_{1}(\cdot, a_{\scriptscriptstyle 2})$ corresponding to 
$\alpha = \alpha_{r + \scriptscriptstyle 1}$ is 
\begin{equation*}
\left(\frac{\varepsilon^{\scriptscriptstyle +}\!(a_{\scriptscriptstyle 2})}{1  -  q^{{\scriptscriptstyle 1} - s_{r + \scalebox{.85}{$\scriptscriptstyle 1$}}}} + \frac{\varepsilon^{\scriptscriptstyle -}\!(a_{\scriptscriptstyle 2})}
{1  +  q^{{\scriptscriptstyle 1} - s_{r + \scalebox{.85}{$\scriptscriptstyle 1$}}}} \right)
\, \cdot \, \prod_{p}\left(1 - \frac{1}{|p|} \right)
\left(\frac{1}{|p|}  +  \frac{1}{2\prod_{k = 1}^{r}\left(1 - |p|^{\, - s_{\scalebox{.75}{$\scriptscriptstyle k$}}}\!\right)} 
+  \frac{1}{2\prod_{k = 1}^{r}\left(1 + |p|^{\, - s_{\scalebox{.75}{$\scriptscriptstyle k$}}}\!\right)}\right).
\end{equation*} 
To compute the other contributions, write 
$
q^{- s_{\scalebox{.75}{$\scriptscriptstyle k$}}} = q^{{\scriptscriptstyle -1\slash 2}}\, \xi_{k}
$ 
($k = 1, \ldots, r$), and put $\xi = q^{ - s_{r + \scalebox{.85}{$\scriptscriptstyle 1$}}}$. Set 
\begin{equation*}
S_{\! p}(\xi_{\scriptscriptstyle 1}, \ldots, \xi_{r}) \, 
: = \, \left(1 - \frac{1}{|p|} \right)\left(\frac{1}{|p|}  +  \frac{\prod_{k = 1}^{r}\big(1 - (q^{\, \scriptscriptstyle - 1\slash 2}\, \xi_{k})^{\deg \, p}\big)^{\!\scriptscriptstyle -1} + \, \prod_{k = 1}^{r}\big(1 + (q^{\, \scriptscriptstyle - 1\slash 2}\, \xi_{k})^{\deg \, p}\big)^{\!\scriptscriptstyle -1}}{2}\right).
\end{equation*} 
Let 
$
\alpha = \sum k_{\scriptscriptstyle i}\alpha_{\scriptscriptstyle i}  + 
\alpha_{r + \scriptscriptstyle 1}
$ 
($k_{\scriptscriptstyle i} = 0$ or $1$) be a root in $\Phi_{\scriptscriptstyle 1},$ and take $w = w_{\scriptscriptstyle \alpha} = w_{\scriptscriptstyle 1}^{k_{\scalebox{.62}{$\scriptscriptstyle 1$}}} 
\cdots \, w_{r}^{k_{\scalebox{.8}{$\scriptscriptstyle r$}}};$ then 
$
\mathrm{M}_{w^{\scriptscriptstyle -1}}(w\mathrm{z}; q) = \mathrm{M}_{w}(\mathrm{z}; q)^{-1}
$ 
is a diagonal matrix. By \eqref{eq: functions-Lwp} and \eqref{eq: Final-expression-Sp}, we have 
\begin{equation*}
S_{\! p}^{w}  = \left(1 - \frac{1}{|p|} \right)
\left(D_{w, \, p}^{\scriptscriptstyle (2)}(\mathrm{sgn}(a)\xi)^{\deg \, p} \, + \; 
\frac{D_{w, \, p}^{\scriptscriptstyle (1)}}{2\prod_{k = 1}^{r}\left(1 - (q^{\, \scriptscriptstyle - 1\slash 2}\, \xi_{k})^{\deg \, p}\right)} 
\left. \, + \;\, \frac{D_{w, \, p}^{\scriptscriptstyle (3)}}{2\prod_{k = 1}^{r}\left(1 + (q^{\, \scriptscriptstyle - 1\slash 2}\, \xi_{k})^{\deg \, p}\right)}\right)\right\vert_{\scriptscriptstyle \xi   \, = \frac{1}{q \, \mathrm{sgn}(a)\, 
\xi_{\scalebox{.8}{$\scriptscriptstyle 1$}}^{k_{\scalebox{.62}{$\scriptscriptstyle 1$}}} \! \cdots \, \xi_{r}^{k_{\scalebox{.8}{$\scriptscriptstyle r$}}}}}
\end{equation*} 
where 
$
\mathrm{M}_{w}\!\left((q^{\, \scriptscriptstyle - 1\slash 2}\, \xi_{\scriptscriptstyle 1})^{\deg \, p}\!, \ldots, (q^{\, \scriptscriptstyle - 1\slash 2}\, \xi_{r})^{\deg \, p}; |p| \right)^{\scriptscriptstyle -1} \! 
= \; \mathrm{diag}(D_{w, \, p}^{\scriptscriptstyle (1)}, D_{w, \, p}^{\scriptscriptstyle (2)}, D_{w, \, p}^{\scriptscriptstyle (3)}).
$ 
This diagonal matrix can be easily computed by using the cocycle relation \eqref{eq: def-cocycle2} and induction on the length of $w;$ we have 
$
\mathrm{M}_{w}(\mathrm{z}; q)^{\scriptscriptstyle -1} \! 
= \prod_{i} 
\mathrm{M}_{w_{\scalebox{.7}{$\scriptscriptstyle i$}}}\!(\mathrm{z}; q)^{ - k_{\scalebox{.75}{$\scriptscriptstyle i$}}}\!,
$ 
and by \eqref{eq: def-cocycle1}, 
$$
\mathrm{M}_{w_{\scalebox{.7}{$\scriptscriptstyle i$}}}\!(\mathrm{z}; q)^{\scriptscriptstyle -1} \! 
= \, \mathrm{diag}\!\left( - \frac{q z_{\scalebox{.9}{$\scriptscriptstyle i$}} (1 - z_{\scalebox{.9}{$\scriptscriptstyle i$}})}
{1 - q z_{\scalebox{.9}{$\scriptscriptstyle i$}}}, 
\, \sqrt{q} z_{\scalebox{.9}{$\scriptscriptstyle i$}}, 
\, \frac{q z_{\scalebox{.9}{$\scriptscriptstyle i$}} 
(1 + z_{\scalebox{.9}{$\scriptscriptstyle i$}})}
{1 + q z_{\scalebox{.9}{$\scriptscriptstyle i$}}} \right) 
\qquad \text{($i = 1, \ldots, r$)}.
$$ 
Accordingly, 
$
S_{\! p}^{w}(\xi_{\scriptscriptstyle 1}, \ldots, \xi_{r}) 
= S_{\! p}(\xi_{\scriptscriptstyle 1}^{\varepsilon_{\scalebox{.8}{$\scriptscriptstyle 1$}}}\!, \ldots, \xi_{r}^{\varepsilon_{r}}),
$ 
where $\varepsilon_{\scalebox{.9}{$\scriptscriptstyle i$}} : = 1 - 2k_{\scalebox{.9}{$\scriptscriptstyle i$}}.$ We also have 
\begin{equation*}
\begin{split}
2^{{\scriptscriptstyle - l(w)}}
\Gamma_{\! w}(a_{\scriptscriptstyle 2}, a; \zeta_{a})
& = (1, \, \mathrm{sgn}(a_{\scriptscriptstyle 2}), \, 0) \, \cdot \, \mathrm{M}_{w}\big(q^{\scriptscriptstyle 1\slash 2}\, \xi_{\scriptscriptstyle 1}, 
\ldots,\, q^{\scriptscriptstyle 1\slash 2}\, \xi_{r}; 1\slash q \big)\, \cdot \,  
^{t}\!(1\slash 2, \, \mathrm{sgn}(a)\slash 2, \, 1\slash 2) \\
& = \, \frac{1}{2}\prod_{i = 1}^{r} \left(\frac{1 -  q^{\scriptscriptstyle 1\slash 2}\, 
\xi_{\scalebox{1.2}{$\scriptscriptstyle i$}}^{\scriptscriptstyle -1}}
{1 - q^{\scriptscriptstyle 1\slash 2}\, \xi_{\scalebox{1.2}{$\scriptscriptstyle i$} }^{} }
\right)^{\! (1  -  \varepsilon_{\scalebox{.8}{$\scriptscriptstyle i$}})\slash 2}   \! 
+ \;\,   \frac{1}{2}\mathrm{sgn}(a)\mathrm{sgn}(a_{\scriptscriptstyle 2})\prod_{i = 1}^{r}
\xi_{\scalebox{1.2}{$\scriptscriptstyle i$}}^{\scriptscriptstyle (\varepsilon_{\scalebox{.8}{$\scriptscriptstyle i$}} - 1)\slash 2}.
\end{split}
\end{equation*} 
The product 
$
\prod_{p} S_{\! p}(\xi_{\scriptscriptstyle 1}, \ldots, \xi_{r})
$ 
does not converge when $|\xi_{k}| = 1,$ but it has meromorphic continuation. Indeed, letting 
$A_{\! p}(\xi_{\scriptscriptstyle 1}, \ldots, \xi_{r})$ be defined by 
\[
A_{\! p}(\xi_{\scriptscriptstyle 1}, \ldots, \xi_{r}) \;\;
= \prod_{1 \, \le \, i \, \le \, j \, \le \, r} 
\left(1 - \frac{(\xi_{\scalebox{1.2}{$\scriptscriptstyle i$}} 
\xi_{\scalebox{1.2}{$\scriptscriptstyle j$}})^{\deg \, p}}{|p|} \right) 
\, \cdot \, \left( 1 \, + \, \frac{- \, 2 + \prod_{k = 1}^{r}\big(1 - (q^{\, \scriptscriptstyle - 1\slash 2}\, \xi_{k})^{\deg \, p}\big)^{\!\scriptscriptstyle -1} + \, \prod_{k = 1}^{r}\big(1 + (q^{\, \scriptscriptstyle - 1\slash 2}\, \xi_{k})^{\deg \, p}\big)^{\!\scriptscriptstyle -1}}{2 \left(1 + |p|^{\scriptscriptstyle -1}\right)} 
\right)
\] 
it is easy to see that the product 
$
\prod_{p} A_{\! p}(\xi_{\scriptscriptstyle 1}, \ldots, \xi_{r})
$ 
is absolutely convergent in the polydisk 
$
|\xi_{k}| \le q^{\scriptscriptstyle \frac{1}{4} - \varepsilon}
$ 
($k = 1, \ldots, r$), with sufficiently small positive $\varepsilon.$ It is clear that 
\[
\prod_{p} S_{\! p}(\xi_{\scriptscriptstyle 1}, \ldots, \xi_{r}) 
= \frac{1}{\zeta(2)} 
\!\prod_{1 \, \le \, i \, \le \, j \, \le \, r} 
\left(1 - \xi_{\scalebox{1.2}{$\scriptscriptstyle i$}} 
\xi_{\scalebox{1.2}{$\scriptscriptstyle j$}}\right)^{\scriptscriptstyle -1} 
\cdot \; \prod_{p} A_{\! p}(\xi_{\scriptscriptstyle 1}, \ldots, \xi_{r}).
\] 
Thus if we write the right-hand side of this as 
$
\zeta(2)^{\scriptscriptstyle -1}
G(\xi_{\scriptscriptstyle 1}, \ldots, \xi_{r}),
$ 
then we can express 
\begin{equation} \label{eq: Main-residue-formula}
\begin{split}
\frak{S}_{\scriptscriptstyle 1} 
& = \frac{1}{\zeta(2)}\sum_{\varepsilon_{\scalebox{.8}{$\scriptscriptstyle i$}} = \pm 1} 
\, \left( \prod_{i = 1}^{r} \left(\frac{1 -  q^{\scriptscriptstyle 1\slash 2}
\, \xi_{\scalebox{1.2}{$\scriptscriptstyle i$}}^{\scriptscriptstyle -1}}
{1 - q^{\scriptscriptstyle 1\slash 2}\, \xi_{\scalebox{1.2}{$\scriptscriptstyle i$}}^{}}
\right)^{\! (1  -  \varepsilon_{\scalebox{.8}{$\scriptscriptstyle i$}})\slash 2}
\cdot \; \frac{G(\xi_{\scriptscriptstyle 1}^{\varepsilon_{\scalebox{.8}{$\scriptscriptstyle 1$}}}\!, \ldots, \xi_{r}^{\varepsilon_{r}}\!)}
{1  -  q^{2}\, \xi_{\scriptscriptstyle 1}^{1  -  \varepsilon_{\scalebox{.8}{$\scriptscriptstyle 1$}}} 
\cdots \;  \xi_{r}^{1  -  \varepsilon_{\scriptscriptstyle r}}\xi^{^{2}}} \right)\\
& + \, \frac{(q \, \xi)\, \mathrm{sgn}(a_{\scriptscriptstyle 2})}{\zeta(2)} 
\!\sum_{\varepsilon_{\scalebox{.8}{$\scriptscriptstyle i$}} = \pm 1} 
\frac{G(\xi_{\scriptscriptstyle 1}^{\varepsilon_{\scalebox{.8}{$\scriptscriptstyle 1$}}}\!, \ldots, \xi_{r}^{\varepsilon_{r}} \!)}
{1  -  q^{2}\, \xi_{\scriptscriptstyle 1}^{1  -  \varepsilon_{\scalebox{.8}{$\scriptscriptstyle 1$}}} 
\cdots \;  \xi_{r}^{1  -  \varepsilon_{\scriptscriptstyle r}}\xi^{^{2}}}.
\end{split}
\end{equation} 
It turns out that this function can be defined and is analytic in an entire neighborhood of 
$|\xi_{k}|  = 1$ ($k = 1, \ldots, r$). \!To see this, \!we need the following lemma -- a variant of \cite[Lemma 2.5.2]{CFKRS}.

\vskip10pt
\begin{lem} \label{average-to-integral-vers1} --- Let $a_{\scalebox{.97}{$\scriptscriptstyle 1$}}, \ldots, a_{r}$ be distinct non-zero complex numbers such that $a_{\scriptscriptstyle i}a_{\!\scriptscriptstyle j} \ne 1$ for all $1\le i, j \le r.$ Suppose $h$ is a symmetric function of $r$ variables, holomorphic on a domain containing 
	$
	\Big(a_{\scriptscriptstyle 1}^{\delta_{\scalebox{.85}{$\scriptscriptstyle 1$}}}\!, \ldots, a_{r}^{\delta_{r}}\Big)
	$ 
	for all 
	$
	(\delta_{\scalebox{.97}{$\scriptscriptstyle 1$}}, \ldots, \delta_{r}) \in \{\pm 1 \}^{r}.
	$ 
	Then 
	\begin{equation*}
	\sum_{\delta_{\!\scriptscriptstyle j} \, = \, \pm 1} 
	\frac{h\big(a_{\scriptscriptstyle 1}^{\delta_{\scalebox{.85}{$\scriptscriptstyle 1$}}}\!, \ldots, a_{r}^{\delta_{r}}\big)}
	{\prod_{1 \, \le \, i \, \le \, j \, \le \, r} 
		\Big(1 - a_{i}^{\delta_{i}}a_{\!\scriptscriptstyle j}^{\delta_{\scalebox{.75}{$\scriptscriptstyle j$}}}\Big)}
	= \frac{(-1)^{r(r + 1)\slash 2}}{r!}\frac{1}{(2\pi \sqrt{-1})^{^{\! r}}}\!\oint \cdots \oint
	h(z_{\scalebox{.97}{$\scriptscriptstyle 1$}}, \ldots, z_{r}) 
	\frac{\prod_{1\le i < j \le r}(z_{\!\scriptscriptstyle j}^{} - z_{\scriptscriptstyle i}^{})^{\scriptscriptstyle 2}
		(1 - z_{\scriptscriptstyle i}^{} z_{\!\scriptscriptstyle j}^{})}
	{\prod_{i, j = 1}^{r} (1 - z_{\scriptscriptstyle i}^{}a_{\!\scriptscriptstyle j}^{})
		(1 - z_{\scriptscriptstyle i}^{}a_{\!\scriptscriptstyle j}^{\scriptscriptstyle -1})}\, 
\frac{d z_{\scalebox{.85}{$\scriptscriptstyle 1$}}^{}}{z_{\scalebox{.85}{$\scriptscriptstyle 1$}}^{\scriptscriptstyle r}} \cdots 
	\frac{d z_{\scriptscriptstyle r}^{}}{z_{\scriptscriptstyle r}^{\scriptscriptstyle r}}
	\end{equation*}
	where each path of integration encloses the $a_{\!\scriptscriptstyle j}^{ \scriptscriptstyle \pm 1}\!,$ but not 
	$
	z_{\!\scriptscriptstyle j}^{} = 0.
	$ 
\end{lem}

\begin{proof} The right-hand side is the sum of the residues at the poles (with non-zero coordinates) of the integrand.

	Fix a pole at $z_{\scriptscriptstyle i}^{} = a_{\!\scriptscriptstyle j_{\scalebox{.62}{$\scriptscriptstyle i$}}}^{\scriptscriptstyle \pm 1}$ for $i = 1, \ldots, r.$ We cannot have $j_{\scriptscriptstyle i} = \! j_{\scriptscriptstyle i'}$ for some $i < i',$ 
as the factor 
	$
	(z_{\scriptscriptstyle i'} \! - z_{\scriptscriptstyle i}^{})^{\scriptscriptstyle 2}(1 - z_{\scriptscriptstyle i}^{} 
	z_{\scriptscriptstyle i'}\!)
	$ 
	in the numerator vanishes when 
	$
	z_{\scriptscriptstyle i}^{} = a_{\!\scriptscriptstyle j_{\scalebox{.62}{$\scriptscriptstyle i$}}}^{\scriptscriptstyle \pm 1}
	$ 
	and 
	$
	z_{\scriptscriptstyle i'} \! = a_{\!\scriptscriptstyle j_{\scalebox{.62}{$\scriptscriptstyle i$}}}^{\scriptscriptstyle \pm 1}\!.
	$ 
	Thus the assignment $\sigma : i \mapsto j_{\scriptscriptstyle i}$ is an element of the symmetric group $\mathbb{S}_{r}.$ If we denote the exponent $\{\pm 1\}$ of 
	$
	a_{\!\scriptscriptstyle j_{\scalebox{.62}{$\scriptscriptstyle i$}}} \! = a_{\scalebox{.95}{$\scriptscriptstyle \sigma(i)$}}
	$ 
	by $\delta_{\scalebox{.95}{$\scriptscriptstyle \sigma(i)$}},$ it follows (upon replacing $z_{\scriptscriptstyle i}$ by 
	$
	z_{\scalebox{.95}{$\scriptscriptstyle \sigma^{\scalebox{.7}{$\scriptscriptstyle -1$}}\!(i)$}}
	$) 
	that the residue of the integrand at the pole 
	$
	\Big(a_{\scalebox{.95}{$\scriptscriptstyle \sigma(1)$}}^{\delta_{\scalebox{.75}{$\scriptscriptstyle \sigma(1)$}}}\!, \ldots, a_{\scalebox{.95}{$\scriptscriptstyle \sigma(r)$}}^{\delta_{\scalebox{.75}{$\scriptscriptstyle \sigma(r)$}}}\!\Big)
	$ 
	is 
	\begin{equation*}
	\frac{h\big(a_{\scriptscriptstyle 1}^{\delta_{\scalebox{.85}{$\scriptscriptstyle 1$}}}\!, \ldots, a_{r}^{\delta_{r}}\big)}
	{\prod_{j = 1}^{r} a_{\!\scriptscriptstyle j}^{r\, \delta_{\scalebox{.85}{$\scriptscriptstyle j$}}}} \;\, \cdot 
	\prod_{1 \, \le \, i \, < \, j \, \le \, r}\frac{\big(a_{\!\scriptscriptstyle j}^{\delta_{\scalebox{.85}{$\scriptscriptstyle j$}}} 
		\! - a_{\scriptscriptstyle i}^{\delta_{i}}\big)^{\!\scriptscriptstyle 2}
		\big(1 - a_{\scriptscriptstyle i}^{\delta_{i}}a_{\!\scriptscriptstyle j}^{\delta_{\scalebox{.85}{$\scriptscriptstyle j$}}}\big)}
	{\big(1 - a_{\!\scriptscriptstyle j}^{\delta_{\scalebox{.85}{$\scriptscriptstyle j$}}}a_{\scriptscriptstyle i}^{}\big)
		\big(1 - a_{\!\scriptscriptstyle j}^{\delta_{\scalebox{.85}{$\scriptscriptstyle j$}}}a_{\scriptscriptstyle i}^{\scriptscriptstyle -1}\big)
		\big(1 - a_{\scriptscriptstyle i}^{\delta_{i}}a_{\!\scriptscriptstyle j}^{}\big)
		\big(1 - a_{\scriptscriptstyle i}^{\delta_{i}}a_{\!\scriptscriptstyle j}^{\scriptscriptstyle -1}\big)}
	\prod_{k = 1}^{r}\bigg(\! -\frac{a_{\scriptscriptstyle k}^{\delta_{\scalebox{.95}{$\scriptscriptstyle k$}}}}
	{1 - a_{\scriptscriptstyle k}^{2\delta_{\scalebox{.95}{$\scriptscriptstyle k$}}}}\bigg).
	\end{equation*} 
	For $i < j,$ we have
	\begin{equation*}
	\frac{\big(a_{\!\scriptscriptstyle j}^{\delta_{\scalebox{.85}{$\scriptscriptstyle j$}}} \! - a_{\scriptscriptstyle i}^{\delta_{i}}\big)^{\!\scriptscriptstyle 2}\big(1 - a_{\scriptscriptstyle i}^{\delta_{i}}a_{\!\scriptscriptstyle j}^{\delta_{\scalebox{.85}{$\scriptscriptstyle j$}}}\big)}
	{\big(1 - a_{\!\scriptscriptstyle j}^{\delta_{\scalebox{.85}{$\scriptscriptstyle j$}}}a_{\scriptscriptstyle i}^{}\big)
		\big(1 - a_{\!\scriptscriptstyle j}^{\delta_{\scalebox{.85}{$\scriptscriptstyle j$}}}
		a_{\scriptscriptstyle i}^{\scriptscriptstyle -1}\big)
		\big(1 - a_{\scriptscriptstyle i}^{\delta_{i}}a_{\!\scriptscriptstyle j}^{}\big)
		\big(1 - a_{\scriptscriptstyle i}^{\delta_{i}}a_{\!\scriptscriptstyle j}^{\scriptscriptstyle -1}\big)} 
	= -\, \frac{a_{\scriptscriptstyle i}^{\delta_{i}}a_{\!\scriptscriptstyle j}^{\delta_{\scalebox{.85}{$\scriptscriptstyle j$}}}}
	{1 - a_{\scriptscriptstyle i}^{\delta_{i}}a_{\!\scriptscriptstyle j}^{\delta_{\scalebox{.85}{$\scriptscriptstyle j$}}}}
	\end{equation*} 
	and the equality in the lemma follows.
\end{proof} 

To apply the formula in the lemma, note that in \eqref{eq: Main-residue-formula} we can express: 
\[
\prod_{i = 1}^{r} \left(\frac{1 -  q^{\scriptscriptstyle 1\slash 2}
\xi_{\scalebox{1.2}{$\scriptscriptstyle i$}}^{\scriptscriptstyle -1}}
{1 - q^{\scriptscriptstyle 1\slash 2}\xi_{\scalebox{1.2}{$\scriptscriptstyle i$}}^{}}
\right)^{\! (1  -  \varepsilon_{\scalebox{.8}{$\scriptscriptstyle i$}})\slash 2}  
\! = \;\; \prod_{i = 1}^{r} \!\frac{1 -  q^{\scriptscriptstyle 1\slash 2}
\xi_{\scalebox{1.2}{$\scriptscriptstyle i$}}^{\varepsilon_{\scalebox{.8}{$\scriptscriptstyle i$}}}}
{1 - q^{\scriptscriptstyle 1\slash 2} \xi_{\scalebox{1.2}{$\scriptscriptstyle i$}}^{}}  
\qquad \text{(for $\varepsilon_{\scalebox{.9}{$\scriptscriptstyle i$}} \in \{\pm 1\},$ $i = 1, \ldots, r$).}
\] 
Thus if we let 
\begin{equation*}
\begin{split}
&\mathscr{H}_{\scriptscriptstyle 1}\!\left(\xi_{\scriptscriptstyle 1}, \ldots, \xi_{r}, u\right) \, = \, \frac{\prod_{i = 1}^{r} \left(1 -  q^{\scriptscriptstyle 1\slash 2}\xi_{\scriptscriptstyle i}\right) }
{1 - u \left(\xi_{\scriptscriptstyle 1} 
	\cdots \,  \xi_{r}\right)^{\scriptscriptstyle -1}}
\, \cdot\, \mathrm{A}\!\left(\xi_{\scriptscriptstyle 1}, \ldots, \xi_{r}\right) \\
&\mathscr{H}_{\scriptscriptstyle 2}\!\left(\xi_{\scriptscriptstyle 1}, \ldots, \xi_{r}, u\right) \, = \, 
\frac{\mathrm{A}\!\left(\xi_{\scriptscriptstyle 1}, \ldots, \xi_{r}\right)}
{1 - u \left(\xi_{\scriptscriptstyle 1} 
	\cdots \,  \xi_{r}\right)^{\scriptscriptstyle -1}}
\end{split}
\end{equation*}
with 
$
\mathrm{A}\!\left(\xi_{\scriptscriptstyle 1}, \ldots, \xi_{r}\right) 
: = \prod_{p} A_{\! p}\!\left(\xi_{\scriptscriptstyle 1}, \ldots, \xi_{r}\right)\!,
$ 
then we can write 
\begin{equation*}
\begin{split}
\frak{S}_{\scriptscriptstyle 1} & = \frac{1}{\zeta(2)}\prod_{i = 1}^{r} \left(1 - q^{\scriptscriptstyle 1\slash 2}\, \xi_{\scalebox{1.2}{$\scriptscriptstyle i$}}\right)^{\!\scriptscriptstyle - 1}
\, \cdot\, \sum_{\varepsilon_{\scalebox{.8}{$\scriptscriptstyle i$}} = \pm 1} 
\frac{\mathscr{H}_{\scriptscriptstyle 1}
\!\left(\xi_{\scriptscriptstyle 1}^{\varepsilon_{\scalebox{.8}{$\scriptscriptstyle 1$}}}\!, \ldots, \xi_{r}^{\varepsilon_{r}}\!, 
q^{\scriptscriptstyle 2}\, \xi_{\scriptscriptstyle 1}^{} \cdots \;  \xi_{r}^{}\, \xi^{\scriptscriptstyle 2}\right)}
{\prod_{1 \, \le \, i \, \le \, j \, \le \, r} 
\big(1 - \xi_{\scalebox{1.1}{$\scriptscriptstyle i$}}^{\varepsilon_{\scalebox{.8}{$\scriptscriptstyle i$}}} 
\xi_{\scalebox{.93}{$\scriptscriptstyle j$}}^{\varepsilon_{\!\scalebox{.65}{$\scriptscriptstyle j$}}}\big)}\\
& + \, \frac{\mathrm{sgn}(a_{\scriptscriptstyle 2}) q\, \xi}{\zeta(2)}
\!\sum_{\varepsilon_{\scalebox{.8}{$\scriptscriptstyle i$}} = \pm 1} 
\frac{\mathscr{H}_{\scriptscriptstyle 2}
\!\left(\xi_{\scriptscriptstyle 1}^{\varepsilon_{\scalebox{.8}{$\scriptscriptstyle 1$}}}\!, \ldots, \xi_{r}^{\varepsilon_{r}}\!, 
	q^{\scriptscriptstyle 2}\, \xi_{\scriptscriptstyle 1}^{} \cdots \;  \xi_{r}^{}\, \xi^{\scriptscriptstyle 2}\right)}
{\prod_{1 \, \le \, i \, \le \, j \, \le \, r} 
	\big(1 - \xi_{\scalebox{1.1}{$\scriptscriptstyle i$}}^{\varepsilon_{\scalebox{.8}{$\scriptscriptstyle i$}}} 
	\xi_{\scalebox{.93}{$\scriptscriptstyle j$}}^{\varepsilon_{\!\scalebox{.65}{$\scriptscriptstyle j$}}}\big)}.
\end{split}
\end{equation*} 
The sum in the right-hand side can be expressed, using Lemma \ref{average-to-integral-vers1}, as a multiple integral, and note that the integrand is\linebreak defined and holomorphic in a neighborhood of $\xi_{\scriptscriptstyle 1} = \cdots = \xi_{r} = 1.$ 
Thus for each $D \in \mathbb{N},$ the coefficient $Q_{\scriptscriptstyle 1}\!(D, q)$ can be obtained from the integral 
\[
\frac{1}{2 \pi \sqrt{-1}}\; \oint\limits_{|\xi| \, = \, {\scalebox{1.5}{$\scriptscriptstyle q$}}^{\, \scriptscriptstyle - 2}} 
\!\frak{S}_{\scriptscriptstyle 1}
\!\left(q^{\scriptscriptstyle -1\slash 2}\!, \ldots, 
q^{\scriptscriptstyle -1\slash 2}\!,  \xi, 1\right) 
\frac{d\xi}{\xi^{\scriptscriptstyle D + 1}}
\] 
see for comparison \cite[Conjecture~5]{AK}.

\subsection{The case $n = 2$} \label{Case n=2}
First the set
$
\Phi_{\scriptscriptstyle 2}
$
is computed as follows. Let $\alpha \in \Phi_{\scriptscriptstyle 2}$
satisfy the condition \eqref{eq: condition*}. Then the coefficients of
$\alpha$ also satisfy the inequality \eqref{eq: condition**}, i.e.
\[
  \sum_{j = 1}^{r} k_{\!\scriptscriptstyle j} \le 3.
\]
It follows that at most three $k_{\!\scriptscriptstyle j}$ ($1 \le j
\le r$) can be non-zero, and hence by \eqref{eq: effect-coeff-roots}, 
$
w_{r + {\scriptscriptstyle 1}}(\alpha)
= \alpha_{\!\scriptscriptstyle j_{\scalebox{.62}{$\scriptscriptstyle 1$}}} \! + 
\alpha_{\!\scriptscriptstyle j_{\scalebox{.62}{$\scriptscriptstyle 2$}}}
$ 
(with $1 \le j_{\scriptscriptstyle 1} < j_{\scriptscriptstyle 2} \le r$), or 
$
w_{r +  {\scriptscriptstyle 1}}(\alpha) =
\alpha_{\!\scriptscriptstyle j_{\scalebox{.62}{$\scriptscriptstyle 1$}}} \!+ 
\alpha_{\!\scriptscriptstyle j_{\scalebox{.62}{$\scriptscriptstyle 2$}}} 
\!+ \alpha_{\!\scriptscriptstyle j_{\scalebox{.62}{$\scriptscriptstyle 3$}}} 
\!+ \alpha_{r + {\scriptscriptstyle 1}}
$ 
for some $1 \le j_{\scriptscriptstyle 1} < j_{\scriptscriptstyle 2} <
j_{\scriptscriptstyle 3} \le r.$ 
However, $w_{\!\scriptscriptstyle j_{\scalebox{.62}{$\scriptscriptstyle 1$}}}
\!(\alpha_{\!\scriptscriptstyle j_{\scalebox{.62}{$\scriptscriptstyle 1$}}} \! + 
\alpha_{\!\scriptscriptstyle j_{\scalebox{.62}{$\scriptscriptstyle 2$}}}\!) =
-\,\alpha_{\!\scriptscriptstyle j_{\scalebox{.62}{$\scriptscriptstyle 1$}}}
 \! + \alpha_{\!\scriptscriptstyle
   j_{\scalebox{.62}{$\scriptscriptstyle 2$}}} \notin \Delta,$ 
and so
\[
  \alpha = \alpha_{\!\scriptscriptstyle j_{\scalebox{.62}{$\scriptscriptstyle 1$}}} \!+ 
\alpha_{\!\scriptscriptstyle j_{\scalebox{.62}{$\scriptscriptstyle 2$}}} \!+ 
\alpha_{\!\scriptscriptstyle j_{\scalebox{.62}{$\scriptscriptstyle 3$}}} 
 \! + 2\alpha_{r +  {\scriptscriptstyle 1}}.
\]
It is now clear that a positive real root $\alpha$ is in
$\Phi_{\scriptscriptstyle 2}$ if and only if it is of the form
\[
  \alpha \;\; = \sum_{\substack{j = 1 \\ j \neq
    j_{\scalebox{.71}{$\scriptscriptstyle 1$}}, \,
    j_{\scalebox{.71}{$\scriptscriptstyle 2$}}, \,
    j_{\scalebox{.71}{$\scriptscriptstyle 3$}}}}^{r}
k_{\!\scriptscriptstyle j}\alpha_{\!\scriptscriptstyle j} 
+ \alpha_{\!\scriptscriptstyle j_{\scalebox{.62}{$\scriptscriptstyle 1$}}}  \!+
\alpha_{\!\scriptscriptstyle j_{\scalebox{.62}{$\scriptscriptstyle 2$}}}  
\! +  \alpha_{\!\scriptscriptstyle j_{\scalebox{.62}{$\scriptscriptstyle 3$}}}  
\! +  2\alpha_{r +  {\scriptscriptstyle 1}}
\]
for some $1 \le j_{\scriptscriptstyle 1} < j_{\scriptscriptstyle 2} <
j_{\scriptscriptstyle 3} \le r,$ and with $k_{\!\scriptscriptstyle j}
\in \{0, 2\}$ for 
$
j \neq \! j_{\scriptscriptstyle 1}, j_{\scriptscriptstyle 2},
j_{\scriptscriptstyle 3}.
$

For $\alpha \in \Phi_{\scriptscriptstyle 2},$ we take the {W}eyl group
element sending $\alpha$ to $\alpha_{r + {\scriptscriptstyle 1}}$ to
be
\[
  w_{\scriptscriptstyle \alpha} :=
  w_{\!\scriptscriptstyle j_{\scalebox{.62}{$\scriptscriptstyle 1$}}}
  \!w_{\!\scriptscriptstyle j_{\scalebox{.62}{$\scriptscriptstyle 2$}}}
  \!w_{\!\scriptscriptstyle j_{\scalebox{.62}{$\scriptscriptstyle 3$}}}
  \!w_{r + {\scriptscriptstyle 1}}
  \!\prod_{\substack{j = 1 \\ k_{\!\scriptscriptstyle j} = 2}}^{r} w_{\!\scriptscriptstyle j}.
\]
This element is in reduced form, since the set of positive real roots
sent by $w_{\scriptscriptstyle \alpha}$ to negative real roots is
\[
  \{\alpha_{\!\scriptscriptstyle j} \!\}_{\!\scriptscriptstyle j \le r
    \,:\,  k_{\!\scalebox{.62}{$\scriptscriptstyle j$}} = 2} \cup
  \{\beta, \beta + \alpha_{\!\scriptscriptstyle
    j_{\scalebox{.62}{$\scriptscriptstyle 1$}}}\!, \, \beta + \alpha_{\!\scriptscriptstyle
    j_{\scalebox{.62}{$\scriptscriptstyle 2$}}}\!,\,
  \beta + \alpha_{\!\scriptscriptstyle j_{\scalebox{.62}{$\scriptscriptstyle 3$}}}\!\}
  \]
where $\beta = \sum \alpha_{\!\scriptscriptstyle j} + \alpha_{r +
  {\scriptscriptstyle 1}},$ the sum being over all $j \le r$ such that
$k_{\!\scriptscriptstyle j} = 2.$

To compute $\frak{S}_{\scriptscriptstyle 2}(\mathrm{z},
a_{\scriptscriptstyle 2}\!),$ consider first the case $r = 3.$
Thus
$
\Phi_{\scriptscriptstyle 2}
= \{\alpha_{\scriptscriptstyle 1} + \alpha_{\scriptscriptstyle 2}
+ \alpha_{\scriptscriptstyle 3} + 2 \alpha_{\scriptscriptstyle 4}\},
$
and
$w_{\scriptscriptstyle \alpha}
= w_{\scriptscriptstyle 1}
w_{\scriptscriptstyle 2}
w_{\scriptscriptstyle 3}
w_{\scriptscriptstyle 4}.$
Let $a \in \{1, \theta_{\scriptscriptstyle 0}\},$ and fix $\zeta_{a}$
such that $\zeta_{a}^{2} = \mathrm{sgn}(a).$ We shall compute
\begin{equation}  \label{eq: Secondary-princ-part-n=2_r=3} 
\frac{1}{2^{\scriptscriptstyle 5}}
\frac{\Gamma_{\!\! w_{\scalebox{.85}{$\scriptscriptstyle \alpha$}}}\!(a_{\scriptscriptstyle 2}, a;
  \zeta_{a})}{1 - \zeta_{a}q^{\scriptscriptstyle 3\slash 2} 
\mathrm{z}^{\scriptscriptstyle \alpha\slash 2}}
\prod_{p}S_{\! p}^{w_{\scalebox{.85}{$\scriptscriptstyle \alpha$}}}(\underline{z}, \zeta_{a}).
\end{equation} 
Using the cocycle relation \eqref{eq: def-cocycle2}, we have 
\[
\mathrm{M}_{w_{\scalebox{.85}{$\scriptscriptstyle \alpha$}}^{\scalebox{.85}{$\scriptscriptstyle -1$}}}(w_{\scriptscriptstyle \alpha}\mathrm{z}) =
\mathrm{M}_{w_{\scalebox{.85}{$\scriptscriptstyle
      \alpha$}}}^{}\!(\mathrm{z})^{\scriptscriptstyle -1} \!
= \mathrm{M}_{w_{\!\scalebox{.80}{$\scriptscriptstyle 1$}}^{}
  \!w_{\!\scalebox{.82}{$\scriptscriptstyle 2$}}^{}
  \!w_{\!\scalebox{.82}{$\scriptscriptstyle 3$}}^{}}
\!(w_{\scriptscriptstyle 4}\mathrm{z})^{\scriptscriptstyle -1}
\mathrm{M}_{w_{\!\scalebox{.82}{$\scriptscriptstyle 4$}}^{}}\!(\mathrm{z})^{\scriptscriptstyle -1}.
\]
Writing
$
\mathrm{M}_{w_{\!\scalebox{.80}{$\scriptscriptstyle 1$}}^{}
  \!w_{\!\scalebox{.82}{$\scriptscriptstyle 2$}}^{}
  \!w_{\!\scalebox{.82}{$\scriptscriptstyle 3$}}^{}}
\!(w_{\scriptscriptstyle 4}\mathrm{z})^{\scriptscriptstyle -1}
\!= \mathrm{diag}(\mathrm{d}_{\scriptscriptstyle 1},
\mathrm{d}_{\scriptscriptstyle 2}, \mathrm{d}_{\scriptscriptstyle 3})
$
and
$
\mathrm{M}_{w_{\!\scalebox{.82}{$\scriptscriptstyle 4$}}^{}}\!(\mathrm{z})^{\scriptscriptstyle -1}
\!= U\mathrm{diag}(\mathrm{e}_{\scriptscriptstyle 1},
\mathrm{e}_{\scriptscriptstyle 2}, \mathrm{e}_{\scriptscriptstyle 3})U,
$
it follows from \eqref{eq: functions-Lwp} that
\begin{equation} \label{eq: Explicit-Lwp-n=2-r=3} 
  \begin{split}
    L_{w_{\scalebox{.85}{$\scriptscriptstyle \alpha$}},\, p}^{\scriptscriptstyle (1)} & \, = \,
    \tfrac{1}{4}\mathrm{e}_{\scriptscriptstyle 1}
    (\mathrm{d}_{\scriptscriptstyle 1} \! +
    \mathrm{d}_{\scriptscriptstyle 3}
    \! + 2\mathrm{sgn}(a)\mathrm{d}_{\scriptscriptstyle 2})\,
    -\, \tfrac{1}{4}\mathrm{e}_{\scriptscriptstyle 3}
    (\mathrm{d}_{\scriptscriptstyle 1} \! +
    \mathrm{d}_{\scriptscriptstyle 3} \! -
    2\mathrm{sgn}(a)\mathrm{d}_{\scriptscriptstyle 2}) \\
    \tfrac{1}{2} \big(L_{w_{\scalebox{.85}{$\scriptscriptstyle
          \alpha$}},\, p}^{\scriptscriptstyle (2)}
    + L_{w_{\scalebox{.85}{$\scriptscriptstyle \alpha$}},\, p}^{\scriptscriptstyle (3)}\big)
\, & = \, \tfrac{1}{8}\mathrm{e}_{\scriptscriptstyle 1}
(\mathrm{d}_{\scriptscriptstyle 1} \! + \mathrm{d}_{\scriptscriptstyle 3}
\! + 2\mathrm{sgn}(a)\mathrm{d}_{\scriptscriptstyle 2})\,
+\, \tfrac{1}{8}\mathrm{e}_{\scriptscriptstyle 3}
(\mathrm{d}_{\scriptscriptstyle 1} \! + \mathrm{d}_{\scriptscriptstyle 3}
\! - 2\mathrm{sgn}(a)\mathrm{d}_{\scriptscriptstyle 2})
\, + \, \tfrac{1}{4}\mathrm{e}_{\scriptscriptstyle 2}
(\mathrm{d}_{\scriptscriptstyle 1} \! - \mathrm{d}_{\scriptscriptstyle 3}) \\
\tfrac{1}{2}\big(L_{w_{\scalebox{.85}{$\scriptscriptstyle \alpha$}},\, p}^{\scriptscriptstyle (3)}
- L_{w_{\scalebox{.85}{$\scriptscriptstyle \alpha$}},\, p}^{\scriptscriptstyle (2)}\big)\, & = \, \tfrac{1}{8}\mathrm{e}_{\scriptscriptstyle 1}
(\mathrm{d}_{\scriptscriptstyle 1} \! + \mathrm{d}_{\scriptscriptstyle 3}
\! + 2\mathrm{sgn}(a)\mathrm{d}_{\scriptscriptstyle 2})\,
+\, \tfrac{1}{8}\mathrm{e}_{\scriptscriptstyle 3}
(\mathrm{d}_{\scriptscriptstyle 1} \! + \mathrm{d}_{\scriptscriptstyle 3}
\! - 2\mathrm{sgn}(a)\mathrm{d}_{\scriptscriptstyle 2})
\, - \, \tfrac{1}{4}\mathrm{e}_{\scriptscriptstyle 2}
(\mathrm{d}_{\scriptscriptstyle 1} \! - \mathrm{d}_{\scriptscriptstyle 3}).
\end{split}
\end{equation}
Here $p$ is a monic linear polynomial.
Setting
$
z_{\scriptscriptstyle k}^{}
\!= q^{-s_{\scalebox{.65}{$\scriptscriptstyle k$}}}
\!:= q^{\, \scriptscriptstyle -1\slash 2}
\xi_{\scriptscriptstyle k}^{\scriptscriptstyle 2}
$
($k = 1, 2, 3$), and hence
$
z_{\scriptscriptstyle 4}^{} \!=
(q^{\scriptscriptstyle 3\slash 4}
\zeta_{a}
\xi_{\scriptscriptstyle 1}
\xi_{\scriptscriptstyle 2}
\xi_{\scriptscriptstyle 3})^{\scriptscriptstyle -1}\!,
$
we have explicitly 
\begin{equation} \label{eq: Explicit-d1-d2-d3} 
  \begin{split} 
&\mathrm{d}_{\scriptscriptstyle 1} 
= \frac{\left(1 - \frac{\xi_{\scriptscriptstyle 1}}{\zeta_{a}
q^{\scriptscriptstyle 3\slash 4}\xi_{\scriptscriptstyle 2}\xi_{\scriptscriptstyle 3}}\right)
\!\left(1 - \frac{\xi_{\scriptscriptstyle 2}}{\zeta_{a}
q^{\scriptscriptstyle 3\slash 4}\xi_{\scriptscriptstyle 1}\xi_{\scriptscriptstyle 3}}\right)
\!\left(1 - \frac{\xi_{\scriptscriptstyle 3}}{\zeta_{a}
q^{\scriptscriptstyle 3\slash 4}\xi_{\scriptscriptstyle 1}\xi_{\scriptscriptstyle 2}}\right)}
{\left(1 - \frac{\zeta_{a}\xi_{\scriptscriptstyle 2}\xi_{\scriptscriptstyle 3}}
{q^{\scriptscriptstyle 1\slash 4}\xi_{\scriptscriptstyle 1}}\right)
\!\left(1 - \frac{\zeta_{a}\xi_{\scriptscriptstyle 1}\xi_{\scriptscriptstyle 3}}
{q^{\scriptscriptstyle 1\slash 4}\xi_{\scriptscriptstyle 2}}\right)
\!\left(1 - \frac{\zeta_{a}\xi_{\scriptscriptstyle 1}\xi_{\scriptscriptstyle 2}}
{q^{\scriptscriptstyle 1\slash 4}\xi_{\scriptscriptstyle 3}}\right)},\;\;\; 
\mathrm{d}_{\scriptscriptstyle 2} = \frac{\zeta_{a}}
{q^{\scriptscriptstyle 3\slash 4}\xi_{\scriptscriptstyle 1}\xi_{\scriptscriptstyle 2}
\xi_{\scriptscriptstyle 3}} \\
&\mathrm{d}_{\scriptscriptstyle 3} 
= \frac{\left(1 + \frac{\xi_{\scriptscriptstyle 1}}{\zeta_{a}
q^{\scriptscriptstyle 3\slash 4}\xi_{\scriptscriptstyle 2}\xi_{\scriptscriptstyle 3}}\right)
\!\left(1 + \frac{\xi_{\scriptscriptstyle 2}}{\zeta_{a}
q^{\scriptscriptstyle 3\slash 4}\xi_{\scriptscriptstyle 1}\xi_{\scriptscriptstyle 3}}\right)
\!\left(1 + \frac{\xi_{\scriptscriptstyle 3}}{\zeta_{a}
q^{\scriptscriptstyle 3\slash 4}\xi_{\scriptscriptstyle 1}\xi_{\scriptscriptstyle 2}}\right)}
{\left(1 + \frac{\zeta_{a}\xi_{\scriptscriptstyle 2}\xi_{\scriptscriptstyle 3}}
{q^{\scriptscriptstyle 1\slash 4}\xi_{\scriptscriptstyle 1}}\right)
\!\left(1 + \frac{\zeta_{a}\xi_{\scriptscriptstyle 1}\xi_{\scriptscriptstyle 3}}
{q^{\scriptscriptstyle 1\slash 4}\xi_{\scriptscriptstyle 2}}\right)
\!\left(1 + \frac{\zeta_{a}\xi_{\scriptscriptstyle 1}\xi_{\scriptscriptstyle 2}}
{q^{\scriptscriptstyle 1\slash 4}\xi_{\scriptscriptstyle 3}}\right)}
\end{split}
\end{equation}
and
\begin{equation} \label{eq: Explicit-e1-e2-e3} 
\mathrm{e}_{\scriptscriptstyle 1} =
\frac{1 - (\zeta_{a}q^{\scriptscriptstyle 3\slash 4}
  \xi_{\scriptscriptstyle 1}\xi_{\scriptscriptstyle 2}\xi_{\scriptscriptstyle 3})^{\scriptscriptstyle -1}}
{1-\zeta_{a}
  q^{\, \scriptscriptstyle -1\slash 4}\xi_{\scriptscriptstyle 1}\xi_{\scriptscriptstyle 2}
  \xi_{\scriptscriptstyle 3}},\;\;\; \mathrm{e}_{2}
= (\zeta_{a}
q^{\scriptscriptstyle 1\slash 4}\xi_{\scriptscriptstyle 1}\xi_{\scriptscriptstyle 2}
\xi_{\scriptscriptstyle 3})^{\scriptscriptstyle -1},\;\;\, \text{and}
\;\;\;\,
\mathrm{e}_{3} =
\frac{1 + (\zeta_{a}q^{\scriptscriptstyle 3\slash 4}
  \xi_{\scriptscriptstyle 1}
  \xi_{\scriptscriptstyle 2}
  \xi_{\scriptscriptstyle 3})^{\scriptscriptstyle -1}}
{1+\zeta_{a}
  q^{\, \scriptscriptstyle -1\slash 4}
  \xi_{\scriptscriptstyle 1}
  \xi_{\scriptscriptstyle 2}
  \xi_{\scriptscriptstyle 3}}.
\end{equation}
Thus, by \eqref{eq: Final-expression-Sp}, we find that
\[
  S_{\! p}^{w_{\!\scalebox{.85}{$\scriptscriptstyle \alpha^{^{}}$}}} \!=
  \frac{1 - \mathrm{sgn}(a)\mathfrak{t}_{\scriptscriptstyle 3}
    \!(\xi_{\scriptscriptstyle 1}^{\scriptscriptstyle \pm 1}\!,\, 
    \xi_{\scriptscriptstyle 2}^{\scriptscriptstyle \pm 1}\!,\, 
    \xi_{\scriptscriptstyle 3}^{\scriptscriptstyle \pm 1}\!)
    q^{\, \scriptscriptstyle -3\slash 2} -
    \cdots -\mathfrak{t}_{\scriptscriptstyle 12}
    (\xi_{\scriptscriptstyle 1}^{\scriptscriptstyle \pm 1}\!,\, 
    \xi_{\scriptscriptstyle 2}^{\scriptscriptstyle \pm 1}\!,\,
    \xi_{\scriptscriptstyle 3}^{\scriptscriptstyle \pm 1}\!)
    q^{\, \scriptscriptstyle -6} \! + q^{\, \scriptscriptstyle -7}}
  {\left(1 - \frac{\mathrm{sgn}(a)
        \xi_{\scriptscriptstyle 1}^{\scriptscriptstyle 2}
        \xi_{\scriptscriptstyle 2}^{\scriptscriptstyle 2}
        \xi_{\scriptscriptstyle 3}^{\scriptscriptstyle 2}}{\sqrt{q}}\right)
    \prod_{\! 1 \, \le \, i \, \le \, 3}
    \left(1 - \frac{\mathrm{sgn}(a)\xi_{\scriptscriptstyle i}^{\scriptscriptstyle 2}}
      {\sqrt{q}}\right)
    \prod_{\! 1 \, \le \, j \, \le \, 3}
    \left(1 - \frac{\mathrm{sgn}(a)
        \xi_{\scriptscriptstyle 1}^{\scriptscriptstyle 2}
        \xi_{\scriptscriptstyle 2}^{\scriptscriptstyle 2}
        \xi_{\scriptscriptstyle 3}^{\scriptscriptstyle 2}}
      {\sqrt{q}\, \xi_{\!\scriptscriptstyle j}^{\scriptscriptstyle 4}}\right)}
\]
where
$
\mathfrak{t}_{\scriptscriptstyle n}
\!(\xi_{\scriptscriptstyle 1}^{\scriptscriptstyle \pm 1}\!,\,
    \xi_{\scriptscriptstyle 2}^{\scriptscriptstyle \pm 1}\!,\,
    \xi_{\scriptscriptstyle 3}^{\scriptscriptstyle \pm 1}\!)
    $
    ($n = 3, \ldots, 12$) are symmetric polynomials in the variables
    $\xi_{\scriptscriptstyle 1}^{\scriptscriptstyle \pm 1}\!,\,
    \xi_{\scriptscriptstyle 2}^{\scriptscriptstyle \pm 1}\!,\,
    \xi_{\scriptscriptstyle 3}^{\scriptscriptstyle \pm 1}
    $
    (i.e., invariant under the hyperoctahedral
group $\mathbb{S}_{2} \wr \mathbb{S}_{3}$) with positive integral
coefficients. \!We regularize
$
S_{\! p}^{w_{\!\scalebox{.85}{$\scriptscriptstyle \alpha^{^{}}$}}}
$
\!by taking its numerator, which we shall denote by
$S_{\! p}^{\scriptscriptstyle \mathrm{reg}}\!.$ Note that the
denominator of $S_{\! p}^{w_{\!\scalebox{.85}{$\scriptscriptstyle \alpha^{^{}}$}}}$ \!is
$
R_{\!p}\!\left(\frac{\zeta_{a} \xi_{\scriptscriptstyle 2}\xi_{\scriptscriptstyle 3}}
  {q^{\scriptscriptstyle 1\slash 4}\xi_{\scriptscriptstyle 1}}\!,
  \frac{\zeta_{a}\xi_{\scriptscriptstyle 1}\xi_{\scriptscriptstyle 3}}
  {q^{\scriptscriptstyle 1\slash 4}\xi_{\scriptscriptstyle 2}}\!,
  \frac{\zeta_{a} \xi_{\scriptscriptstyle 1}\xi_{\scriptscriptstyle 2}}
  {q^{\scriptscriptstyle 1\slash 4}\xi_{\scriptscriptstyle 3}}\right)^{\!\scriptscriptstyle -1}\!,
$
where
\begin{equation} \label{eq: Local-factor-residue-r=3}
  R_{\! p}(z_{\scriptscriptstyle 1}^{}, z_{\scriptscriptstyle 2}^{}, z_{\scriptscriptstyle 3}^{}) =
  \left(1 - qz_{\scriptscriptstyle 1}^{\scriptscriptstyle 2}
    z_{\scriptscriptstyle 2}^{\scriptscriptstyle 2}
    z_{\scriptscriptstyle 3}^{\scriptscriptstyle 2}\right)^{\scriptscriptstyle -1}
  \; \cdot \prod_{1 \, \le \, i \, \le \, j \, \le \, 3}
  (1 - z_{\scriptscriptstyle i}^{} z_{\!\scriptscriptstyle j}^{})^{\scriptscriptstyle-1}
\end{equation}
is the local factor of the ``modified'' residue at
$z_{\scriptscriptstyle 4}^{}  = 1\slash q$
of the function
$
Z(z_{\scriptscriptstyle 1}^{}, z_{\scriptscriptstyle 2}^{},
z_{\scriptscriptstyle 3}^{}, z_{\scriptscriptstyle 4}^{})
: = Z^{\scriptscriptstyle (1)}\!(\mathbf{s}; \chi_{\scriptscriptstyle 1}\!,
\chi_{\scriptscriptstyle 1}\!),
$
see \cite[Section~4, eqn (20)]{Dia}.

\vskip5pt
\begin{rem} \label{Crucial-Remark} --- It is easy to check that in fact the functions 
\[
\left(\frac{L_{w_{\!\scalebox{.85}{$\scriptscriptstyle \alpha^{^{}}$}},\, p}^{\scriptscriptstyle (1)}}
{q^{\scriptscriptstyle 3\slash 4}\zeta_{a}
	\xi_{\scriptscriptstyle 1}
	\xi_{\scriptscriptstyle 2}
	\xi_{\scriptscriptstyle 3}}\right)
R_{\!p}\!\left(\frac{\zeta_{a} \xi_{\scriptscriptstyle 2}\xi_{\scriptscriptstyle 3}}
{q^{\scriptscriptstyle 1\slash 4}\xi_{\scriptscriptstyle 1}}\!,
\frac{\zeta_{a}\xi_{\scriptscriptstyle 1}\xi_{\scriptscriptstyle 3}}
{q^{\scriptscriptstyle 1\slash 4}\xi_{\scriptscriptstyle 2}}\!,
\frac{\zeta_{a} \xi_{\scriptscriptstyle 1}\xi_{\scriptscriptstyle 2}}
{q^{\scriptscriptstyle 1\slash 4}\xi_{\scriptscriptstyle 3}}\right)^{\!\scriptscriptstyle -1}
 \;\;\; \mathrm{and}\;\;\;\;\;\;
\frac{L_{w_{\!\scalebox{.85}{$\scriptscriptstyle \alpha^{^{}}$}},\, p}^{\scriptscriptstyle (3)}
	\pm \, L_{w_{\!\scalebox{.85}{$\scriptscriptstyle \alpha^{^{}}$}},\, p}^{\scriptscriptstyle (2)}}
{2\prod_{j = 1}^{3} (1 \mp q^{\,\scriptscriptstyle -1\slash 2}
	\xi_{\!\scriptscriptstyle j}^{\scriptscriptstyle 2})}\, \cdot \,
	R_{\!p}\!\left(\frac{\zeta_{a} \xi_{\scriptscriptstyle 2}\xi_{\scriptscriptstyle 3}}
	{q^{\scriptscriptstyle 1\slash 4}\xi_{\scriptscriptstyle 1}}\!,
	\frac{\zeta_{a}\xi_{\scriptscriptstyle 1}\xi_{\scriptscriptstyle 3}}
	{q^{\scriptscriptstyle 1\slash 4}\xi_{\scriptscriptstyle 2}}\!,
	\frac{\zeta_{a} \xi_{\scriptscriptstyle 1}\xi_{\scriptscriptstyle 2}}
	{q^{\scriptscriptstyle 1\slash 4}\xi_{\scriptscriptstyle 3}}\right)^{\!\scriptscriptstyle -1}
\] 
computed using \eqref{eq: Explicit-Lwp-n=2-r=3} -- \eqref{eq: Local-factor-residue-r=3}, are symmetric in the variables
$\xi_{\scriptscriptstyle 1}^{\scriptscriptstyle \pm 1}\!,\, 
\xi_{\scriptscriptstyle 2}^{\scriptscriptstyle \pm 1}\!,\, 
\xi_{\scriptscriptstyle 3}^{\scriptscriptstyle \pm 1}\!.
$ 
These symmetries will be used in an essential way later in this section. 
\end{rem}

\vskip5pt
Now when $r = 3,$ the function
$
Z^{(c)}(\mathbf{s}; \chi_{a_{\scriptscriptstyle 2} c_{\scriptscriptstyle 2}}, 
\chi_{a_{\scriptscriptstyle 1} c_{\scriptscriptstyle 1}}\!)
$
is meromorphic on all of $\mathbb{C}^{\scriptscriptstyle 4}.$ Using
Proposition \ref{MS-residue-general}, it follows (as in \cite{Dia})
that the product in \eqref{eq: Secondary-princ-part-n=2_r=3} is
\[
  \prod_{p} S_{\! p}^{w_{\!\scalebox{.85}{$\scriptscriptstyle \alpha^{^{}}$}}}
  \!=\, \left(1 - \mathrm{sgn}(a)\sqrt{q}
    \,\xi_{\scriptscriptstyle 1}^{\scriptscriptstyle 2}
    \xi_{\scriptscriptstyle 2}^{\scriptscriptstyle 2}
    \xi_{\scriptscriptstyle 3}^{\scriptscriptstyle 2}
  \right)^{\scriptscriptstyle -1}
  \; \cdot \prod_{1 \, \le \, i \, \le \, j \, \le \,  3}
  \left(1 - \frac{\mathrm{sgn}(a)\sqrt{q}
      \,\xi_{\scriptscriptstyle 1}^{\scriptscriptstyle 2}
      \xi_{\scriptscriptstyle 2}^{\scriptscriptstyle 2}
      \xi_{\scriptscriptstyle 3}^{\scriptscriptstyle 2}}
    {\xi_{\scriptscriptstyle i}^{\scriptscriptstyle 2}
      \xi_{\!\scriptscriptstyle j}^{\scriptscriptstyle 2}}\right)^{\!\scriptscriptstyle -1}
  \!\prod_{p} S_{\! p}^{\scriptscriptstyle \mathrm{reg}}
\]
where, for a general monic irreducible,
$
S_{\! p}^{\scriptscriptstyle \mathrm{reg}}
$
is obtained by replacing $\mathrm{sgn}(a) \mapsto \chi_{a}(p),$
$\xi_{\scriptscriptstyle k} \mapsto \xi_{\scriptscriptstyle k}^{\deg\,p}\!,$ and
$q \mapsto |p|.$ Note that the product over $p$ in the right-hand side
converges only in a neighborhood of $\xi_{\scriptscriptstyle k} \! = 1$ (for $k = 1, 2, 3$).
\!When
$
\xi_{\scriptscriptstyle 1}
\! = \xi_{\scriptscriptstyle 2}
\! = \xi_{\scriptscriptstyle 3} \! = \xi,$ we have
$
S_{\! p}^{\scriptscriptstyle \mathrm{reg}}
\!= P\left(\chi_{a}(p) \xi^{2\deg\,p}\!, |p|^{-\frac{1}{2}}\right)\!,
$
with $P(x, y)$ given by
\begin{align*}
P(x, y) =\; & \big(1 - y^{\scriptscriptstyle 2}\big)
            \big(1 - xy\big)
            \big(1 - x^{\scriptscriptstyle -1}y\big) \\
          & \cdot \Big[1
            +\big(x + x^{\scriptscriptstyle -1} \big)y
            \,+ \, \big(x + x^{\scriptscriptstyle -1} \big)^{\!\scriptscriptstyle 2} y^{\scriptscriptstyle 2}
            - \, 4\big(x + x^{\scriptscriptstyle -1}\big)y^{\scriptscriptstyle 3}
            -\, 5\big(x + x^{\scriptscriptstyle -1}\big)^{\!\scriptscriptstyle 2}y^{\scriptscriptstyle 4}
            +\, \big(x + x^{\scriptscriptstyle -1}\big)
            \big(3x + x^{\scriptscriptstyle -1}\big)
            \big(x + 3x^{\scriptscriptstyle -1}\big)y^{\scriptscriptstyle 5} \\
             &- \, \big(x+x^{\scriptscriptstyle -1}\big)
               \big(7+3x^{\scriptscriptstyle 2}+ 3x^{\scriptscriptstyle -2}\big)y^{\scriptscriptstyle 7}
            + \, \big(8+5x^{\scriptscriptstyle 2}+5x^{\scriptscriptstyle -2}\big)y^{\scriptscriptstyle 8}
            - \, \big(x+x^{\scriptscriptstyle -1}\big)y^{\scriptscriptstyle 9}
               - \, y^{\scriptscriptstyle 10}\Big]. \\
\end{align*}
In particular,
\[
  P(1, y) = (1 - y)^{\scriptscriptstyle 5} (1 + y)
  (1 + 4 \, y + 11 y^{\scriptscriptstyle 2} + 10 \, y^{\scriptscriptstyle 3} 
  - 11 y^{\scriptscriptstyle 4} + 11 y^{\scriptscriptstyle 6}
  - 4 \, y^{\scriptscriptstyle 7} - y^{\scriptscriptstyle 8})
\]
which is precisely the polynomial occurring in \cite{Zha1} and \cite{Dia}.

By \eqref{eq: function-Gammaw}, the remaining part of
\eqref{eq: Secondary-princ-part-n=2_r=3} is given by
\begin{equation*}
  \begin{split}
  2^{-5}\Gamma_{\!\! w_{\scalebox{.85}{$\scriptscriptstyle \alpha$}}}
  \!(a_{\scriptscriptstyle 2}, a;
  \zeta_{a})
   = \bigg[&\frac{\psi(z_{\scriptscriptstyle 1}^{}, z_{\scriptscriptstyle 2}^{},
  z_{\scriptscriptstyle 3}^{}, z_{\scriptscriptstyle 4}^{}; q)
  + \psi(z_{\scriptscriptstyle 1}^{}, z_{\scriptscriptstyle 2}^{},
  z_{\scriptscriptstyle 3}^{}, - z_{\scriptscriptstyle 4}^{}; q)}
{4 \, q^{\scriptscriptstyle 3\slash 2} z_{\scriptscriptstyle 1}^{}
  z_{\scriptscriptstyle 2}^{}
  z_{\scriptscriptstyle 3}^{} z_{\scriptscriptstyle
    4}^{\scriptscriptstyle 4}}\\
& + \frac{\mathrm{sgn}(a_{\scriptscriptstyle 2})(
  \phi(z_{\scriptscriptstyle 1}^{}, z_{\scriptscriptstyle 2}^{},
  z_{\scriptscriptstyle 3}^{}, z_{\scriptscriptstyle 4}^{}; q)
  - \phi(- z_{\scriptscriptstyle 1}^{}, - z_{\scriptscriptstyle 2}^{},
  - z_{\scriptscriptstyle 3}^{}, z_{\scriptscriptstyle 4}^{}; q))}
{4 \, q^{\scriptscriptstyle 3\slash 2} z_{\scriptscriptstyle 1}^{}
  z_{\scriptscriptstyle 2}^{}
  z_{\scriptscriptstyle 3}^{} z_{\scriptscriptstyle
    4}^{\scriptscriptstyle 4}}\\
& - \frac{\mathrm{sgn}(a)}
{2 \, q^{\scriptscriptstyle 3} z_{\scriptscriptstyle 1}^{}
  z_{\scriptscriptstyle 2}^{} z_{\scriptscriptstyle 3}^{} z_{\scriptscriptstyle 4}^{\scriptscriptstyle 4}}
\!\left(\!\frac{1 - \mathrm{sgn}(a_{\scriptscriptstyle 2}) z_{\scriptscriptstyle 4}^{}}
  {1 - \mathrm{sgn}(a_{\scriptscriptstyle 2}) q z_{\scriptscriptstyle 4}^{}}\!\right)
\bigg]_{z_{\scalebox{.6}{$\scriptscriptstyle 1$}}= \frac{
    \xi_{\scalebox{.75}{$\scriptscriptstyle 1$}}^{\scalebox{.75}{$\scriptscriptstyle 2$}}}{\sqrt{q}},\,
  z_{\scalebox{.6}{$\scriptscriptstyle 2$}}= \frac{\xi_{\scalebox{.75}{$\scriptscriptstyle 2$}}^{\scalebox{.75}{$\scriptscriptstyle 2$}}}{\sqrt{q}},\,
  z_{\scalebox{.6}{$\scriptscriptstyle 3$}}=
  \frac{\xi_{\scalebox{.75}{$\scriptscriptstyle 3$}}^{\scalebox{.75}{$\scriptscriptstyle 2$}}}
  {\sqrt{q}},\,
  z_{\scalebox{.6}{$\scriptscriptstyle 4$}}=
  \frac{1}{q^{\scalebox{.75}{$\scriptscriptstyle 3\slash 4$}}\zeta_{a}
  \xi_{\scalebox{.75}{$\scriptscriptstyle 1$}}
  \xi_{\scalebox{.75}{$\scriptscriptstyle 2$}}
  \xi_{\scalebox{.75}{$\scriptscriptstyle 3$}}}}
\end{split}
\end{equation*}
where the functions
$
\phi(z_{\scriptscriptstyle 1}^{}, z_{\scriptscriptstyle 2}^{},
z_{\scriptscriptstyle 3}^{}, z_{\scriptscriptstyle 4}^{}; q)
$
and
$
\psi(z_{\scriptscriptstyle 1}^{}, z_{\scriptscriptstyle 2}^{},
  z_{\scriptscriptstyle 3}^{}, z_{\scriptscriptstyle 4}^{}; q)
$
are given by
\begin{equation} \label{eq: Functions-phi-psi}
  \begin{split}
& \phi(z_{\scriptscriptstyle 1}^{}, z_{\scriptscriptstyle 2}^{},
z_{\scriptscriptstyle 3}^{}, z_{\scriptscriptstyle 4}^{}; q) =
\frac{(1 - \sqrt{q}\, z_{\scriptscriptstyle 1}^{} z_{\scriptscriptstyle
    4}^{}) (1 - \sqrt{q}\, z_{\scriptscriptstyle 2}^{} z_{\scriptscriptstyle 4}^{})
(1 - \sqrt{q}\, z_{\scriptscriptstyle 3}^{} z_{\scriptscriptstyle 4}^{})  
(1 - q z_{\scriptscriptstyle 4}^{\scriptscriptstyle 2})}{
(1 - q^{\scriptscriptstyle 3\slash 2} z_{\scriptscriptstyle 1}^{}
z_{\scriptscriptstyle 4}^{})
(1 - q^{\scriptscriptstyle 3\slash 2} z_{\scriptscriptstyle 2}^{}
z_{\scriptscriptstyle 4}^{})
(1 - q^{\scriptscriptstyle 3\slash 2} z_{\scriptscriptstyle 3}^{}
z_{\scriptscriptstyle 4}^{}) (1 - q^{\scriptscriptstyle 2}
z_{\scriptscriptstyle 4}^{\scriptscriptstyle 2})}, \\
& \psi(z_{\scriptscriptstyle 1}^{}, z_{\scriptscriptstyle 2}^{},
  z_{\scriptscriptstyle 3}^{}, z_{\scriptscriptstyle 4}^{}; q) = -\,
  \frac{(1 - q^{\scriptscriptstyle 3\slash 2} z_{\scriptscriptstyle 4}^{})
  \phi(z_{\scriptscriptstyle 1}^{}, z_{\scriptscriptstyle 2}^{},
z_{\scriptscriptstyle 3}^{}, z_{\scriptscriptstyle 4}^{}; q)
}{\sqrt{q} - q z_{\scriptscriptstyle 4}^{}}. 
  \end{split}
\end{equation}
For
$
\xi_{\scriptscriptstyle 1}
\! = \xi_{\scriptscriptstyle 2}
\! = \xi_{\scriptscriptstyle 3}
\! = \!1,$
one obtains, up to a factor of $\frac{1}{4},$ the values given in
\cite[Table1, p.~20]{Dia}.

\vskip10pt
{\bf The case $r \ge 4.$} We fix a root $\alpha \in
\Phi_{\scriptscriptstyle 2},$
\[
  \alpha = \alpha_{\scriptscriptstyle 1}
   \! +\alpha_{\scriptscriptstyle 2}
   + \alpha_{\scriptscriptstyle 3}
   + 2\!\sum_{j \in J}\alpha_{\!\scriptscriptstyle j}
  +2 \alpha_{r + {\scriptscriptstyle 1}} \qquad \text{(where $J \subseteq \{4,\ldots, r\}$).}
\]
Write $\alpha = \alpha' \! + \alpha_{\scalebox{1.0}{$\scriptscriptstyle J$}}^{},$ where
$
\alpha' := \alpha_{\scriptscriptstyle 1}
   \! +\alpha_{\scriptscriptstyle 2}
   + \alpha_{\scriptscriptstyle 3}
  +2 \alpha_{r + {\scriptscriptstyle 1}} \in \Phi_{\scriptscriptstyle 2}.
  $
  Then
  $
  w_{\!\scriptscriptstyle \alpha^{^{}}} \! = w_{\!{\scriptscriptstyle \alpha'}}
  w_{\!\scalebox{1.0}{$\scriptscriptstyle J$}}^{},
  $
  where
  $
w_{\!\scalebox{1.0}{$\scriptscriptstyle J$}}^{} := \prod_{j \, \in \,
  J}w_{\!\scriptscriptstyle j}.
$
Applying \eqref{eq: def-cocycle2}, we see that\linebreak
$
\mathrm{M}_{w_{\scalebox{.85}{$\scriptscriptstyle
      \alpha$}}^{\scalebox{.85}{$\scriptscriptstyle
      -1$}}}(w_{\!\scriptscriptstyle \alpha^{^{}}}\!\mathrm{z})
= \mathrm{M}_{w_{\scalebox{.85}{$\scriptscriptstyle \alpha'$}}}
\!(w_{\!\scalebox{1.0}{$\scriptscriptstyle
    J$}}^{}\mathrm{z})^{\scriptscriptstyle -1} 
\mathrm{M}_{w_{\!\scalebox{.9}{$\scriptscriptstyle J$}}^{}}
\!(\mathrm{z})^{\scriptscriptstyle -1}\!.
$
Thus, if we let
$
\mathrm{M}_{w_{\!\scalebox{.9}{$\scriptscriptstyle J$}}^{}}
\!(\mathrm{z})^{\scriptscriptstyle -1}
\!= \mathrm{diag}(\mathrm{b}_{\scriptscriptstyle 1},
\mathrm{b}_{\scriptscriptstyle 2}, \mathrm{b}_{\scriptscriptstyle 3}),
$
and $p$ a monic linear polynomial, then, by \eqref{eq: functions-Lwp},
we have
\[
 L_{w_{\!\scalebox{.85}{$\scriptscriptstyle \alpha^{^{}}$}},\, p}^{\scriptscriptstyle (1)}
 =\, \mathrm{b}_{\scriptscriptstyle 2}
 L_{w_{\!\scalebox{.85}{$\scriptscriptstyle \alpha'$}},\, p}^{\scriptscriptstyle (1)},\;\;\;
 L_{w_{\!\scalebox{.85}{$\scriptscriptstyle \alpha^{^{}}$}},\, p}^{\scriptscriptstyle (2)}
    \!+ L_{w_{\!\scalebox{.85}{$\scriptscriptstyle \alpha^{^{}}$}},\, p}^{\scriptscriptstyle (3)}
    =\, \mathrm{b}_{\scriptscriptstyle 1}
    \big(L_{w_{\!\scalebox{.85}{$\scriptscriptstyle \alpha'$}},\, p}^{\scriptscriptstyle (2)}
    \!+ L_{w_{\!\scalebox{.85}{$\scriptscriptstyle \alpha'$}},\,
      p}^{\scriptscriptstyle (3)}\big),
    \;\, \mathrm{and}\;\;\,
L_{w_{\!\scalebox{.85}{$\scriptscriptstyle \alpha^{^{}}$}},\, p}^{\scriptscriptstyle (3)}
\!- L_{w_{\!\scalebox{.85}{$\scriptscriptstyle \alpha^{^{}}$}},\,
  p}^{\scriptscriptstyle (2)}
=\, \mathrm{b}_{\scriptscriptstyle 3}
\big(L_{w_{\!\scalebox{.85}{$\scriptscriptstyle \alpha'$}},\, p}^{\scriptscriptstyle (3)}
\!- L_{w_{\!\scalebox{.85}{$\scriptscriptstyle \alpha'$}},\,
  p}^{\scriptscriptstyle (2)}\big)
\]
with
$
L_{w_{\!\scalebox{.85}{$\scriptscriptstyle \alpha'$}},\, p}^{\scriptscriptstyle (i)}
$
($i = 1, 2, 3$) computed by using \eqref{eq: Explicit-Lwp-n=2-r=3},
\eqref{eq: Explicit-d1-d2-d3} and \eqref{eq: Explicit-e1-e2-e3}.
Indeed, the functions
$
L_{w_{\!\scalebox{.85}{$\scriptscriptstyle \alpha'$}},\, p}^{\scriptscriptstyle (i)}
$
are evaluated at
\[
  w_{\!\scalebox{1.0}{$\scriptscriptstyle J$}}^{}
  (z_{\scriptscriptstyle 1}^{}, z_{\scriptscriptstyle 2}^{},
  z_{\scriptscriptstyle 3}^{}, z_{r + {\scriptscriptstyle 1}}^{}) =
   \left(z_{\scriptscriptstyle 1}^{}, z_{\scriptscriptstyle 2}^{},
     z_{\scriptscriptstyle 3}^{}, z_{r + {\scriptscriptstyle 1}}^{}
     \cdot \prod_{j \,\in \, J}q^{\scriptscriptstyle 1\slash 2}
     z_{\!\scriptscriptstyle j}^{}\right)
 \]
 subject to the condition that
 $
\mathrm{z}^{\scriptscriptstyle \alpha\slash 2}
= \zeta_{a}^{\scriptscriptstyle -1}
q^{\scriptscriptstyle - (d(\alpha) + 1)\slash 4}.
$ 
Accordingly, if we set
$
z_{\scriptscriptstyle k}^{}  \!= q^{\scriptscriptstyle -1\slash 2}
\xi_{\scriptscriptstyle k}^{\scriptscriptstyle 2}
$ 
($k = 1, \ldots, r$), then we can write 
\[
z_{r + {\scriptscriptstyle 1}}^{}
\!\prod_{j \,\in \, J}q^{\scriptscriptstyle 1\slash 2}
     z_{\!\scriptscriptstyle j}^{} =
\frac{1}{q^{\scriptscriptstyle 3\slash 4}\zeta_{a}
\xi_{\scriptscriptstyle 1}\xi_{\scriptscriptstyle 2}\xi_{\scriptscriptstyle 3}}.
\]
It now follows from \eqref{eq: Final-expression-Sp} that
\begin{equation}  \label{eq: Sp_w-alpha}
S_{\! p}^{w_{\!\scalebox{.85}{$\scriptscriptstyle \alpha^{^{}}$}}}
\!= \, \big(1 - q^{\scriptscriptstyle -1}\big)
\left[\frac{\mathrm{b}_{\scriptscriptstyle 2}
 L_{w_{\!\scalebox{.85}{$\scriptscriptstyle \alpha'$}},\, p}^{\scriptscriptstyle (1)}}
{q^{\scriptscriptstyle 3\slash 4}\zeta_{a}
  \xi_{\scriptscriptstyle 1}
\xi_{\scriptscriptstyle 2}
\xi_{\scriptscriptstyle 3}
\!\prod_{j \,\in \, J}\xi_{\!\scriptscriptstyle j}^{\scriptscriptstyle 2}}
+\frac{\mathrm{b}_{\scriptscriptstyle 1}
  \big(L_{w_{\!\scalebox{.85}{$\scriptscriptstyle \alpha'$}},\, p}^{\scriptscriptstyle (2)}
    \!+ L_{w_{\!\scalebox{.85}{$\scriptscriptstyle \alpha'$}},\,
      p}^{\scriptscriptstyle (3)}\big)}
  {2 \prod_{j = 1}^{r} (1 - q^{\,\scriptscriptstyle -1\slash 2}
\xi_{\!\scriptscriptstyle j}^{\scriptscriptstyle 2})}
  + \frac{\mathrm{b}_{\scriptscriptstyle 3}
\big(L_{w_{\!\scalebox{.85}{$\scriptscriptstyle \alpha'$}},\, p}^{\scriptscriptstyle (3)}
\!- L_{w_{\!\scalebox{.85}{$\scriptscriptstyle \alpha'$}},\,
  p}^{\scriptscriptstyle (2)}\big)}
  {2 \prod_{j = 1}^{r} (1 + q^{\,\scriptscriptstyle -1\slash 2}
\xi_{\!\scriptscriptstyle j}^{\scriptscriptstyle 2})}\right]
\end{equation}
where the entries of the diagonal matrix
$
\mathrm{M}_{w_{\!\scalebox{.9}{$\scriptscriptstyle J$}}^{}}
\!(\mathrm{z})^{\scriptscriptstyle -1}
$
are given by
\[
  \mathrm{b}_{\scriptscriptstyle 1} =
  \prod_{j \in  J} \left(\frac{1 - q^{\,\scriptscriptstyle -1\slash 2}
      \xi_{\!\scriptscriptstyle j}^{\scriptscriptstyle 2}}
    {1 - q^{\,\scriptscriptstyle -1\slash 2}
\xi_{\!\scriptscriptstyle j}^{\scriptscriptstyle - 2}}\!\right),\;\;\;
  \mathrm{b}_{\scriptscriptstyle 2} =
  \prod_{j \in  J}\xi_{\!\scriptscriptstyle j}^{\scriptscriptstyle 2},\;\, \mathrm{and}\;\;\,
  \mathrm{b}_{\scriptscriptstyle 3} =
  \prod_{j \in  J} \left(\frac{1 + q^{\,\scriptscriptstyle -1\slash 2}
      \xi_{\!\scriptscriptstyle j}^{\scriptscriptstyle 2}}
    {1 + q^{\,\scriptscriptstyle -1\slash 2}
\xi_{\!\scriptscriptstyle j}^{\scriptscriptstyle - 2}}\!\right).
\]
Note that the first term in
$
S_{\!
  p}^{w_{\!\scalebox{.85}{$\scriptscriptstyle \alpha^{^{}}$}}}
$
\!is
\[
\frac{L_{w_{\!\scalebox{.85}{$\scriptscriptstyle \alpha'$}},\, p}^{\scriptscriptstyle (1)}
\big(q^{\,\scriptscriptstyle -1\slash 2}
\xi_{\scriptscriptstyle 1}^{\scriptscriptstyle 2},
q^{\,\scriptscriptstyle -1\slash 2}
\xi_{\scriptscriptstyle 2}^{\scriptscriptstyle 2},
q^{\,\scriptscriptstyle -1\slash 2}
\xi_{\scriptscriptstyle 3}^{\scriptscriptstyle 2}; \zeta_{a}\big)}
{q^{\scriptscriptstyle 3\slash 4}\zeta_{a}
  \xi_{\scriptscriptstyle 1}
\xi_{\scriptscriptstyle 2}
\xi_{\scriptscriptstyle 3}}
= \frac{L_{w_{\!\scalebox{.85}{$\scriptscriptstyle \alpha'$}},\, p}^{\scriptscriptstyle (1)}
\Big(q^{\,\scriptscriptstyle -1\slash 2}
\xi_{\scriptscriptstyle 1}^{\scriptscriptstyle 2},
q^{\,\scriptscriptstyle -1\slash 2}
\xi_{\scriptscriptstyle 2}^{\scriptscriptstyle 2},
q^{\,\scriptscriptstyle -1\slash 2}
\xi_{\scriptscriptstyle 3}^{\scriptscriptstyle 2},
\frac{1}{q^{\scriptscriptstyle 3\slash 4}\zeta_{a}
  \xi_{\scriptscriptstyle 1}
\xi_{\scriptscriptstyle 2}
\xi_{\scriptscriptstyle 3}};
\mathrm{sgn}(a); q\Big)}
{q^{\scriptscriptstyle 3\slash 4}\zeta_{a}
  \xi_{\scriptscriptstyle 1}
\xi_{\scriptscriptstyle 2}
\xi_{\scriptscriptstyle 3}}.
\]
Moreover, $S_{\! p}^{w_{\!\scalebox{.85}{$\scriptscriptstyle \alpha^{^{}}$}}}$
\!is obtained by replacing
$
\xi_{\!\scriptscriptstyle j}^{} \mapsto
\xi_{\!\scriptscriptstyle j}^{\scriptscriptstyle - 1}
$
(for $j \in J$) in
$
S_{\! p}^{w_{\!\scalebox{.85}{$\scriptscriptstyle \alpha'$}}}\!.
$
We regularize
$
S_{\! p}^{w_{\!\scalebox{.85}{$\scriptscriptstyle \alpha'$}}}
$
\!by setting 
\begin{equation} \label{eq: Sp-regularized}
  S_{\! p}^{\scriptscriptstyle \mathrm{reg}} : =
S_{\! p}^{w_{\!\scalebox{.85}{$\scriptscriptstyle \alpha'$}}}
\!\!\, R_{\! p}^{\scriptscriptstyle (3)}
  \!\left(\frac{\zeta_{a} \xi_{\scriptscriptstyle 2}\xi_{\scriptscriptstyle 3}}
  {q^{\scriptscriptstyle 1\slash 4}\xi_{\scriptscriptstyle 1}}\!,
  \frac{\zeta_{a}\xi_{\scriptscriptstyle 1}\xi_{\scriptscriptstyle 3}}
  {q^{\scriptscriptstyle 1\slash 4}\xi_{\scriptscriptstyle 2}}\!,
  \frac{\zeta_{a} \xi_{\scriptscriptstyle 1}\xi_{\scriptscriptstyle 2}}
  {q^{\scriptscriptstyle 1\slash 4}\xi_{\scriptscriptstyle 3}}\right)^{\!\scriptscriptstyle -1}
\!\cdot \; \prod_{i = 1}^{3}\prod_{j = 4}^{r}
\left(1 - \frac{
    \xi_{\scriptscriptstyle i}^{\scriptscriptstyle 2}
    \xi_{\!\scriptscriptstyle j}^{\scriptscriptstyle 2}}{q}\right)
\left(1- \frac{
    \xi_{\scriptscriptstyle i}^{\scriptscriptstyle -2}
    \xi_{\!\scriptscriptstyle j}^{\scriptscriptstyle 2}}{q}\right)
\prod_{4 \, \le \, k \, \le \, l \, \le \, r}\left(1 -
  \frac{\xi_{\scriptscriptstyle k}^{\scriptscriptstyle 2}
    \xi_{\scriptscriptstyle l}^{\scriptscriptstyle 2}}{q}\right)
\end{equation}
where
$
R_{\! p}^{\scriptscriptstyle (3)}
(z_{\scriptscriptstyle 1}^{}, z_{\scriptscriptstyle 2}^{}, z_{\scriptscriptstyle 3}^{})
= R_{\! p}(z_{\scriptscriptstyle 1}^{}, z_{\scriptscriptstyle 2}^{}, z_{\scriptscriptstyle 3}^{})
$
is the local factor \eqref{eq: Local-factor-residue-r=3}. This is a
polynomial in $q^{\scriptscriptstyle -1\slash 2}$ of the form
\[
  S_{\! p}^{\scriptscriptstyle \mathrm{reg}} = 1 \,
  +\, O\Big(q^{\scriptscriptstyle -\frac{3}{2}}\Big).
\]
As before, we define $S_{\! p}^{\scriptscriptstyle \mathrm{reg}},$ for
arbitrary $p,$ by simply substituting $\mathrm{sgn}(a) \mapsto
\chi_{a}(p),$ $\xi_{\scriptscriptstyle k} \mapsto \xi_{\scriptscriptstyle k}^{\deg\,p}\!,$ and
$q \mapsto |p|$ in the above formula. The product
$\prod_{p} S_{\! p}^{\scriptscriptstyle \mathrm{reg}}$ converges
absolutely, as long as each of the first three variables
$\xi_{\scriptscriptstyle 1},
\xi_{\scriptscriptstyle 2},
\xi_{\scriptscriptstyle 3}
$
is in a small neighborhood of the unit circle,
and
$
|\xi_{\scriptscriptstyle k}| < q^{\scriptscriptstyle \varepsilon}
$
($k=4, \ldots, r$) for sufficiently small positive $\varepsilon.$ Let $N_{\! p}({\bf \xi}, \chi_{a}),$ where
$
{\bf \xi} : =
$
$
(\xi_{\scriptscriptstyle 1}, \ldots, \xi_{r}),
$
be defined by
\begin{equation*}
 \begin{split}
N_{\! p}({\bf \xi}, \chi_{a}) = \; & R_{\! p}^{\scriptscriptstyle (3)}
  \!\left(\frac{(\zeta_{a} \xi_{\scriptscriptstyle
        2}\xi_{\scriptscriptstyle 3}\!)^{\scriptscriptstyle \deg\,p}}
  {|p|^{\scriptscriptstyle 1\slash 4}\xi_{\scriptscriptstyle 1}^{\scriptscriptstyle \deg\,p}}\!,
  \frac{(\zeta_{a}\xi_{\scriptscriptstyle 1}\xi_{\scriptscriptstyle 3}\!)^{\scriptscriptstyle \deg\,p}}
  {|p|^{\scriptscriptstyle 1\slash 4}\xi_{\scriptscriptstyle 2}^{\scriptscriptstyle \deg\,p}}\!,
  \frac{(\zeta_{a} \xi_{\scriptscriptstyle 1}\xi_{\scriptscriptstyle 2}\!)^{\scriptscriptstyle \deg\,p}}
  {|p|^{\scriptscriptstyle 1\slash 4}\xi_{\scriptscriptstyle 3}^{\scriptscriptstyle \deg\,p}}\right)\\
& \cdot\, \prod_{i = 1}^{3}\prod_{j = 4}^{r}
\left(1 - \frac{
    \xi_{\scriptscriptstyle i}^{\scriptscriptstyle 2 \deg\,p}
    \xi_{\!\scriptscriptstyle j}^{\scriptscriptstyle 2 \deg\,p}}{|p|}\right)^{\!\scriptscriptstyle -1}
\!\left(1- \frac{
    \xi_{\scriptscriptstyle i}^{\scriptscriptstyle -2 \deg\,p}
    \xi_{\!\scriptscriptstyle j}^{\scriptscriptstyle 2 \deg\,p}}{|p|}\right)^{\!\scriptscriptstyle -1}
\cdot \prod_{4 \, \le \, k \, \le \, l \, \le \, r}\left(1 -
  \frac{\xi_{\scriptscriptstyle k}^{\scriptscriptstyle 2 \deg\,p}
    \xi_{\scriptscriptstyle l}^{\scriptscriptstyle 2 \deg\,p}}{|p|}\right)^{\!\scriptscriptstyle -1}
    \end{split}
  \end{equation*}
  and set $N({\bf \xi}, \chi_{a}) := \prod_{p}N_{\! p}({\bf \xi},
  \chi_{a}).$ Thus
  \begin{equation*}
    \begin{split}
      N({\bf \xi}, \chi_{a}) = \, & \left(1 - \mathrm{sgn}(a)\sqrt{q}
    \,\xi_{\scriptscriptstyle 1}^{\scriptscriptstyle 2}
    \xi_{\scriptscriptstyle 2}^{\scriptscriptstyle 2}
    \xi_{\scriptscriptstyle 3}^{\scriptscriptstyle 2}
  \right)^{\scriptscriptstyle -1}
  \, \cdot \prod_{1 \, \le \, i \, \le \, j \, \le \,  3}
  \left(1 - \frac{\mathrm{sgn}(a)\sqrt{q}
      \,\xi_{\scriptscriptstyle 1}^{\scriptscriptstyle 2}
      \xi_{\scriptscriptstyle 2}^{\scriptscriptstyle 2}
      \xi_{\scriptscriptstyle 3}^{\scriptscriptstyle 2}}
    {\xi_{\scriptscriptstyle i}^{\scriptscriptstyle 2}
      \xi_{\!\scriptscriptstyle j}^{\scriptscriptstyle
        2}}\right)^{\!\scriptscriptstyle -1}\\
  &\cdot\, \prod_{i = 1}^{3}\prod_{j = 4}^{r}
\big(1 - 
   \xi_{\scriptscriptstyle i}^{\scriptscriptstyle 2}
    \xi_{\!\scriptscriptstyle j}^{\scriptscriptstyle 2}\big)^{\scriptscriptstyle -1}
\big(1- 
    \xi_{\scriptscriptstyle i}^{\scriptscriptstyle -2}
    \xi_{\!\scriptscriptstyle j}^{\scriptscriptstyle 2}\big)^{\scriptscriptstyle -1}
\, \cdot \prod_{4 \, \le \, k \, \le \, l \, \le \, r} \big(1 -
  \xi_{\scriptscriptstyle k}^{\scriptscriptstyle 2}
    \xi_{\scriptscriptstyle l}^{\scriptscriptstyle 2}\big)^{\scriptscriptstyle -1}\!.
    \end{split}
    \end{equation*} 
    For every monic irreducible $p,$ let $R_{\! p}^{\scriptscriptstyle (r)}\!(\underline{z})$ 
    denote the local factor of the (modified) residue at $z_{r + {\scriptscriptstyle 1}} = 1\slash q$ of the function
$
Z(\mathrm{z}):=
Z^{\scriptscriptstyle (1)}\!(\mathbf{s}; \chi_{\scriptscriptstyle 1}\!,
\chi_{\scriptscriptstyle 1}\!).
$ 
For $a\in \{1, \, \theta_{\scriptscriptstyle 0} \},$ define 
$\mathscr{R}_{\! p}^{\scriptscriptstyle (r)}\!({\bf \xi}, \chi_{a})$ by 
\[
\mathscr{R}_{\! p}^{\scriptscriptstyle (r)}\!({\bf \xi}, \chi_{a}) : = 
R_{\! p}^{\scriptscriptstyle (r)}
\!\left(\frac{(\zeta_{a} \xi_{\scriptscriptstyle 2}\xi_{\scriptscriptstyle 3}\!)^{\scriptscriptstyle \deg\,p}}
{|p|^{\scriptscriptstyle 1\slash 4}\xi_{\scriptscriptstyle 1}^{\scriptscriptstyle \deg\,p}}\!,
\frac{(\zeta_{a}\xi_{\scriptscriptstyle 1}\xi_{\scriptscriptstyle 3}\!)^{\scriptscriptstyle \deg\,p}}
{|p|^{\scriptscriptstyle 1\slash 4}\xi_{\scriptscriptstyle 2}^{\scriptscriptstyle \deg\,p}}\!,
\frac{(\zeta_{a} \xi_{\scriptscriptstyle 1}\xi_{\scriptscriptstyle 2}\!)^{\scriptscriptstyle \deg\,p}}
{|p|^{\scriptscriptstyle 1\slash 4}\xi_{\scriptscriptstyle 3}^{\scriptscriptstyle \deg\,p}}\!, 
\frac{\xi_{\scriptscriptstyle 4}^{\scriptscriptstyle 2 \deg\,p}}
{|p|^{\scriptscriptstyle 3\slash 4} 
(\zeta_{a} \xi_{\scriptscriptstyle 1}
\xi_{\scriptscriptstyle 2}
\xi_{\scriptscriptstyle 3}\!)^{^{\deg\,p}}}, \ldots, 
\frac{\xi_{\scriptscriptstyle r}^{\scriptscriptstyle 2 \deg\,p}}
{|p|^{\scriptscriptstyle 3\slash 4} 
	(\zeta_{a} \xi_{\scriptscriptstyle 1}
	\xi_{\scriptscriptstyle 2}
	\xi_{\scriptscriptstyle 3}\!)^{^{\deg\,p}}}\right)
\] 
and put 
$
\mathscr{R}^{\scriptscriptstyle (r)}\!({\bf \xi}, \chi_{a}) = 
\prod_{p}\mathscr{R}_{\! p}^{\scriptscriptstyle (r)}\!({\bf \xi}, \chi_{a}).
$ 
Then the product over $p$ in \eqref{eq: Secondary-princ-parts-general} 
(in the variables $\xi_{\scriptscriptstyle 1}, \ldots, \xi_{r}$) corresponding to the 
$(\alpha', \zeta_{a})$-term is given by the formula 
\begin{equation} \label{eq: reg-formula-princ-parts} 
\mathscr{R}^{\scriptscriptstyle (r)}\!({\bf \xi}, \chi_{a}) 
\prod_{p}\frac{S_{\! p}^{w_{\!\scalebox{.85}{$\scriptscriptstyle \alpha'$}}}\!\!({\bf \xi}, \chi_{a})}
{\mathscr{R}_{\! p}^{\scriptscriptstyle (r)}\!({\bf \xi}, \chi_{a})};
\end{equation} 
this product converges initially when 
$
\xi_{\scriptscriptstyle 1},
\xi_{\scriptscriptstyle 2},
\xi_{\scriptscriptstyle 3}
$ 
are in a small neighborhood of the unit circle, 
and 
$
|\xi_{\scriptscriptstyle k}| < q^{\scriptscriptstyle \varepsilon}
$ 
for $k \ge 4.$ Thus we cannot immediately simplify 
$
\prod_{p}\mathscr{R}_{\! p}^{\scriptscriptstyle (r)}\!({\bf \xi}, \chi_{a})
$ 
in \eqref{eq: reg-formula-princ-parts}. To show that this product does indeed cancel 
$\mathscr{R}^{\scriptscriptstyle (r)}\!({\bf \xi}, \chi_{a}),$ and thus obtain a completely explicit formula for 
\eqref{eq: reg-formula-princ-parts}, we shall need the following lemma.

\vskip10pt
\begin{lem}  \label{Euler-prod-convergence-Rp} --- For sufficiently small $\varepsilon > 0,$ the series defining 
	$R_{\! p}^{\scriptscriptstyle (r)}\slash R_{\! p}^{\scriptscriptstyle (3)}$ is absolutely convergent, and 
	\[
	\frac{R_{\! p}^{\scriptscriptstyle (r)}\!(\xi_{\scriptscriptstyle 1}^{\scriptscriptstyle \deg\,p}\!, 
		\ldots, \xi_{r}^{\scriptscriptstyle \deg\,p})}{R_{\! p}^{\scriptscriptstyle (3)}
		\!(\xi_{\scriptscriptstyle 1}^{\scriptscriptstyle \deg\,p}\!, 
		\xi_{\scriptscriptstyle 2}^{\scriptscriptstyle \deg\,p}\!, 
		\xi_{\scriptscriptstyle 3}^{\scriptscriptstyle \deg\,p})} 
	= 1 \, + \, O\left(|p|^{\scriptscriptstyle - 1 - \varepsilon} \right)
	\] 
	in the polydisk $\Omega_{\varepsilon} : |\xi_{\scriptscriptstyle k}| < 
	q^{\, \scriptscriptstyle - \frac{1}{4} + \frac{\varepsilon}{2}}
	$ 
	for $k \le 3,$ and $|\xi_{\scriptscriptstyle k}| < 
	q^{\, \scriptscriptstyle - 3 - 11\varepsilon}$ 
	for $k \ge 4.$ 
\end{lem}

\begin{proof} \!We can obviously assume that $p$ is linear. By \eqref{eq: f-e-loc} and 
	\eqref{eq: Local-factor-residue-principal}, we can write 
	\begin{equation*}
	\begin{split}
		R_{\! p}^{\scriptscriptstyle (r)}({\bf \xi}) & 
		= \left(1 - q^{\scriptscriptstyle -1} \right) \, \cdot \, \left(\tfrac{1}{2}, 1, \tfrac{1}{2}\right)
		\tilde{{\bf f}}\left({\bf \xi}, q^{\scriptscriptstyle -1}; q\right) \\
	&= \left(1 - q^{\scriptscriptstyle -1} \right) \, \cdot \, 
	\left(\tfrac{1}{2}, 1, \tfrac{1}{2}\right) 
	\mathrm{M}_{w_{\!\scalebox{.9}{$\scriptscriptstyle \alpha'$}}^{\scalebox{.9}{$\scriptscriptstyle -1$}}}
\!\left({\bf \xi}, q^{\scriptscriptstyle -1}; q\right)
	\tilde{{\bf f}}\left(w_{\!{\scriptscriptstyle \alpha'}}^{\scriptscriptstyle -1}\big({\bf \xi}, q^{\scriptscriptstyle -1}\big);\,  q\right)
	\end{split}
	\end{equation*} 
	where 
	$
	w_{\!{\scriptscriptstyle \alpha'}} = 
	w_{\scriptscriptstyle 1}
	w_{\scriptscriptstyle 2}
	w_{\scriptscriptstyle 3}
	w_{r + {\scriptscriptstyle 1}}.
	$ 
	The entries of 
	\[ 
	\xi_{\scriptscriptstyle 1}^{\scriptscriptstyle 2} 
	\xi_{\scriptscriptstyle 2}^{\scriptscriptstyle 2} 
	\xi_{\scriptscriptstyle 3}^{\scriptscriptstyle 2} 
	R_{\! p}^{\scriptscriptstyle (3)}\!\left(\xi_{\scriptscriptstyle 1}, \xi_{\scriptscriptstyle 2}, \xi_{\scriptscriptstyle 3}\right)^{\scriptscriptstyle -1} \, \cdot \, \left(\tfrac{1}{2}, 1, \tfrac{1}{2}\right)
	\mathrm{M}_{w_{\!\scalebox{.9}{$\scriptscriptstyle \alpha'$}}^{\scalebox{.9}{$\scriptscriptstyle -1$}}}\!\left({\bf \xi}, q^{\scriptscriptstyle -1}; q\right)
	\] 
	are polynomials in 
	$
	\xi_{\scriptscriptstyle 1}, \xi_{\scriptscriptstyle 2}, \xi_{\scriptscriptstyle 3},
	$ 
	and 
	$
	w_{\!{\scriptscriptstyle \alpha'}}^{\scriptscriptstyle -1}
	\big({\bf \xi}, q^{\scriptscriptstyle -1}\big) 
	= \left(w_{\!{\scriptscriptstyle \alpha'}}^{\scriptscriptstyle  -1}\big({\bf \xi}, q^{\scriptscriptstyle  -1}\big)_{k}\right)_{1 \, \le k \, \le \,  r+1}
	$ 
	with 
	\[ 
	w_{\!{\scriptscriptstyle \alpha'}}^{\scriptscriptstyle -1}\big({\bf \xi}, q^{\scriptscriptstyle -1}\big)_{k} = 
	\begin{cases} 
	\left(\xi_{\scriptscriptstyle 1}\xi_{\scriptscriptstyle 2}\xi_{\scriptscriptstyle 3}\right) 
	\slash \xi_{\scriptscriptstyle k} & \text{if } k \le 3 \\ 
	q\, \xi_{\scriptscriptstyle 1}\xi_{\scriptscriptstyle 2}\xi_{\scriptscriptstyle 3}\xi_{\scriptscriptstyle k} & \text{if } 4 \le k \le r \\ 
	\left(q^{\scriptscriptstyle 3\slash 2}\xi_{\scriptscriptstyle 1}\xi_{\scriptscriptstyle 2}\xi_{\scriptscriptstyle 3}\right)^{\scriptscriptstyle -1} & \text{if } k = r+1.
	\end{cases}
	\] 
	Thus 
	$
	R_{\! p}^{\scriptscriptstyle (r)}({\bf \xi})
	\slash R_{\! p}^{\scriptscriptstyle (3)}\!\left(\xi_{\scriptscriptstyle 1}, \xi_{\scriptscriptstyle 2}, \xi_{\scriptscriptstyle 3}\right)
	$ 
	is holomorphic for 
	$
	-\frac{1}{4} - \varepsilon < \log_{q}|\xi_{\scriptscriptstyle k}| < -\frac{1}{4}
	$ 
	($k \le 3$) and 
	$
	\log_{q}|\xi_{\scriptscriptstyle k}| < - 3 - 8\varepsilon$ ($k \ge 4$).

	Now, for $i < j \le r,$ let $w_{\scriptscriptstyle ij} \in W$ be defined by 
	$
	w_{\scriptscriptstyle ij} \!= \!w_{\scriptscriptstyle i}w_{\!\scriptscriptstyle j}w_{r+{\scriptscriptstyle 1}}
	w_{\scriptscriptstyle i}w_{\!\scriptscriptstyle j},
	$ 
	and put $w^{\scriptscriptstyle \ast} \!\! = \!w_{\scriptscriptstyle 12}w_{\scriptscriptstyle 13}w_{\scriptscriptstyle 23}.$ By applying the cocycle relation \eqref{eq: def-cocycle2}, it is straightforward to check that the 
	function $R_{\! p}^{\scriptscriptstyle (r)}\slash R_{\! p}^{\scriptscriptstyle (3)}$ is invariant under $w^{\scriptscriptstyle \ast}\!.$ Note that this element acts on $\xi_{\scriptscriptstyle k}^{}$ by 
\[
\xi_{\scriptscriptstyle k}^{} \overset{w^{\ast}}{\longrightarrow}
 \begin{cases} 
\left(\sqrt{q}\, \xi_{\scriptscriptstyle k}^{} \right)^{\scriptscriptstyle -1} & \text{if } k \le 3 \\ 
q^{\scriptscriptstyle 3 \slash 2}
\xi_{\scriptscriptstyle 1}^{\scriptscriptstyle 2} 
\xi_{\scriptscriptstyle 2}^{\scriptscriptstyle 2} 
\xi_{\scriptscriptstyle 3}^{\scriptscriptstyle 2} 
\xi_{\scriptscriptstyle k}^{} & \text{if } k \ge 4. 
\end{cases}
\] 
Accordingly, the function 
$
R_{\! p}^{\scriptscriptstyle (r)}({\bf \xi}) \slash R_{\! p}^{\scriptscriptstyle (3)}
\!\left(\xi_{\scriptscriptstyle 1}, \xi_{\scriptscriptstyle 2}, \xi_{\scriptscriptstyle 3}\right)
$ 
is holomorphic for 
$
- \frac{1}{4} < \log_{q} |\xi_{\scriptscriptstyle k}| 
< - \frac{1}{4} + \frac{\varepsilon}{2}$ ($k \le 3$), and 
$ 
\log_{q} |\xi_{\scriptscriptstyle k}| < - 3 \linebreak 
- 11\varepsilon$ ($k \ge 4$); 
it is clear that in this region the function is represented by the scalar function 
\[
\left(1 - q^{\scriptscriptstyle -1}\right)R_{\! p}^{\scriptscriptstyle (3)}
\!\left(\xi_{\scriptscriptstyle 1}, \xi_{\scriptscriptstyle 2}, \xi_{\scriptscriptstyle 3}\right)^{\scriptscriptstyle -1} 
\, \cdot \, \left(\tfrac{1}{2}, 1, \tfrac{1}{2}\right)
\mathrm{M}_{w_{\!\scalebox{.9}{$\scriptscriptstyle \alpha'$}}^{\scalebox{.9}{$\scriptscriptstyle -1$}}w^{\ast}}
\!\left({\bf \xi}, q^{\scriptscriptstyle -1}; q\right) 
\tilde{{\bf f}}\left(w_{\!{\scriptscriptstyle \alpha'}}^{\scriptscriptstyle -1} w^{\scriptscriptstyle \ast}\big({\bf \xi}, q^{\scriptscriptstyle -1}\big);\,  q\right).
\] 
Since, by our assumptions, the vector function $\tilde{{\bf f}}\left({\bf \xi}, \xi_{r+{\scriptscriptstyle 1}}; q\right)$ continues meromorphically to the open convex cone $X^{\ast}\!,$ it follows at once that $R_{\! p}^{\scriptscriptstyle (r)}\slash R_{\! p}^{\scriptscriptstyle (3)}$ is meromorphic in the polydisk $\Omega_{\varepsilon}$ stated in the lemma.

To see that no pole can in fact occur in this region, \!we argue as follows. \!We show that: 
\begin{equation} \label{eq: pos-roots-omega-eps-ineq}
d(\alpha) - 1 + \sum_{k = 1}^{r} 2n_{\scriptscriptstyle k}\log_{q} |\xi_{\scriptscriptstyle k}| - 2n_{r+{\scriptscriptstyle 1}} < 0
\end{equation} 
for all 
$
\alpha = \sum n_{\scriptscriptstyle k}\alpha_{\scriptscriptstyle k} \in \Delta^{\mathrm{re}}_{+},$ and ${\bf \xi} \in \Omega_{\varepsilon}.$ Indeed, since 
$l(w^{\scriptscriptstyle \ast}) = 11,$ 
\[ 
w^{\scriptscriptstyle \ast}\!(\alpha) < 0 \iff \alpha \in 
\{\alpha_{\scriptscriptstyle 1}, \alpha_{\scriptscriptstyle 2}, \alpha_{\scriptscriptstyle 3}\} 
\cup \left\{\sum_{k \le 3} n_{\scriptscriptstyle k}\alpha_{\scriptscriptstyle k} + \alpha_{r+{\scriptscriptstyle 1}}: 
\text{with }n_{\scriptscriptstyle k} \in \{0, 1\},\text{not all zero}\right\} 
\cup \{\alpha_{\scriptscriptstyle 1} + \alpha_{\scriptscriptstyle 2} + \alpha_{\scriptscriptstyle 3} 
+ 2\alpha_{r+{\scriptscriptstyle 1}}\}.
\] 
All these roots satisfy \eqref{eq: pos-roots-omega-eps-ineq}, as the reader can easily verify. If 
$w^{\scriptscriptstyle \ast}\!(\alpha) > 0,$ then by computing the coefficients of $\alpha_{\scriptscriptstyle k}$ (for $k \le 3$) in 
$w^{\scriptscriptstyle \ast}\!(\alpha),$ one finds that
\begin{equation} \label{eq: Ineq-pos-roots-coeff}
n_{\scriptscriptstyle 1} 
+ n_{\scriptscriptstyle 2} 
+n_{\scriptscriptstyle 3} 
\le 6\sum_{k = 4}^{r}n_{\scriptscriptstyle k}.
\end{equation} 
Thus, if $\log_{q} |\xi_{\scriptscriptstyle k}| < - \frac{1}{4} + \frac{\varepsilon}{2}$ for $k \le 3,$ 
and $\log_{q} |\xi_{\scriptscriptstyle k}| < - 3 - 11\varepsilon$ for $k \ge 4,$ then the expression in 
\eqref{eq: pos-roots-omega-eps-ineq} is smaller than 
\[ 
(\tfrac{1}{2} + \varepsilon)(n_{\scriptscriptstyle 1} + n_{\scriptscriptstyle 2} + n_{\scriptscriptstyle 3}) - (5 + 22\varepsilon)
\sum_{k = 4}^{r} n_{\scriptscriptstyle k}
\] 
which by \eqref{eq: Ineq-pos-roots-coeff} is $\le 0.$ It follows that 
\begin{equation*}
\begin{split}
\frac{R_{\! p}^{\scriptscriptstyle (r)}({\bf \xi})}
{R_{\! p}^{\scriptscriptstyle (3)}\!\left(
\xi_{\scriptscriptstyle 1}, 
\xi_{\scriptscriptstyle 2}, 
\xi_{\scriptscriptstyle 3}\right)} 
& = \left(1 - q^{\scriptscriptstyle -1}\right)
R_{\! p}^{\scriptscriptstyle (3)}\!\left(
\xi_{\scriptscriptstyle 1}, 
\xi_{\scriptscriptstyle 2}, 
\xi_{\scriptscriptstyle 3}\right)^{\scriptscriptstyle -1} \, \cdot \, \left(\tfrac{1}{2}, 1, \tfrac{1}{2}\right)
\tilde{{\bf f}}\left({\bf \xi}, q^{\scriptscriptstyle -1}; q\right) \\
& = \left(1 - q^{\scriptscriptstyle -1}\right)\sum B_{\scriptscriptstyle 
n_{\scalebox{.7}{$\scriptscriptstyle 1$}}\!, \, 
n_{\scalebox{.7}{$\scriptscriptstyle 2$}}\!,\, 
n_{\scalebox{.7}{$\scriptscriptstyle 3$}}\!, \ldots,  
n_{\scalebox{.95}{$\scriptscriptstyle r$}}} 
\xi_{\scriptscriptstyle 1}^{\scriptscriptstyle 
n_{\scalebox{.7}{$\scriptscriptstyle 1$}}}
\xi_{\scriptscriptstyle 2}^{\scriptscriptstyle 
n_{\scalebox{.7}{$\scriptscriptstyle 2$}}}
\xi_{\scriptscriptstyle 3}^{\scriptscriptstyle 
n_{\scalebox{.7}{$\scriptscriptstyle 3$}}} 
\cdots \; \xi_{r}^{\scriptscriptstyle 
n_{\scalebox{.95}{$\scriptscriptstyle r$}}}
\end{split}
\end{equation*} 
the power series being normally convergent on $\Omega_{\varepsilon}.$ Note
that 
$
B_{\scriptscriptstyle 0, \ldots, 0} 
= \left(1 - q^{\scriptscriptstyle -1} \right)^{\scriptscriptstyle -1}
$ 
so that the constant term of 
$
R_{\! p}^{\scriptscriptstyle (r)} \slash R_{\! p}^{\scriptscriptstyle (3)}
$ 
is 1.

For fixed $n_{\scriptscriptstyle 4}, \ldots,n_{r},$ let us temporarily denote the subseries 
$
\sum_{n_{\scalebox{.7}{$\scriptscriptstyle 1$}},\, 
n_{\scalebox{.7}{$\scriptscriptstyle 2$}},\, 
n_{\scalebox{.7}{$\scriptscriptstyle 3$}} \, \ge \, 0} 
B_{\scriptscriptstyle 
	n_{\scalebox{.7}{$\scriptscriptstyle 1$}}\!, \, 
	n_{\scalebox{.7}{$\scriptscriptstyle 2$}}\!,\, 
	n_{\scalebox{.7}{$\scriptscriptstyle 3$}}\!, \ldots, 
	n_{\scalebox{.95}{$\scriptscriptstyle r$}}} 
 \xi_{\scriptscriptstyle 1}^{n_{\scalebox{.7}{$\scriptscriptstyle 1$}}}
\xi_{\scriptscriptstyle 2}^{n_{\scalebox{.7}{$\scriptscriptstyle 2$}}}
\xi_{\scriptscriptstyle 3}^{n_{\scalebox{.7}{$\scriptscriptstyle 3$}}}
$ 
by 
$
F_{n_{\scalebox{.7}{$\scriptscriptstyle 4$}}, \ldots, n_{\scalebox{.95}{$\scriptscriptstyle r$}}}\!(
\xi_{\scriptscriptstyle 1}, 
\xi_{\scriptscriptstyle 2}, 
\xi_{\scriptscriptstyle 3}).
$ 
\!The invariance of $R_{\! p}^{\scriptscriptstyle (r)}\slash R_{\! p}^{\scriptscriptstyle (3)}$ under $w^{\scriptscriptstyle \ast}$ 
implies that 
\[
F_{n_{\scalebox{.7}{$\scriptscriptstyle 4$}}, \ldots, n_{\scalebox{.95}{$\scriptscriptstyle r$}}}
\!(
\xi_{\scriptscriptstyle 1}, 
\xi_{\scriptscriptstyle 2}, 
\xi_{\scriptscriptstyle 3}) = 
\left(q^{\scriptscriptstyle 3 \slash 2}
\xi_{\scriptscriptstyle 1}^{\scriptscriptstyle 2}
\xi_{\scriptscriptstyle 2}^{\scriptscriptstyle 2}
\xi_{\scriptscriptstyle 3}^{\scriptscriptstyle 2}\right)^{n_{\scalebox{.7}{$\scriptscriptstyle 4$}} +  \cdots + n_{\scalebox{.95}{$\scriptscriptstyle r$}}}
\!\!F_{n_{\scalebox{.7}{$\scriptscriptstyle 4$}}, \ldots, n_{\scalebox{.95}{$\scriptscriptstyle r$}}}
\!\left(\frac{1}{\sqrt{q} \, \xi_{\scriptscriptstyle 1}}, 
\frac{1}{\sqrt{q} \, \xi_{\scriptscriptstyle 2}}, 
\frac{1}{\sqrt{q} \, \xi_{\scriptscriptstyle 3}}\right)
\] 
when $|\xi_{\scriptscriptstyle k}| = q^{\scriptscriptstyle -1 \slash 4}.$ Thus by applying Cauchy’s integral formula
\[
B_{n_{\scalebox{.7}{$\scriptscriptstyle 1$}}\!, \ldots, 
n_{\scalebox{.95}{$\scriptscriptstyle r$}}} = 
\frac{1}{\big(2\pi \sqrt{-1}\big)^{\!{\scriptscriptstyle 3}}}
\, \oint\limits_{|\xi_{\scalebox{.8}{$\scriptscriptstyle 1$}}|\,= \,
	q^{\, \scalebox{.8}{$\scriptscriptstyle  -1\slash 4$}}} \;
\oint\limits_{|\xi_{\scalebox{.8}{$\scriptscriptstyle 2$}}|\,= \,
	q^{\, \scalebox{.8}{$\scriptscriptstyle  -1\slash 4$}}} \;
\oint\limits_{|\xi_{\scalebox{.8}{$\scriptscriptstyle 3$}}|\,= \,
	q^{\, \scalebox{.8}{$\scriptscriptstyle  -1\slash 4$}}}
\!\frac{\left(q^{\scriptscriptstyle 3\slash 2}
	\xi_{\scriptscriptstyle 1}^{\scriptscriptstyle 2}
	\xi_{\scriptscriptstyle 2}^{\scriptscriptstyle 2}
	\xi_{\scriptscriptstyle 3}^{\scriptscriptstyle 2}\right)^{n_{\scalebox{.7}{$\scriptscriptstyle 4$}}
		+ \cdots + n_{\scalebox{.95}{$\scriptscriptstyle r$}}}
	\!\!F_{n_{\scalebox{.7}{$\scriptscriptstyle 4$}}, \ldots, 
		n_{\scalebox{.95}{$\scriptscriptstyle r$}}}\!\left(\frac{1}{\sqrt{q}\,
		\xi_{\scalebox{.85}{$\scriptscriptstyle 1$}}},
	\frac{1}{\sqrt{q}\,
		\xi_{\scalebox{.85}{$\scriptscriptstyle 2$}}},
	\frac{1}{\sqrt{q}\,
		\xi_{\scalebox{.85}{$\scriptscriptstyle 3$}}}\right)}{
	\xi_{\scriptscriptstyle 1}^{n_{\scalebox{.7}{$\scriptscriptstyle 1$}}
		+ \, \scalebox{.95}{$\scriptscriptstyle 1$}}
	\xi_{\scriptscriptstyle 2}^{n_{\scalebox{.7}{$\scriptscriptstyle 2$}}
		+ \, \scalebox{.95}{$\scriptscriptstyle 1$}}
	\xi_{\scriptscriptstyle 3}^{n_{\scalebox{.7}{$\scriptscriptstyle 3$}}
		+ \, \scalebox{.95}{$\scriptscriptstyle 1$}}}
\; d \xi_{\scriptscriptstyle 3}\,
d \xi_{\scriptscriptstyle 2}\,
d\xi_{\scriptscriptstyle 1}
\]
we find that
$
B_{n_{\scalebox{.7}{$\scriptscriptstyle 1$}}\!, \ldots, 
	n_{\scalebox{.95}{$\scriptscriptstyle r$}}} \!= 0
$
if
$
2(n_{\scriptscriptstyle 4} + \cdots + n_{r}) < 
\max\{n_{\scriptscriptstyle 1},
n_{\scriptscriptstyle 2},
n_{\scriptscriptstyle 3}\}.
$
On the other hand, if
$
|\xi_{\scriptscriptstyle k}| = q^{\scriptscriptstyle -\frac{1}{2} - \varepsilon}
$ 
for all $k \le r,$ we have 
\[
\big|R_{\! p}^{\scriptscriptstyle (3)}\!\left(
\xi_{\scriptscriptstyle 1}, 
\xi_{\scriptscriptstyle 2}, 
\xi_{\scriptscriptstyle 3}\right)^{\scriptscriptstyle -1} \big| 
\le \big(1 + q^{\scriptscriptstyle - 2 - 6\varepsilon} \big) \big(1 +
q^{\scriptscriptstyle - 1 - 2\varepsilon} \big)^{\scriptscriptstyle 6} \ll 1
\] 
and, by our assumptions on $\tilde{f}(\mathrm{z}; q),$ the same
estimate holds for the entries of $\tilde{{\bf{f}}}\left({\bf \xi}, q^{\scriptscriptstyle -1}; q\right).$ By applying Cauchy's ine-\linebreak 
qualities, \!we find that 
$
\left|B_{n_{\scalebox{.7}{$\scriptscriptstyle 1$}}\!, \ldots, 
	n_{\scalebox{.95}{$\scriptscriptstyle r$}}}\right| \ll 
q^{\scriptscriptstyle \left(\frac{1}{2} + \, \varepsilon \right)
	(n_{\scalebox{.7}{$\scriptscriptstyle 1$}}
	+ \cdots + n_{\scalebox{.95}{$\scriptscriptstyle r$}})}\!.
$ 
Thus, for ${\bf \xi} \in \Omega_{\varepsilon},$ \!we have 
\begin{equation*}
\begin{split}
\left|\frac{R_{\! p}^{\scriptscriptstyle (r)}({\bf \xi})}{R_{\! p}^{\scriptscriptstyle (3)}
\!\left(\xi_{\scriptscriptstyle 1}, 
\xi_{\scriptscriptstyle 2}, 
\xi_{\scriptscriptstyle 3} \right)} -1 \right| 
&\ll -1 \, + \, \sum_{n_{\scalebox{.7}{$\scriptscriptstyle 4$}}\, \ge \, 0} \cdots 
\sum_{n_{\scalebox{.95}{$\scriptscriptstyle r$}} \,\ge \, 0} 
q^{\scriptscriptstyle \left(- \frac{5}{2} - 10\,\varepsilon \right)(n_{\scalebox{.7}{$\scriptscriptstyle 4$}} 
+ \cdots + n_{\scalebox{.95}{$\scriptscriptstyle r$}})}
\sum_{n_{\scalebox{.7}{$\scriptscriptstyle 1$}} = 0}^{2(n_{\scalebox{.7}{$\scriptscriptstyle 4$}} 
+ \cdots + n_{\scalebox{.95}{$\scriptscriptstyle r$}})}  \;
\sum_{n_{\scalebox{.7}{$\scriptscriptstyle 2$}} = 0}^{2(n_{\scalebox{.7}{$\scriptscriptstyle 4$}} 
+ \cdots + n_{\scalebox{.95}{$\scriptscriptstyle r$}})} \; 
\sum_{n_{\scalebox{.7}{$\scriptscriptstyle 3$}} = 0}^{2(n_{\scalebox{.7}{$\scriptscriptstyle 4$}} 
+ \cdots + n_{\scalebox{.95}{$\scriptscriptstyle r$}})} 
q^{\scriptscriptstyle \left(\frac{1}{4} + \frac{3\varepsilon}{2}\right)
(n_{\scalebox{.7}{$\scriptscriptstyle 1$}} 
+  n_{\scalebox{.7}{$\scriptscriptstyle 2$}} 
+  n_{\scalebox{.7}{$\scriptscriptstyle 3$}})}\\
&  \ll -1 \, + \, \sum_{n_{\scalebox{.7}{$\scriptscriptstyle 4$}}\, \ge \, 0} \cdots 
\sum_{n_{\scalebox{.95}{$\scriptscriptstyle r$}} \ge 0} q^{\, \scriptscriptstyle -(1 +\varepsilon)
(n_{\scalebox{.7}{$\scriptscriptstyle 4$}} + \cdots + n_{\scalebox{.95}{$\scriptscriptstyle r$}})}\\
& \ll q^{\, \scriptscriptstyle -1 - \varepsilon}
\end{split} 
\end{equation*} 
which completes the proof.
\end{proof}

\vskip10pt
\begin{prop} \label{Regularization-Justification} --- For sufficiently small $\varepsilon > 0,$ we have 
	\[ 
	\mathscr{R}^{\scriptscriptstyle (r)}\!({\bf \xi}, \chi_{a}) 
	\prod_{p}\frac{S_{\! p}^{w_{\!\scalebox{.85}{$\scriptscriptstyle \alpha'$}}}\!\!({\bf \xi}, \chi_{a})}
	{\mathscr{R}_{\! p}^{\scriptscriptstyle (r)}\!({\bf \xi}, \chi_{a})} 
	= N({\bf \xi}, \chi_{a})\prod_{p}S_{\! p}^{\scriptscriptstyle \mathrm{reg}}\!({\bf \xi}, \chi_{a})	
	\] 
	for 
	$
	\xi_{\scriptscriptstyle 1},
	\xi_{\scriptscriptstyle 2},
	\xi_{\scriptscriptstyle 3}
	$ 
	in a small neighborhood of the unit circle, and 
	$
	|\xi_{\scriptscriptstyle k}| < q^{\scriptscriptstyle \varepsilon}
	$ 
	for $k \ge 4.$
\end{prop}

\begin{proof} By definition, 
	$
	S_{\! p}^{w_{\!\scalebox{.85}{$\scriptscriptstyle \alpha'$}}}\!\!({\bf \xi}, \chi_{a}) 
	= N_{\! p}({\bf \xi}, \chi_{a})
	S_{\! p}^{\scriptscriptstyle \mathrm{reg}}\!({\bf \xi}, \chi_{a}).
	$ 
	Thus, by taking 
	$
	{\bf \xi} =
	(\xi_{\scriptscriptstyle 1}, \ldots, \xi_{r})
	$ 
	as stated in the proposition, we can write 
	\[ 
	\prod_{p}\frac{S_{\! p}^{w_{\!\scalebox{.85}{$\scriptscriptstyle \alpha'$}}}\!\!({\bf \xi}, \chi_{a})}
	{\mathscr{R}_{\! p}^{\scriptscriptstyle (r)}\!({\bf \xi}, \chi_{a})} 
	 = \prod_{p}S_{\! p}^{\scriptscriptstyle \mathrm{reg}}\!({\bf \xi}, \chi_{a}) 
	\, \cdot \, \prod_{p} \frac{N_{\! p}({\bf \xi}, \chi_{a})}{\mathscr{R}_{\! p}^{\scriptscriptstyle (r)}\!({\bf \xi}, \chi_{a})};
	\] 
	the product in the left-hand side converges by our assumptions. \!To conclude that 
	\[
\prod_{p} \frac{N_{\! p}({\bf \xi}, \chi_{a})}{\mathscr{R}_{\! p}^{\scriptscriptstyle (r)}\!({\bf \xi}, \chi_{a})}
=  \frac{N({\bf \xi}, \chi_{a})}{\mathscr{R}^{\scriptscriptstyle (r)}\!({\bf \xi}, \chi_{a})}
\] 
it suffices to show that 
\begin{equation}  \label{Local-asympt-estimate} 
N_{\! p}({\bf \xi}, \chi_{a})^{\scriptscriptstyle -1}\mathscr{R}_{\! p}^{\scriptscriptstyle (r)}\!({\bf \xi}, \chi_{a})
= 1 \, + \, O\left(|p|^{- 1 - \eta_{\scalebox{.9}{$\scriptscriptstyle \varepsilon$}}} \right) 
\qquad{\text{(with $\eta_{\scalebox{1.1}{$\scriptscriptstyle \varepsilon$}} \sim \tfrac{1}{2}$ as $\varepsilon \to 0$)}}
\end{equation} 
uniformly in 
$
\mathrm{D}_{\varepsilon} : = 
\{\text{${\bf \xi}: q^{- \varepsilon\slash 3} 
< |\xi_{\scriptscriptstyle k}| < q^{\varepsilon\slash 2}$ 
for $k = 1, 2, 3$ and $|\xi_{\scriptscriptstyle k}| < q^{\varepsilon\slash 2}$ 
for $k=4, \ldots, r$}\};
$ 
here it is understood that the function 
$
N_{\! p}({\bf \xi}, \chi_{a})^{\scriptscriptstyle -1}
\mathscr{R}_{\! p}^{\scriptscriptstyle (r)}\!({\bf \xi}, \chi_{a})
$ 
is holomorphic in $\mathrm{D}_{\varepsilon}.$ Indeed, assuming \eqref{Local-asympt-estimate} for the moment, the product 
\[
\prod_{p} N_{\! p}({\bf \xi}, \chi_{a})^{\scriptscriptstyle -1}
\mathscr{R}_{\! p}^{\scriptscriptstyle (r)}\!({\bf \xi}, \chi_{a})
\] 
converges absolutely, hence it is a holomorphic function, in $\mathrm{D}_{\varepsilon}.$ 
Restricting to $|\xi_{\scriptscriptstyle k}| < q^{-\varepsilon}$ ($k \ge 4$), we clearly have 
\begin{equation*}
\begin{split}
&\prod_{i = 1}^{3}\prod_{j = 4}^{r}
\left(1 - 
\xi_{\scriptscriptstyle i}^{\scriptscriptstyle 2}
\xi_{\!\scriptscriptstyle j}^{\scriptscriptstyle 2}\right)
\left(1- 
\xi_{\scriptscriptstyle i}^{\scriptscriptstyle -2}
\xi_{\!\scriptscriptstyle j}^{\scriptscriptstyle 2}\right)
\;\; \cdot \prod_{4 \, \le \, k \, \le \, l \, \le \, r} \big(1 -
\xi_{\scriptscriptstyle k}^{\scriptscriptstyle 2}
\xi_{\scriptscriptstyle l}^{\scriptscriptstyle 2}\big)\\ 
& = \prod_{p}\left\{\prod_{i = 1}^{3}\prod_{j = 4}^{r}
\left(1 - \frac{
	\xi_{\scriptscriptstyle i}^{\scriptscriptstyle 2 \deg\,p}
	\xi_{\!\scriptscriptstyle j}^{\scriptscriptstyle 2 \deg\,p}}{|p|}\right)
\!\left(1- \frac{
	\xi_{\scriptscriptstyle i}^{\scriptscriptstyle -2 \deg\,p}
	\xi_{\!\scriptscriptstyle j}^{\scriptscriptstyle 2 \deg\,p}}{|p|}\right)
\;\; \cdot \prod_{4 \, \le \, k \, \le \, l \, \le \, r}\left(1 -
\frac{\xi_{\scriptscriptstyle k}^{\scriptscriptstyle 2 \deg\,p}
\xi_{\scriptscriptstyle l}^{\scriptscriptstyle 2 \deg\,p}}{|p|}\right)\right\}.
\end{split}
\end{equation*} 
Put 
$
\xi_{\scriptscriptstyle 1}' 
= \frac{\xi_{\scriptscriptstyle 2}\xi_{\scriptscriptstyle 3}}
{q^{\scriptscriptstyle 1\slash 4}\xi_{\scriptscriptstyle 1}}, 
\xi_{\scriptscriptstyle 2}' 
= \frac{\xi_{\scriptscriptstyle 1}\xi_{\scriptscriptstyle 3}}
{q^{\scriptscriptstyle 1\slash 4}\xi_{\scriptscriptstyle 2}}, 
\xi_{\scriptscriptstyle 3}' 
= \frac{\xi_{\scriptscriptstyle 1}\xi_{\scriptscriptstyle 2}}
{q^{\scriptscriptstyle 1\slash 4}\xi_{\scriptscriptstyle 3}}, 
$ 
and 
$
\xi_{\scriptscriptstyle k}' 
= \frac{\xi_{\scriptscriptstyle k}^{\scriptscriptstyle 2}}
{q^{\scriptscriptstyle 3\slash 4} 
\xi_{\scriptscriptstyle 1}
\xi_{\scriptscriptstyle 2}
\xi_{\scriptscriptstyle 3}}
$ 
for $k \ge 4.$ Thus 
\[
\xi_{\scriptscriptstyle 1} 
= q^{\frac{1}{4}}
\sqrt{\xi_{\scriptscriptstyle 2}' \xi_{\scriptscriptstyle 3}'},\;\, 
\xi_{\scriptscriptstyle 2} 
= q^{\frac{1}{4}}
\sqrt{\xi_{\scriptscriptstyle 1}' \xi_{\scriptscriptstyle 3}'},\;\, 
\xi_{\scriptscriptstyle 3} 
= q^{\frac{1}{4}}
\sqrt{\xi_{\scriptscriptstyle 1}' \xi_{\scriptscriptstyle 2}'}  
\] 
with the obvious choice of the square root. \!Note that 
\[
q^{\,  - \frac{1}{4} - \frac{\varepsilon}{3}} 
< |\xi_{\scriptscriptstyle k}'| 
< q^{\, - \frac{1}{4} + \frac{\varepsilon}{2}} 
\implies 
q^{\, - \frac{\varepsilon}{3}} 
< |\xi_{\scriptscriptstyle k}| 
< q^{\frac{\varepsilon}{2}} 
\;\;  \text{(for $k \le 3$)} \;\; \text{and} \;\;
|\xi_{\scriptscriptstyle k}'| 
< q^{\, - \frac{3}{4} - \frac{\varepsilon}{2}} 
\implies  
|\xi_{\scriptscriptstyle k}| 
< q^{\frac{\varepsilon}{2}} 
\;\;  \text{(for $k \ge 4$).}
\] 
It follows that the product 
$
\prod_{p}\left(R_{\! p}^{\scriptscriptstyle (r)}\slash R_{\! p}^{\scriptscriptstyle (3)}\right), 
$ 
as a function of the variables 
$
{\bf \xi'} \! =
(\xi_{\scriptscriptstyle 1}', \ldots, \xi_{r}'),
$ 
is holomorphic in the domain 
$
\mathrm{D}_{\varepsilon}' \!: = 
\!\{\text{${\bf \xi}' \!: \! q^{ - \frac{1}{4}  - \frac{\varepsilon}{3}} 
\! < |\xi_{\scriptscriptstyle k}'| < q^{ - \frac{1}{4} + \frac{\varepsilon}{2}}$ 
for $k\le 3,$ $|\xi_{\scriptscriptstyle k}'| < q^{ - \frac{3}{4} - \frac{7\varepsilon}{2}}$ 
for $k \ge 4$}\}.
$ 
On the other hand, by Lemma \ref{Euler-prod-convergence-Rp}, the (original) series\linebreak 
of 
$
R_{\! p}^{\scriptscriptstyle (r)}\slash R_{\! p}^{\scriptscriptstyle (3)}
$ 
is absolutely convergent, and 
\[
	\frac{R_{\! p}^{\scriptscriptstyle (r)}\!(\xi_{\scriptscriptstyle 1}'^{\scriptscriptstyle \deg\,p}\!, 
		\ldots, \xi_{r}'^{\scriptscriptstyle \deg\,p})}{R_{\! p}^{\scriptscriptstyle (3)}
		\!(\xi_{\scriptscriptstyle 1}'^{\scriptscriptstyle \deg\,p}\!, 
		\xi_{\scriptscriptstyle 2}'^{\scriptscriptstyle \deg\,p}\!, 
		\xi_{\scriptscriptstyle 3}'^{\scriptscriptstyle \deg\,p})} 
	= 1 \, + \, O\left(|p|^{\scriptscriptstyle - 1 - \varepsilon} \right)
	\] 
	when 
	$|\xi_{\scriptscriptstyle k}'| < 
	q^{\, - \frac{1}{4} + \frac{\varepsilon}{2}}
	$ 
	for $k \le 3,$ and $|\xi_{\scriptscriptstyle k}'| < 
	q^{\, \scriptscriptstyle - 3 - 11\varepsilon}
	$ 
	for $k \ge 4.$ If 
	$
	|\xi_{\scriptscriptstyle k}'| < 
	q^{\, - \frac{1}{2} - \varepsilon}
	$ 
	($k \le 3$), 
	\[
	R^{\scriptscriptstyle (3)}
	\!(\xi_{\scriptscriptstyle 1}', 
	\xi_{\scriptscriptstyle 2}', 
	\xi_{\scriptscriptstyle 3}')
	= \prod_{p}R_{\! p}^{\scriptscriptstyle (3)}
	\!(\xi_{\scriptscriptstyle 1}'^{\scriptscriptstyle \deg\,p}\!, 
	\xi_{\scriptscriptstyle 2}'^{\scriptscriptstyle \deg\,p}\!, 
	\xi_{\scriptscriptstyle 3}'^{\scriptscriptstyle \deg\,p})
	\] 
	and thus, subject to \eqref{Local-asympt-estimate}, the proposition follows by analytic continuation. 
	
	It remains to establish the asymptotic formula \eqref{Local-asympt-estimate}. 
	By applying the functional equation \eqref{eq: f-e-loc}, we have 
	\begin{equation*}
	\begin{split}
	\mathscr{R}_{\! p}^{\scriptscriptstyle (r)}\!({\bf \xi}, \chi_{a}) & = 
	R_{\! p}^{\scriptscriptstyle (r)}
	\!\left(\frac{(\zeta_{a} \xi_{\scriptscriptstyle 2}\xi_{\scriptscriptstyle 3}\!)^{\scriptscriptstyle \deg\,p}}
	{|p|^{\scriptscriptstyle 1\slash 4}\xi_{\scriptscriptstyle 1}^{\scriptscriptstyle \deg\,p}}\!,
	\frac{(\zeta_{a}\xi_{\scriptscriptstyle 1}\xi_{\scriptscriptstyle 3}\!)^{\scriptscriptstyle \deg\,p}}
	{|p|^{\scriptscriptstyle 1\slash 4}\xi_{\scriptscriptstyle 2}^{\scriptscriptstyle \deg\,p}}\!,
	\frac{(\zeta_{a} \xi_{\scriptscriptstyle 1}\xi_{\scriptscriptstyle 2}\!)^{\scriptscriptstyle \deg\,p}}
	{|p|^{\scriptscriptstyle 1\slash 4}\xi_{\scriptscriptstyle 3}^{\scriptscriptstyle \deg\,p}}\!, 
	\frac{\xi_{\scriptscriptstyle 4}^{\scriptscriptstyle 2 \deg\,p}}
	{|p|^{\scriptscriptstyle 3\slash 4} 
		(\zeta_{a} \xi_{\scriptscriptstyle 1}
		\xi_{\scriptscriptstyle 2}
		\xi_{\scriptscriptstyle 3}\!)^{^{\deg\,p}}}, \ldots, 
	\frac{\xi_{\scriptscriptstyle r}^{\scriptscriptstyle 2 \deg\,p}}
	{|p|^{\scriptscriptstyle 3\slash 4} 
		(\zeta_{a} \xi_{\scriptscriptstyle 1}
		\xi_{\scriptscriptstyle 2}
		\xi_{\scriptscriptstyle 3}\!)^{^{\deg\,p}}}\right)\\
	& = \left(1 - |p|^{\scriptscriptstyle -1}\right) 
	\, \cdot \, \left(\tfrac{1}{2}, \scalebox{1.5}{$\scriptscriptstyle \chi_{\scalebox{.8}{$\scriptscriptstyle a$}}\!(p)$}, \tfrac{1}{2}\right)
	\!\left.\mathrm{M}_{w_{\scalebox{.85}{$\scriptscriptstyle \alpha'$}}}
\big(\mathrm{z}; |p|\big)^{\scriptscriptstyle -1} 
\tilde{{\bf f}}\big(\mathrm{z}; |p|\big)
\right\vert_{\scriptscriptstyle z_{\scalebox{.6}{$\scriptscriptstyle 1$}} = 
\frac{\xi_{\scalebox{.75}{$\scriptscriptstyle 1$}}^{\scalebox{.75}{$\scriptscriptstyle 2\deg\, p$}}}{\sqrt{|p|}}\!, 
\ldots, \, z_{\scalebox{.72}{$\scriptscriptstyle r$}} = 
\frac{\xi_{\scalebox{.9}{$\scriptscriptstyle r$}}^{\scalebox{.75}{$\scriptscriptstyle 2\deg\, p$}}}{\sqrt{|p|}}\!,\,  
z_{\scalebox{.72}{$\scriptscriptstyle r$} \scalebox{.6}{$\scriptscriptstyle + 1$}} = 
\frac{1}{|p|^{\scalebox{.75}{$\scriptscriptstyle 3 \slash 4$}} \left(\zeta_{a} 
\xi_{\scalebox{.75}{$\scriptscriptstyle 1$}}
\xi_{\scalebox{.75}{$\scriptscriptstyle 2$}}
\xi_{\scalebox{.75}{$\scriptscriptstyle 3$}}\!\right)^{\scalebox{.75}{$\scriptscriptstyle \deg\,p$}}}}
\end{split}
\end{equation*} 
where we recall that 
$
w_{\!{\scriptscriptstyle \alpha'}} \!= 
w_{\scriptscriptstyle 1}
w_{\scriptscriptstyle 2}
w_{\scriptscriptstyle 3}
w_{r + {\scriptscriptstyle 1}}.
$ 
By Lemma \ref{Local-vect-estimate}, the vector function 
$\tilde{{\bf f}}\big(\mathrm{z}; |p|\big)$ satisfies 
\[
\tilde{{\bf f}}\big(\mathrm{z}; |p|\big) 
\; =\,   ^{^{^{t}}}\!\!\bigg(\prod_{j = 1}^{r}\left(1 - z_{\!\scriptscriptstyle j}\right)^{\scriptscriptstyle -1}\!, \, 
z_{r + {\scriptscriptstyle 1}}, 
\prod_{j = 1}^{r}\left(1 + z_{\!\scriptscriptstyle j}\right)^{\scriptscriptstyle -1}
\bigg) 
\, + \, O_{\varepsilon}\left(|p|^{-\frac{3}{2} + 4(r+1)\varepsilon}\right)
\] 
as long as $|z_{\scriptscriptstyle i}| \le |p|^{- \frac{1}{2} + \varepsilon}$ ($i = 1, \ldots, r$), and 
$|z_{r + {\scriptscriptstyle 1}}| \le |p|^{- \frac{3}{4} + \varepsilon}.$ By computing 
$ 
\left(\tfrac{1}{2}, \scalebox{1.5}{$\scriptscriptstyle \chi_{\scalebox{.8}{$\scriptscriptstyle a$}}\!(p)$}, \tfrac{1}{2}\right)
\!\mathrm{M}_{w_{\scalebox{.85}{$\scriptscriptstyle \alpha'$}}}
\big(\mathrm{z}; |p|\big)^{\scriptscriptstyle -1}
$ 
using \eqref{eq: def-cocycle2}, one finds that 
\begin{equation*}
\begin{split}
& \left(1 - |p|^{\scriptscriptstyle -1}\right)
N_{\! p}({\bf \xi}, \chi_{a})^{\scriptscriptstyle -1}
\, \cdot \, \left(\tfrac{1}{2}, \scalebox{1.5}{$\scriptscriptstyle \chi_{\scalebox{.8}{$\scriptscriptstyle a$}}\!(p)$}, \tfrac{1}{2}\right)
\!\mathrm{M}_{w_{\scalebox{.85}{$\scriptscriptstyle \alpha'$}}}
\big(\mathrm{z}; |p|\big)^{\scriptscriptstyle -1} 
\; ^{^{^{t}}}\!\!\bigg(\prod_{j = 1}^{r}\left(1 - z_{\!\scriptscriptstyle j}\right)^{\scriptscriptstyle -1}\!, \, 
z_{r + {\scriptscriptstyle 1}}, 
\prod_{j = 1}^{r}\left(1 + z_{\!\scriptscriptstyle j}\right)^{\scriptscriptstyle -1}
\bigg)\\
& = 1 \, + \, O\left(|p|^{- 1 - \eta'_{\scalebox{.9}{$\scriptscriptstyle \varepsilon$}}} \right)
\end{split}
\end{equation*} 
uniformly for ${\bf \xi} \in \mathrm{D}_{\varepsilon},$ 
with 
$
\eta'_{\scalebox{1.1}{$\scriptscriptstyle \varepsilon$}} \sim \tfrac{1}{2}
$ 
as $\varepsilon \to 0,$ and $z_{\!\scriptscriptstyle j}$ ($j \le r + 1$) as above. This gives \eqref{Local-asympt-estimate}, which completes the proof.\end{proof}

It remains to compute the factor 
$
2^{-l(w_{\scalebox{.85}{$\scriptscriptstyle \alpha$}})}\Gamma_{\!\! w_{\scalebox{.85}{$\scriptscriptstyle \alpha$}}}\!(a_{\scriptscriptstyle 2}, a; \zeta_{a} \!).
$ 
For ease of notation, let 
$
\overline{\mathrm{M}}_{w}(\mathrm{z}; q) = \mathrm{M}_{w}(q\mathrm{z}; 1\slash q)
$ 
for $w \in W.$ Then 
$ 
\overline{\mathrm{M}}_{w_{\!\scalebox{.85}{$\scriptscriptstyle \alpha^{^{}}$}}}\!\!(\mathrm{z}; q) = \overline{\mathrm{M}}_{w_{\!\scalebox{.9}{$\scriptscriptstyle J$}}^{}}\!(\mathrm{z}; q)  
\overline{\mathrm{M}}_{w_{\!\scalebox{.85}{$\scriptscriptstyle \alpha'$}}}\!(w_{\!\scalebox{1.0}{$\scriptscriptstyle J$}}^{}\mathrm{z}; q), 
$ 
and 
\begin{equation} \label{eq: Global-cocycle-formula} 
\overline{\mathrm{M}}_{w_{\!\scalebox{.85}{$\scriptscriptstyle \alpha'$}}}\!(\mathrm{z}; q) 
= 
\frac{1}{2 \, q^{\scriptscriptstyle 3\slash 2} z_{\scriptscriptstyle 1}^{}
z_{\scriptscriptstyle 2}^{}
z_{\scriptscriptstyle 3}^{} 
z_{r + {\scriptscriptstyle 1}}^{\scriptscriptstyle 4}}
\begin{pmatrix} 
\psi(z_{\scriptscriptstyle 1}^{}, z_{\scriptscriptstyle 2}^{},
z_{\scriptscriptstyle 3}^{}, z_{r + {\scriptscriptstyle 1}}^{}; q) & 
- \frac{2}{q^{\scalebox{.95}{$\scriptscriptstyle 3 \slash 2$}}}\!\left(\!\frac{1 - q z_{r + {\scalebox{.85}{$\scriptscriptstyle 1$}}}^{\scalebox{.9}{$\scriptscriptstyle 2$}}}
{1 - q^{\scalebox{.9}{$\scriptscriptstyle 2$}} z_{r + {\scalebox{.85}{$\scriptscriptstyle 1$}}}^{\scalebox{.9}{$\scriptscriptstyle 2$}}}\!\right)& 
\psi(z_{\scriptscriptstyle 1}^{}, z_{\scriptscriptstyle 2}^{},
z_{\scriptscriptstyle 3}^{}, - z_{r + {\scriptscriptstyle 1}}^{}; q) \\
\phi(z_{\scriptscriptstyle 1}^{}, z_{\scriptscriptstyle 2}^{},
z_{\scriptscriptstyle 3}^{}, z_{r + {\scriptscriptstyle 1}}^{}; q) & 
- \frac{2 z_{r + {\scalebox{.85}{$\scriptscriptstyle 1$}}}^{}}{q^{\scalebox{.95}{$\scriptscriptstyle 3 \slash 2$}}}\!\Big(\!\frac{q - 1}
{1 - q^{\scalebox{.9}{$\scriptscriptstyle 2$}} 
z_{r + {\scalebox{.85}{$\scriptscriptstyle 1$}}}^{\scalebox{.9}{$\scriptscriptstyle 2$}}}\!\Big) & 
- \, \phi(- z_{\scriptscriptstyle 1}^{}, - z_{\scriptscriptstyle 2}^{},
- z_{\scriptscriptstyle 3}^{}, z_{r +{\scriptscriptstyle 1}}^{}; q)\\
-\, \psi(- z_{\scriptscriptstyle 1}^{}, - z_{\scriptscriptstyle 2}^{},
- z_{\scriptscriptstyle 3}^{}, - z_{r + {\scriptscriptstyle 1}}^{}; q) & 
- \frac{2}{q^{\scalebox{.95}{$\scriptscriptstyle 3 \slash 2$}}}
\!\left(\! \frac{1 - q z_{r + {\scalebox{.85}{$\scriptscriptstyle 1$}}}^{\scalebox{.9}{$\scriptscriptstyle 2$}}}
{1 - q^{\scalebox{.9}{$\scriptscriptstyle 2$}} 
z_{r + {\scalebox{.85}{$\scriptscriptstyle 1$}}}^{\scalebox{.9}{$\scriptscriptstyle 2$}}} \!\right) & 
-\, \psi(- z_{\scriptscriptstyle 1}^{}, - z_{\scriptscriptstyle 2}^{},
- z_{\scriptscriptstyle 3}^{}, z_{r + {\scriptscriptstyle 1}}^{}; q)  \\
\end{pmatrix}.
\end{equation} 
Here $\phi$ and $\psi$ are given by \eqref{eq: Functions-phi-psi}. The formula \eqref{eq: function-Gammaw} now yields: \begin{equation*}
\begin{split}
2^{-l(w_{\scalebox{.85}{$\scriptscriptstyle \alpha$}})}\Gamma_{\!\! w_{\scalebox{.85}{$\scriptscriptstyle \alpha$}}}\!(a_{\scriptscriptstyle 2}, a; \zeta_{a} \!) 
= \, &\frac{1}{2}\left(\prod_{j \, \in \, J} \frac{1 - z_{\!\scriptscriptstyle j}^{\scriptscriptstyle -1}}{1 - qz_{\!\scriptscriptstyle j}},\, 
\mathrm{sgn}(a_{\scriptscriptstyle 2})q^{\scriptscriptstyle -|J| \slash 2} 
\cdot \prod_{j \, \in \, J} \frac{1}{z_{\!\scriptscriptstyle  j}},\, 0\right)\\
&\cdot \left.\overline{\mathrm{M}}_{w_{\!\scalebox{.85}{$\scriptscriptstyle \alpha'$}}}\!(\mathrm{z}; q)\right\vert_{ 
z_{\scalebox{.95}{$\scriptscriptstyle r$} + {\scalebox{.82}{$\scriptscriptstyle 1$}}} = 
\frac{1}
{\scalebox{1.3}{$\scriptscriptstyle q$}^{\scalebox{.82}{$\scriptscriptstyle 3 \slash 2$}} 
\zeta_{a} (\scalebox{1.3}{$\scriptscriptstyle z$}_{\scalebox{.75}{$\scriptscriptstyle 1$}}^{} 
\scalebox{1.3}{$\scriptscriptstyle z$}_{\scalebox{.75}{$\scriptscriptstyle 2$}}^{} 
\scalebox{1.3}{$\scriptscriptstyle z$}_{\scalebox{.75}{$\scriptscriptstyle 3$}}^{}\!)
^{\scalebox{.82}{$\scriptscriptstyle 1 \slash 2$}}}} 
\cdot \;\, ^{\scalebox{1.2}{$\scriptscriptstyle t$}}\!(1,  \mathrm{sgn}(a),  1).
\end{split}
\end{equation*}

{\bf The computation of $Q_{\scriptscriptstyle 2}(D, q).$} Putting everything together, we can now compute the contribution
\[
\frak{S}_{\scriptscriptstyle 2}(\mathrm{z}, a_{\scriptscriptstyle 2} \!)  
=  \tfrac{1}{2}\, 
\cdot\sum_{\alpha \, \in \, \Phi_{\scalebox{.75}{$\scriptscriptstyle 2$}}}\;
\sum_{a \, \in \, \{1, \, \theta_{\scalebox{.75}{$\scriptscriptstyle 0$}} \}}\;
\sum_{\zeta_{a}^{\scalebox{.85}{$\scriptscriptstyle 2$}} \, = \,\mathrm{sgn}(a)}
\frac{\Gamma_{\!\! w_{\scalebox{.85}{$\scriptscriptstyle
        \alpha$}}}\!(a_{\scriptscriptstyle 2}, a; \zeta_{a} \!)}
{2^{l(w_{\scalebox{.85}{$\scriptscriptstyle \alpha$}})}}
\,\big(1 \, - \, \zeta_{a}q^{\scriptscriptstyle (d(\alpha) + 1)\slash 4}
\mathrm{z}^{\scriptscriptstyle \alpha\slash 2}\!\big)^{\scriptscriptstyle -1}
S_{\scriptscriptstyle \alpha}(\underline{z}, \chi_{a})
\]
of the principal parts at the poles corresponding to the roots in
$\Phi_{\scriptscriptstyle 2}.$ Here
$S_{\scriptscriptstyle \alpha}(\underline{z}, \chi_{a} \!)$ is defined
as follows. \!Express $\alpha$ as\linebreak
\[
  \alpha \;\; = \sum_{\substack{j = 1 \\ j \neq
      j_{\scalebox{.71}{$\scriptscriptstyle 1$}}, \,
      j_{\scalebox{.71}{$\scriptscriptstyle 2$}}, \,
      j_{\scalebox{.71}{$\scriptscriptstyle 3$}}}}^{r}
  k_{\!\scriptscriptstyle j}
  \alpha_{\!\scriptscriptstyle j} +
  \alpha_{\!\scriptscriptstyle j_{\scalebox{.62}{$\scriptscriptstyle 1$}}} \! +
  \alpha_{\!\scriptscriptstyle j_{\scalebox{.62}{$\scriptscriptstyle 2$}}} \! +
  \alpha_{\!\scriptscriptstyle j_{\scalebox{.62}{$\scriptscriptstyle 3$}}}
 \! + 2\alpha_{r+ {\scriptscriptstyle 1}}
\]
for some
$
1 \le j_{\scriptscriptstyle 1} < j_{\scriptscriptstyle 2} <
j_{\scriptscriptstyle 3} \le r,
$
and with $k_{\!\scriptscriptstyle j} \in \{0, 2\}$ for
$
j \ne \! j_{\scriptscriptstyle 1}, j_{\scriptscriptstyle 2},
j_{\scriptscriptstyle 3}.$ 
Write
$
\{1, \ldots, r \} \setminus \{j_{\scriptscriptstyle 1}, j_{\scriptscriptstyle 2},
j_{\scriptscriptstyle 3} \} = \{j_{\scriptscriptstyle 4},
j_{\scriptscriptstyle 5}, \cdots \},
$
and, for $i \le r,$ define
$
\delta_{\!\scriptscriptstyle j_{\scalebox{.62}{$\scriptscriptstyle i$}}}
= - 1
$
or $1$ according as
$
k_{\!\scriptscriptstyle j_{\scalebox{.62}{$\scriptscriptstyle i$}}}\! = 2
$
or not. If we set, as before,
$
z_{\scriptscriptstyle k} \! = q^{\scriptscriptstyle -1\slash 2}
\xi_{\scriptscriptstyle k}^{\scriptscriptstyle 2}
$
for $k\le r,$ then, by Proposition \ref{Regularization-Justification}, 
\[
  S_{\scriptscriptstyle \alpha}(\underline{z}, \chi_{a})
  : = N\Big(\xi_{\! j_{\scalebox{.64}{$\scriptscriptstyle 1$}}}^{}\!,
  \xi_{\! j_{\scalebox{.64}{$\scriptscriptstyle 2$}}}^{}\!,
  \xi_{\! j_{\scalebox{.64}{$\scriptscriptstyle 3$}}}^{}\!,
  \xi_{\! j_{\scalebox{.64}{$\scriptscriptstyle
        4$}}}^{\scriptscriptstyle \delta_{\!\scriptscriptstyle j_{\scalebox{.64}{$\scriptscriptstyle 4$}}}}\!, \ldots,
  \xi_{\! j_{\scalebox{.80}{$\scriptscriptstyle
        r$}}}^{\scriptscriptstyle \delta_{\!\scriptscriptstyle
      j_{\scalebox{.80}{$\scriptscriptstyle r$}}}}\!,\chi_{a}\! \Big)
  \prod_{p} S_{\! p}^{\mathrm{reg}}\Big(\xi_{\! j_{\scalebox{.64}{$\scriptscriptstyle 1$}}}^{}\!,
  \xi_{\! j_{\scalebox{.64}{$\scriptscriptstyle 2$}}}^{}\!,
  \xi_{\! j_{\scalebox{.64}{$\scriptscriptstyle 3$}}}^{}\!,
  \xi_{\! j_{\scalebox{.64}{$\scriptscriptstyle
        4$}}}^{\scriptscriptstyle \delta_{\!\scriptscriptstyle j_{\scalebox{.64}{$\scriptscriptstyle 4$}}}}\!, \ldots,
  \xi_{\! j_{\scalebox{.80}{$\scriptscriptstyle
        r$}}}^{\scriptscriptstyle \delta_{\!\scriptscriptstyle
      j_{\scalebox{.80}{$\scriptscriptstyle r$}}}}\!, \chi_{a}\! \Big)
\] 
where $S_{\! p}^{\scriptscriptstyle \mathrm{reg}}\!({\bf \xi}, \chi_{a})$ is given explicitly by \eqref{eq: Sp-regularized}, 
\eqref{eq: Sp_w-alpha}, and \eqref{eq: Local-factor-residue-r=3}, \!with $\mathrm{sgn}(a) \mapsto
\chi_{a}(p),$ $\xi_{\scriptscriptstyle k} \mapsto \xi_{\scriptscriptstyle k}^{\deg\,p}\!\!,$ and
$q \mapsto |p|.$ Notice that, due to the singularities of 
\[
\prod_{i = 1}^{3}\prod_{j = 4}^{r}
\big(1 - 
\xi_{\scriptscriptstyle i}^{\scriptscriptstyle 2}
\xi_{\!\scriptscriptstyle j}^{\scriptscriptstyle 2}\big)^{\scriptscriptstyle -1}
\big(1- 
\xi_{\scriptscriptstyle i}^{\scriptscriptstyle -2}
\xi_{\!\scriptscriptstyle j}^{\scriptscriptstyle 2}\big)^{\scriptscriptstyle -1}
\, \cdot \prod_{4 \, \le \, k \, \le \, l \, \le \, r} \big(1 -
\xi_{\scriptscriptstyle k}^{\scriptscriptstyle 2}
\xi_{\scriptscriptstyle l}^{\scriptscriptstyle 2}\big)^{\scriptscriptstyle -1}
  \]
  in $N({\bf \xi}, \chi_{a}),$ we cannot take
$
\xi_{i} = \xi_{\! j}
$
right away, for any $1 \le i \ne \! j \le r.$ To see that $\frak{S}_{\scriptscriptstyle 2}(\mathrm{z}, a_{\scriptscriptstyle 2} \!)$ is indeed defined when $z_{\scriptscriptstyle k} = q^{\scriptscriptstyle -1\slash 2}\!,$\linebreak 
for all $k \le r,$ we first note that 
$
\prod_{p} S_{\! p}^{\scriptscriptstyle \mathrm{reg}}\!({\bf \xi}, \chi_{a})	
$ 
is symmetric in the variables
$\xi_{\scriptscriptstyle 1}^{\scriptscriptstyle \pm 1}\!,\,
\xi_{\scriptscriptstyle 2}^{\scriptscriptstyle \pm 1}\!,\,
\xi_{\scriptscriptstyle 3}^{\scriptscriptstyle \pm 1}\!,
$ 
and (separately) in $\xi_{\scriptscriptstyle k},$ for $4\le k \le r.$ 
This follows at once from \eqref{eq: Sp-regularized}, \eqref{eq: Sp_w-alpha}, and Remark \ref{Crucial-Remark}. Moreover, 
the functions $\mathscr{G}_{\scriptscriptstyle 1}({\bf \xi}, \zeta_{a})$ and $\mathscr{G}_{\scriptscriptstyle 2}({\bf \xi}, \zeta_{a})$ defined by 
\begin{equation*}
\begin{split} 
\mathscr{G}_{\scriptscriptstyle 1}({\bf \xi}, \zeta_{a}) : = \, &\frac{\prod_{j = 1}^{r}\big(1 - q^{\scriptscriptstyle 1\slash 2}\xi_{\!\scriptscriptstyle j}^{\scriptscriptstyle 2}\big)}
{\left(1 - \mathrm{sgn}(a)\sqrt{q}
	\,\xi_{\scriptscriptstyle 1}^{\scriptscriptstyle 2}
	\xi_{\scriptscriptstyle 2}^{\scriptscriptstyle 2}
	\xi_{\scriptscriptstyle 3}^{\scriptscriptstyle 2}
	\right)\prod_{1 \le i \le j \le  3}
	\left(1 - \frac{\mathrm{sgn}(a)\sqrt{q}
		\,\xi_{\scriptscriptstyle 1}^{\scriptscriptstyle 2}
		\xi_{\scriptscriptstyle 2}^{\scriptscriptstyle 2}
		\xi_{\scriptscriptstyle 3}^{\scriptscriptstyle 2}}
	{\xi_{\scriptscriptstyle i}^{\scriptscriptstyle 2}
		\xi_{\!\scriptscriptstyle j}^{\scriptscriptstyle
			2}}\right)}\\
&\cdot  (1, 0, 0)\,\overline{\mathrm{M}}_{w_{\!\scalebox{.85}{$\scriptscriptstyle \alpha'$}}}
\bigg(q^{\,\scriptscriptstyle -1\slash 2}
\xi_{\scriptscriptstyle 1}^{\scriptscriptstyle 2},
q^{\,\scriptscriptstyle -1\slash 2}
\xi_{\scriptscriptstyle 2}^{\scriptscriptstyle 2},
q^{\,\scriptscriptstyle -1\slash 2}
\xi_{\scriptscriptstyle 3}^{\scriptscriptstyle 2},
\frac{1}{q^{\scriptscriptstyle 3\slash 4}\zeta_{a}
	\xi_{\scriptscriptstyle 1}
	\xi_{\scriptscriptstyle 2}
	\xi_{\scriptscriptstyle 3}}; q\bigg)
\; ^{\scalebox{1.2}{$\scriptscriptstyle t$}}\!(1,  \mathrm{sgn}(a),  1)
\end{split}
\end{equation*} 
and 
\begin{equation*}
\begin{split} 
\mathscr{G}_{\scriptscriptstyle 2}({\bf \xi}, \zeta_{a}) : = \, & \frac{\xi_{\scriptscriptstyle 1} \cdots \, \xi_{r}}
{\left(1 - \mathrm{sgn}(a)\sqrt{q}
	\,\xi_{\scriptscriptstyle 1}^{\scriptscriptstyle 2}
	\xi_{\scriptscriptstyle 2}^{\scriptscriptstyle 2}
	\xi_{\scriptscriptstyle 3}^{\scriptscriptstyle 2}
	\right)\prod_{1 \le i \le j \le 3}
	\left(1 - \frac{\mathrm{sgn}(a)\sqrt{q}
		\,\xi_{\scriptscriptstyle 1}^{\scriptscriptstyle 2}
		\xi_{\scriptscriptstyle 2}^{\scriptscriptstyle 2}
		\xi_{\scriptscriptstyle 3}^{\scriptscriptstyle 2}}
	{\xi_{\scriptscriptstyle i}^{\scriptscriptstyle 2}
		\xi_{\!\scriptscriptstyle j}^{\scriptscriptstyle 2}}\right)}\\
&\cdot (0, 1, 0)\, 
\overline{\mathrm{M}}_{w_{\!\scalebox{.85}{$\scriptscriptstyle \alpha'$}}}
\bigg(q^{\,\scriptscriptstyle -1\slash 2}
\xi_{\scriptscriptstyle 1}^{\scriptscriptstyle 2},
q^{\,\scriptscriptstyle -1\slash 2}
\xi_{\scriptscriptstyle 2}^{\scriptscriptstyle 2},
q^{\,\scriptscriptstyle -1\slash 2}
\xi_{\scriptscriptstyle 3}^{\scriptscriptstyle 2},
\frac{1}{q^{\scriptscriptstyle 3\slash 4}\zeta_{a}
	\xi_{\scriptscriptstyle 1}
	\xi_{\scriptscriptstyle 2}
	\xi_{\scriptscriptstyle 3}}; q\bigg) 
\; ^{\scalebox{1.2}{$\scriptscriptstyle t$}}\!(1,  \mathrm{sgn}(a),  1)
\end{split}
\end{equation*} 
($
w_{\!{\scriptscriptstyle \alpha'}}
= w_{\!\scriptscriptstyle 1}^{} \!w_{\!\scriptscriptstyle 2}^{} w_{\!\scriptscriptstyle 3}^{} w_{\!r + {\scriptscriptstyle 1}}^{}
$) 
satisfy the same symmetries. \!If we set 
\[
\mathscr{K}_{\scriptscriptstyle 1, 2}(\mathbf{x}, u, \zeta_{a}) \, : =\,  
\frac{\mathscr{G}_{\scriptscriptstyle 1, 2}\!\left({\mathbf x}, \zeta_{a}\right)
\prod_{p} S_{\! p}^{\scriptscriptstyle \mathrm{reg}}\!\left({\mathbf x}, \chi_{a}\right)}
{(1  -  q^{\scriptscriptstyle 3\slash 4} 
	\zeta_{a}\,
	x_{\scriptscriptstyle 4}^{\, \scalebox{.95}{$\scriptscriptstyle -1$}} 
	\cdots\,  x_{r}^{\, \scalebox{.95}{$\scriptscriptstyle -1$}}u)
\prod_{i = 1}^{3}\prod_{\! j = 4}^{r}
(1 - 
x_{\scriptscriptstyle i}^{\scriptscriptstyle 2}
x_{\!\scriptscriptstyle j}^{\scriptscriptstyle 2})
(1- 
x_{\scriptscriptstyle i}^{\scriptscriptstyle -2}
x_{\!\scriptscriptstyle j}^{\scriptscriptstyle 2})
\, \cdot \prod_{4 \le k \le l \le r} (1 -
x_{\scriptscriptstyle k}^{\scriptscriptstyle 2}
x_{\scriptscriptstyle l}^{\scriptscriptstyle 2})}
\] 
and define $\mathscr{L}_{\scriptscriptstyle 1, 2}({\bf \xi}, z_{r + {\scriptscriptstyle 1}}, \zeta_{a})$ by 
\[
\mathscr{L}_{\scriptscriptstyle 1, 2}({\bf \xi}, z_{r + {\scriptscriptstyle 1}}, \zeta_{a}) \, : = \,  
\frac{1}{2^{3} 3! (r - 3)!}
\sum_{\sigma \, \in \, \mathbb{S}_{r}} \;\, \sum_{\delta_{\scalebox{.75}{$\scriptscriptstyle \sigma(i)$}}  \, = \, \pm 1} 
\!\mathscr{K}_{\scriptscriptstyle 1, 2}\!\left(\xi_{\scalebox{.95}{$\scriptscriptstyle \sigma(1)$}}^{\delta_{\scalebox{.75}{$\scriptscriptstyle \sigma(1)$}}}\!, 
\ldots, \xi_{\scalebox{.95}{$\scriptscriptstyle \sigma(r)$}}^{\delta_{\scalebox{.75}{$\scriptscriptstyle \sigma(r)$}}}, 
z_{r + {\scriptscriptstyle 1}} \cdot \prod_{j = 1}^{r} \xi_{\!\scriptscriptstyle j}^{}, \zeta_{a}\!\right)
\] 
where the first sum is over all elements of the symmetric group $\mathbb{S}_{r}$ on $r$ letters, then we have the following:

\vskip10pt
\begin{thm} --- For 
	$
	\xi_{\scalebox{.97}{$\scriptscriptstyle 1$}}, \ldots, \xi_{r}
	$ 
	distinct elements on the unit circle such that 
	$
	\xi_{\scriptscriptstyle i}\xi_{\scriptscriptstyle j} \ne 1$ 
	and 
	$
	|\xi_{\scriptscriptstyle i} - 1| < \rho,
	$ 
	for all $1\le i, j \le r$ and a sufficiently small positive $\rho,$ we have 
	\[
\frak{S}_{\scriptscriptstyle 2}
\!\left(q^{\scriptscriptstyle -1\slash 2}{\bf \xi}^{\scriptscriptstyle 2}\!, z_{r + {\scriptscriptstyle 1}}, a_{\scriptscriptstyle 2}\right)  
	=  \tfrac{1}{4}\!\sum_{\zeta_{a}^{\scalebox{.85}{$\scriptscriptstyle 2$}}  \, = \, \pm 1}
	\left\{\mathscr{L}_{\scriptscriptstyle 1}({\bf \xi}, z_{r + {\scriptscriptstyle 1}}, \zeta_{a})
	\prod_{j = 1}^{r}\big(1 - q^{\scriptscriptstyle 1\slash 2}\xi_{\!\scriptscriptstyle j}^{\scriptscriptstyle 2}\big)^{\scriptscriptstyle -1}
	+ \, \frac{\mathrm{sgn}(a_{\scriptscriptstyle 2})}{\xi_{\scriptscriptstyle 1} \cdots \, \xi_{r}}
	\!\mathscr{L}_{\scriptscriptstyle 2}({\bf \xi}, z_{r + {\scriptscriptstyle 1}}, \zeta_{a}) \right\}
\] 
where 
$
q^{\scriptscriptstyle -1\slash 2}{\bf \xi}^{\scriptscriptstyle 2} 
: = \left(q^{\scriptscriptstyle -1\slash 2}\xi_{\scriptscriptstyle 1}^{\scriptscriptstyle 2}, \ldots, 
q^{\scriptscriptstyle -1\slash 2}\xi_{r}^{\scriptscriptstyle 2}\right).
$ 
Moreover, for complex $u$ with $|u| < q^{- 3\slash 4}(1 - \rho)^{r - 3}\!,$ we have 
\begin{equation} \label{eq: IR}
	\begin{split}
	& \sum_{\sigma \, \in \, \mathbb{S}_{r}} \;\, \sum_{\delta_{\scalebox{.75}{$\scriptscriptstyle \sigma(i)$}}  \, = \, \pm 1} 
	\!\mathscr{K}_{\scriptscriptstyle 1, 2}\!\left(\xi_{\scalebox{.95}{$\scriptscriptstyle \sigma(1)$}}^{\delta_{\scalebox{.75}{$\scriptscriptstyle \sigma(1)$}}}\!, 
	\ldots, \xi_{\scalebox{.95}{$\scriptscriptstyle \sigma(r)$}}^{\delta_{\scalebox{.75}{$\scriptscriptstyle \sigma(r)$}}}\!, u, \zeta_{a}\!\right)\\
	& = \frac{(-1)^{r(r  +  {\scriptscriptstyle 1})\slash {\scriptscriptstyle 2}}}{(2\pi \sqrt{-1})^{^{r}}}\!
	\oint\limits_{\scalebox{.85}{$\scriptscriptstyle |z_{\scalebox{.55}{$\scriptscriptstyle 1$}} - 1| 
			\,= \, $} \scalebox{.92}{$\scriptscriptstyle \rho$}}  \hskip-1.7pt \cdots 
	\hskip-1.7pt\oint\limits_{\scalebox{.85}{$\scriptscriptstyle |z_{\scalebox{.7}{$\scriptscriptstyle r$}} - 1| \, = \, $} 
		\scalebox{.92}{$\scriptscriptstyle \rho$}} 
	\mathscr{G}_{\scriptscriptstyle 1, 2}\!\left(\mathrm{z}, \zeta_{a}\right)
	(1  -  q^{\scriptscriptstyle 3\slash 4} 
	\zeta_{a}\,
	z_{\scriptscriptstyle 4}^{\, \scalebox{.95}{$\scriptscriptstyle -1$}} 
	\cdots\,  z_{r}^{\, \scalebox{.95}{$\scriptscriptstyle -1$}}u)^{\scriptscriptstyle -1}
	\!\prod_{p} S_{\! p}^{\scriptscriptstyle \mathrm{reg}}\!\left(\mathrm{z}, \chi_{a}\right)\\ 
	&\hskip102pt \cdot \, \frac{\prod_{1 \, \le \, i  \, < \, j \, \le \, r} (z_{\scriptscriptstyle i}^{} - z_{\!\scriptscriptstyle j}^{})^{\mathrm{\bf{e}}_{\scalebox{.65}{$\scriptscriptstyle ij$}}}
		(1 - z_{\scriptscriptstyle i}^{} z_{\!\scriptscriptstyle j}^{}) \, \cdot \, \prod_{1 \, \le \, k  \, \le \, l \, \le \, 3}
		(1 - z_{\scriptscriptstyle k}^{} z_{\scriptscriptstyle l}^{})} 
	{\prod_{i,  j = 1}^{r} \left(1 - z_{\scriptscriptstyle i}^{}\, \xi_{\scriptscriptstyle j}^{}\right)
		\left(1 - z_{\scriptscriptstyle i}^{}\, \xi_{\scriptscriptstyle j}^{\scriptscriptstyle -1}\right) 
		\, \cdot\, \prod_{k = 1}^{3}\prod_{l = 4}^{r} (1 + z_{\scriptscriptstyle k}^{} z_{\scriptscriptstyle l}^{})   
		(z_{\scriptscriptstyle k}^{\scriptscriptstyle -1} \! + z_{\scriptscriptstyle k}^{\scriptscriptstyle -2} z_{\scriptscriptstyle l}^{})   
		\prod_{4 \, \le \, k  \, \le \, l \, \le \, r} (1 + z_{\scriptscriptstyle k}^{} z_{\scriptscriptstyle l}^{})} 
	\frac{d \mathrm{z}}{\mathrm{z}^{\scriptscriptstyle r}}
	\end{split}
	\end{equation} 
	where $\mathrm{\bf{e}}_{\scalebox{.85}{$\scriptscriptstyle ij$}} = 1$ or $2$ according as $i\le 3$ and $j\ge 4$ or not, and 
	$\mathrm{z}^{r} : = z_{\scalebox{.98}{$\scriptscriptstyle 1$}}^{r} \cdots\,  z_{r}^{r}.$
\end{thm}

\begin{proof} By Proposition \ref{Regularization-Justification}, the function 
	$
	\prod_{p} S_{\! p}^{\scriptscriptstyle \mathrm{reg}}\!(\mathrm{z}, \chi_{a})	
	=  N(\mathrm{z}, \chi_{a})^{\scriptscriptstyle -1}
	\!\cdot N(\mathrm{z}, \chi_{a})\prod_{p}S_{\! p}^{\scriptscriptstyle \mathrm{reg}}\!(\mathrm{z}, \chi_{a})	
	$ 
	is holomorphic on a small polydisk 
	$
	|z_{\scriptscriptstyle i} - 1| \le \rho
	$ 
	($i = 1, \ldots, r$), and by \eqref{eq: Global-cocycle-formula}, \eqref{eq: Functions-phi-psi}, the same is true for 
	$ 
	\mathscr{G}_{\scriptscriptstyle 1, 2}\!\left(\mathrm{z}, \zeta_{a}\right)
	(1  -  q^{\scriptscriptstyle 3\slash 4} 
	\zeta_{a}\,
	z_{\scriptscriptstyle 4}^{\, \scalebox{.95}{$\scriptscriptstyle -1$}} 
	\cdots\,  z_{r}^{\, \scalebox{.95}{$\scriptscriptstyle -1$}}u)^{\scriptscriptstyle -1}\!.
	$ 
	As all these functions are symmetric in the variables 
	$
	z_{\scriptscriptstyle 1}^{\, \scriptscriptstyle \pm 1}\!,\, 
	z_{\scriptscriptstyle 2}^{\, \scriptscriptstyle \pm 1}\!,\,
	z_{\scriptscriptstyle 3}^{\, \scriptscriptstyle \pm 1}\!,
	$ 
	and (separately) symmetric in $z_{\scriptscriptstyle k}$ ($k \ge 4$), our first assertion can be verified by a simple counting argument. To establish the integral formula, \!we apply the following slight generalization of Lemma \ref{average-to-integral-vers1}:
	 
\vskip5pt
\begin{lem} --- Let 
	$
	a_{	\scalebox{.97}{$\scriptscriptstyle 1$}}, \ldots, a_{r}
	$ 
	be distinct elements on the unit circle such that $a_{\scriptscriptstyle i}a_{\! \scriptscriptstyle j} \ne 1$ and $|a_{\scriptscriptstyle i} - 1| < \varepsilon$, for all $1\le i, j \le r$ and a small positive $\varepsilon.$ Suppose $h$ is a function of $r$ complex variables, holomorphic on the polydisk $|z_{\scriptscriptstyle i} - 1| \le \varepsilon,$ $i = 1, \ldots, r,$ and, for $0 \le m < r,$ define 
	$K_{m}(\mathrm{z})$ by 
	\[
	K_{m}(\mathrm{z}) = \frac{h(\mathrm{z})}
	{\prod_{k = 1}^{m}\prod_{l = m + 1}^{r} (1 - z_{\scriptscriptstyle k}^{\scriptscriptstyle 2} 
		z_{\scriptscriptstyle l}^{\scriptscriptstyle 2})   
		(1 - z_{\scriptscriptstyle k}^{\scriptscriptstyle -2} 
		z_{\scriptscriptstyle l}^{\scriptscriptstyle 2})   
		\prod_{m + 1 \, \le \, k  \, \le \, l \, \le \, r} 
		(1 - z_{\scriptscriptstyle k}^{\scriptscriptstyle 2} 
		z_{\scriptscriptstyle l}^{\scriptscriptstyle 2})}.
	\] 
	Then we have 
	\begin{equation*}
	\begin{split}
	&\frac{(-1)^{r(r  +  {\scriptscriptstyle 1})\slash {\scriptscriptstyle 2}}}{(2\pi \sqrt{-1})^{^{r}}}\!
	\oint\limits_{\scalebox{.85}{$\scriptscriptstyle |z_{\scalebox{.55}{$\scriptscriptstyle 1$}} - 1| 
			= \varepsilon$}}  \hskip-1.5pt\cdots 
	\hskip-1.5pt\oint\limits_{\scalebox{.85}{$\scriptscriptstyle |z_{\scalebox{.7}{$\scriptscriptstyle r$}} - 1| = \varepsilon$}} 
	h(\mathrm{z}) 
	\frac{\prod_{1 \, \le \, i  \, < \, j \, \le \, r} (z_{\scriptscriptstyle i}^{} - z_{\!\scriptscriptstyle j}^{})^{\mathrm{\bf{e}}_{\scalebox{.65}{$\scriptscriptstyle ij$}}}
		(1 - z_{\scriptscriptstyle i}^{} z_{\!\scriptscriptstyle j}^{}) \, \cdot \, \prod_{1 \, \le \, k  \, \le \, l \, \le \, m}
		(1 - z_{\scriptscriptstyle k}^{} z_{\scriptscriptstyle l}^{})} 
	{\prod_{i,  j = 1}^{r} (1 - z_{\scriptscriptstyle i}^{}a_{\!\scriptscriptstyle j}^{})
		(1 - z_{\scriptscriptstyle i}^{}a_{\!\scriptscriptstyle j}^{\scriptscriptstyle -1}) 
		\, \cdot\, \prod_{k = 1}^{m}\prod_{l = m + 1}^{r} (1 + z_{\scriptscriptstyle k}^{} z_{\scriptscriptstyle l}^{})   
		(z_{\scriptscriptstyle k}^{\scriptscriptstyle -1} \! + z_{\scriptscriptstyle k}^{\scriptscriptstyle -2} z_{\scriptscriptstyle l}^{})   
		\prod_{m + 1 \, \le \, k  \, \le \, l \, \le \, r} (1 + z_{\scriptscriptstyle k}^{} z_{\scriptscriptstyle l}^{})} 
	\frac{d \mathrm{z}}{\mathrm{z}^{\scriptscriptstyle r}}\\
	&  =  \sum_{\sigma \, \in \, \mathbb{S}_{r}} \; \sum_{\delta_{\scalebox{.75}{$\scriptscriptstyle \sigma(i)$}}  \, = \, \pm 1} 
	K_{m}\Big(a_{\scalebox{.95}{$\scriptscriptstyle \sigma(1)$}}^{\delta_{\scalebox{.75}{$\scriptscriptstyle \sigma(1)$}}}\!, \ldots, 
	a_{\scalebox{.95}{$\scriptscriptstyle \sigma(r)$}}^{\delta_{\scalebox{.75}{$\scriptscriptstyle \sigma(r)$}}} \!\Big) 
	\end{split}
	\end{equation*} 
	where $\mathrm{\bf{e}}_{\scalebox{.85}{$\scriptscriptstyle ij$}} = 1$ or $2$ according as $i\le m$ and $j\ge m+1$ or not, and 
	$\mathrm{z}^{r} : = z_{\scalebox{.98}{$\scriptscriptstyle 1$}}^{r} \cdots\,  z_{r}^{r}.$
\end{lem}

\begin{proof} \!We proceed as in the proof of Lemma \ref{average-to-integral-vers1}. \!The poles of the integrand (inside the contour) occur at 
	$
	z_{\scriptscriptstyle i}^{} \! = a_{\scalebox{.95}{$\scriptscriptstyle \sigma(i)$}}^{\delta_{\scalebox{.75}{$\scriptscriptstyle \sigma(i)$}}}\!,
	$ 
	for $\sigma \in \mathbb{S}_{r},$ and $\delta_{\scalebox{1.0}{$\scriptscriptstyle \sigma(i)$}} \! \in \{ \pm 1\}$ for $1 \le i \le r.$ Writing 
	\[
	\frac{\prod_{1 \, \le \, i  \, < \, j \, \le \, r} (z_{\scriptscriptstyle i}^{} - z_{\!\scriptscriptstyle j}^{})^{\mathrm{\bf{e}}_{\scalebox{.65}{$\scriptscriptstyle ij$}}}} 
	{\prod_{k = 1}^{m}\prod_{l = m+1}^{r} (z_{\scriptscriptstyle k}^{\scriptscriptstyle -1} \! + z_{\scriptscriptstyle k}^{\scriptscriptstyle -2} z_{\scriptscriptstyle l}^{})} = \frac{\prod_{1 \, \le \, i  \, < \, j \, \le \, r}(z_{\!\scriptscriptstyle j}^{} - z_{\scriptscriptstyle i}^{})^{\scriptscriptstyle 2}} 
	{\prod_{k = 1}^{m}\prod_{l = m+1}^{r} (1 - z_{\scriptscriptstyle k}^{\scriptscriptstyle -2} z_{\scriptscriptstyle l}^{\scriptscriptstyle 2})} 
	\] 
	the identity stated in the lemma follows easily from the fact that 
	\[
	\frac{(-1)^{\scalebox{1.05}{$\scriptscriptstyle r(r  -  1)\slash 2$}}
		\prod_{1 \, \le \, i  \, < \, j \, \le \, r}(z_{\!\scriptscriptstyle j}^{} - z_{\scriptscriptstyle i}^{})^{\scriptscriptstyle 2}
		(1 - z_{\scriptscriptstyle i}^{} z_{\!\scriptscriptstyle j}^{})} 
	{\prod_{i = 1}^{r} \prod_{\! j \ne i} (1 - z_{\scriptscriptstyle i}^{}z_{\!\scriptscriptstyle j}^{})
		(1 - z_{\scriptscriptstyle i}^{}z_{\!\scriptscriptstyle j}^{\scriptscriptstyle -1}) 
		\, \cdot \, \prod_{i = 1}^{r}  z_{\scriptscriptstyle i}^{r - {\scriptscriptstyle 1}} (1 - z_{\scriptscriptstyle i}^{\scriptscriptstyle 2})}
	\;\; =  \prod_{1 \le k \le l \le r} (1 -  z_{\scriptscriptstyle k}^{} z_{\scriptscriptstyle l}^{} )^{\scriptscriptstyle -1}
	\] 
	with 
	$
	z_{\scriptscriptstyle i}^{} : = a_{\scalebox{.95}{$\scriptscriptstyle \sigma(i)$}}^{\delta_{\scalebox{.75}{$\scriptscriptstyle \sigma(i)$}}}
	$ 
	for $i = 1, \ldots, r.$ 
\end{proof}

The second assertion in our theorem follows at once from this lemma by taking $m = 3,$ 
\[
h(\mathrm{z}) = 
\frac{\mathscr{G}_{\scriptscriptstyle 1, 2}\!\left(\mathrm{z}, \zeta_{a}\right)
\!\prod_{p} S_{\! p}^{\scriptscriptstyle \mathrm{reg}}\!\left(\mathrm{z}, \chi_{a}\right)}
{1  -  q^{\scriptscriptstyle 3\slash 4} 
	\zeta_{a}\,
	z_{\scriptscriptstyle 4}^{\, \scalebox{.95}{$\scriptscriptstyle -1$}} 
	\cdots\,  z_{r}^{\, \scalebox{.95}{$\scriptscriptstyle -1$}}u}
\] 
and $a_{\scriptscriptstyle i} = \xi_{\scriptscriptstyle i}$ for all $i.$ This completes the proof of the theorem. 
\end{proof} 

One should note that the assumption on 
$
\xi_{\scriptscriptstyle 1}, \ldots, \xi_{r}
$ 
(or $a_{\scalebox{.97}{$\scriptscriptstyle 1$}}, \ldots, a_{\scriptscriptstyle r}$ in the lemma) being on the unit circle, 
made solely to simplify the exposition, can be removed by analytic continuation.

\vskip10pt
\begin{cor} --- For $D \in \mathbb{N}$ and $r \ge 4,$ 
	the coefficient $Q_{\scriptscriptstyle 2}(D, q)$ is given by 
	\[
Q_{\scriptscriptstyle 2}(D, q)  
=  \scalebox{1.45}{$\scriptscriptstyle \frac{1}{2^{\scalebox{.8}{$\scriptscriptstyle 5$}} 3! (r - 3)!}$}
\, \cdot \sum_{\zeta_{a}^{\scalebox{.85}{$\scriptscriptstyle 2$}}  \, = \, \pm 1}
\zeta_{a}^{\scriptscriptstyle D}
\left\{\left(1 - q^{\scriptscriptstyle 1\slash 2}\right)^{- r}
\!\cdot \tilde{\mathscr{L}}_{\scriptscriptstyle 1}(D, \zeta_{a})
\, + \, \tilde{\mathscr{L}}_{\scriptscriptstyle 2}(D, \zeta_{a}) \right\}
\] 
where $\tilde{\mathscr{L}}_{\scriptscriptstyle 1, 2}(D, \zeta_{a})$ has the integral representation 
\begin{equation*}
\begin{split}
\tilde{\mathscr{L}}_{\scriptscriptstyle 1, 2}(D, \zeta_{a})  = 
\frac{(-1)^{r(r  +  {\scriptscriptstyle 1})\slash {\scriptscriptstyle 2}}}{(2\pi \sqrt{-1})^{^{r}}}\!
\oint  & \cdots \oint 
\frac{\mathscr{G}_{\scriptscriptstyle 1, 2}\!\left(\mathrm{z}, \zeta_{a}\right)
\prod_{p} S_{\! p}^{\scriptscriptstyle \mathrm{reg}}\!\left(\mathrm{z}, \chi_{a}\right)}
{z_{\scriptscriptstyle 4}^{\scriptscriptstyle D} 
	\cdots\,  z_{r}^{\scriptscriptstyle D}}\\ 
& \hskip-1pt \cdot \, \frac{\prod_{1 \, \le \, i  \, < \, j \, \le \, r} (z_{\scriptscriptstyle i}^{} - z_{\!\scriptscriptstyle j}^{})^{\mathrm{\bf{e}}_{\scalebox{.65}{$\scriptscriptstyle ij$}}}
	(1 - z_{\scriptscriptstyle i}^{} z_{\!\scriptscriptstyle j}^{}) \, \cdot \, \prod_{1 \, \le \, k  \, \le \, l \, \le \, 3}
	(1 - z_{\scriptscriptstyle k}^{} z_{\scriptscriptstyle l}^{})} 
{\prod_{i = 1}^{r} \left(1 - z_{\scriptscriptstyle i}^{}\right)^{{\scriptscriptstyle 2} r} 
\cdot\, \prod_{k = 1}^{3}\prod_{l = 4}^{r} (1 + z_{\scriptscriptstyle k}^{} z_{\scriptscriptstyle l}^{})   
	(z_{\scriptscriptstyle k}^{\scriptscriptstyle -1} \! + z_{\scriptscriptstyle k}^{\scriptscriptstyle -2} z_{\scriptscriptstyle l}^{})   
	\prod_{4 \, \le \, k  \, \le \, l \, \le \, r} (1 + z_{\scriptscriptstyle k}^{} z_{\scriptscriptstyle l}^{})} 
\frac{d \mathrm{z}}{\mathrm{z}^{\scriptscriptstyle r}}
\end{split}
\end{equation*} 
the integrals being taken over sufficiently small circles $|z_{\scriptscriptstyle i}^{} - 1| = \rho,$ $i = 1, \ldots, r.$ 
\end{cor}

\begin{proof} Consider the integral 
	\begin{equation} \label{eq: Int-rep-Q2} 
\frac{1}{2 \pi \sqrt{-1}}\; \oint\limits_{|z| \, = \, q^{\, \scriptscriptstyle - 2}} 
\frak{S}_{\scriptscriptstyle 2}
\!\left(q^{\scriptscriptstyle -1\slash 2}{\bf \xi}^{\scriptscriptstyle 2}\!, z, 1\right) 
\frac{d z}{{z}^{\scriptscriptstyle D + 1}}
\end{equation}
where, for ease of notation, we set 
$z : = z_{r + {\scriptscriptstyle 1}},$ 
and 
$
q^{\scriptscriptstyle -1\slash 2}{\bf \xi}^{\scriptscriptstyle 2} 
: = \left(q^{\scriptscriptstyle -1\slash 2}\xi_{\scriptscriptstyle 1}^{\scriptscriptstyle 2}, \ldots, 
q^{\scriptscriptstyle -1\slash 2}\xi_{r}^{\scriptscriptstyle 2}\right).
$ 
Using the identity 
\[
\frak{S}_{\scriptscriptstyle 2}
\!\left(q^{\scriptscriptstyle -1\slash 2}{\bf \xi}^{\scriptscriptstyle 2}\!, z, 1\right)  
=  \tfrac{1}{4}\!\sum_{\zeta_{a}^{\scalebox{.85}{$\scriptscriptstyle 2$}}  \, = \, \pm 1}
\left\{\mathscr{L}_{\scriptscriptstyle 1}({\bf \xi}, z, \zeta_{a})
\prod_{j = 1}^{r}\big(1 - q^{\scriptscriptstyle 1\slash 2}\xi_{\!\scriptscriptstyle j}^{\scriptscriptstyle 2}\big)^{\scriptscriptstyle -1}
+ \, \frac{1}{\xi_{\scriptscriptstyle 1} \cdots \, \xi_{r}}
\!\mathscr{L}_{\scriptscriptstyle 2}({\bf \xi}, z, \zeta_{a}) \right\}
\] 
and the definition of the functions $\mathscr{L}_{\scriptscriptstyle 1, 2}({\bf \xi}, z, \zeta_{a}),$ we can express 
\eqref{eq: Int-rep-Q2} via $\eqref{eq: IR}$ as a sum of two multiple integrals invol-\\ving the rational function 
\[ 
\left(1  -  q^{\scriptscriptstyle 3\slash 4} 
\zeta_{a}\,
z_{\scriptscriptstyle 4}^{\, \scalebox{.95}{$\scriptscriptstyle -1$}} 
\cdots\,  z_{r}^{\, \scalebox{.95}{$\scriptscriptstyle -1$}}
\, \xi_{\scriptscriptstyle 1}^{} \cdots \; \xi_{r}^{} 
\, z\right)^{\scriptscriptstyle -1} 
\cdot \, 
\prod_{i,  j = 1}^{r} \left(1 - z_{\scriptscriptstyle i}^{}\, \xi_{\scriptscriptstyle j}^{}\right)^{\scriptscriptstyle -1}
\!\left(1 - z_{\scriptscriptstyle i}^{}\, \xi_{\scriptscriptstyle j}^{\scriptscriptstyle -1}\right)^{\scriptscriptstyle -1}\!;
\] 
this function exhibits the dependence upon either of $\xi_{\scriptscriptstyle 1}^{}, \ldots, \xi_{r}^{}$ or $z$ in the integrands. \!Taking $\rho$ even smaller if necessary so that $q^{\, \scriptscriptstyle - 2} < q^{\, \scriptscriptstyle - 3\slash 4}(1 - \rho)^{r - {\scriptscriptstyle 3}}\!,$ and setting $\xi_{\scriptscriptstyle 1}^{}  = \cdots = \xi_{r}^{} = 1,$ our conclusion follows at once by integrating with respect to $z,$\linebreak 
and applying the simple formula 
\[
\frac{1}{2 \pi \sqrt{-1}}\; \oint\limits_{|z| \, = \, q^{\scriptscriptstyle - 2}} 
\frac{d z}{z^{\scriptscriptstyle D + 1}(1 - c z)} = c^{\scriptscriptstyle D}
\] 
with 
$ 
c = q^{\scriptscriptstyle 3\slash 4} \zeta_{a}\, z_{\scriptscriptstyle 4}^{\scriptscriptstyle -1} 
\cdots\,  z_{r}^{\scriptscriptstyle -1}.
$ 
(Note that $|c z| < 1$ when $|z| \le q^{\, \scriptscriptstyle - 2}$).
\end{proof}

We have the following:

\vskip10pt
\begin{prop} --- The functions
$\tilde{\mathscr{L}}_{\scriptscriptstyle 1, 2}(D, \zeta_{a})$ are
polynomials in $D$ of degree $(r - 3) (r + 10)\slash 2$ with leading
terms given by
\[
3\cdot 
2^{{\scriptscriptstyle 25} - {\scriptscriptstyle 7}r}
(r - 3)!\,
\frac{0!\, 1!\, 2!\, \cdots \, (r - 4)!}
{7!\,  9!\,  11!\,  \cdots \, (2r - 1)!}
\mathscr{G}_{\scriptscriptstyle 1, 2}\!\left({\bf 1}, \zeta_{a}\right)
\prod_{p} S_{\! p}^{\scriptscriptstyle \mathrm{reg}}\!\left({\bf 1}, \chi_{a}\right)
\, \cdot \; D^{(r - {\scriptscriptstyle 3}) (r + {\scriptscriptstyle 10})\slash
  {\scriptscriptstyle 2}} 
\]
where we set ${\bf 1} : = (1, 1, \ldots, 1).$
\end{prop}

\begin{proof} \!The integral
$
\tilde{\mathscr{L}}_{\scriptscriptstyle 1, 2}(D, \zeta_{a})
$
is the coefficient of
$
\prod_{i = 1}^{r} \left(z_{\scriptscriptstyle i} - 1\right)^{{\scriptscriptstyle 2} r -
 {\scriptscriptstyle 1}}
 $ 
 in the power series expansion of
\[ 
\frac{(-1)^{r(r  +  {\scriptscriptstyle 1})\slash {\scriptscriptstyle 2}}\,
  \mathscr{G}_{\scriptscriptstyle 1, 2}\!\left(\mathrm{z}, \zeta_{a}\right)
  \prod_{p} S_{\! p}^{\scriptscriptstyle \mathrm{reg}}
  \!\left(\mathrm{z}, \chi_{a}\right)
\, \cdot \, \prod_{1 \, \le \, i  \, < \, j \, \le \, r} 
(z_{\scriptscriptstyle i}^{} - z_{\!\scriptscriptstyle
  j}^{})^{\mathrm{\bf{e}}_{\scalebox{.65}{$\scriptscriptstyle ij$}}}
(1 - z_{\scriptscriptstyle i}^{} z_{\!\scriptscriptstyle j}^{}) 
\, \cdot \, \prod_{1 \, \le \, k  \, \le \, l \, \le \, 3}
(1 - z_{\scriptscriptstyle k}^{} z_{\scriptscriptstyle l}^{})} 
{(z_{\scriptscriptstyle 4}^{\scriptscriptstyle D}  \cdots\, z_{r}^{\scriptscriptstyle D})
\, \mathrm{z}^{\scriptscriptstyle r}
\prod_{k = 1}^{3}\prod_{l = 4}^{r} 
(1 + z_{\scriptscriptstyle k}^{} z_{\scriptscriptstyle l}^{})   
(z_{\scriptscriptstyle k}^{\scriptscriptstyle -1} 
\! + z_{\scriptscriptstyle k}^{\scriptscriptstyle -2} z_{\scriptscriptstyle l}^{})   
\, \cdot \, \prod_{4 \, \le \, k  \, \le \, l \, \le \, r} 
(1 + z_{\scriptscriptstyle k}^{} z_{\scriptscriptstyle l}^{})} 
\]
around $z_{\scriptscriptstyle i} = 1,$ $i = 1, \ldots, r.$
Substituting $x_{\scriptscriptstyle i} = z_{\scriptscriptstyle i} - 1,$
and using the binomial expansion
\[
 (1 + x)^{\scriptscriptstyle - D}  = 1 - D x
 + \frac{D (D + 1)}{2} x^{\scriptscriptstyle 2}
 - \frac{D (D + 1)(D + 2)}{6} x^{\scriptscriptstyle 3} + \cdots
\]
it is clear that
$
\tilde{\mathscr{L}}_{\scriptscriptstyle 1, 2}(D, \zeta_{a})
$
is a polynomial in $D.$ To compute the highest power of $D$
in this polynomial, notice that every monomial in the expansion of
\[
\prod_{4 \, \le \, i  \, < \, j \, \le \, r} 
(z_{\scriptscriptstyle i}^{} - z_{\!\scriptscriptstyle
  j}^{})^{\scriptscriptstyle 2}
(1 - z_{\scriptscriptstyle i}^{} z_{\!\scriptscriptstyle j}^{})
= (-1)^{(r - {\scriptscriptstyle 3}) (r - {\scriptscriptstyle 4})\slash {\scriptscriptstyle 2}}
\; \cdot \prod_{4 \, \le \, i  \, < \, j \, \le \, r} 
(x_{\scriptscriptstyle i}^{} - x_{\!\scriptscriptstyle j}^{})^{\scriptscriptstyle 2}
(x_{\scriptscriptstyle i}^{} \, x_{\!\scriptscriptstyle j}^{}
+ x_{\scriptscriptstyle i}^{} + x_{\!\scriptscriptstyle j}^{})
\]
has degree at least $3 (r - 3) (r - 4)\slash 2.$
Thus the highest power of $D$ occurring in the polynomial is at most 
\[
(2 r - 1)(r - 3) - \frac{3 (r - 3) (r - 4)}{2}
= \frac{(r - 3)(r + 10)}{2} 
\]
and the total contribution to
$
D^{(r - {\scriptscriptstyle 3}) (r + {\scriptscriptstyle 10})\slash
  {\scriptscriptstyle 2}}
$
is given by 
\begin{equation*}
\begin{split}
\mathscr{G}_{\scriptscriptstyle 1, 2}\!\left({\bf 1}, \zeta_{a}\right)
\prod_{p} S_{\! p}^{\scriptscriptstyle \mathrm{reg}}\!\left({\bf 1}, \chi_{a}\right)
&\, \cdot 
\, \frac{1}{(2\pi \sqrt{-1})^{^{{\scriptscriptstyle 3}}}}\!
\oint \!\!\oint \!\!\oint \!\frac{(x_{\scriptscriptstyle 1} + 2)
  (x_{\scriptscriptstyle 2} + 2)
  (x_{\scriptscriptstyle 3} + 2)
  \prod_{1 \, \le \, i  \, < \, j \, \le \, 3}
  (x_{\scriptscriptstyle i} - x_{\!\scriptscriptstyle j})^{\scriptscriptstyle 2}
(x_{\scriptscriptstyle i}\, x_{\!\scriptscriptstyle j} +
x_{\scriptscriptstyle i}
+ x_{\!\scriptscriptstyle j})^{\scriptscriptstyle 2}}
{x_{\scriptscriptstyle 1}^{5}
  x_{\scriptscriptstyle 2}^{5}
  x_{\scriptscriptstyle 3}^{5}}\,
dx_{\scriptscriptstyle 1}
dx_{\scriptscriptstyle 2}
dx_{\scriptscriptstyle 3}\\
&\cdot \, \frac{(-1)^{\scriptscriptstyle r}\, 
2^{- (r - {\scriptscriptstyle 3}) (r + {\scriptscriptstyle 10})\slash
 {\scriptscriptstyle 2}}}{(2\pi \sqrt{-1})^{^{r {\scriptscriptstyle - 3}}}}
\!\oint  \!\cdots \!\oint  
\!\frac{\Delta(x_{\scriptscriptstyle 4}^{}, \ldots, x_{r}^{})
  \Delta(x_{\scriptscriptstyle 4}^{\scriptscriptstyle 2},
  \ldots, x_{r}^{\scriptscriptstyle 2})} 
{\prod_{i = 4}^{r} x_{\scriptscriptstyle i}^{{\scriptscriptstyle 2} r}}
      e^{\scalebox{1.}{$\scriptscriptstyle
 \scalebox{.85}{$\scriptscriptstyle -D$}\sum_{k = 4}^{r}$}\,
x_{\scalebox{.65}{$\scriptscriptstyle k$}}}
dx_{\scriptscriptstyle 4} \, \cdots \, dx_{r}.
\end{split}
\end{equation*}
Here
$
\Delta(y_{\scriptscriptstyle 1}, \ldots, y_{n})
$
is the Vandermonde determinant
\[
\Delta(y_{\scriptscriptstyle 1}, \ldots, y_{n})
= \mathrm{det}\,  (y_{\scriptscriptstyle i}^{\,\scriptscriptstyle j-1})_{1 \le i, j \le n} 
\;\; =\prod_{1 \le  i  <  j  \le n}
(y_{\!\scriptscriptstyle j} - y_{\scriptscriptstyle i}).
\]
The first integral equals $-48,$ and the second integral can be
computed as in \cite[Proposition~2.1]{GHRR}. Thus, by substituting
$
y_{\scriptscriptstyle k} = -D x_{\scriptscriptstyle k},
$
it follows, as in loc. cit., that
\begin{equation*}
  \begin{split}
& \frac{1}{(2\pi \sqrt{-1})^{^{r {\scriptscriptstyle - 3}}}}
\!\oint  \!\cdots \!\oint  
\!\frac{\Delta(x_{\scriptscriptstyle 4}^{}, \ldots, x_{r}^{})
  \Delta(x_{\scriptscriptstyle 4}^{\scriptscriptstyle 2},
  \ldots, x_{r}^{\scriptscriptstyle 2})} 
{\prod_{i = 4}^{r} x_{\scriptscriptstyle i}^{{\scriptscriptstyle 2} r}}
      e^{\scalebox{1.}{$\scriptscriptstyle
 \scalebox{.85}{$\scriptscriptstyle -D$}\sum_{k = 4}^{r}$}\,
x_{\scalebox{.65}{$\scriptscriptstyle k$}}}
dx_{\scriptscriptstyle 4} \, \cdots \, dx_{r} \\
& = (r - 3)!\,
\frac{0!\, 1!\, 2!\, \cdots \, (r - 4)!}
{7!\,  9!\,  11!\,  \cdots \, (2r - 1)!}
\,\mathrm{det}\,  \left(\binom{2r + 1 - 2j}{i - 1}\right)_{1 \le i, j \le r-3}
\cdot \; (- D)^{(r - {\scriptscriptstyle 3}) (r + {\scriptscriptstyle 10})\slash
  {\scriptscriptstyle 2}} 
\end{split}
\end{equation*}
where $\binom{a}{b}$ is the binomial coefficient. The determinant
equals
$
(-2)^{(r - {\scriptscriptstyle 3}) (r - {\scriptscriptstyle 4})\slash
  {\scriptscriptstyle 2}}\!,
$
which can be seen as follows. The first row is $1, 1, \ldots, 1,$ and
so, by replacing the $j$-th column ($j = 2, \ldots, r - 3$) by the
difference between the $(j-1)$-th and $j$-th columns, we are
reduced (after removing the first row and the first column)
to the computation of an $(r - 4) \times (r - 4)$ determinant
with the first row $2, 2, \ldots, 2.$ Applying the same procedure,
we are further reduced to the computation of an
$(r - 5) \times (r - 5)$ determinant with the first row
$4, 4, \ldots, 4.$ Continuing, it follows by applying the identity
\[
\binom{a + 2}{b} \, - \, \binom{a}{b}
\, = \, \binom{a + 1}{b - 1} \, + \, \binom{a}{b - 1}
\]
that at the $k$-th step, we get an $(r - k - 2) \times (r - k - 2)$
determinant with the first row
$
2^{\scriptscriptstyle k - 1}\!, 2^{\scriptscriptstyle k - 1}\!, \ldots,
2^{\scriptscriptstyle k - 1}\!.
$
Thus
\[
\mathrm{det}\,  \left(\binom{2r + 1 - 2j}{i - 1}\right)_{1 \le i, j \le r-3}
= \;\, \prod_{k = 1}^{r - 3} (- 1)^{r {\scriptscriptstyle - k - 3}}
2^{\scriptscriptstyle k - 1}
= (-2)^{(r - {\scriptscriptstyle 3}) (r - {\scriptscriptstyle 4})\slash
  {\scriptscriptstyle 2}}\!.
\]
From this, our last assertion follows immediately.
\end{proof}

Thus the leading coefficient of $Q_{\scriptscriptstyle 2}(D, q)$ is 
\begin{equation*}
\begin{split}
2^{{\scriptscriptstyle 19} - {\scriptscriptstyle 7}r}
\cdot \frac{0!\, 1!\, 2!\, \cdots \, (r - 4)!}
{7!\,  9!\,  11!\,  \cdots \, (2r - 1)!}
\!\sum_{\zeta_{a}^{\scalebox{.85}{$\scriptscriptstyle 2$}} \, =  \, \pm 1}
\!\zeta_{a}^{\scriptscriptstyle D}
\left\{(1, 1, 0)\,\overline{\mathrm{M}}_{w_{\!\scalebox{.85}{$\scriptscriptstyle \alpha'$}}}
\!\!\left(q^{\,\scriptscriptstyle -1\slash 2}\!,\,
q^{\,\scriptscriptstyle -1\slash 2}\!,\,
q^{\,\scriptscriptstyle -1\slash 2}\!,\,
q^{\scriptscriptstyle - 3\slash 4}\zeta_{a}^{\scriptscriptstyle -1}; q\right)
\,^{\scalebox{1.2}{$\scriptscriptstyle t$}}\!(1,  \mathrm{sgn}(a),  1)
L\big(\tfrac{1}{2}, \chi_{a}\big)^{\!\scriptscriptstyle 7}
\prod_{p} S_{\! p}^{\scriptscriptstyle \mathrm{reg}}\!\left({\bf 1}, \chi_{a}\right)\right\}
\end{split}
\end{equation*} 
where, by \eqref{eq: Global-cocycle-formula} and \eqref{eq: Functions-phi-psi}, 
\begin{equation*}
\begin{split}
& (1, 1, 0)\,\overline{\mathrm{M}}_{w_{\!\scalebox{.85}{$\scriptscriptstyle \alpha'$}}}
\!\!\left(q^{\,\scriptscriptstyle -1\slash 2}\!,\,
q^{\,\scriptscriptstyle -1\slash 2}\!,\,
q^{\,\scriptscriptstyle -1\slash 2}\!,\,
q^{\scriptscriptstyle - 3\slash 4}\zeta_{a}^{\scriptscriptstyle -1}; q\right)
\,^{\scalebox{1.2}{$\scriptscriptstyle t$}}\!(1,  \mathrm{sgn}(a),  1)\\
&= \begin{cases} 
1 \, +  \, q^{\,\scriptscriptstyle 1\slash 4} \, + \, 10\, q^{\,\scriptscriptstyle 1\slash 2} 
\, + \, 7 q^{\,\scriptscriptstyle 3\slash 4} \, + \, 20\, q \, + \, 7 q^{\,\scriptscriptstyle 5\slash 4} 
\, + \, 10\, q^{\,\scriptscriptstyle 3\slash 2} \, + \, q^{\,\scriptscriptstyle 7\slash 4} \, + \, 
q^{\scriptscriptstyle 2}  &\mbox{if $\zeta_{a} = \mathrm{sgn}(a) = 1$} \\
1 \, -  \, q^{\,\scriptscriptstyle 1\slash 4} \, + \, 10\, q^{\,\scriptscriptstyle 1\slash 2} 
\, - \, 7 q^{\,\scriptscriptstyle 3\slash 4} \, + \, 20\, q \, - \, 7 q^{\,\scriptscriptstyle 5\slash 4} 
\, + \, 10\, q^{\,\scriptscriptstyle 3\slash 2} \, - \, q^{\,\scriptscriptstyle 7\slash 4} \, + \, 
q^{\scriptscriptstyle 2} &\mbox{if $\zeta_{a} = - 1$ and $\mathrm{sgn}(a) = 1$} \\
1 \, -  \, i q^{\,\scriptscriptstyle 1\slash 4} \, - \, 4\, q^{\,\scriptscriptstyle 1\slash 2} 
\, + \, 7i q^{\,\scriptscriptstyle 3\slash 4} \, + \, 6\, q \, - \, 7i q^{\,\scriptscriptstyle 5\slash 4} 
\, - \, 4\, q^{\,\scriptscriptstyle 3\slash 2} \, + \, i q^{\,\scriptscriptstyle 7\slash 4} \, + \, 
q^{\scriptscriptstyle 2} &\mbox{if $\zeta_{a} = i$ and $\mathrm{sgn}(a) = - 1$}\\
1 \, +  \, i q^{\,\scriptscriptstyle 1\slash 4} \, - \, 4\, q^{\,\scriptscriptstyle 1\slash 2} 
\, - \, 7i q^{\,\scriptscriptstyle 3\slash 4} \, + \, 6\, q \, + \, 7i q^{\,\scriptscriptstyle 5\slash 4} 
\, - \, 4\, q^{\,\scriptscriptstyle 3\slash 2} \, - \, i q^{\,\scriptscriptstyle 7\slash 4} \, + \, 
q^{\scriptscriptstyle 2} &\mbox{if $\zeta_{a} = - i$ and $\mathrm{sgn}(a) = - 1,$}
\end{cases} 
\end{split}
\end{equation*} 
and, by \eqref{eq: Explicit-Lwp-n=2-r=3} -- \eqref{eq: Local-factor-residue-r=3}, 
\eqref{eq: Sp_w-alpha} and \eqref{eq: Sp-regularized}, 
$
S_{\! p}^{\scriptscriptstyle \mathrm{reg}}\!\left({\bf 1}, \chi_{a}\right)  
= P_{r}\left(\chi_{a}(p)\slash \sqrt{|p|}\right) 
$ 
with 
\begin{equation*} 
\begin{split} 
P_{r}(t) = &\, \left(1 -  t\right)^{\!(r^{\scalebox{.85}{$\scriptscriptstyle 2$}} + 7r - 14)\slash 2}    
\!\left(1 +  t \right)^{\!(r^{\scalebox{.85}{$\scriptscriptstyle 2$}}  + 7r - 28)\slash 2}\\
&\cdot \, \left[\left(t^{} +  t^{2} \right)
\left(t^{} + 6 t^{2} + t^{3}\right) \, + \, 
\tfrac{1}{2}\left(1 +  t^{} \right)^{4 - r}  \! + \, 
\tfrac{1}{2}\left(1  -  t^{}\right)^{- r} 
\left(1 + 10\, t^{} + 20\, t^{2} + 10\, t^{3} + t^{4} \right) \right].  
\end{split}
\end{equation*} 
Note that 
\[
P_{r}(t) =  1 - 14 (r - 2) t^{3} - 
\frac{r^{4} + 12 r^{3} + 59 r^{2} - 696 r + 1164}{12} t^{4} + O\big(t^{5}\big)
\] 
which implies the absolute convergence of the product 
$
\prod_{p} S_{\! p}^{\scriptscriptstyle \mathrm{reg}}\!\left({\bf 1}, \chi_{a}\right).
$

%


\end{document}